\documentclass{amsbook}
\usepackage{amsthm,amsmath,amsfonts,latexsym,amssymb,mathrsfs,color}

\newtheorem{theorem}{Theorem}[chapter]
\newtheorem{lemma}[theorem]{Lemma}
\newtheorem{proposition}[theorem]{Proposition}
\newtheorem{corollary}[theorem]{Corollary}

\newtheorem{problem}[theorem]{Problem}
\newtheorem{question}[theorem]{Question}

\theoremstyle{definition}
\newtheorem{definition}[theorem]{Definition}
\newtheorem{example}[theorem]{Example}

\theoremstyle{remark}
\newtheorem{remark}[theorem]{Remark}

\numberwithin{section}{chapter} \numberwithin{equation}{chapter}

\begin{document}
\frontmatter

\title[Matrix Functions of Bounded Type]
  {Matrix Functions of Bounded Type: An Interplay
  \\  Between Function Theory and Operator Theory}

\author{\Large Ra{\'u}l\ E.\ Curto}

\author{\Large In Sung Hwang}

\author{\Large Woo Young Lee}

\maketitle

%
%
%
%

\setcounter{page}{4}\tableofcontents

\chapter*{Abstract}

In this paper, we study matrix functions of bounded type from the
viewpoint of describing an interplay between function theory and
operator theory. \ We first establish a criterion on the
coprime-ness of two singular inner functions and obtain several
properties of the Douglas-Shapiro-Shields factorizations of matrix
functions of bounded type. \ We propose a new notion of
tensored-scalar singularity, and then answer questions on Hankel
operators with matrix-valued bounded type symbols. \ We also examine
an interpolation problem related to a certain functional equation on
matrix functions of bounded type; this can be seen as an extension
of the classical Hermite-Fej\' er Interpolation Problem for matrix
rational functions. \ We then extend the $H^\infty$-functional
calculus to an $\overline{H^\infty}+H^\infty$-functional calculus
for the compressions of the shift. \ Next, we consider the
subnormality of Toeplitz operators with matrix-valued bounded type
symbols and, in particular, the matrix-valued version of Halmos's
Problem 5; we then establish a matrix-valued version of Abrahamse's
Theorem. \ We also solve a subnormal Toeplitz completion problem of
$2\times 2$ partial block Toeplitz matrices. \ Further, we establish
a characterization of hyponormal Toeplitz pairs with matrix-valued
bounded type symbols, and then derive rank formulae for the
self-commutators of hyponormal Toeplitz pairs.


\renewcommand{\thefootnote}{}
\footnote{ 2010 \textit{Mathematics Subject Classification.} Primary
30J05, 30H10, 30H15, 47A13, 47A56, 47B20, 47B35; Secondary 30J10, 30J15, 30H35, 47A20, 47A57\\
\smallskip
\indent\textit{Key words.} Functions of bounded type, matrix
functions of bounded type, coprime inner functions,
Douglas-Shapiro-Shields factorizations, tensored-scalar singularity,
compressions of the shifts, $H^\infty$-functional calculus,
interpolation problems, Toeplitz operators, Hankel operators,
subnormal, hyponormal, jointly hyponormal, Halmos' Problem 5,
Abrahamse's Theorem, Toeplitz pairs.\\
\smallskip
\indent The work of the first named author was partially supported
by NSF Grant DMS-1302666. \ The work of the second named author was
supported by NRF(Korea) grant No. 2016R1A2B4012378. \
The work of the third named author was supported by NRF(Korea) grant No. 2015R1D1A1A01058036.\\
\smallskip
\indent
Affiliations:\\
\indent Ra{\'u}l\ E.\ Curto: Department of Mathematics, University
of Iowa, Iowa City, IA 52242, U.S.A.
email: raul-curto@uiowa.edu\\
\indent In Sung Hwang: Department of Mathematics, Sungkyunkwan
University, Suwon 440-746, Korea,
email: ihwang@skku.edu\\
\indent Woo Young Lee: Department of Mathematics, Seoul National
University, Seoul 151-742, Korea, email: wylee@snu.ac.kr

}

\bigskip

%
%
%
%
%
%

\mainmatter


\maketitle

%
%
%
%

\chapter{Introduction}

\medskip

A function $\varphi \in L^\infty$ is said to be of bounded type or
in the Nevanlinna class if $\varphi$ can be written as the quotient
of two functions in $H^\infty$. \ This class of functions has been
extensively studied in the literature. \ However, it seems to be
quite difficult to determine whether the given function $\varphi$ is
of bounded type if we only look at its Fourier series expansion. \
The well known criterion for detecting ``bounded type" employs
Hankel operators - the function $\varphi$ is of bounded type if and
only if the kernel of the Hankel operator $H_\varphi$ with symbol
$\varphi$ is nonzero (cf. \cite{Ab}). \ The class of functions of
bounded type plays an important role in the study of function theory
and operator theory. \ Indeed, for functions of bounded type, there
is a nice connection between function theory and operator theory. \
In 1970, P.R. Halmos posed the following problem, listed as Problem
5 in his lectures ``Ten problems in Hilbert space" \cite{Hal1},
\cite{Hal2}: Is every subnormal Toeplitz operator  either normal or
analytic\,? \ 

Many authors have given partial answers to Halmos'
Problem 5. \ In 1984, Halmos' Problem 5 was answered in the negative
by C. Cowen and  J. Long \cite{CoL} - they found a symbol not of
bounded type which is non-analytic and induces a non-normal
subnormal Toeplitz operator. \ To date, researchers have been
unable to characterize subnormal Toeplitz operators in terms of
their symbols. \ The most interesting partial answer to Halmos's
Problem 5 was given by M.B. Abrahamse \cite{Ab} - every subnormal
Toeplitz operator $T_\varphi$ whose symbol $\varphi$ is such that
$\varphi$ or $\overline\varphi$ is of bounded type is either normal
or analytic. \ Besides that, there are several fruitful interplays
between function theory and operator theory, for functions of
bounded type. \

In the present paper we explore matrix functions of bounded type in
$L^\infty_{M_n}$ (the space of $n\times n$ matrix-valued bounded
measurable functions on the unit circle); that is, matrix functions
whose entries are  of bounded type. \ We concentrate on the
connections between function theory and operator theory. \

For the function-theoretic aspects, we focus on coprime inner
functions, the Douglas-Shapiro-Shields factorization, tensored-scalar
singularity, an interpolation problem, and a functional calculus. \
First of all,  we consider the following intrinsic question: How
does one determine the coprime-ness of two inner functions in
$H^\infty$? \ This question is easy for Blaschke products. \ Thus
we are interested in the following question: When are two singular
inner functions coprime? \ Naturally, a measure-theoretic problem
arises at once since singular inner functions correspond to their
singular measures. \ We answer this question in Chapter 3. \ As we
may expect, we will show that two singular inner functions are
coprime if and only if the corresponding singular measures are
mutually singular (cf. Theorem \ref{thmcoprime3}). \ To prove this,
we use a notion of infimum of finite positive Borel measures on the
unit circle $\mathbb T$, and more generally, on a locally compact
Hausdorff space. \ The key point of this argument is to decide when
is the infimum of two finite positive Borel measures nonzero (cf.
Theorem \ref{thminfmeasure8}). \

To properly understand matrix functions of bounded
type, we need to factorize those functions into coprime products of
matrix inner functions and the adjoints of matrix
$H^\infty$-functions. \ In general, every matrix function of bounded
type can be represented by a left or a right coprime factorization -
the so-called Douglas-Shapiro-Shields factorization (cf. Remark
\ref{DSS}). \ In particular, this factorization is very helpful and
somewhat inevitable for the study of Hankel and Toeplitz operators
with such symbols. \ In Chapter 4, we consider several properties of
left or right coprime factorizations for matrix functions of bounded
type. \ First of all, we consider the following question: If $A$ is
a matrix $H^\infty$-function and $\Theta$ is a matrix inner
function, does it follow that $A$ and $\Theta$ are left coprime if
and only if $A$ and $\Theta$ are right coprime? \ In other words,
do the notions of ``left" coprime-ness and ``right" coprime-ness
coincide for $A$ and $\Theta$? \ The answer to this question is
negative in general. \ However, we can show that the answer is
affirmative if $\Theta$ is diagonal-constant, that is, $\Theta$ is a
diagonal inner function, constant along the diagonal (cf. Theorem
\ref{lem33.9}). \ We also show that if $\Phi$ is a matrix
$H^\infty$-function such that $\Phi^*$ is of bounded type and its
determinant is  not identically zero then the degree of the inner
part of its inner-outer factorization is less than or equal to the
degree of the inner part of its left coprime factorization (cf.
Corollary \ref{thmcodimgh}). \ In fact, the degree of the inner part
of the left coprime factorization is equal to the degree of the
inner part of the right coprime factorization (cf.  Lemma
\ref{lem2.9}). \ 

On the other hand, it is well known that the
composition of two inner functions is again an inner function. \ But
we cannot guarantee that the composition of two Blaschke products is
again a Blaschke product. \ Thus, we are interested in the question:
If $\theta$ and $\delta$ are coprime finite Blaschke products and
$\omega$ is an inner function, are the compositions $\theta\circ
\omega$ and $\delta\circ \omega$ coprime? \ We prove that the
answer is affirmative when the common zeros of $\theta$ and $\delta$
lie in some ``range set" of $\omega$ at its singularity - almost
everywhere in the open unit disk (cf. Theorem \ref{bcothm}). \
Moreover, we show that if $\Phi$ is a matrix rational function whose
determinant is not identically zero, then for a finite Blaschke
product $\omega$ the inner part of the right coprime factorization
of the composition $\Phi\circ \omega$ is exactly the composition of
the inner part of $\Phi$ and $\omega$ (cf. Theorem
\ref{factorization6879}); this is still true for matrix functions of
bounded type if $\omega$ is a Blaschke factor (cf. Theorem
\ref{lem5.17}). \

We next ask the question: How does one define a singularity for
matrix functions of bounded type? \ Conventionally, the singularity
(or the existence of a pole) of matrix $L^\infty$-functions is
defined by a singularity (or a pole) of some entry of the matrix
function (cf. \cite{BGR}, \cite{BR}). \ However, we propose another
notion of singularity which is more suitable for our study of Hankel
and Toeplitz operators. \ In Chapter 5, we give a new notion, that of
``tensored-scalar singularity." \ This new definition uses the
Hankel operator as a characterization of functions of bounded type
via the kernel of the Hankel operator. \ This notion provides an
answer to the question: Under which conditions does it follow that
if the product of two Hankel operators with matrix-valued bounded
type symbols is zero then either of them is zero?  \ It is well
known that the answer to this question for scalar-valued cases is
affirmative, but is negative for matrix-valued cases unless certain
assumptions are made about the symbols. \ Here we show that if
either of the two symbols has a tensored-scalar singularity then the
answer is affirmative (cf. Theorem \ref{lem33.5}). \ We also
consider the question: If $\Phi$ and $\Psi$ are matrix functions of
bounded type, when is $H_\Phi^* H_\Phi=H_\Psi^* H_\Psi$ (where
$H_\Phi, H_\Psi$ are Hankel operators)? \ We answer this question as follows: If $\Phi$ or $\Psi$ has a tensored-scalar singularity
then two products are equal only when the co-analytic parts of
$\Phi$ and $\Psi$ coincide up to a unitary constant left factor (cf.
Theorem \ref{lem3.13}). \

As transitional aspects from function theory to operator theory, in
Chapter 6 we consider an interpolation problem for matrix functions
of bounded type and a functional calculus for the compressions of
the shift operator.  \ We consider an interpolation problem
involving the following matrix-valued functional equation: When is  
$\Phi-K\Phi^*$ a matrix $H^\infty$-function (where $\Phi$ is a
matrix $L^\infty$-function and $K$ is an unknown matrix
$H^\infty$-function)? \ In other words, when does there exist a matrix $H^\infty$-function $K$ such that $\Phi-K\Phi^*$ is a matrix $H^\infty$-function? \ If $\Phi$ is a matrix-valued
{\it rational} function, this interpolation problem reduces to the
classical Hermite-Fej\' er Interpolation Problem. \ We here examine
this interpolation problem for the matrix functions of bounded type
(cf. Theorems \ref{cor32.378} and \ref{thm32.378}). \ On the other
hand, it is well known that the functional calculus for polynomials
of the compressions of the shift results in the  Hermite-Fej\' er
matrix via  the classical Hermite-Fej\' er Interpolation Problem. \
We also extend the polynomial calculus to the $H^\infty$-functional
calculus (the so-called Sz.-Nagy-Foia\c s functional calculus) via
the Triangularization Theorem, and then extend it to a
$\overline{H^\infty}+H^\infty$-functional calculus for the
compressions of the shift. \

Chapters 7 - 9 are devoted to operator-theoretic aspects. \ In
Chapter 7 we consider the subnormality of Toeplitz operators with
matrix-valued symbols and, in particular, the matrix-valued version
of Halmos's Problem 5 (\cite{Hal1}, \cite{Hal2}): Which subnormal
Toeplitz operators with matrix-valued symbols are either normal or
analytic\,? \ In 1976, M.B. Abrahamse \cite{Ab} showed that if
$\varphi\in L^\infty$ is such that $\varphi$ or $\overline \varphi$
is of bounded type, if $T_\varphi$ is hyponormal, and if the kernel
of the self-commutator of $T_\varphi$ is invariant under $T_\varphi$
then $T_\varphi$ is either normal or analytic. \ The aim of this
chapter is to establish a matrix-valued version of Abrahamse's
Theorem. \ In fact, a straightforward matrix-valued version of
Abrahamse's Theorem may fail. \ Recently, it was shown in
\cite{CHL1} that if $\Phi$ and $\Phi^*$ are matrix functions of
bounded type with the constraint that the inner part of the right
coprime factorization of the co-analytic part $\Phi_-$ is
diagonal-constant, then a matrix-valued version of Abrahamse's
Theorem holds for $T_\Phi$. \ Also, it was shown in \cite{CHKL} that
if $\Phi$ is a matrix-valued {\it rational} function then the above
``diagonal-constant" condition can be weakened to the condition of
``having a nonconstant diagonal-constant inner divisor." \ From the results of Chapter 5 (cf.~Lemma \ref{lem55.2-1}), we
can see that those conditions of \cite{CHL1} and \cite{CHKL} are
special cases of the condition of ``having a tensored-scalar
singularity." \ Indeed, in Chapter 7, we will show that if the
symbol has a tensored-scalar singularity then we get a full-fledged
matrix-valued version of Abrahamse's Theorem (cf.~Theorem
\ref{thm2.2}). \ In particular, if the symbol is scalar-valued then
it vacuously has a tensored-scalar singularity, so that the
matrix-valued version reduces to the original Abrahamse's Theorem. \

Given a partially specified operator matrix with some known entries,
the problem of finding suitable operators to complete the given
partial operator matrix so that the resulting matrix satisfies
certain given properties is called a {\it completion problem}. \ A
{subnormal completion} of a partial operator matrix is a particular
specification of the unspecified entries resulting in a subnormal
operator. \ In Chapter 8,  we solve the following ``subnormal
Toeplitz completion" problem: find the unspecified Toeplitz entries
of the partial block Toeplitz matrix
$$
A:=\begin{pmatrix} T_{\overline b_\alpha} & ? \\ ?& T_{\overline
b_\beta}\end{pmatrix}\quad\hbox{ ($\alpha,\beta\in\mathbb D$)}
$$
so that $A$ becomes subnormal, where $b_\lambda$ is a Blaschke
factor whose zero is $\lambda$. \ We can here show that the
unspecified entries $?$ have symbols that are matrix functions of
bounded type. \ Thus this problem reduces to a problem on the
subnormality of ``bounded type" Toeplitz operators. \ Recently, in
\cite{CHL2}, we have considered this completion problem for the
cases $\alpha=\beta=0$. \ The solution given in \cite[Theorem
5.1]{CHL2} relies upon very intricate and long computations using
the symbol involved. \ However our solution in this chapter provides
a shorter and more insightful proof by employing the results of the
previous chapter. \ Our solution also shows that 2-hyponormality,
subnormality and normality coincide for this completion, except in a
special case (cf. Theorem \ref{thm4.2}). \

On the other hand, normal Toeplitz operators were characterized by a
property of their symbols in the early 1960's by A. Brown and P.R.
Halmos \cite{BH}. \ The exact nature of the relationship between the
symbol $\varphi\in L^\infty$ and the hyponormality of the Toeplitz
operator $T_{\varphi}$ was understood in 1988 via Cowen's Theorem
\cite{Co2} -- this elegant and useful theorem has been used in the
works \cite{CuL1}, \cite{CuL2}, \cite{FL}, \cite{Gu1}, \cite{Gu2},
\cite{GS}, \cite{HKL1}, \cite{HKL2}, \cite{HL1}, \cite{HL2},
\cite{HL3}, \cite{Le}, \cite{NT}, \cite{Zhu}, and others; these works have been
devoted to the study of hyponormality for Toeplitz operators on
$H^2$. \  Particular attention has been paid to Toeplitz operators
with polynomial symbols or rational symbols \cite{HL1}, \cite{HL2},
\cite{HL3}. \ However, the case of arbitrary symbol $\varphi$,
though solved in principle by Cowen's theorem, is in practice very
complicated. \ Indeed, it may not even be possible to find tractable
necessary and sufficient condition for the hyponormality of
$T_\varphi$ in terms of the Fourier coefficients of the symbol
$\varphi$ unless certain assumptions are made about it. \  To date,
tractable criteria for the cases of trigonometric polynomial symbols
and rational symbols have been derived from a Carath\' eodory-Schur
interpolation problem \cite{Zhu} and a tangential Hermite-Fej\' er
interpolation problem \cite{Gu1} or the classical Hermite-Fej\' er
interpolation problem \cite{HL3}. \ Recently, C. Gu, J. Hendricks
and D. Rutherford \cite{GHR} have considered the hyponormality of
Toeplitz operators with matrix-valued symbols and characterized it
in terms of their symbols. \  In particular they showed that if
$T_\Phi$ is a hyponormal Toeplitz operator with matrix-valued symbol
$\Phi$, then its symbol $\Phi$ is normal, i.e.,
$\Phi^*\Phi=\Phi\Phi^*$. \  Their characterization resembles Cowen's
Theorem except for an additional condition -- the normality of the
symbol. \

In 1988, two important developments took place in the field of
operator theory. The first one was the introduction of the notion of
``joint hyponormality" for $n$-tuples of operators and the second
one was the characterization of hyponormality of Toeplitz operators
in terms of their symbols (via Cowen's Theorem) as we have remarked
just above. \ Since then, it has become natural to consider joint
hyponormality for tuples of Toeplitz operators. \ The notion of
joint hyponormality was first formally introduced by A. Athavale
\cite{At}. \ He conceived joint hyponormality as a notion at least
as strong as requiring that the linear span of the operator
coordinates consist of hyponormal operators, the latter notion being
called weak joint hyponormality. \  Joint hyponormality and weak
joint hyponormality have been studied by A. Athavale \cite{At}, J.
Conway and W. Szymanski \cite{CS}, R. Curto \cite{Cu}, R. Curto and
W.Y. Lee \cite{CuL1}, R. Curto, P. Muhly, and J. Xia \cite{CMX}, R.
Douglas, V. Paulsen and K. Yan \cite{DPY}, R. Douglas and K. Yan
\cite{DY}, D. Farenick and R. McEachin \cite{FM}, C. Gu \cite{Gu2},
S. McCullough and V. Paulsen \cite{McCP1},\cite{McCP2}, D. Xia
\cite{Xi}, and others. \  Joint hyponormality originated from
questions about commuting normal extensions of commuting operators,
and it has also been considered with an aim at understanding the gap
between hyponormality and subnormality for single operators. \ The
study of jointly hyponormal Toeplitz tuples started with D. \
Farenick and R. McEachin \cite{FM}. \  They studied operators that
form jointly hyponormal pairs in the presence of the unilateral
shift: precisely, they showed that if $U$ is the unilateral shift on
the Hardy space $H^2$, then the joint hyponormality of the pair
$(U,T)$ implies that $T$ is necessarily a Toeplitz operator. \ This
result invites us to consider the joint hyponormality for pairs of
Toeplitz operators. \ In Chapter 9, we consider (jointly) hyponormal
Toeplitz pairs with matrix-valued bounded type symbols. \ In their
research monograph \cite{CuL1}, the authors studied hyponormality of
pairs of Toeplitz operators (called Toeplitz pairs) when both
symbols are trigonometric polynomials. \ The core of the main result
of \cite{CuL1} is that the  hyponormality of $\mathbf{T}\equiv
(T_\varphi, T_\psi)$ ($\varphi, \psi$ trigonometric polynomials)
forces that the co-analytic parts of $\varphi$ and $\psi$
necessarily coincide up to a constant multiple, i.e.,
\begin{equation}\label{5.1}
\varphi-\beta\psi\in H^2\ \ \hbox{for some $\beta\in\mathbb{C}$}.
\end{equation}
In \cite{HL4}, (\ref{5.1*}) was extended for Toeplitz pairs whose symbols are
rational functions with some constraint. \
As a result, the following question arises at once: Does (\ref{5.1})
still hold for Toeplitz pairs whose symbols are {\it matrix-valued}
trigonometric polynomials or rational functions? \

Chapter 9 is concerned with this question. \ More generally, we give
a characterization of hyponormal
 Toeplitz pairs with bounded type symbols  by using the theory established
in the previous chapters. \ Consequently, we will answer the above
question (cf. Corollary \ref{cor5.23}). \ Indeed, (\ref{5.1}) is
still true for matrix-valued trigonometric polynomials under some
invertibility and commutativity assumptions on the Fourier
coefficients of the symbols (those assumptions always hold vacuously
for scalar-valued cases). \ In fact, this follows from our core idea
(Lemma \ref{thm5.19}) that if $\Phi$ and $\Psi$ are matrix functions
of bounded type whose inner parts of right coprime factorizations of
analytic parts commute and whose co-analytic parts have a common
tensored-scalar pole then the hyponormality of the Toeplitz pair
$(T_\Phi, T_\Psi)$ implies that the common tensored-scalar pole has
the same order. \ Consequently, we give a characterization of the
(joint) hyponormality of Toeplitz pairs with bounded type symbols
(cf. Theorem \ref{thm5.22-1}). \ We also consider the
self-commutators of the Toeplitz pairs with matrix-valued rational
symbols and derive rank formulae for them (cf. Theorem
\ref{thm5.34}). \

Chapter 10 is devoted to concluding remarks and open questions.

\bigskip

{\it Acknowledgments.} The authors are deeply grateful to the
referee for many helpful suggestions, which significantly improved both the
content and the presentation. \ In particular, we are indebted to the referee for supplying the present version of the proof
of Lemma \ref{lem1.4}, which is more transparent than the original proof.

%
%
%
%
%

\chapter{Preliminaries}

The main ingredients of this paper are functions of bounded type,
Hankel operators, Toeplitz operators, and hyponormality. \ First of
all, we review the notion of functions of bounded type and a few
essential facts about Hankel and Toeplitz operators, and for that we
will use \cite{Ab}, \cite{BS}, \cite{Do1}, \cite{GGK}, \cite{MAR},
\cite{Ni2} and \cite{Pe}. \ Let $\mathcal{H}$ and $\mathcal{K}$ be
complex Hilbert spaces, let $\mathcal{B(H,K)}\label{B(H,K)}$ be the
set of bounded linear operators from $\mathcal{H}$ to $\mathcal{K}$,
and write $\mathcal{B(H)}\label{B(H)}:=\mathcal{B(H,H)}$. \ For
$A,B\in\mathcal{B(H)}$, we let $[A,B]\label{[A,B]}:=AB-BA$. \  An
operator $T\in\mathcal{B(H)}$ is said to be normal if
$[T^*,T]\label{[T^*,T]}=0$, hyponormal if $[T^*,T]\ge 0$. \ For an
operator $T\in \mathcal{B(H)}$, we write $\hbox{ker}\,T\label{ker}$
and $\hbox{ran}\, T\label{ran}$ for the kernel and the range of $T$,
respectively. \ For a subset $\mathcal M$ of a Hilbert space
$\mathcal H$, $\hbox{cl}\, \mathcal M\label{clM}$ and $\mathcal
M^\perp\label{Mperp}$ denote the closure and the orthogonal
complement of $\mathcal M$, respectively. \ Also, let
$\mathbb{T}\label{T}\equiv \partial\mathbb{D}$ be the unit circle
(where $\mathbb D$ denotes the open unit disk in the complex plane
$\mathbb C$). \ Recall that $L^\infty\equiv L^\infty(\mathbb T)$ is
the set of bounded measurable functions on $\mathbb T$, that the
Hilbert space $L^2\equiv L^2({\mathbb T})\label{L^2T}$ has a
canonical orthonormal basis given by the trigonometric functions
$e_n(z)=z^n$, for all $n\in {\mathbb Z}$, and that the Hardy space
$H^2\equiv H^2({\mathbb T})\label{H^2T}$ is the closed linear span
of $\{e_n: n \geq 0 \}$. \  An element $f\in L^2$ is said to be
analytic if $f\in H^2$. \ Let $H^\infty\label{H^infty} :=L^\infty
\cap H^2$, i.e., $H^\infty$ is the set of bounded analytic functions
on $\mathbb{D}$. \ Given a function $\varphi\in
L^\infty\label{L^infty}$, the Toeplitz operator
$T_\varphi\label{Tvarphi}$ and
 the Hankel operator $H_\varphi\label{Hvarphi}$ with symbol $\varphi$
on $H^2$ are defined by
\begin{equation}\label{1.000}
T_\varphi g:=P(\varphi g) \quad\hbox{and}\quad H_\varphi
g:=JP^\perp(\varphi g) \qquad (g\in H^2),
\end{equation}
where $P\label{P}$ and $P^\perp\label{Pperp}$ denote the orthogonal
projections from $L^2$ onto $H^2$ and $(H^2)^\perp$, respectively,
and $J$ denotes the unitary operator from $L^2$ onto $L^2$ defined
by $J\label{J}(f)(z)=\overline z f(\overline z)$ for $f\in L^2$. \

For $\varphi\in L^\infty$, we write
$$
\varphi_+\equiv P \varphi\in H^2\quad\text{and}\quad \varphi_-\equiv
\overline{P^\perp \varphi}\in zH^2.
$$
Thus we may write $\varphi=\overline{\varphi_-}+\varphi_+$. \

We recall that a function $\varphi\in L^\infty$ is said to be of
{\it bounded type} (or in the Nevanlinna class $\mathcal N$) if
there are functions $\psi_1,\psi_2\in H^\infty$  such that
$$
\varphi(z)=\frac{\psi_1(z)}{\psi_2(z)}\quad\hbox{for almost all
$z\in \mathbb{T}$.}
$$
Let $BMO\label{BMO}$ denote the set of functions of bounded mean
oscillation in $L^1$. \ It is well-known that $L^\infty\subseteq
BMO\subseteq L^2$ and that if $f\in L^2$, then $H_f$ is bounded on
$H^2$ if and only if $P^\perp f\in BMO$ (where $P^{\perp}$ is the
orthogonal projection of $L^2$ onto $(H^2)^{\perp}$) (cf.
\cite{Pe}). \ Thus if $\varphi\in L^\infty$, then
$\overline{\varphi_-}, \ \overline{\varphi_+}\in BMO$, so that
$H_{\overline{\varphi_-}}$ and $H_{\overline{\varphi_+}}$ are well
understood. \ We recall \cite[Lemma 3]{Ab} that if $\varphi\in
L^\infty$ then
\begin{equation}\label{1.1}
\hbox{$\varphi$ is of bounded type}\ \Longleftrightarrow\
\hbox{ker}\, H_\varphi\ne \{0\}\,.
\end{equation}
Assume now that both $\varphi$ and $\overline\varphi$ are of bounded
type. \  Then from the Beurling's Theorem, $\text{ker}\,
H_{\overline{\varphi_-}}=\theta_0 H^2$ and $\text{ker}\,
H_{\overline{\varphi_+}}=\theta_+ H^2$ for some inner functions
$\theta_0, \theta_+$. \ We thus have
$b:={\overline{\varphi_-}}\theta_0 \in H^2$, and hence we can write
\begin{equation}\label{2.2}
\varphi_-=\theta_0\overline{b}  \ \ \text{~and similarly~} \varphi_+
=\theta_+\overline{a} \text{~for some~} a \in H^2.
\end{equation}
By Kronecker's Lemma \cite[p. 217]{Ni2}, if $f\in H^\infty$ then
$\overline f$ is a rational function if and only if $\hbox{rank}\,
H_{\overline f}<\infty$, which implies that
\begin{equation}\label{2.4}
\hbox{$\overline{f}$ is rational} \ \Longleftrightarrow\
f=\theta\overline b\ \ \hbox{with a finite Blaschke product
$\theta$}.
\end{equation}
If $T_\varphi$ is hyponormal then since $T_{\varphi\psi}-T_\varphi
T_\psi=H_{\overline\varphi}^*H_\psi$ ($\varphi, \psi\in L^\infty$)
and hence,
\begin{equation}\label{2.3}
[T_\varphi^*, T_\varphi]=H_{\overline\varphi}^*
H_{\overline\varphi}-H_\varphi^* H_\varphi=
H_{\overline{\varphi_+}}^*
H_{\overline{\varphi_+}}-H_{\overline{\varphi_-}}^*
H_{\overline{\varphi_-}},
\end{equation}
it follows that $||H_{\overline{\varphi_+}} f||\ge
||H_{\overline{\varphi_-}} f||$ for all $f\in H^2$, and hence
$$
\theta_+ H^2= \text{ker}\, H_{\overline{\varphi_+}}\subseteq
\text{ker}\, H_{\overline{\varphi_-}}=\theta_0 H^2,
$$
which implies that $\theta_0$ divides $\theta_+$, i.e.,
$\theta_+=\theta_0\theta_1$ for some inner function $\theta_1$. \
For an inner function $\theta$, we write
$$
\mathcal H(\theta):=H^2\ominus \theta H^2.
$$
Note that if $f=\theta \overline a \in L^2$, then $f\in H^2$ if and
only if $a\in \mathcal H(z\theta)$; in particular, if $f(0)=0$ then
$a\in \mathcal H(\theta)$. \  Thus, if
$\varphi\equiv\overline{\varphi_-}+\varphi_+\in L^\infty$ is such
that $\varphi$ and $\overline\varphi$ are of bounded type such that
$\varphi_+(0)=0$ and $T_\varphi$ is hyponormal, then we can write
$$
\varphi_+=\theta_0\theta_1\bar a\quad\text{and}\quad
\varphi_-=\theta_0 \bar b, \qquad\text{where $a\in
\mathcal{H}(\theta_0\theta_1)$ and $b\in\mathcal{H}(\theta_0)$.}
$$

We turn our attention to the case of matrix functions. \

Let $M_{n\times r}\label{M_nr}$ denote the set of all $n\times r$
complex matrices and write $M_n\label{M_n}:=M_{n\times n}$. \  For
$\mathcal X$ a Hilbert space, let $L^2_{\mathcal X}\label{L2X}\equiv
L^2_{\mathcal X}(\mathbb T)$ be the Hilbert space of $\mathcal
X$-valued norm square-integrable measurable functions on
$\mathbb{T}$ and let $L^{\infty}_{\mathcal X}\label{LiX}\equiv
L^{\infty}_{\mathcal X}(\mathbb T)$ be the set of $\mathcal
X$-valued bounded measurable functions on $\mathbb{T}$. \ We also
let $H^2_{\mathcal X}\label{H2X}\equiv H^2_{\mathcal X}(\mathbb T)$
be the corresponding Hardy space and $H^{\infty}_{\mathcal
X}\label{HiX}\equiv H^{\infty}_{\mathcal X}(\mathbb T)
=L^\infty_{\mathcal X}\cap H^2_{\mathcal X}$. \  We observe that
$L^2_{\mathbb{C}^n}= L^2\otimes \mathbb{C}^n$ and
$H^2_{\mathbb{C}^n}= H^2\otimes \mathbb{C}^n$. \

\begin{definition}
For a matrix-valued function $\Phi \equiv \begin{pmatrix}
\varphi_{ij}\end{pmatrix}\in L^\infty_{M_{n}}$, we say that $\Phi$
is of {\it bounded type} if each entry $\varphi_{ij}$ is of bounded
type, and we say that $\Phi$ is {\it rational} if each entry
$\varphi_{ij}$ is a rational function. \
\end{definition}

\medskip

Let $\Phi\equiv \begin{pmatrix} \varphi_{ij}\end{pmatrix} \in
L^\infty_{M_n}$ be such that $\Phi^*$ is of bounded type. \ Then
each $\overline\varphi_{ij}$ is of bounded type. \ Thus in view of
(\ref{2.2}), we may write
$\varphi_{ij}=\theta_{ij}\overline{b}_{ij}$, where $\theta_{ij}$ is
inner and $\theta_{ij}$ and $b_{ij}$ are coprime, in other words,
there does not exist a nonconstant inner divisor of $\theta_{ij}$
and $b_{ij}$. \ Thus if $\theta$ is the least common multiple of
$\{\theta_{ij}:i,j=1,2, \cdots, n \}$, then we may write
\begin{equation}\label{2.6}
\Phi=\begin{pmatrix}\varphi_{ij}\end{pmatrix}=
\begin{pmatrix}\theta_{ij}\overline{b}_{ij}\end{pmatrix} =
\begin{pmatrix}\theta \overline{a}_{ij} \end{pmatrix} \equiv \theta
A^* \quad \hbox{(where $A\equiv \begin{pmatrix}a_{ji} \end{pmatrix}
\in H^{2}_{M_n}$)}.
\end{equation}
In particular, $A(\alpha)$ is nonzero whenever $\theta(\alpha)=0$
and $|\alpha|<1$. \

\medskip

For $\Phi\in L^\infty_{M_n}$, we write
$$
\Phi_+\label{Phi_+}:=P_n(\Phi)\in H^2_{M_n} \quad\hbox{and}\quad
\Phi_-\label{Phi_-}:=\bigl[P_n^\perp (\Phi)\bigr]^* \in H^2_{M_n}.
$$
Thus we may write $\Phi=\Phi_-^*+\Phi_+\,$. \ However, it will often
be convenient to allow the constant term in $\Phi_-$. \ Hence, if
there is no confusion we may assume that $\Phi_-$ shares the
constant term with $\Phi_+$: in this case, $\Phi(0) = \Phi_+(0) +
\Phi_-(0)^*$. \ If $\Phi=\Phi_-^*+\Phi_+\in L^\infty_{M_n}$ is such
that $\Phi$ and $\Phi^*$ are of bounded type, then in view of
(\ref{2.6}), we may write
\begin{equation}\label{2.6-1}
\Phi_+= \theta_1 A^* \quad\hbox{and}\quad \Phi_-= \theta_2 B^*,
\end{equation}
where $\theta_1$ and $\theta_2$ are inner functions and $A,B\in
H^{2}_{M_n}$. \  In particular, if $\Phi\in L^\infty_{M_n}$ is
rational then the $\theta_i$ can be chosen as finite Blaschke
products, as we observed in (\ref{2.4}). \

\medskip

We now introduce the notion of Hankel operators and Toeplitz
operators with matrix-valued symbols. \ If $\Phi$ is a matrix-valued
function in $L^\infty_{M_n}$, then $T_\Phi: H^2_{\mathbb{C}^n}\to
H^2_{\mathbb{C}^n}$ denotes Toeplitz operator with symbol $\Phi$
defined by
$$
T_\Phi\label{T_Phi} f:=P_n\label{P_n}(\Phi f)\quad \hbox{for}\ f\in
H^2_{\mathbb{C}^n},
$$
where $P_n$ is the orthogonal projection of $L^2_{\mathbb{C}^n}$
onto $H^2_{\mathbb{C}^n}$. \  A Hankel operator with symbol $\Phi\in
L^\infty_{M_n}$ is an operator $H_\Phi: H^2_{\mathbb{C}^n}\to
H^2_{\mathbb{C}^n}$ defined by
$$
H_\Phi\label{H_Phi} f := J_n P_n^\perp (\Phi f)\quad \hbox{for}\
f\in H^2_{\mathbb{C}^n},
$$
where $P_n^\perp$ is the orthogonal projection of
$L^2_{\mathbb{C}^n}$ onto $(H^2_{\mathbb{C}^n})^\perp$ and $J_n$
denotes the unitary operator from $L^2_{\mathbb{C}^n}$ onto
$L^2_{\mathbb{C}^n}$ given by $J_n(f)(z):= \overline{z}
f(\overline{z})$ for $f \in L^2_{\mathbb{C}^n}$.  \ For $\Phi\in
L^\infty_{M_{n\times m}}$, write
$$
\widetilde\Phi\label{widetildePhi} (z):=\Phi^*(\overline z). \
$$
A matrix-valued function $\Theta\label{Theta}\in
H^\infty_{M_{n\times m}}$ is called {\it inner} if
$\Theta^*\Theta=I_m$ almost everywhere on $\mathbb{T}$, where $I_m$
denotes the $m\times m$ identity matrix. \ If there is no confusion
we write simply $I$ for $I_m$. \ The following basic relations can
be easily derived:
\begin{align}
&T_\Phi^*=T_{\Phi^*},\ \  H_\Phi^*= H_{\widetilde \Phi} \quad
(\Phi\in
L^\infty_{M_n});\label{1.2}\\
&T_{\Phi\Psi}-T_\Phi T_\Psi = H_{\Phi^*}^*H_\Psi \quad
(\Phi,\Psi\in L^\infty_{M_n});\label{1.3}\\
&H_\Phi T_\Psi = H_{\Phi\Psi},\ \
H_{\Psi\Phi}=T_{\widetilde{\Psi}}^*H_\Phi \quad (\Phi\in
L^\infty_{M_n}, \Psi\in H^\infty_{M_n});\label{1.4}\\
&H_\Phi^* H_\Phi - H_{\Theta \Phi}^* H_{\Theta\Phi} =H_\Phi^*
H_{\Theta^*}H_{\Theta^*}^*H_\Phi \quad (\Theta\in H^\infty_{M_n}
 \ \hbox{inner,}  \ \Phi\in L^\infty_{M_n}).\label{1.5}
\end{align}
A matrix-valued trigonometric polynomial $\Phi \in
L^{\infty}_{M_{n\times m}}$ is of the form
\begin{equation*}
\Phi(z)=\sum_{j=-m}^{N} A_j z^j \quad (A_j \in M_{n\times m}),
\label{outer}
\end{equation*}
where $A_N$ and $A_{-m}$ are called the \textit{outer} coefficients
of $\Phi$. \ For matrix-valued functions $A:=\sum_{j=-\infty}^\infty
A_j z^j\in L^2_{M_{n\times m}}$ and $B:=\sum_{j=-\infty}^\infty B_j
z^j\in L^2_{M_{n\times m}}$, we define the inner product of $A$ and
$B$ by
$$
\langle A,B\rangle \label{AB}:=\int_{\mathbb T} \hbox{tr}\,(B^*A)\,
d\mu=\sum_{j=-\infty}^\infty \hbox{tr}\,(B_j^*A_j)\,,
$$
where $\hbox{tr}\,(\cdot)\label{tr}$ denotes the trace of a matrix
and define $||A||_2\label{A_2}:=\langle A,A\rangle^{\frac{1}{2}}$. \
We also define, for $A\in L^\infty_{M_{n\times m}}$,
$$
||A||_\infty\label{A_infty}:=\hbox{ess sup}_{t\in\mathbb T}
||A(t)||\quad \hbox{($||\cdot||$ denotes the spectral norm of a
matrix)}. \
$$
Finally, the {\it shift} operator $S$ on $H^2_{\mathbb C^n}$ is
defined by $S\label{S}:=T_{zI}$. \

\medskip

The following fundamental result will be useful in the sequel. \

\medskip

\noindent {\bf The Beurling-Lax-Halmos Theorem.}\label{beur}
(\cite{FF}, \cite{Ni2}) \ \ {\it A nonzero subspace $M$ of
$H^2_{\mathbb C^n}$ is invariant for the shift operator $S$ on
$H^2_{\mathbb C^n}$ if and only if $M=\Theta H^2_{\mathbb C^m}$,
where $\Theta$ is an inner matrix function in
$H^{\infty}_{M_{n\times m}}\label{H^inftyMnm}$ \hbox{\rm ($m\le
n$)}. \  Furthermore, $\Theta$ is unique up to a unitary constant
right factor; that is, if $M=\Delta H^2_{\mathbb{C}^r}$ {\rm (}for
$\Delta$ an inner function in $H^{\infty}_{M_{n\times r}}${\rm )},
then $m=r$ and $\Theta=\Delta W$, where $W$ is a unitary matrix
mapping $\mathbb C^m$ onto $\mathbb C^m$. \  }

\bigskip

As customarily done, we say that two matrix-valued functions $A$ and
$B$ are {\it equal} if they are equal up to a unitary constant right
factor. \  Observe that, by (\ref{1.4}), for $\Phi\in
L^\infty_{M_n}$,
$$
H_\Phi S=H_{\Phi\cdot zI} =H_{zI\cdot\Phi}=S^*H_\Phi,
$$
which implies that the kernel of a Hankel operator $H_\Phi$ is an
invariant subspace of the shift operator on $H^2_{\mathbb C^n}$. \
Thus, if $\hbox{ker}\,{H_{\Phi}} \ne \{0\}$, then by the
Beurling-Lax-Halmos Theorem,
$$
\hbox{ker}\, H_\Phi=\Theta H^2_{\mathbb{C}^m}
$$
for some inner matrix function $\Theta$. \  We note that $\Theta$
need not be a square matrix.\label{square} \

\medskip

However, we have:

\begin{lemma}\label{gu2} \cite[Theorem 2.2]{GHR}\
For $\Phi\in L^\infty_{M_n}$, the following are equivalent:
\begin{itemize}
\item[(a)] $\Phi$ is of bounded type;
\item[(b)] $\ker\, H_\Phi=\Theta H^2_{\mathbb{C}^n}$ for
some square inner matrix function $\Theta$.
\end{itemize}
\end{lemma}

\medskip



We recall that an operator $T\in\mathcal{B(H)}$ is said to be
subnormal if $T$ has a normal extension, i.e.,
$T=N\vert_{\mathcal{H}}$, where $N$ is a normal operator on some
Hilbert space $\mathcal{K}\supseteq \mathcal{H}$ such that $\mathcal
H$ is invariant for $N$. \ The Bram-Halmos criterion for
subnormality (\cite{Br}, \cite{Con}) states that an operator $T\in
\mathcal{B(H)}$ is subnormal if and only if $\sum_{i,j}(T^ix_j, T^j
x_i)\ge 0$ for all finite collections
$x_0,x_1,\cdots,x_k\in\mathcal{H}$. \  It is easy to see that this
is equivalent to the following positivity test:
\begin{equation}\label{1.6}
\begin{pmatrix}
\hbox{$[T^*,T]$}& \hbox{$[T^{*2},T]$}& \hdots & \hbox{$[T^{*k},T]$}\\
\hbox{$[T^{*}, T^2]$}& \hbox{$[T^{*2},T^2]$} & \hdots & \hbox{$[T^{*k},T^2]$}\\
\vdots & \vdots & \ddots & \vdots\\
\hbox{$[T^*, T^k]$} & \hbox{$[T^{*2}, T^k]$} & \hdots &
\hbox{$[T^{*k},T^k]$}
\end{pmatrix}\ge 0
\qquad\text{(all $k\ge 1$)}\,.
\end{equation}
Condition (\ref{1.6}) provides a measure of the gap between
hyponormality and subnormality. \  In fact the positivity condition
(\ref{1.6}) for $k=1$ is equivalent to the hyponormality of $T$,
while subnormality requires the validity of (\ref{1.6}) for all $k$.
\ For $k\ge 1$, an operator $T$ is said to be {\it
$k$-hyponormal}\label{khypo} if $T$ satisfies the positivity
condition (\ref{1.6}) for a fixed $k$. \  Thus the Bram-Halmos
criterion can be stated as: $T$ is subnormal if and only if $T$ is
$k$-hyponormal for all $k\ge 1$. \ The notion of $k$-hyponormality
has been considered by many authors aiming at understanding  the
bridge between hyponormality and subnormality. \ In view of
(\ref{1.6}), between hyponormality and subnormality there exists a
whole slew of increasingly stricter conditions, each expressible in
terms of the joint hyponormality of the tuples $(I,T,T
^2,\hdots,T^k)$. \  Given an $n$-tuple $\mathbf T=(T_1,\hdots, T_n)$
of operators on $\mathcal{H}$, we let $[\mathbf T^*,\mathbf
T]\label{bfT^*bfT} \in\mathcal{B(H\oplus\cdots\oplus H)}$ denote the
{\it self-commutator} of $\mathbf T$, defined by
$$
[\mathbf T^*, \mathbf T]:=
\begin{pmatrix}
\hbox{$[T_1^*, T_1]$}& \hbox{$[T_2 ^*, T_1]$}& \hdots & \hbox{$[T_n^*,T_1]$}\\
\hbox{$[T_1^*, T_2]$}& \hbox{$[T_2 ^*, T_2]$}& \hdots &\hbox{$[T_n^*,T_2]$}\\
\vdots & \vdots & \ddots & \vdots\\
\hbox{$[T_1^*,T_n]$} & \hbox{$[T_2^*,T_n]$} & \hdots &
\hbox{$[T_n^*,T_n]$}
\end{pmatrix}.
$$
By analogy with the case $n=1$, we shall say (\cite{At}, \cite{CMX})
that $\mathbf T$ is {\it jointly hyponormal} (or simply, {\it
hyponormal}) if $[\mathbf T^ *,\mathbf T]$ is a positive operator on
$\mathcal{H} \oplus \cdots \oplus \mathcal{H}$. \  $\mathbf T$ is
said to be {\it normal} if $\mathbf T$ is commuting and every $T_i$
is a normal operator, and {\it subnormal} if $\mathbf T$ is the
restriction of a normal $n$-tuple to a common invariant subspace. \
Clearly, the normality, subnormality or hyponormality of an
$n$-tuple requires as a necessary condition that every coordinate in
the tuple be normal, subnormal or hyponormal, respectively. \
Normality and subnormality require that the coordinates commute, but
hyponormality does not. \

In 1988, the hyponormality of the Toeplitz operators $T_{\varphi}$
was characterized in terms of their symbols $\varphi$  via Cowen's
Theorem \cite{Co2}, which follows.

\medskip

\noindent {\bf Cowen's Theorem.} (\cite{Co2},
\cite{NT})\label{cowen} \ \ {\it For each $\varphi\in L^\infty$, let
$$
\mathcal{E}(\varphi)\label{Ephi}\equiv \Bigl\{k\in H^\infty:\
||k||_\infty\le 1\ \hbox{and}\ \varphi-k\overline\varphi\in
H^\infty\Bigr\}. \
$$
Then $T_\varphi$ is hyponormal if and only if $\mathcal{E}(\varphi)$
is nonempty. \  }

\medskip

To study hyponormality (resp. normality and subnormality) of the
Toeplitz operator $T_\varphi$ with symbol $\varphi$ we may, without
loss of generality, assume that $\varphi(0)=0$; this is because
hyponormality (resp. normality and subnormality) is invariant under
translations by scalars. \

In 2006, C. Gu, J. Hendricks and D. Rutherford \cite{GHR} have
considered the hyponormality of Toeplitz operators with
matrix-valued symbols and characterized it in terms of their
symbols. \ Their characterization resembles Cowen's Theorem except
for an additional condition -- the normality of the symbol. \

\medskip

\begin{lemma}\label{gu}
{\rm (Hyponormality of Block Toeplitz Operators) (\cite{GHR})} \ For
each $\Phi\in L^\infty_{M_n}$, let
$$
\mathcal{E}(\Phi)\label{Ebigphi}:=\Bigl\{K\in H^\infty_{M_n}:\
||K||_\infty \le 1\ \ \hbox{and}\ \ \Phi-K \Phi^*\in
H^\infty_{M_n}\Bigr\}. \
$$
Then $T_\Phi$ is hyponormal if and only if $\Phi$ is normal and
$\mathcal{E}(\Phi)$ is nonempty. \end{lemma}

\medskip


M. Abrahamse \cite[Lemma 6]{Ab} showed that if $T_\varphi$ is
hyponormal, if $\varphi\notin H^\infty$, and if $\varphi$ or
$\overline{\varphi}$ is of bounded type then both $\varphi$ and
$\overline{\varphi}$ are of bounded type. \ However, in contrast to
the scalar-valued case, $\Phi^*$ may not be of bounded type even
though $T_\Phi$ is hyponormal, $\Phi\notin H^\infty_{M_n}$ and
$\Phi$ is of bounded type. \ But we have a one-way implication:
\begin{equation}\label{gu3}
\hbox{$T_\Phi$ is hyponormal and $\Phi^*$ is of bounded type}\
\Longrightarrow \ \hbox{$\Phi$ is of bounded type}
\end{equation}
(see \cite [Corollary 3.5 and Remark 3.6]{GHR}). \ Thus, whenever we
deal with hyponormal Toeplitz operators $T_\Phi$ with symbols $\Phi$
satisfying that both $\Phi$ and $\Phi^*$ are of bounded type (e.g.,
$\Phi$ is a matrix-valued rational function), it suffices to assume
that only $\Phi^*$ is of bounded type. \  In spite of this fact, for
convenience, we will assume that $\Phi$ and $\Phi^*$ are of bounded
type whenever we deal with bounded type symbols. \

\bigskip


\newpage

\noindent{\Large \bf Notations}\label{notation}

\bigskip

\begin{itemize}
\item Let $\theta$ be an inner function in $H^\infty$. Then

\smallskip

$I_\theta:\label{Itheta}=\theta I\equiv \left(\begin{smallmatrix}
\theta\\ &\ddots\\ &&\theta\end{smallmatrix}\right)$ (where $I$ is
the identity matrix)

$\mathcal{Z}(\theta)\label{ztheta}:=\hbox{the set of all zeros of
$\theta$}$

\medskip

\item $b_\lambda (z)\label{blambdaz}:=\frac{z-\lambda}{1-\overline\lambda z}$ ($\lambda\in\mathbb D$), a Blaschke factor

\medskip

\item $H^2_0\label{h20}\equiv (H^2_{M_n})_0:=zH^2_{M_n}$

\medskip

\item Let $\Theta\in H^\infty_{M_{n\times m}}$ be an inner matrix function. Then

\smallskip

$\mathcal H(\Theta)\label{HTheta}:= H^2_{\mathbb C_n}\ominus \Theta
H^2_{\mathbb C_m}$

$\mathcal H_{\Theta}\label{H_Theta}:= H^2_{M_{n \times m}}\ominus
\Theta H^2_{M_m}$

$\mathcal K_{\Theta}\label{K_Theta}:= H^2_{M_{n \times m}}\ominus
H^2_{M_n}\Theta$

\smallskip
\noindent If $\Theta=I_{\theta}$ for an inner function $\theta$,
then $\mathcal{H}_\Theta=\mathcal{K}_\Theta$. \ If there is no
confusion then we write, for brevity, $\mathcal H_\theta$,
$\mathcal{K}_\theta$ for $\mathcal H_{I_\theta}$,
$\mathcal{K}_{I_\theta}$. \

\medskip

\item For $\mathcal X$ a closed subspace of $H^2_{M_n}$,
$P_{\mathcal X}$ denotes the orthogonal projection from $H^2_{M_n}$
onto $\mathcal{X}$.

\medskip

\item
If $\Phi\in L^\infty_{M_n}$ and $\Delta_1$ and $\Delta_2$ are inner
matrix functions in $H^\infty_{M_n}$, we write
$$
\Phi_{\Delta_1, \Delta_2}\label{PhiLow}:=
P_{(H^{2}_{M_n})^\perp}(\Phi_-^* \Delta_1) +
P_{H_0^2}(\Delta_2^*\Phi_+);
$$
$$
\Phi^{\Delta_1, \Delta_2}\label{Phiupper}:= P_{(H^{2}_{M_n})^\perp}
(\Delta_1\Phi_-^* ) + P_{H_0^2}(\Phi_+\Delta_2^*),
$$
and abbreviate
$$
\Phi_{\Delta}\label{Phi_delta}\equiv \Phi_{\Delta,
\Delta}\quad\hbox{and}\quad \Phi^{\Delta}\label{Phi^delta}\equiv
\Phi^{\Delta, \Delta}.\label{qw}
$$
If $\Delta_i:=I_{\delta_i}$ for some inner functions $\delta_i \
(i=1,2)$, then $\Phi_{\Delta_1, \Delta_2}=\Phi^{\Delta_1,
\Delta_2}$. \ If there is no confusion then we write, for brevity,
$$
\Phi_{I_{\delta_1},\Delta_2}\equiv \Phi_{\delta_1,\Delta_2},\ \
\Phi_{\Delta_1, I_{\delta_2}}\equiv \Phi_{\Delta_1,\delta_2},\ \
\Phi^{I_{\delta}}\equiv \Phi^\delta\ \ \hbox{and etc}
$$
(i.e., we write $\delta$ for $I_{\delta}$ in this representation).

\medskip

\item
For an inner function $\theta$, $U_\theta$ denotes the compression
of the shift operator $U\equiv T_z$ : i.e.,
$$
U_\theta\label{U_theta}= P_{\mathcal H (\theta)} U \vert_{\mathcal H
(\theta)}.
$$
More generally, for $A \in L^{\infty}_{M_n}$ and an inner function
$\Theta \in H^{\infty}_{M_n}$, we write
$$
(T_A)_{\Theta}\label{compress}=P_{\mathcal H(\Theta)} T_A|_{\mathcal
H(\Theta)},
$$
which is called the compression of $T_A$ to $\mathcal{H}(\Theta)$.

\end{itemize}

%
%
%
%
%
%

\newpage


%
%
%
%

\chapter{Coprime inner functions}

\medskip

To understand functions of bounded type, we need to factorize those
functions into a coprime product of an inner function and the
complex conjugate of an $H^2$-function. \ Thus we are interested in
the following question: When are two inner functions in $H^\infty$
coprime\,? \ Naturally, a measure-theoretic problem arises at once,
since singular inner functions correspond to their singular measures. \
In this chapter, we answer this question. \

A nonzero sequence $\{\alpha_j\}$ in $\mathbb D$ satisfying
$\sum_{j=1}^\infty (1-|\alpha_j|)<\infty$ is called a Blaschke
sequence. \ If $\{\alpha_j\}$ is a Blaschke sequence and $k$ is an
integer, $k\ge 0$, then the function
$$
b(z):=z^k\prod_{j=1}^\infty \frac{|\alpha_j|}{\alpha_j}
\left(\frac{z-\alpha_j}{1-\overline\alpha_j z}\right)
$$
is called a Blaschke product. \ The factor $z^k$ in the definition
of the Blaschke product is to allow $b$ to have a zero at $0$. \ If
$\{\alpha_j\}$ is a finite sequence then $b$ is called a finite
Blaschke product. \

Recall that an inner function $\theta\in H^\infty$ can be written as
$$
\theta (z)= c\,  b(z) \,  \hbox{exp}\left(-\int_{\mathbb
T}\frac{t+z}{t-z}d\mu(t)\right) \quad(z \in \mathbb D),
$$
where $c$ is a constant of modulus $1$, $b$ is a Blaschke product,
and $\mu$ is a finite positive Borel measure on $\mathbb T$ which is
singular with respect to Lebesgue measure. \ It is evident that if
$b_1$ and $b_2$ are Blaschke products then
\begin{equation}\label{Blaco}
\hbox{$b_1$ and $b_2$ are not coprime}\ \Longleftrightarrow \
\mathcal Z(b_1)\cap \mathcal{Z}(b_2)\ne \emptyset\,.
\end{equation}
Thus the difficulty in determining coprime-ness of two inner
functions is caused by an inner functions having no zeros in
$\mathbb D$, which is called {\it a singular inner function}:
$$
\theta(z)\equiv \hbox{exp}\left(-\int_{\mathbb
T}\frac{t+z}{t-z}d\mu(t)\right),
$$
where $\mu$ will be called the {\it singular measure} of $\theta$. \
Thus we are interested in the following question:
\begin{equation}\label{cosing}
\hbox{When are two singular inner functions coprime?}
\end{equation}
Question (\ref{cosing}) seems to be well known to experts; however,
we have not been able to find an answer in the literature. \ Thus,
below we give an answer to Question (\ref{cosing}). \

\medskip

Let $\mu$ be a Borel measure on a locally compact Hausdorff space
$X$. \ Recall that the support of $\mu$ is defined as the set:
$$
\hbox{supp}(\mu):=\Bigl\{ x \in X: \mu(G)>0 \ \hbox{for every open
neighborhood} \ G\ \hbox{of}\ x\Bigr\}.
$$
An equivalent definition of support of $\mu$ is as the largest
closed set $S \subseteq X$ such that for every open subset $U$ of
$X$ for which $S\cap U\ne \emptyset$,
\begin{equation}\label{supp}
\mu(S^c)=0\ \ \hbox{and}\ \ \mu (U \cap S) > 0.
\end{equation}
On the other hand, if there is a Borel set $A$ such that
$\mu(E)=\mu(A\cap E)$ for every Borel set $E$, we say that $\mu$ is
{\it concentrated on $A$} (cf. \cite{Ru}). \ This is equivalent to
the condition that $\mu(E)=0$ whenever $E\cap A=\emptyset$. \ Also
we say that two Borel measures $\mu$ and $\nu$ are {\it mutually
singular} (and write $\mu\perp\nu$) if there are disjoint Borel sets
$A$ and $B$ such that $\mu$ is concentrated on $A$ and $\nu$ is
concentrated on $B$. \

We note that if $\hbox{supp}(\mu)\cap\hbox{supp}(\nu)=\emptyset$,
then $\mu\perp \nu$. \ However the converse is not true: for
example, if $m$ is the Lebesgue measure on $[0,1]$ and  $\delta_1$
is the Dirac measure at $1$, then $m\perp \delta_{1}$, but
$\hbox{supp}(m)\cap\hbox{supp}(\delta_1)=\{1\}$. \

\medskip

How does one define the infimum of a family of finite positive Borel
measures? \ \ Let $\mathcal S\equiv \{\mu_1,\mu_2,\cdots\}$ be a
countable family of finite positive Borel measures on a locally
compact Hausdorff space $X$. \ For any Borel set $E$, define $\mu
(E)$ by
\begin{equation}\label{infmea}
\mu (E):=\inf \sum_{k} \mu_k (E_k),
\end{equation}
where the $\mu_k$ runs through the family $\mathcal S$ and where
$\{E_1, E_2, \cdots\}$ runs through all partitions of $E$ into Borel
sets (cf. \cite[p.84]{Ga}). \ Then we can show that $\mu$ is a
positive Borel measure satisfying the following:
\begin{itemize}
\item[(i)] \ $\mu \leq \mu_k$ for all $k=1,2,\cdots$;

\item[(ii)] \ $\mu$ is the maximum positive Borel measure on $X$  satisfying (i)
in the sense that: if $\lambda$ is a positive Borel measure on $X$
satisfying $\lambda\leq \mu_k$ for all $k=1,2,\cdots$, then
$\lambda\leq \mu$.
\end{itemize}
The measure $\mu$ of (\ref{infmea}) is called the {\it infimum} of
the family $\mathcal S$ of measures, and we write
$$
\mu:=\inf_k\,(\mu_k).
$$
We note that the infimum of two nonzero measures may be zero. \ For
example, let $X=[0,1]$ and suppose $m$ is the Lebesgue measure on
$X$ and $\delta_1$ is the Dirac measure at $1$. \ Then $\mu(X)\equiv
\inf\,(m,\delta_1)(X)\le m(\{1\}) + \delta_1 ([0,1))=0$.

\bigskip

\begin{theorem}\label{thminfmeasure2} \  Let $\{\mu_k\}$ be a countable
family of  finite positive Borel measures on a locally compact
Hausdorff space $X$. \ If $\mu_k$ is concentrated on $A_k$  for each
$k \in \mathbb Z_+$, then $\mu\equiv\inf_k\,(\mu_k)$ is concentrated
on $\bigcap_n A_n$.
\end{theorem}

\begin{proof} \ Write
$$
A\equiv\bigcap_k A_k \quad\hbox{and}\quad \mu \equiv \inf_k(\mu_k).
$$
Let $E$ be a Borel set and suppose $\{E_n\}$ is a partition of $E$
into Borel sets. \  \ Let
$$
F_n:=E_n \cap A \quad\hbox{and}\quad  A_n^{(k)}:=E_n \setminus A_k.
$$
Then we can see that
$$
E=\Bigl(\cup_n F_n\Bigr) \bigcup \Bigl(\cup_{n,k} A_n^{(k)} \Bigr).
$$
Let $\{C_n\}$ be the collection consisting of $\{F_n,
A_n^{(k)}\}_{n,k}$. \ We then have
\begin{equation}\label{aaa}
\aligned \mu(E)
& \leq \inf \sum_k \mu_k (C_k)\\
& \leq \inf \sum_k \mu_k (F_k)\\
& =\inf \sum_k \mu_k (E_k \cap A).
\endaligned
\end{equation}
Since $E\cap A=\bigcup_n(E_n \cap A )$, it follows from (\ref{aaa})
that
$$
\mu(E) \leq \mu(E \cap A ).
$$
This completes the proof.
\end{proof}

\begin{corollary}\label{corinfmeasure2} \
Let $\{\mu_k\}$ be a countable family of  finite positive Borel
measures on a locally compact Hausdorff space $X$. \ If $\mu_i \perp
\mu_j$ for some $i,j$, then $\inf_k\,(\mu_k)=0$.
\end{corollary}

\begin{proof} \
Immediate from Theorem \ref{thminfmeasure2}.
\end{proof}

\medskip

\begin{theorem}\label{thminfmeasure8} \
Let $\mu_1$ and $\mu_2$ be finite positive Borel measures on a Borel
$\sigma$-algebra $\mathfrak B$ in a locally compact Hausdorff space
$X$. \ Then
$$
\mu_1 \perp \mu_2 \Longleftrightarrow \inf\, (\mu_1, \mu_2)=0.
$$
\end{theorem}

\begin{proof} \
($\Rightarrow$) \ This follows from Theorem \ref{thminfmeasure2}.

($\Leftarrow$)  \ Suppose $\mu_1$ and $\mu_2$ are finite positive
Borel measures. \ By the Lebesgue decomposition of $\mu_1$ relative
to $\mu_2$, there exists a unique pair $\{\mu_a, \mu_s\}$ of finite
positive measures on $\mathfrak B$ such that
$$
\mu_1=\mu_a+\mu_s, \quad \mu_a \ll\mu_2, \quad \mu_s \perp \mu_2.
$$
Let $h\in L^1(\mu_2)$ be the Radon-Nikodym derivative of $\mu_a$
with respect to $\mu_2$: that is,
\begin{equation}\label{sharp}
\mu_a(E)=\int_E h\, d\mu_2\quad (E\in\mathfrak B).
\end{equation}
Note that $h$ is a nonnegative measurable function. \ Assume that
$\mu_1$ and $\mu_2$ are not mutually singular. \ Then $h \neq 0$
$[\mu_2]$. \ Define
$$
H:=\Bigl\{x \in X: 0<h(x)\leq 1\Bigr\}.
$$
There are two cases to consider.

\medskip Case 1 ($\mu_2(H)\neq 0$): \
Define
$$
\lambda(E):=\int_{E} h \cdot \chi_{H} \,d\mu_2 \quad (E\in\mathfrak
B).
$$
Since $ h \cdot \chi_{H}$ is a nonnegative measurable function, it
follows that $\lambda$ is a positive measure. \ Observe that
$$
\lambda(H)=\int_{H} h  d\mu_2>0.
$$
For each $E\in\mathfrak B$, we have, by (\ref{sharp}),
$$
\lambda(E)=\int_{E} h \cdot \chi_{H} d\mu_2 \leq \int_{E} h
d\mu_2=\mu_a(E)\leq \mu_1(E)
$$
and by definition of $H$,
$$
\lambda(E)=\int_{E} h \cdot \chi_{H} d\mu_2 \leq \int_{E}
d\mu_2=\mu_2(E).
$$
But since $\lambda(H)>0$, it follows that $0 \neq \lambda \leq
\inf\, (\mu_1, \mu_2)$, which implies $\inf\, (\mu_1, \mu_2)\neq 0$.

\medskip

Cases 2 ($\mu_2(H)= 0$): \ For $m=2,3,4,\cdots$, define
$$
H_m:=\Bigl\{x \in X: 1<h(x)\leq m\Bigr\}.
$$
Since $h \neq 0 \ [\mu_2]$, there exists $N\in\{2,3,\cdots\}$ such
that $\mu_2(H_N)\neq 0$. \ Define
$$
\lambda(E):=\frac{1}{N}\int_{E} h \cdot \chi_{H_N} d\mu_2 \quad
(E\in \mathfrak B).
$$
Then $\lambda$ is a positive measure. \ Observe that
$$
\lambda(H_N)=\frac{1}{N}\int_{H_N} h \, d\mu_2 > \frac{1}{N}
\int_{H_N}  d\mu_2 =\frac{1}{N}\mu_2(H_N) >0.
$$
For each $E\in\mathfrak B$,
$$
\lambda(E)=\frac{1}{N}\int_{E} h \cdot \chi_{H_N} d\mu_2 \leq
\frac{1}{N}\int_{E} h d\mu_2 \leq \int_{E} h  d\mu_2=\mu_a(E)\leq
\mu_1(E),
$$
and by definition of $H_N$,
$$
\lambda(E)=\frac{1}{N}\int_{E} h \cdot \chi_{H_N} d\mu_2 \leq
\int_{E}  d\mu_2=\mu_2(E).
$$
Since $\lambda (H_N)>0$, it follows  that $0 \neq \lambda \leq
\hbox{inf}\, (\mu_1, \mu_2)$, which implies $\hbox{inf}\, (\mu_1,
\mu_2)\neq 0$. \ This completes the proof.
\end{proof}

\begin{corollary}\label{thminfmeasure88} \
Let $\{\mu_1, \mu_2, \cdots, \mu_m\}$ be a finite collection of
finite positive Borel measures on  a locally compact Hausdorff space
$X$. \ Then the followings are equivalent:
\begin{itemize}

\item[(i)] \ $\mu\equiv \inf(\mu_1, \mu_2, \cdots, \mu_m) \neq 0$

\item[(ii)] \ If $\mu_k$ is concentrated on $A_k$  for each $k=1,2 \cdots, m$,
then $\bigcap_{k=1}^m A_k \neq \emptyset$
\end{itemize}
\end{corollary}

\begin{proof} \
This follows at once from Theorems \ref{thminfmeasure2} and
\ref{thminfmeasure8}.
\end{proof}

\medskip

\begin{remark}
Corollary \ref{thminfmeasure88} may fail for a {\it countable}
collection of finite positive Borel measures. For example, let
$\lambda$ be a nonzero finite positive Borel measure on a locally
compact Hausdorff space $X$ and let
$$
\mu_k:=\frac{\lambda}{k} \quad(k \in \mathbb Z_+).
$$
Then clearly, $\mu\equiv \inf_k(\mu_k) =0$. \ \ Suppose that $\mu_k$
is concentrated on $A_k$  for each $k \in \mathbb Z_+$. \ Then
$\lambda$ is concentrated on $A_k$ for each $k \in \mathbb Z_+$. \
Write
$$
E:= \bigcap_{k=1}^{\infty}A_k \quad \hbox{and} \quad
E_n:=\bigcap_{k=1}^n A_k.
$$
Then
$$
E_1 \supseteq E_2 \supseteq E_3 \supseteq \cdots \quad \hbox{and}
\quad E=\bigcap_{n=1}^{\infty}E_n.
$$
We now claim  that
\begin{equation}\label{zx}
\lambda(E_n)=\lambda(X)  \quad  \hbox{for each} \  n \in \mathbb
Z_+.
\end{equation}
To prove (\ref{zx}), we use mathematical induction. \ Since
$\lambda$ is concentrated on $A_1=E_1$, it follows that
$\lambda(E_1)=\lambda(E_1 \cap X)=\lambda(X)$. \ Suppose
$\lambda(E_k)=\lambda(X)$ ($k \geq 1$). \ Since $\lambda$ is
concentrated on $A_{k+1}$, we have
$$
\lambda(E_{k+1})=\lambda(E_k\cap A_{k+1})=\lambda(E_k)=\lambda(X),
$$
which proves (\ref{zx}). \ Since $\lambda$ is a finite measure,
$\lambda(X)=\lambda(E_n)$ converges to $\lambda(E)$, so that
$\lambda(E)=\lambda(X)\neq 0$, and hence $E\ne \emptyset$.
\hfill$\square$
\end{remark}

\medskip

\begin{theorem}\label{nonzero4} \ Let  $\mu_1, \  \mu_2$ be finite
positive regular Borel measures on a locally compact Hausdorff space
$X$ such that
$$
S \equiv \hbox{\rm supp}(\mu_1) \cap \hbox{\rm supp}(\mu_2)  \neq
\phi.
$$
If there exists $x \in S$ and an open neighborhood $N$ of $x$ such
that
$$
 m \leq \left\{\frac{\mu_2(N_x)}{\mu_1(N_x)}: x \in N_x, \ \hbox{an open subset of} \  N\right\} \leq M
$$
for some $m, M>0$, then $\mu\equiv \inf\,(\mu_1, \mu_2)\neq 0$.
\end{theorem}

\noindent {\it Remark}. Since $x \in S$, $\mu_i(N_x)\neq 0$ for each
$i=1,2$.

\begin{proof}[Proof of Theorem \ref{nonzero4}] \
Suppose that there exists $x \in S$ and an open neighborhood $N$ of
$x$ such that
$$
 m \leq \left\{\frac{\mu_2(N_x)}{\mu_1(N_x)}: x \in N_x, \ \hbox{an open subset of} \  N\right\} \leq M,
$$
for some $0<m<1<M<\infty$. \ We will show that
\begin{equation}\label{nonzero64}
\mu(N) \neq 0.
\end{equation}
Note that $\mu_1$ and $\mu_2$ are regular Borel measures. \ Let $G$
be a Borel set and $\epsilon>0$. \ Suppose $ x \in G$. \ Then there
exists an open set $V\supseteq G$ such that
$$ \mu_1(N \cap G)+\epsilon > \mu_1(N \cap V) \geq
\frac{1}{M}\mu_2(N \cap V).
$$
Thus
$$
\aligned \mu_1(N \cap G)+\mu_2(N\cap G^c)+\epsilon
&\geq \frac{1}{M}\mu_2(N \cap V)+\mu_2(N\cap G^c)\\
&\geq \frac{1}{M}\Bigl(\mu_2(N \cap G)+\mu_2(N\cap G^c)\Bigr)\\
&\geq \frac{1}{M}\mu_2(N),
\endaligned
$$
which gives
$$
\mu_1(N \cap G)+\mu_2(N\cap G^c)\geq \frac{1}{M}\mu_2(N).
$$
If $x \notin G$, then similarly, we have
$$
\mu_1(N \cap G)+\mu_2(N\cap G^c) \geq m\,\mu_1(N).
$$
Thus, it follows that
$$
\mu(N) = \inf\,(\mu_1,\mu_2)(N)\geq
\hbox{min}\left(\frac{1}{M}\mu_2(N), m \mu_1(N) \right)>0,
$$
which proves (\ref{nonzero64}).
\end{proof}


We now have:
\medskip

\begin{theorem}\label{thmcoprime3}  \
Let $\theta_1, \theta_2 \in H^{\infty}$ be singular inner functions
with singular measures $\mu_1$ and $\mu_2$, respectively. \ Then
$\theta_1$ and $\theta_2$ are coprime if and only if $\mu_1\perp
\mu_2$.
\end{theorem}

\begin{proof} \
($\Rightarrow$) \ For $i=1,2$, write
$$
\theta_i(z)= \hbox{exp}\left(-\int_{\mathbb
T}\frac{t+z}{t-z}d\mu_i(t)\right) \quad(z \in \mathbb D),
$$
where each  $\mu_i$ is a finite positive Borel measure on $\mathbb
T$ which is singular with respect to Lebesgue measure. \ Assume that
$\mu_1$ and $\mu_2$ are not mutually singular. \ Thus, by Theorem
\ref{thminfmeasure8}, $\mu\equiv \inf (\mu_1, \mu_2)\ne 0$. \ Since
\begin{equation}\label{absolutely}
\mu(E)\leq \mu_1(E) \quad \hbox{and} \quad \mu(E)\leq \mu_2(E) \quad
\hbox{for each Borel set} \ E\subseteq \mathbb T,
\end{equation}
it follows that $\mu$ is absolutely continuous with respect to both
$\mu_1$ and $\mu_2$. \ In particular, $\mu$ is a finite positive
Borel measure on $\mathbb T$. \ Also, since $\mu_1$ is singular with
respect to Lebesgue measure, $\mu$ is singular with respect to
Lebesgue measure. \ Thus, it follows from (\ref{absolutely}) that
$\mu_i^{\prime}:=\mu_i-\mu \ (i=1,2)$ is a finite positive Borel
measure which is singular with respect to Lebesgue measure. \ Since
$\mu_i^{\prime}(\mathbb T) \leq \mu_i(\mathbb T)<\infty$, we can see
that $\mu_i^\prime$ is regular for each $i=1,2$. \ Observe that
$$
\hbox{exp}\left(-\int_{\mathbb T}\frac{t+z}{t-z}d\mu_1(t)\right)
=\hbox{exp}\left(-\int_{\mathbb T}\frac{t+z}{t-z}d\mu(t)\right)
\cdot \hbox{exp}\left(-\int_{\mathbb T}\frac{t+z}{t-z}d\mu_1^{\prime}(t)\right)\\
$$
Thus,
$$
\theta(z):=\hbox{exp}\left(-\int_{\mathbb
T}\frac{t+z}{t-z}d\mu(t)\right) \qquad(z \in \mathbb D).
$$
is a nonconstant inner divisor of $\theta_1$. \ Similarly we can
show that $\theta$ is also a nonconstant inner divisor of
$\theta_2$. \ Hence $\theta_1$ and $\theta_2$ are not coprime.

($\Leftarrow$) Assume that $\theta_1$ and $\theta_2$ are not
coprime. \ Thus there exists a nonconstant common inner divisor
$\omega$ of $\theta_1$ and $\theta_2$: i.e.,
\begin{equation}\label{333}
\theta_1=\omega\theta_1^\prime\quad\hbox{and}\quad
\theta_2=\omega\theta_2^\prime.
\end{equation}
Note that $\omega$ is also a singular inner function, so that we may
write
$$
\omega(z)=\hbox{exp}\left(-\int_{\mathbb
T}\frac{t+z}{t-z}d\mu(t)\right),
$$
where $\mu$ is a nonzero finite positive Borel measure on $\mathbb
T$ which is singular with respect to Lebesgue measure. \ If
$\mu_1^\prime$ and $\mu_2^\prime$ are singular measures
corresponding to $\theta_1^\prime$ and $\theta_2^\prime$,
respectively, then it follows from (\ref{333}) that
$$
d\mu_i=d(\mu+\mu_i^\prime)\quad (i=1,2).
$$
Thus we can see that $\mu(E)\le \mu_i(E)$ for every Borel set $E$ of
$\mathbb T$. \
Since $\mu$ is nonzero, it follows from Theorem \ref{thminfmeasure2}
that
$\mu_1$ and $\mu_2$ are not mutually singular.
\end{proof}

\begin{remark}
By a similar argument as in the proof of Theorem \ref{thmcoprime3},
we can show that if $\theta_1$ and $\theta_2$ are singular inner
functions with singular measures $\mu_1$ and $\mu_2$, respectively
and if $\mu:=\inf\,(\mu_1, \mu_2)$, then
$$
\theta(z)=\hbox{exp}\left(-\int_{\mathbb
T}\frac{t+z}{t-z}d\mu(t)\right).
$$
is the greatest common inner divisor of $\theta_1$ and $\theta_2$
\end{remark}

\bigskip

%
%
%
%
%
%

\newpage

\chapter{Douglas-Shapiro-Shields factorizations}

\medskip

To understand matrix functions of bounded type, we need to factor
those functions into a coprime product of matrix inner functions and
the adjoints of matrix $H^\infty$-functions; this is the so-called
Douglas-Shapiro-Shields factorization. \ This factorization is very
helpful and somewhat unavoidable for the study of Hankel and
Toeplitz operators with bounded type symbols. \ In this chapter, we
consider several properties of the Douglas-Shapiro-Shields
factorization for matrix functions of bounded type. \

For a matrix-valued function $\Phi\in H^2_{M_{n\times r}}$, we say
that $\Delta\in H^2_{M_{n\times m}}$ is a {\it left inner divisor}
of $\Phi$ if $\Delta$ is an inner matrix function such that
$\Phi=\Delta A$ for some $A \in H^{2}_{M_{m\times r}}$ ($m\le n$).
We also say that two matrix functions $\Phi\in H^2_{M_{n\times r}}$
and $\Psi\in H^2_{M_{n\times m}}$ are {\it left coprime} if the only
common left inner divisor of both $\Phi$ and $\Psi$ is a unitary
constant, and that $\Phi\in H^2_{M_{n\times r}}$ and $\Psi\in
H^2_{M_{m\times r}}$ are {\it right coprime} if $\widetilde\Phi$ and
 $\widetilde\Psi$ are left coprime.
Two matrix functions $\Phi$ and $\Psi$ in $H^2_{M_n}$ are said to be
{\it coprime} if they are both left and right coprime. \ We note
that if $\Phi\in H^2_{M_n}$ is such that $\hbox{det}\,\Phi\ne 0$,
 then any left inner divisor $\Delta$ of $\Phi$ is
square, i.e., $\Delta\in H^2_{M_n}$\label{square2}: indeed, if
$\Phi=\Delta A$ with $\Delta\in H^2_{M_{n\times r}}$ ($r<n$) then
$\hbox{rank}\,\Phi(z)\le \hbox{rank}\,\Delta(z)\le r <n$, so that
$\hbox{det}\,\Phi(z)=0$. \
If $\Phi\in H^2_{M_n}$ is such that $\hbox{det}\,\Phi\ne 0$, then we
say that $\Delta\in H^2_{M_{n}}$ is a {\it right inner divisor} of
$\Phi$ if $\widetilde{\Delta}$ is a left inner divisor of
$\widetilde{\Phi}$.

Let $\{\Theta_i\in H^\infty_{M_n}: i\in J\}$ be a family of inner
matrix functions. \  The greatest common left inner divisor
$\Theta_d\label{thetad}$ and the least common left inner multiple
$\Theta_m\label{thetam}$ of the family $\{\Theta_i\in
H^\infty_{M_n}: i\in J\}$ are the inner functions defined by
$$
\Theta_d H^2_{\mathbb C^p}=\bigvee_{i \in J}\Theta_{i}H^2_{\mathbb
C^n} \quad \hbox{and} \quad \Theta_m H^2_{\mathbb C^q}=\bigcap_{i\in
J}\Theta_{i}H^2_{\mathbb C^n}.
$$
Similarly, the greatest common right inner divisor
$\Theta_d^{\prime}$ and the least common right inner multiple
$\Theta_m^{\prime}$ of the family $\{\Theta_i\in H^\infty_{M_n}:
i\in J\}$ are the inner functions defined by
$$
\widetilde{\Theta}_d^{\prime} H^2_{\mathbb C^r}=\bigvee_{i \in
J}\widetilde{\Theta}_{i} H^2_{\mathbb C^n} \quad \hbox{and} \quad
\widetilde{\Theta}_m^{\prime} H^2_{\mathbb C^s}=\bigcap_{i \in J}
\widetilde{\Theta}_{i}H^2_{\mathbb C^n}.
$$
The Beurling-Lax-Halmos Theorem guarantees that $\Theta_d$ and
$\Theta_m$ exist and are unique up to a unitary constant right
factor, and $\Theta_d^{\prime}$ and $\Theta_m^{\prime}$ are unique
up to a unitary constant left factor. \  We write
$$
\begin{aligned}
&\Theta_d =\hbox{left-g.c.d.}\label{lgcd}\,\{\Theta_i: i\in
J\},\quad
\Theta_m=\hbox{left-l.c.m.}\label{llcm}\,\{\Theta_i: i\in J\},\\
&\Theta_d^\prime=\hbox{right-g.c.d.}\label{rgcd}\,\{\Theta_i: i\in
J\},\quad
\Theta_m^\prime=\hbox{right-l.c.m.}\label{rlcm}\,\{\Theta_i: i\in
J\}.
\end{aligned}
$$
If $n=1$, then
$\hbox{left-g.c.d.}\,\{\cdot\}=\hbox{right-g.c.d.}\,\{\cdot\}$
(simply denoted $\hbox{g.c.d.}\label{GCD}\,\{\cdot\}$) and
$\hbox{left-l.c.m.}\,\{\cdot\}=\hbox{right-l.c.m.}\,\{\cdot\}$
(simply denoted $\hbox{l.c.m.}\label{LCM}\,\{\cdot\}$). \ In
general, it is not true that
$\hbox{left-g.c.d.}\,\{\cdot\}=\hbox{right-g.c.d.}\,\{\cdot\}$ and
$\hbox{left-l.c.m.}\,\{\cdot\}=\hbox{right-l.c.m.}\,\{\cdot\}$.

\medskip

However, we have:

\begin{lemma}\label{lem2.1}
Let $\Theta_i:=I_{\theta_i}$ for an inner function $\theta_i \ (i
\in J)$. \
\begin{enumerate}
\item[(a)] $\hbox{\rm left-g.c.d.}\,\{\Theta_i: i\in J\}=\hbox{\rm right-g.c.d.}\,\{\Theta_i: i\in J\}
=I_{\theta_d}$, where $\theta_d=\text{\rm g.c.d.}\,\{\theta_i : i
\in J \}$.
\item[(b)] $\hbox{\rm left-l.c.m.}\,\{\Theta_i: i\in J\}=\hbox{\rm right-l.c.m.}\,\{\Theta_i: i\in J\}
=I_{\theta_m}$, where $\theta_m=\text{\rm l.c.m.}\,\{\theta_i : i
\in J \}$.
\end{enumerate}
\end{lemma}

\begin{proof}
See \cite[Lemma 2.1]{CHL2}.
\end{proof}

\medskip

In view of Lemma \ref{lem2.1}, if $\Theta_i=I_{\theta_i}$ for an
inner function $\theta_i$ ($i\in J)$, we can define the greatest
common inner divisor $\Theta_d$ and the least common inner multiple
$\Theta_m$ of the $\Theta_i$ by
$$
\Theta_d\equiv\hbox{g.c.d.}\,\{\Theta_i:i\in
J\}:=I_{\theta_d},\quad\hbox{where} \ \theta_d=\hbox{\rm
g.c.d.}\,\{\theta_i : i \in J \}
$$
and
$$
\Theta_m\equiv\hbox{l.c.m.}\,\{\Theta_i:i\in
J\}:=I_{\theta_m},\quad\hbox{where} \ \theta_m=\hbox{\rm
l.c.m.}\,\{\theta_i : i \in J \}.
$$
Both $\Theta_d$ and $\Theta_m$ are {\it diagonal-constant} inner
functions, i.e., diagonal inner functions, and constant along the
diagonal. \

\bigskip

\begin{remark}\label{DSS}
By contrast with scalar-valued functions, in (\ref{2.6}), $I_\theta$
and $A$ need not be (right) coprime. \ If
$\Omega=\hbox{left-g.c.d.}\,\{I_\theta, A\}$ in the representation
(\ref{2.6}), that is,
$$
\Phi=\theta A^*\,,
$$
then $I_\theta =\Omega \Omega_{\ell}$ and $A=\Omega A_{\ell}$ for
some inner matrix $\Omega_{\ell}$ (where $\Omega_{\ell}\in
H^2_{M_n}$ because $\hbox{det}\,(I_\theta)\neq 0$) and some $A_l \in
H^{2}_{M_n}$. \  Therefore if $\Phi^*\in L^\infty_{M_n}$ is of
bounded type then we can write
\begin{equation}\label{2.8}
\Phi={A_{\ell}}^*\Omega_{\ell},\quad\hbox{where $A_{\ell}$ and
$\Omega_{\ell}$ are left coprime.}
\end{equation}
In this case, $A_{\ell}^*\Omega_{\ell}$ is called the {\it left
coprime factorization} of $\Phi$ and write, briefly,
$$
\Phi=A_{\ell}^*\Omega_{\ell}\ \ \hbox{(left coprime).}
$$
Similarly, we can write
\begin{equation}\label{2.9}
\Phi=\Omega_r A_r^*, \quad\hbox{where $A_r$ and $\Omega_r$ are right
coprime.}
\end{equation}
In this case, $\Omega_r A_r^*$ is called the {\it right coprime
factorization} of $\Phi$ and we write, succinctly,
$$
\Phi=\Omega_r A_r^*\ \ \hbox{(right coprime).}
$$
We often say that (\ref{2.9}) is the {\it Douglas-Shapiro-Shields
factorization} of $\Phi$ and (\ref{2.8}) is  the {\it left
Douglas-Shapiro-Shields factorization} of $\Phi$ (cf. \cite{DSS},
\cite{Fu}). \ We also say that $\Omega_\ell$ and $\Omega_r$ are
called the {\it inner parts} of those factorizations.
\end{remark}

\begin{remark} (\cite[Corollary 2.5]{GHR}; \cite[Remark 2.2]{CHL2}) \
As a consequence of the Beurling-Lax-Halmos Theorem, we can see that
\begin{equation}\label{RCD}
\Phi=\Omega_r A_r^*\ \hbox{(right coprime)}\ \Longleftrightarrow
\hbox{ker}\,H_{\Phi^*}=\Omega_r H_{\mathbb C^n}^2.
\end{equation}
\end{remark}

\bigskip

If $M$ is a nonzero closed subspace of $\mathbb C^n$ then the matrix
function
$$
b_\lambda P_M+(I-P_M)\quad \hbox{($P_M:=$ the orthogonal projection
of $\mathbb C^n$ onto $M$)}
$$
is called a {\it Blaschke-Potapov factor}\,; an $n\times n$ matrix
function $D$ is called {\it a finite Blaschke-Potapov product} if
$D$ is of the form
$$
D \label{Dz}=\nu \prod_{m=1}^M \Bigl(b_m P_m + (I-P_m)\Bigr),
$$
where $\nu$ is an $n\times n$  unitary constant matrix, $b_m$ is a
Blaschke factor, and $P_m$ is an orthogonal projection in $\mathbb
C^n$ for each $m=1,\cdots, M$. \  In particular, a scalar-valued
function $D$ reduces to a finite Blaschke product $D=\nu
\prod_{m=1}^M b_m$. \  It is known (cf. \cite{Po}) that an $n\times
n$ matrix function $D$ is rational and inner if and only if it can
be represented as a finite Blaschke-Potapov product. \ If $\Phi\in
L^\infty_{M_n}$ is rational then $\Omega_\ell$ and $\Omega_r$ in
(\ref{2.8}) and (\ref{2.9}) can be chosen as finite Blaschke-Potapov
products. \ We can see (cf. \cite{CHL2}) that every inner divisor of
$I_{b_{\lambda}} \in H_{M_n}^\infty$ is a Blaschke-Potapov factor.

\medskip

In what follows, we examine the left and right coprime
factorizations of matrix $H^2$-functions. \ We begin with:

\begin{lemma}\label{lem2.4}
Let $\Theta\in H^\infty_{M_{n\times m}}$ be an inner matrix function
and $A\in H^2_{M_{n\times m}}$. \  Then the following hold:
\begin{enumerate}
\item[(a)] $A \in \mathcal K_{\Theta} \Longleftrightarrow \Theta A^* \in H^2_0$;
\item[(b)] $A \in \mathcal H_{\Theta} \Longleftrightarrow A^* \Theta  \in H^2_0$;
\item[(c)] $P_{H_0^2}(\Theta A^*)=\Theta\Bigl(P_{\mathcal K_{\Theta}}A\Bigr)^*$.
\end{enumerate}
\end{lemma}

\begin{proof}
Immediate from a direct calculation. \
\end{proof}

\begin{lemma}\label{lem2.5}
Let $\Theta\in H^\infty_{M_{n\times m}}$ be an inner matrix function
and $B\in H^{2}_{M_{n \times m}}$. \  Then the following hold:
\begin{enumerate}
\item[(a)] $P_{\mathcal K_{\Theta}}\label{PKTheta} B=
\widetilde{P_{\mathcal
H_{\widetilde{\Theta}}}\label{PHtheta}\widetilde{B}}$ \  if $n=m$;
\item[(b)] $P_{\mathcal K_{\Theta}}(\Lambda B)
=\Lambda (P_{\mathcal K_{\Theta}}B)$ for any constant matrix
$\Lambda \in M_n$;
\item[(c)] $P_{\mathcal H_{\Theta}}(B\Lambda )=(P_{\mathcal H_{\Theta}}B)\Lambda $ for any constant matrix
$\Lambda \in M_m$. \
\end{enumerate}
In particular, if $n=m$ and $\Theta= I_\theta$ for an inner function
$\theta$ and $\Lambda$ commutes with $B$ then
$$
\Lambda(P_{\mathcal K_{\Theta}}B) =(P_{\mathcal
K_{\Theta}}B)\Lambda.
$$
\end{lemma}

\begin{proof} Write $B:=B_1+B_2$, where $B_1:=P_{\mathcal K_{\Theta}}B$ and
$B_2:=B_3 \Theta$ for some $B_3 \in H_{M_n}^2$. \  Then $
\widetilde{B}=\widetilde{B_1}+\widetilde{\Theta}\widetilde{B_3}$. \
If $n=m$, then $\widetilde{\Theta}$ is an inner matrix function.  \
Since $B_1 \in \mathcal K_{\Theta}$, by Lemma \ref{lem2.4} (a), we
have $\Theta B_1^* \in H^2_0$, so that ${\widetilde{B}_1}^*
\widetilde{\Theta}=\widetilde{\Theta B_1^*} \in H^2_0$. \ Thus it
follows from Lemma \ref{lem2.4} (b) that $\widetilde{B}_1 \in
\mathcal H_{\widetilde{\Theta}}$, and hence
$$
P_{\mathcal H_{\widetilde{\Theta}}}\widetilde{B}=\widetilde{B}_1
=\widetilde{P_{\mathcal K_{\Theta}}B},
$$
giving (a). \  Observe that
$$
P_{\mathcal K_{\Theta}}(\Lambda B )=P_{\mathcal
K_{\Theta}}(\Lambda(B_1+B_3\Theta) ) =P_{\mathcal
K_{\Theta}}(\Lambda B_1 ).
$$
Since $B_1 \in \mathcal K_{\Theta}$, it follows from Lemma
\ref{lem2.4} (a) that $\Theta B_1^* \in H^2_0$. \  Thus
$$
\Theta(\Lambda B_1 )^*=\Theta  B_1^* \Lambda^* \in H^2_0,
$$
which implies $\Lambda B_1 \in \mathcal K_{\Theta}$, and hence
$P_{\mathcal K_{\Theta}}(\Lambda B )=\Lambda B_1
=\Lambda(P_{\mathcal K_{\Theta}}B) $. \  This proves (b). \  The
statement (c) follows from (a) and (b). \  The last assertion
follows at once from (b) and (c) because $\mathcal
H_{\Theta}=\mathcal K_{\Theta}$ if $\Theta=I_\theta$. \
\end{proof}

\begin{lemma}\label{lem22.55}
Let $\Theta\in H^\infty_{M_{n}}$ and $\Delta \in
H^{\infty}_{M_{n\times m}}$ be inner matrix functions. Then the
following hold:
\begin{enumerate}
\item[(a)] $\mathcal K_{\Theta \Delta}=\mathcal K_{\Delta} \bigoplus \mathcal K_{\Theta}\Delta$;
\item[(b)]  $\mathcal H_{\Theta \Delta }=\mathcal H_{\Delta} \bigoplus \Delta \mathcal
H_{\Theta}$.
\end{enumerate}
\end{lemma}

\begin{proof} \
The inclusion $\mathcal K_{\Delta} \subseteq \mathcal K_{\Theta
\Delta}$ is obvious; it also follows from Lemma \ref{lem2.4}(a) that
$\mathcal K_{\Theta}\Delta \subseteq \mathcal K_{\Theta \Delta}$. \
Thus
$$
\mathcal K_{\Delta} \bigoplus \mathcal K_{\Theta}\Delta \subseteq
\mathcal K_{\Theta \Delta}.
$$
For the reverse inclusion, suppose $M \in \mathcal K_{\Theta
\Delta}$. \
 \ Write $M_1:= P_{\mathcal K_{\Delta}}M$. \
 Then $M-M_1=Q \Delta $ for some $Q \in H^2_{M_{n\times m}}$. \
Since $M_1 \in \mathcal K_{\Delta} \subseteq  \mathcal K_{\Theta
\Delta} $, we have $Q \Delta=M-M_1\in \mathcal K_{\Theta \Delta}$.
\ Thus, by Lemma \ref{lem2.4}(a), $\Theta Q^*=(\Theta \Delta)(Q
\Delta)^*
 \in H^2_0$, and hence, again by Lemma
\ref{lem2.4}(a), $Q \in \mathcal K_{\Theta}$, which gives (a). \ The
argument for (b) is similar.
\end{proof}

\begin{lemma}\label{lem22.88} \
Suppose $A,B,C\in H^\infty_{M_n}$ with $AB=BA$.  \ If $A$ and $B$
are left coprime and $A$ and $C$ are left coprime then $A$ and $BC$
are left coprime.
\end{lemma}

\begin{proof} \ If $A$ and $B$ are left coprime
and $A$ and $C$ are left coprime then it follows  (cf.
\cite[p.242]{FF}) that
$$
H^2_{\mathbb C^n}=AH^2_{\mathbb{C}^n}\bigvee BH^2_{\mathbb{C}^n}
=AH^2_{\mathbb{C}^n}\bigvee CH^2_{\mathbb{C}^n}.
$$
Then the limit argument together with the commutativity of $A$ and
$B$ shows that
$$
H^2_{\mathbb C^n}=AH^2_{\mathbb{C}^n}\bigvee BCH^2_{\mathbb{C}^n},
$$
which gives the result.
\end{proof}

\begin{corollary}\label{lem2.3}
Let $A, B\in H^\infty_{M_n}$ and $\theta$ be an inner function. \ If
$A$ and $B$ are left coprime, and $A$ and $I_\theta$ are left
coprime, then $A$ and $\theta B$ are left coprime. \
\end{corollary}

\begin{proof} Immediate from Lemma \ref{lem22.88}.
\end{proof}

\begin{corollary}\label{cor22.99}  \
Suppose that $A,B,C\in H^\infty_{M_n}$ with $AB=BA$.  \ If $A$ and
$B$ are right coprime and $A$ and $C$ are right coprime then $A$ and
$CB$ are right coprime.
\end{corollary}

\begin{proof} \
This follows from Lemma \ref{lem22.88} together with the fact that
$A$ and $B$ are right coprime if and only if $\widetilde{A}$ and
$\widetilde{C}$ are left coprime.
\end{proof}

\bigskip

We recall the inner-outer factorization of vector-valued functions.
\ Let $D$ and $E$ be Hilbert spaces. \ If $F$ is a function with
values in $\mathcal{B}(E,D)$ such that $F(\cdot)e\in H^2_{D}$ for
each $e\in E$, then $F$ is called a strong $H^2$-function. \ The
strong $H^2$-function $F$ is called an {\it inner} function if
$F(\cdot)$ is an isometric  operator from $D$ into $E$. \ Write
$\mathcal{P}_{E}\label{pe}$ for the set of all polynomials with
values in $E$, i.e., $p(\zeta)=\sum_{k=0}^n \widehat{p}(k)\zeta^k$,
$\widehat{p}(k)\in E$. \ Then the function $Fp=\sum_{k=0}^n
F\widehat{p}(k) z^k$ belongs to $H^2_{D}$. \ The strong
$H^2$-function $F$ is called {\it outer} if $\hbox{cl}\,
F\cdot\mathcal{P}_E=H^2_{D}$. \ Note that every $F\in H^2_{M_n}$ is
a strong $H^2$-function. \ We then have an analogue of the scalar
Factorization Theorem. \

\bigskip

\noindent{\bf Inner-Outer Factorization.} (cf. \cite{Ni})\quad Every
strong $H^2$-function $F$ with values in $\mathcal{B}(E, D)$ can be
expressed in the form
$$
F=F_iF_e,
$$
where $F_e$ is an outer function with values in $\mathcal{B}(E,
D^\prime)$ and $F_i$ is an inner function with values in
$\mathcal{B}(D^\prime,D)$, for some Hilbert space $D^\prime$.

\bigskip

If $\dim D=\dim E<\infty$, then outer operator functions can be
detected by outer scalar functions.

\medskip

\noindent {\bf Helson-Lowdenslager Theorem.}\label{HLT} (cf.
\cite{Ni2})\quad If $\hbox{dim}\, D <\infty$ and $F$ is a strong
$H^2$-function with values in $\mathcal B(D)$, then $F$ is outer if
and only if $\det F$ is outer.

\bigskip

We then have:

\begin{lemma}\label{lem0.3}
Let $\Phi\in H^2_{M_n}$. \ Then $\det\Phi\ne 0$ if and only if every
left inner divisor of $\Phi$ is square.
\end{lemma}

\begin{proof}
We assume to the contrary that every left inner divisor of $\Phi$ is
square, but $\det\, \Phi=0$.   \ Since $\Phi$ is a strong
$H^2$-function, we have an inner-outer factorization of $\Phi$ of
the form
$$
\Phi=\Phi_i\, \Phi_e\,,
$$
where  $\Phi_i\in H^2_{M_{n\times r}}$ is inner and $\Phi_e\in
H^2_{M_{r \times n}}$ is outer. \ By assumption, $r=n$. \ We thus
have $0=\hbox{det}\,\Phi=\hbox{det}\,\Phi_i\, \hbox{det}\,\Phi_e$,
which implies that $\hbox{det}\,\Phi_e=0$. \ By the
Helson-Lowdenslager Theorem, $0$ should be an outer function, a
contradiction. \ The converse was noticed in p.\pageref{square}.
\end{proof}

\bigskip

We introduce a notion of the ``characteristic scalar inner" function
of a matrix inner function (cf. \cite[p. 81]{He}, \cite{SFBK}).

\medskip

\begin{definition}\label{def5.1} \
Let $\Delta\in H^\infty_{M_n}$ be inner. \ We say that {\it $\Delta$
has a scalar inner multiple $\theta\in H^\infty$} if there exists
$\Delta_0\in H^\infty_{M_n}$ such that
$$
\Delta\Delta_0=\Delta_0\Delta = I_\theta.
$$
We define
$$
m_{\Delta}\label{ddelta}:= \hbox{g.c.d.}\,\Bigl\{\theta :\ \theta \
\hbox{is a  scalar inner multiple of $\Delta$} \Bigr\}\,.
$$
In view of Lemma \ref{lem2.1}, $m_\Delta$ is the minimal $\theta$ so
that $\Delta\Delta_0=\Delta_0\Delta = I_\theta$ for some
$\Delta_0\in H^\infty_{M_n}$. \ We say that $m_{\Delta}$ is the {\it
characteristic scalar inner function} of $\Delta$.
\end{definition}

\medskip

{\it Remark.} \ In Definition \ref{def5.1}, it is enough to assume $\Delta\Delta_0 = I_\theta$.  For, given two matrices $A$ and $B$ such that $AB=\lambda I$, with $\lambda \ne 0$, it is straightforward to verify that $BA=AB$.

\medskip

It is well known (cf. \cite[Proposition 5.6.1]{SFBK}) that if $\Delta\in
H^\infty_{M_n}$ is inner then $\Delta$ has a scalar inner multiple.
\ Thus $m_\Delta$ is well-defined. \ We would like to remark that
the notion of $m_\Delta$ arises in the Sz.-Nagy--Foia\c s theory of
contraction operators $T$ of class $C_0$ (completely nonunitary
contractions $T$ for which there exists a nonzero function $u\in
H^\infty$ such that $u(T)=0$): the {\it minimal function} $m_T$ of
the $C_0$-contraction operator $T$ amounts to our $m_{\Theta_T}$,
where $\Theta_T$ is the characteristic function of $T$ (cf.
\cite{Be}, \cite{SFBK}).

\bigskip

If $\Delta\equiv\bigl(\delta_{ij}\bigr)\in H^\infty_{M_n}$ is inner,
we may ask, how do we find $m_{\Delta}$\,? \ The following lemma
gives an answer. \ To see this, we first observe that if
$\Delta\equiv \bigl(\delta_{ij}\bigr)\in H^\infty_{M_n}$ is inner
then $\ker H_{\Delta^*}=\Delta H^2_{\mathbb C^n}$ and hence, by
Lemma \ref{gu2}, $\Delta^*$ is of bounded type, and we may write
$\delta_{ij}=\theta_{ij}\overline{b}_{ij}$ ($\theta_{ij}$ inner;
$\theta_{ij}$ and $b_{ij}\in H^\infty$ coprime). \

\medskip

\begin{lemma}\label{lem3.5} \
Let $\Delta\in H^\infty_{M_n}$ be inner. \ Thus we may write
$\Delta\equiv \bigl(\theta_{ij}\overline{b}_{ij}\bigr)$, where
$\theta_{ij}$ is an inner function and $\theta_{ij}$ and $b_{ij}\in
H^\infty$ are coprime for each $i,j=1,2,\cdots,n$. \ Then
$$
m_\Delta=\hbox{\rm l.c.m.}\,\Bigl\{\theta_{ij}:i,j=1,2,\cdots,
n\Bigr\}.
$$
\end{lemma}

\begin{proof}
Observe that
\begin{equation}
\Delta=\bigl(\theta_{ij}\overline{b}_{ij}\bigr)=\bigl(\theta
\overline{a}_{ij}\bigr)= \theta A^* \quad (A \in H^{2}_{M_n}).
\end{equation}
Since $\Delta $ is inner, it follows that $A^*A=I$, so that $\Delta
A=I_\theta$. \ This says that $m_\Delta$ is an inner divisor of
$\theta$.  \ Thus, we may write
\begin{equation}\label{1.4-1}
\zeta m_\Delta=\theta\quad \hbox{($\zeta$ scalar-inner)}.
\end{equation}
We now want to show that $\zeta$ is constant. \ Let
$I_{m_\Delta}=\Delta\Delta_0$ for some inner function $\Delta_0$.  \
It follows from (\ref{1.4-1}) that $\zeta \Delta \Delta_0=I_\theta
=\Delta A$, and hence $\zeta \Delta_0 =A$. \ Since by (\ref{1.4-1}),
$\zeta$ is an inner divisor of $\theta$, it follows that
$$
\bigl(\theta_{ij}\overline{b}_{ij}\bigr)=\bigl(\theta
\overline{a}_{ij}\bigr)= (\theta\overline{\zeta})\Delta_0^*.
$$
But since $\theta_{ij}$ and $b_{ij}$ are coprime, $\theta_{ij}$ is
inner divisor of $\theta \overline{\zeta}$ for each $i,j=1,2,
\cdots, n$. \ Thus $\theta$ is an inner divisor of $\theta
\overline{\zeta}$, and therefore $\zeta$ is constant. \ This
completes the proof.
\end{proof}

\bigskip

We observe:

\begin{lemma}\label{lem1.4}
If $\theta\in H^2$ is a nonconstant inner function, then
$$
\bigcap_{n=1}^{\infty}{\theta}^nH^2=\{0\}.
$$
\end{lemma}

\begin{proof} \
If $\theta$ is a nonconstant inner function, then the wandering
subspace construction from the Halmos proof of the
Beurling-Lax-Halmos theorem implies that (cf. \cite[p. 204]{SFBK})
$$
H^2=\bigoplus_{j=0}^\infty \theta^j \mathcal{H}(\theta)
$$
and more generally,
$$
\theta^n H^2=\bigoplus_{j=n}^\infty \theta^j \mathcal{H}(\theta).
$$
Thus any $f\in H^2$ has an orthogonal decomposition as
$f=\sum_{j=0}^\infty \theta^j f_j$ with $f_j\in
\mathcal{H}(\theta)$. \ If $f\in \theta^n H^2$, then $f_j=0$ for
$0\le j<n$. \
Hence, if $f\in \cap_{n=0}^\infty \theta^n H^2$, then $f_j=0$
for $j=0,1,2,\hdots$, or $f=0$.
\end{proof}

\medskip

\begin{lemma}\label{lem3.8} \rm{(The Local Rank)}\quad
Let $E$ and $D$ be Hilbert spaces, $\hbox{dim}E <\infty$, and let
$\Phi$ be a strong $H^2$-function with values in $\mathcal B(D, E)$.
Denote $\hbox{Rank}\,\Phi:=\hbox{max}_{\zeta \in \mathbb
D}\,\hbox{rank}\Phi(\zeta)$, where $\hbox{rank}\Phi(\zeta):=\dim
\Phi(\zeta)(D)$. \ Then $\bigl\{ f\in H^2_D:
\Phi(\zeta)f(\zeta)\equiv 0\bigr\}=\vartheta H^2_{D^{\prime}}$
($\vartheta$ inner) for some Hilbert space $D^\prime$. \ In this
case, $\hbox{dim}\,D=\hbox{dim}\,D^{\prime}+\hbox{Rank}\,\Phi$.
\end{lemma}

\begin{proof}
See \cite[p.44]{Ni}.
\end{proof}

\begin{lemma}\label{lem1.7}
Let $A \in H^{\infty}_{M_{n}}$ and $\Theta=I_\theta$ for some
nonconstant inner function $\theta$. \ If $\Theta$  and $A$ are left
{\rm (}or right{\rm )} coprime, then $\det\,A\ne 0$.
\end{lemma}

\begin{proof} Suppose $\Theta$ and $A$ are left coprime. \
Then by the Beurling-Lax-Halmos Theorem, we have
\begin{equation}\label{1-5}
{\Theta} H^2_{\mathbb C^n} \bigvee A H^2_{\mathbb C^n} =
H^2_{\mathbb C^n}\quad \hbox{(cf. \cite[p.242]{FF})}.
\end{equation}
Assume to the contrary that $\det A=0$. \ Then $\det
\widetilde{A}=\widetilde{\det A}=0$, and hence
$\hbox{Rank}\,\widetilde{A}=m<n$. \ Since $\widetilde{A} \in
H^{\infty}_{M_{n}}$ is a strong $H^2$-function, it follows from
Lemma \ref{lem3.8} that
$$
\hbox{ker}\widetilde{A}\equiv \Bigl\{ f\in H^2_{\mathbb C^n}:
\widetilde{A}(\zeta)f(\zeta)\equiv 0\Bigr\}=\vartheta
H^2_{D^{\prime}},
$$
where $\vartheta$ is inner and $\hbox{dim}\,D^{\prime}=n-m$. \ Let
$U:H^2_{\mathbb C^{n-m}} \to H^2_{D^{\prime}}$ be unitary and put
$\widetilde{\Omega}:=\vartheta U \in H^2_{M_{n \times (n-m)}}$. \
Then $\hbox{ker}\widetilde{A}=\widetilde{\Omega} H^2_{\mathbb
C^{n-m}}$, so that $\widetilde{A}\widetilde{\Omega}=0$, and hence
$\Omega A=0$. \ It thus follows from (\ref{1-5}) that
\begin{equation}\label{1-6}
\Omega \Theta H^2_{\mathbb C^n}= \Omega\Bigl( \Theta H^2_{\mathbb
C^n} \bigvee A H^2_{\mathbb C^n}\Bigr) =\Omega H^2_{\mathbb C^n}.
\end{equation}
Write $\Omega:=\bigl(\omega_{ij}\bigr)\in H^2_{M_{(n-m)\times n}}$.
\ Since $\Theta=I_\theta$, it follows from (\ref{1-6}) that for each
$p=1,2,\cdots$, $(\overline{\theta}^p\omega_{ij})H^2_{\mathbb
C^n}=\Theta^{*p}\Omega H^2_{\mathbb C^n}=\Omega H^2_{\mathbb C^n}
\in H^2_{\mathbb C^{n-m}}$, which implies that for each $1\leq i\leq
n-m$ and $1\leq j\leq n$,
$$
\omega_{ij}\in \bigcap_{p=1}^{\infty} \theta^p H^2.
$$
Since $\theta$ is not constant, it follows from Lemma \ref{lem1.4}
that $\omega_{ij}=0$ for all $i,j$, and hence $\Omega=0$, a
contradiction. \ If $\Theta$ and $A$ are right coprime then
$\widetilde\Theta$ and $\widetilde A$ are left coprime. \ Thus by
what we showed just before, $\det \widetilde A\ne 0$, and hence
$\det A\ne 0$. This completes the proof.
\end{proof}

\bigskip

The following theorem plays a key role in the theory of coprime
factorizations. \

\medskip

\begin{theorem}\label{lem33.9}
For $A \in H_{M_n}^2$  and $\Theta:=I_\theta$ for some scalar inner
function $\theta$, then $\Theta$ and $A$ are right {\rm (}or
left{\rm )} coprime if and only if $\theta$ and $\det A$ are
coprime.
\end{theorem}

\begin{proof}
We first prove the theorem when $A\equiv \Delta\in H^\infty_{M_n}$
is inner. \

If $\Delta$ is a diagonal-constant inner function then this is
trivial. \ Thus we may assume that $\Delta$ is not
diagonal-constant. \ Suppose that $\theta$ and $\det \Delta$ are not
coprime.
 Write $m_\Delta:=\delta$ and
$I_\delta=\Delta\Delta_0=\Delta_0\Delta$ for a nonconstant inner
function $\Delta_0$. \ Thus $(\det \Delta)(\det \Delta_0)=\delta^n$,
and hence $\theta$ and $\delta^n$ (and hence $\delta$) are not
coprime. \ Put $\omega:=\hbox{g.c.d.}\,(\theta,\, \delta)$. \ Then
we can write $\Theta\equiv I_\theta =I_{\omega}I_{\theta_1}$ and
\begin{equation}\label{5.3}
I_\delta= \Delta\Delta_0 =I_\omega I_{\delta_1}
\end{equation}
for some inner functions $\theta_1$ and $\delta_1$. \ If
$\delta=\omega$, then $I_\delta$ is an inner divisor of $I_\theta$,
so that evidently, $\Theta$ and $\Delta$ are not right coprime. \ We
now suppose $\delta\ne \omega$. \ We then claim that
\begin{equation}\label{5.4}
\hbox{$\Delta$ and $I_\omega$ are not right coprime.}
\end{equation}
Towards a proof of (\ref{5.4}), we assume to the contrary that
$\Delta$ and $I_\omega$ are right coprime. \ Since by (\ref{5.3}),
$\overline\omega\Delta\Delta_0=I_{\delta_1} \in H^2_{M_n}$, it
follows from (\ref{1.4}) that
\begin{equation}\label{5.4-1}
0=H_{\overline\omega \Delta \Delta_0}=H_{\overline\omega\Delta}
T_{\Delta_0}.
\end{equation}
But since $\Delta$ and $I_\omega$ are right coprime, it follows from
(\ref{RCD}) together with (\ref{5.4-1}) that
$$
\Delta_0H^2_{\mathbb C^n}=\hbox{ran} \, T_{\Delta_0} \subseteq
\hbox{ker}H_{\overline\omega\Delta}=\omega H^2_{\mathbb C^n},
$$
which implies that $I_\omega$ is a left inner divisor of $\Delta_0$
(cf. \cite[Corollary IX.2.2]{FF}), so that $\overline\omega
\Delta_0$ should be inner. \ Consequently,
$I_{\delta_1}=I_{\delta\overline\omega}=\Delta(\overline\omega
\Delta_0)$,
 which contradicts the
definition of $m_\Delta$. \ This proves (\ref{5.4}). \ Since
$\Theta=I_{\theta_1} I_\omega$, it follows that $\Theta$ and
$\Delta$ are not right coprime. \

Conversely, suppose that $\Theta$ and $\Delta$ are not right
coprime. \ Thus $\Theta=\Theta_1 G$ and $\Delta=\Delta_1 G$, where
$G \in H^{\infty}_{M_n}$ is not unitary constant. \ Thus $\det G$ is
a common inner divisor of $\det \Theta$ and $\det \Delta$. \ Since $
G$ is not constant, it follows from the Fredholm theory of block
Toeplitz operators (cf. \cite{Do1}) that
$$
0<\dim \mathcal H(G)=\dim \hbox{ker}\,T_{G^*}=-\hbox{ind}\,
T_{G}=-\hbox{ind}\, T_{\det G},
$$
(where $\hbox{ind}\,(\cdot)$ denotes the Fredholm index) which
implies that $\det G$ is not constant, and hence $\det \Delta$ and
$\theta$ are not coprime. \

Now we prove the general case of $A\in H^2_{M_n}$. \ In view of
Lemma \ref{lem1.7}, we may assume that $\hbox{det}\, A\ne 0$. \ Then
by Lemma \ref{lem0.3}, the left and the right inner divisors of $A$
are square. \ Thus $A$ has the following inner-outer factorizations
of the form:
$$
A=A_e\, A_i,
$$
where  $A_i\in H^\infty_{M_n}$ is inner and $A_e\in H^2_{M_n}$ is
outer. \ Hence by what we proved just before and Helson-Lowdenslager
Theorem (p.\pageref{HLT}), we can see that $\Theta$ and $A$ are
right coprime if and only if $\Theta$ and $A_i$ are right coprime if
and only if $\theta$ and $\det A_i$ are coprime if and only if
$\theta$ and $\det A$ are coprime. \

For the case of left coprime-ness, we use
$\widetilde{\Phi}(z):=\Phi^*(\overline z)$. \ By the case of the
right coprime-ness and the fact that $\widetilde{\det A}= \det
\widetilde{A}$, it follows that $\theta$ and $\det A$ are coprime if
and only if $\widetilde\theta$ and $\det \widetilde{A}$ are coprime
if and only if $\widetilde{\Theta}$ and $\widetilde{A}$ are right
coprime if and only if $\Theta$ and $A$ are left coprime. \
This completes the proof.
\end{proof}

\bigskip

Theorem \ref{lem33.9} shows that
if $A\in H^2_{M_n}$ and $\Theta:=I_\theta$ for some scalar inner function $\theta$,
then ``left" coprime-ness and ``right" coprime-ness coincide for $A$ and $\Theta$.
Thus if $\theta$ is an inner function
then we shall say that $A\in H^2_{M_n}$ and $I_\theta$ are {\it
coprime} whenever they are right or left coprime. \ Hence if in the
representations (\ref{2.8}) and (\ref{2.9}), $\Omega_\ell$ or
$\Omega_r$ is of the form $I_\theta$ with an inner function
$\theta$, then we shall write
\begin{equation}\label{2.2-R}
\Phi=\theta A^*\ \ \hbox{(coprime)}.
\end{equation}

\medskip

\begin{lemma}\label{lem544444.2}
Let $\Omega, \Delta\in H^\infty_{M_n}$ be inner and $\Theta\equiv
I_\theta$ for an inner function $\theta$. \ If $\Omega$ and $\Theta$
are coprime and $\Delta$ and $\Theta$ are coprime, then $\hbox{\rm
left-l.c.m.}\,(\Omega, \Delta)$ and $\Theta$ are coprime.
\end{lemma}

\begin{proof} \ Suppose that $\Omega$ and $\Theta$
are coprime and $\Delta$ and $\Theta$ are coprime. \
Then by Theorem \ref{lem33.9},  $m_\Omega\equiv \omega$ and ${\theta}$ are
coprime and $m_{\Delta}\equiv \delta$ and ${\theta}$ are
coprime, so that by Corollary \ref{lem2.3},
$\hbox{l.c.m.}\,({\omega}, {\delta})$ and ${\theta}$ are coprime. \
Thus
\begin{equation}\label{coprime}
\hbox{l.c.m.}\,({\omega}, {\delta}) H^2_{\mathbb C^n} \bigvee
{\theta} H^2_{\mathbb C^n}=H^2_{\mathbb C^n}.
\end{equation}
Observe that
$$
\hbox{left-l.c.m.}\,(\Omega, \Delta) H^2_{\mathbb C^n}=\Omega
H^2_{\mathbb C^n} \cap \Delta H^2_{\mathbb C^n} \supseteq
{\omega}H^2_{\mathbb C^n} \cap {\delta} H^2_{\mathbb C^n}
=\hbox{l.c.m.}\,({\omega}, {\delta}) H^2_{\mathbb C^n}.
$$
It thus follows from (\ref{coprime}) that
$$
\hbox{left-l.c.m.}\,(\Omega, \Delta) H^2_{\mathbb C^n} \bigvee
\Theta H^2_{\mathbb C^n}\supseteq \hbox{l.c.m.}\,({\omega},
{\delta}) H^2_{\mathbb C^n} \bigvee {\theta} H^2_{\mathbb
C^n}=H^2_{\mathbb C^n},
$$
which implies that $\hbox{left-l.c.m.}\,(\Omega, \Delta)$ and
$\Theta$ are left coprime and hence coprime. \
\end{proof}

\begin{lemma}\label{lem544444.2-1} \
Let $\Delta_1, \Delta_2\in H^\infty_{M_n}$ be inner. \ Then there
exist inner matrix functions $\Omega_1, \Omega_2 \in H^\infty_{M_n}$
such that $\hbox{\rm left-l.c.m.}\,(\Delta_1, \Delta_2) =\Delta_1
\Omega_1=\Delta_2 \Omega_2$.
\end{lemma}

\begin{proof} \ Observe that
$$
\hbox{left-l.c.m.}\,(\Delta_1, \Delta_2) H^2_{\mathbb C^n}= \Delta_1
H^2_{\mathbb C^n} \cap \Delta_2 H^2_{\mathbb C^n} \subseteq \Delta_i
H^2_{\mathbb C^n} \  (i=1,2),
$$
which gives the result (cf. \cite[Corollary IX.2.2]{FF}). \
\end{proof}

\begin{proposition}\label{pro2.6}
Let $\Phi \in H_{M_n}^{2}$ be of the form $\Phi=\Theta A^* $ {\rm
(right coprime)}. \ If $\Delta$ is an inner matrix function in
$H^2_{M_n}$, put $\Phi_{\Delta}:=P_{H^2_0}(\Delta^*\Phi)$ {\rm (cf.
p.\pageref{notation})}. Then the following hold:
\begin{itemize}
\item[(a)]
If $\Delta$ is a left inner divisor of $\Theta$, then
$$
\Phi_{\Delta}=\Theta_1 A_1^*\quad\hbox{\rm  (right coprime)},
$$
where $\Theta_1= \Delta^*\Theta$  and $A_1:=P_{\mathcal
K_{\Theta_1}}(A)$.
\item[(b)]
If $\Omega:=\hbox{\rm left-g.c.d.}(\Delta, \Theta)$ and
$\Omega^*\Delta=I_{\delta_1}$ for an inner function $\delta_1$, then
$$
\Phi_{\Delta}=\Theta_1 A_1^*\quad\hbox{\rm (right coprime)},
$$
where $\Theta_1= \Omega^*\Theta$  and $A_1:=P_{\mathcal
K_{\Theta_1}}(\delta_1 A)$;
\item[(c)] \ If $\Omega:=\hbox{\rm left-g.c.d.} (\Delta, \Theta)$ and
$\Omega^*\Theta=I_{\theta_1}$ for an inner function $\theta_1$, then
$$
\Phi_{\Delta}={\theta_1}A_1^*\quad\hbox{\rm (coprime)},
$$
where $A_1:=P_{\mathcal K_{{\theta_1}}}(A \Delta_1)$ with
$\Delta_1=\Omega^*\Delta$.
\end{itemize}
\end{proposition}

\begin{proof}  Let $\Theta=\Delta \Theta_1 $ for an inner
function $\Theta_1$. \  Then it follows from Lemma \ref{lem2.4} that
$$
\Phi_{\Delta}=P_{H_0^{2}}(\Delta^*\Phi) =P_{H_0^2}(\Delta^*\Theta
A^*) =\Theta_1\Bigl(P_{\mathcal{K}_{\Theta_1}}(A)\Bigr)^*.
$$
Suppose $\Theta_1$ and $A_1:=P_{\mathcal K_{\Theta_1}}(A)$ are not
right coprime. \ Put $\Theta_2:=\hbox{right-g.c.d.} (\Theta_1,
A_1)$. \ Then $\Theta_2$ is not unitary constant and we may write
\begin{equation}\label{22.12}
\Theta_1=\Theta_3 \Theta_2 \ \hbox{and} \ A_1=A_2\Theta_2
\quad\hbox{(for some $\Theta_3, A_2 \in H^2_{M_n}$)}.
\end{equation}
Thus $\Theta_2$ is a common right inner divisor of $A$ and $\Theta$.
\ This is a contradiction. \ This proves (a).

For (b), we write $\Omega\equiv \hbox{left-g.c.d.}\,(\Delta,
\Theta)$. \ Then by assumption, we may write
$\Delta=\delta_1\Omega$. \ Put $\Theta_1:= \Omega^*\Theta$. \ Then
$I_{\delta_1}$ and $\Theta_1$ are (left) coprime. \ It follows from
Lemma \ref{lem2.4} that
$$
\aligned \Phi_{\Delta}&=P_{H_0^{2}}(\Delta^*\Phi)
=P_{H_0^2}(I_{\delta_1}^*\Theta_1 A^*)\\
&=P_{H_0^2}\bigl(\Theta_1
({\delta_1}A)^*\bigr)=\Theta_1\Bigl(P_{\mathcal{K}_{\Theta_1}}({\delta_1}A)\Bigr)^*.
\endaligned
$$
Note that $\Theta_1$ and $P_{\mathcal{K}_{\Theta_1}}({\delta_1}A)$
are right coprime if and only if $\Theta_1$ and ${\delta_1}A$ are
right coprime. \ Now we will show that $\Theta_1$ and ${\delta_1}A$
are right coprime. \ Since $\Theta_1$ and $I_{\delta_1}$ are left
coprime, it follows from Theorem \ref{lem33.9} that $\Theta_1$ and
$I_{\delta_1}$ are right coprime. \ Since $\Theta$ and $A$ are right
coprime, it follows that  $\Theta_1$ and $A$ are right coprime. \
Thus it follows from Corollary \ref{cor22.99} that $\Theta_1$ and
${\delta_1}A$ are right coprime. \ This proves (b). \ For (c), write
$\Omega\equiv \hbox{left-g.c.d.} (\Delta, \Theta)$. \ Then by
assumption, we may write
$$
\Delta=\Omega \Delta_1 \quad \hbox{and} \quad \Theta=\theta_1
\Omega,
$$
where $\Delta_1$ and $I_{\theta_1}$ is (left) coprime. \ Then it
follows from Lemma \ref{lem2.4} that
$$
\aligned \Phi_{\Delta}&=P_{H_0^{2}}(\Delta^*\Phi)
=P_{H_0^2}(\Delta_1^* I_{\theta_1} A^*)\\
&=P_{H_0^2}(I_{\theta_1} (A
\Delta_1)^*)=I_{\theta_1}\Bigl(P_{\mathcal{K}_{{\theta_1}}}(A
\Delta_1)\Bigr)^*.
\endaligned
$$
Since $I_{\theta_1}$ and $\Delta_1$ are coprime, and $I_{\theta_1}$
and $A$ are coprime, it follows from Lemma  \ref{lem22.88} that
$I_{\theta_1}$ and $A \Delta_1$ are coprime so that $I_{\theta_1}$
and $P_{\mathcal K_{{\theta_1}}}(A \Delta_1)$ are coprime.
 \ This completes the proof. \
\end{proof}

\begin{corollary}\label{cor2.9}
Let $\theta$ and $\delta$ be finite Blaschke products. \ If $\Phi
\in H_{M_n}^{2}$ is of the form
$$
\Phi=\theta A^* \ \ \hbox{\rm (coprime)}\,,
$$
then
\begin{equation}\label{112.13}
\Phi_{\delta} \equiv P_{H_0^2} (\overline\delta \Phi) =\theta_1
A_1^* \ \ \hbox{\rm (coprime)}\,,
\end{equation} where
$\theta_1=\overline{\hbox{\rm g.c.d.}(\delta, \theta)}\,\theta$ and
$A_1:=P_{\mathcal K_{\theta_1}}(\delta_1 A)$ with
$\delta_1=\overline{\hbox{\rm g.c.d.}(\delta, \theta)}\,\delta$. \
Moreover, $A_1(\alpha)$ is invertible for each $\alpha \in \mathcal
Z(\theta_1)$.
\end{corollary}

\begin{proof} \  The assertion (\ref{112.13}) follows from Proposition \ref{pro2.6}.  \
Put $A=(a_{ij})_{ij=1}^n$. \ Since $\theta$ is a finite Blaschke
product, $\theta_1$ is also a finite Blaschke product. \ For each
$\alpha\in \mathcal Z(\theta_1)$, we have
$\frac{1}{1-\overline{\alpha}z}\in \mathcal H (\theta_1)$, and hence
\begin{equation}\label{112.133}
\aligned A_1(\alpha)
&=\bigl(P_{\mathcal{K}_{\theta_1}}(\delta_1 A)\bigr)(\alpha)\\
&=\Bigl((P_{\mathcal{H}(\theta_1)}(\delta_1 a_{ij}))(\alpha)\Bigr)_{i,j=1}^n\\
&=\Biggl(\Bigl \langle P_{\mathcal{H}(\theta_1)}(\delta_1 a_{ij}),
   \ \frac{1}{1-\overline{\alpha}z}\Bigr \rangle \Biggr )_{i,j=1}^n\\
&=\Biggl(\Bigl \langle \delta_1 a_{ij}, \
   \frac{1}{1-\overline{\alpha}z}\Bigr \rangle \Biggr )_{i,j=1}^n\\
&=\bigl (~\delta_1(\alpha) a_{ij}(\alpha)~\bigr)_{i,j=1}^n\\
&=\delta_1(\alpha)A(\alpha).
\endaligned
\end{equation}
Since $\theta_1$ and $\delta_1$ are coprime and $I_\theta$ and $A$
are coprime, it follows from Theorem \ref{lem33.9} that
$\delta_1(\alpha) \neq 0$ and $A(\alpha)$ is invertible. \ Thus, by
(\ref{112.133}), $A_1(\alpha)$ is invertible. \
\end{proof}

\medskip

\begin{corollary}\label{cor22.6}
Let $\Phi \in H_{M_n}^{\infty}$ be of the form $\Phi=B^* \Theta$
\hbox{\rm (left coprime)}. \ If $\Omega$ is a right inner divisor of
$\Theta$ and $\Phi^\Omega:=P_{H_0^2} (\Phi\Omega^*)$, then
$$
\Phi^{\Omega}=B_1^* \Delta_1\ \ \hbox{\rm (left coprime)},
$$
where $\Delta_1= \Theta \Omega^*$  and $B_1:=P_{\mathcal
H_{\Delta_1}}B$.
\end{corollary}

\begin{proof} \ Since
$\widetilde{\Phi}=\widetilde{\Theta} \widetilde{B}^*$ (right
coprime) and $\widetilde{\Omega}$ is a left inner divisor of
$\widetilde{\Theta}$, it follows from Proposition \ref{pro2.6} that
\begin{equation*}\label{12.13}
\widetilde{\Phi}_{\widetilde{\Omega}}:=P_{H^2_0}
(\widetilde\Omega^*\widetilde\Phi)=\Theta_1 A_1^* \quad(\hbox{right
coprime}),
\end{equation*}
where $\Theta_1= \widetilde{\Omega}^*\widetilde{\Theta}$  and
$A_1:=P_{\mathcal K_{{\Theta_1}}}\widetilde{B}$. \ It follows from
Lemma \ref{lem2.5} that
\begin{equation}\label{22.13}
\widetilde{\Phi}_{\widetilde{\Omega}}
=\Theta_1\left(\widetilde{P_{\mathcal
H_{\widetilde{\Theta}_1}}B}\right)^*.
\end{equation}
Since $\widetilde{\Phi}_{\widetilde{\Omega}}
=\widetilde{P_{H^2_0}(\Phi \Omega^*)}=\widetilde{\Phi^{\Omega}}$, it
follows from (\ref{22.13}) that
$$
\Phi^{\Omega}=\left(P_{\mathcal H_{\widetilde{\Theta}_1}}B\right)^*
\widetilde{\Theta}_1=B_1^*\Delta_1 \quad (\hbox{right coprime}),
$$
where $B_1=P_{\mathcal H_{\Delta_1}}B$ and
$\Delta_1=\widetilde{\Theta}_1=\Theta \Omega^*$. \ This completes
the proof.
\end{proof}

\bigskip


Recall that the composition of two inner functions is again an inner
function. \ In general, we cannot guarantee that the composition of
two Blaschke products is again a Blaschke product. \ However, by
Frostman's Theorem (\cite{Ga}, \cite{MAR}), if $\theta$ is an inner
function in $H^\infty$, then for almost all $\alpha \in \mathbb D$
(with respect to area measure on $\mathbb D$), the function
$b_{\alpha} \circ \theta$ is a Blaschke product. \

\medskip

We are now interested in the following:

\begin{question}\label{q2-27}
Let $\theta, \delta$ and $\omega$ be inner functions in
$H^{\infty}$. \ Suppose $\theta$ and $\delta$ are coprime. \
\begin{itemize}
\item[(a)] If $\omega$ is a finite Blaschke product, are
$\omega \circ \theta$ and $\omega \circ \delta$ coprime?
\item[(b)] If $\theta$ and $\delta$ are finite Blaschke products,
are $\theta \circ \omega$ and $\delta \circ \omega$ coprime?
\end{itemize}
\end{question}

\bigskip

The following example shows that the answer to Question \ref{q2-27}
(a) is negative.

\begin{example}\label{e2-28} \
Let $\alpha_1, \cdots, \alpha_m, \beta_1, \cdots, \beta_n  \in
\mathbb D$ satisfy the following properties
\begin{itemize}
\item[(a)] $\alpha_i \neq \beta_j$ for each $i,j$;
\item[(b)] $\eta:=\prod_{i=1}^m \alpha_i=\prod_{j=1}^n\beta_j$.
\end{itemize}
Put
$$
\theta:=\prod_{i=1}^m b_{-\alpha_i}, \quad \delta:=\prod_{j=1}^n
b_{-\beta_j}, \quad \hbox{and} \quad \omega:=b_{\eta}.
$$
Then $\theta$ and $\delta$ are coprime. \ However, we have
$$
(\omega \circ \theta)(0)=\omega(\eta)=0=(\omega \circ \delta)(0),
$$
which implies that $\omega \circ \theta$ and $\omega \circ \delta$
are not coprime. \hfil$\square$
\end{example}

\bigskip

To examine Question \ref{q2-27} (b), we recall the notion of
``capacity of zero" for a set in $\mathbb D$. \ Recall (cf.
\cite[p.78]{Ga}) that a set $L\subseteq \mathbb D$ is said to have
{\it positive capacity} if for some compact subset $K\subseteq L$,
there exists a nonzero positive measure $\mu$ on $K$ such that
$$
G_\mu (z)\label{gmu}\equiv \int_K \log\, \frac{1}{|b_\alpha(z)|}
d\mu(\alpha)\ \hbox{is bounded  on $\mathbb D$.}
$$
The function $G_\mu (z)$ is called the {\it Green's potential}. \ It
is known that every set of positive area has positive capacity, so
that every set of capacity zero in $\mathbb D$ has area measure
zero. \

On the other hand, if $\theta \in H^{\infty}$ is an inner function,
then $z_0 \in \mathbb T$ is called a {\it singularity} of $\theta$
if $\theta$ does not extend analytically from $\mathbb D$ to $z_0$.
\ Write
$$
\theta\equiv b \cdot s,
$$
where $b$ is a Blaschke product and $s$ is a singular inner function
with singular measure $\mu$. \ It is known (\cite[Theorems 6.1 and
6.2]{Ga}) that $z_0$ is a singularity of $\theta$ if and only if
$z_0$ is an accumulation point of $\mathcal Z(b)$ or $z_0 \in
\hbox{supp}(\mu)$. \

\medskip

Now let $f \in H^{\infty}$ and let $z_0 \in \mathbb T$.  \ Then the
{\it range set}  of $f$ at $z_0$ is defined as
$$
\mathcal R(f, z_0)\label{rfz0} := \bigcap_{r>0}f\bigl(\mathbb D \cap
\Delta(z_0,r)\bigr),
$$
where $\Delta(z_0, r):=\{z:\ |z-z_0|<r\}$. \ Thus we can see that
the range set is the set of values assumed infinitely often in each
neighborhood of $z_0$. \ If $f$ is analytic across  $z_0$, and
non-constant then $\mathcal R(f, z_0)=\emptyset$.

We recall:

\medskip

\begin{lemma}{\rm (}\cite[Theorem 6.6]{Ga}{\rm )}\label{garnett77} \
Let $\theta$ be an inner function and let $z_0$ be a singularity of
$\theta$. \ Then
$$
\mathcal R(\theta, z_0)=\mathbb D \setminus L,
$$
where $L$ is a set of capacity zero. \ Moreover, if $\alpha \in
\mathcal R(\theta, z_0)$, then $b_{\alpha} \circ \theta$ is a
Blaschke product.
\end{lemma}

\medskip

As we have remarked, a set of capacity zero has area measure zero,
so that Lemma \ref{garnett77} shows that $\mathbb D$ is the range
set of an inner function $\theta$ at its singularity, except
possibly for a set of measure zero. \

\medskip

We then have:

\begin{theorem}\label{bcothm} \
Suppose that $\theta$ and $\delta$ are finite Blaschke products. \
If $\omega$ is an inner function satisfying one of the following:
\begin{itemize}
\item[(i)]$\omega$ is a finite Blaschke
product;
\item[(ii)] $\mathcal{Z}(\theta)\bigcup
\mathcal{Z}(\delta)\subseteq \bigcup \Bigl\{\mathcal R(\omega, z_0):
z_0 \ \hbox{is a singularity of} \  \omega
 \Bigr\}$,
\end{itemize}
then we have
$$
\hbox{$\theta$ and $\delta$ are coprime}\ \Longleftrightarrow\
\hbox{$\theta \circ \omega$ and $\delta \circ \omega$ are coprime}.
$$
\end{theorem}

\begin{proof} \ If $\omega$ is a finite Blaschke product,
then obviously, $\theta \circ \omega$ and $\delta \circ \omega$ are
also finite Blaschke products. \ Suppose that (ii) holds. \ It then
follows from Lemma \ref{garnett77} that $\theta \circ \omega$ and
$\delta \circ \omega$ are  Blaschke products. \ Observe that
\begin{equation}\label{blaschke1}
\mathcal Z(\theta \circ \omega) \cap \mathcal Z(\delta \circ
\omega)=\Bigl\{\alpha \in \mathbb D:\omega(\alpha) \in \mathcal
Z(\theta) \cap \mathcal Z(\delta)\Bigr\}.
\end{equation}
If (i) or (ii) holds, then it is clear that $ \omega(\mathbb
D)\supseteq \mathcal{Z}(\theta)\cap \mathcal{Z}(\delta). $ \ It thus
follows from (\ref{blaschke1}) that
$$
\aligned
\theta \ \hbox{and} \ \delta\  \hbox{are coprime} &\Longleftrightarrow \mathcal Z(\theta) \cap \mathcal Z(\delta)=\emptyset\\
& \Longleftrightarrow \mathcal Z(\theta \circ \omega) \cap \mathcal Z(\delta \circ \omega)=\emptyset\\
&\Longleftrightarrow \theta \circ \omega \ \hbox{and} \ \delta \circ
\omega\ \hbox{are coprime},
\endaligned
$$
which gives the result.
\end{proof}

\begin{corollary}\label{corbcothm} \
Suppose that $\{\theta_1, \theta_2,\cdots, \theta_N\}$ is the set of
finite Blaschke products. \ If $\omega$ is an inner function
satisfying one of the following:
\begin{itemize}
\item[(i)] $\omega$ is a finite Blaschke
product;
\item[(ii)]
$\bigcup_{i=1}^N \mathcal{Z}(\theta_i) \subseteq \bigcup
\Bigl\{\mathcal R(\omega, z_0): z_0 \ \hbox{is a singularity of} \
\omega
 \Bigr\}$,
\end{itemize}
then
\begin{equation}\label{blaschke21-1}
\hbox{\rm g.c.d.}\{\theta_k : 1\le k \le N \}=1 \
\Longleftrightarrow \ \hbox{\rm g.c.d.}\{ \theta_k \circ \omega :
1\le k\le N \}=1.
\end{equation}
\end{corollary}

\begin{proof} \
Immediate from Theorem \ref{bcothm}.
\end{proof}

\medskip

\begin{theorem} \label{lemou} \ If $f \in H^2$  is an outer function
and $\omega \in H^{\infty}$ is an inner function, then $f \circ
\omega$ is an outer function.
\end{theorem}
\begin{proof} \ Let $f$ be an outer function and $\omega$ be an inner function. \
Since $f$ is outer, there exists a sequence of polynomials $p_n$
such that $f \cdot p_n \to 1$ as  $n \to \infty$. \ Since $p_n \circ
\omega \in H^2$, there exists a sequence of polynomials $q_n^{(m)}$
such that
$$
q_n^{(m)} \to p_n \circ \omega \quad \hbox{as} \ m \to \infty \quad
\hbox{(for each $n \in \mathbb N$)}.
$$
We thus have
$$
(f\circ \omega)\cdot q_n^{(m)} \longrightarrow (f \cdot p_n)\circ
\omega  \quad \hbox{as} \ m \to \infty.
$$
On the other hand, since
$$
(f \cdot p_n)\circ \omega \longrightarrow 1\circ \omega =1 \quad
\hbox{as} \ n \to \infty,
$$
it follows that $ 1 \in \hbox{cl}\, (f \circ \omega)
\cdot\mathcal{P}_{\mathbb C}. $ \ Thus we can conclude that $f\circ
\omega$ is  outer. \ This completes the proof. \ \end{proof}

\bigskip

\begin{corollary} \label{lemou234} \
Suppose that $\varphi \in H^{\infty}$. \ Then we can write
$$
\varphi=\varphi_i  \varphi_e \quad \hbox{\rm (inner-outer
factorization)},
$$
then, for an inner function $\omega$, we have
$$
\varphi \circ \omega =(\varphi_i \circ \omega) \cdot (\varphi_e
\circ \omega) \quad \hbox{\rm (inner-outer factorization)}.
$$
\end{corollary}

\begin{proof} \ Since $\varphi_i \circ \omega$ is an inner function, it follow from Theorem
\ref{lemou}.
\end{proof}

\bigskip

On the other hand, we may ask what is the relation between the inner
parts of the inner-outer factorization and the coprime factorization
for bounded analytic functions whose conjugates are of bounded type.
\ Let us consider a simple example: take $\varphi(z)=z-2$. \ Then
$$
\varphi(z)=
\begin{cases}
1\cdot(z-2) \quad(\hbox{inner-outer factoriztion})\\
z\cdot (\overline{1-2z}) \quad (\hbox{coprime}),
\end{cases}
$$
which shows that the degree of the inner part of the inner-outer
factorization is less than the degree of the inner part of the
coprime factorization. \ Indeed, this is not unexpected.  \ In what
follows, we will show that this phenomenon holds even for
matrix-valued cases.

\bigskip

\begin{lemma}\label{lem0.3lem} $($\cite[p. 21]{Ni}$)$ \
Let $F\in H^2_{M_{m\times n}}$ $(n\geq m)$. \ The following re
equivalent:
\begin{itemize}
\item[(a)]  $F$ is an outer function;
\item[(b)] $\hbox{\rm g.c.d.}\Bigl\{\bigl(\det (F|_{\mathbb C^m})\bigr)_i\Bigr\}=1$,
\end{itemize}
where $F|_{\mathbb C^m}$ run through all minors of $F$ in
$H^2_{M_m}$ with respect to the canonical bases.
\end{lemma}

\medskip

\begin{theorem} \label{thmouter567} \ Suppose $F \in H^{2}_{m \times n}$ $(n\geq m)$
is a rational outer function. \ Then, for a finite Blaschke product
$\omega \in H^{\infty}$, $F \circ \omega$ is an outer function.

\end{theorem}
\begin{proof} \
Suppose $F$ is a rational outer function. \ Then by Lemma
\ref{lem0.3lem},
\begin{equation}\label{outer6789}
\hbox{\rm g.c.d.}\Bigl\{\Bigl(\det (F|_{\mathbb
C^m})\Bigr)_i\Bigr\}=1.
\end{equation}
Since $ \det \bigl((F \circ \omega)|_{\mathbb C^m}\bigr)=\det
\bigl((F|_{\mathbb C^m}) \circ \omega\bigr)=\Bigl(\det
\bigl(F|_{\mathbb C^m}\bigr)\Bigr) \circ \omega$, it follows that
\begin{equation}\label{outer33}
\Bigl(\det \bigl((F \circ \omega)|_{\mathbb
C^m}\bigr)\Bigr)_i=\Bigl(\bigl(\det (F|_{\mathbb C^m})\bigr) \circ
\omega\Bigr)_i.
\end{equation}
Thus it follows from Corollary \ref{lemou234} and (\ref{outer33})
that
\begin{equation}\label{outer3353}
\Bigl(\det \bigl((F \circ \omega)|_{\mathbb
C^m}\bigr)\Bigr)_i=\Bigl(\det \bigl(F|_{\mathbb C^m}\bigr)\Bigr)_i
\circ \omega.
\end{equation}
Since $F$ is rational, $\det \bigl(F|_{\mathbb C^m}\bigr)$ is also
rational. \ It thus follows from (\ref{outer6789}),
(\ref{outer3353}) and Corollary \ref{corbcothm} that
$$
\hbox{\rm g.c.d.}\Bigl\{\Bigl(\det \bigl(F \circ \omega|_{\mathbb
C^m}\bigr)\Bigr)_i\Bigr\} =\hbox{\rm g.c.d.}\Bigl\{\Bigl(\det
\bigl(F|_{\mathbb C^m}\bigr)\Bigr)_i \circ \omega \Bigr\}=1.
$$
Thus by Lemma \ref{lem0.3lem}, $F \circ \omega$ is an outer
function. \
\end{proof}

\bigskip

Suppose $\Phi \in H_{M_n}^2$ is such that $\Phi^*$ is of bounded
type. \  Then in view of (\ref{2.9}), we may write $\Phi=\Theta A^*$
(right coprime). \  Now we define the degree of $\Phi$ by
\begin{equation}\label{2.18}
\hbox{deg}\, (\Phi):=\hbox{dim}\, \mathcal H (\Theta)\,.
\end{equation}
Since by the Beurling-Lax-Halmos Theorem, the right coprime
factorization  is unique (up to a unitary constant right factor), it
follows that $\hbox{deg}\,(\Phi)\label{degphi}$ is well-defined. \
Moreover, by using the well-known Fredholm theory of block Toeplitz
operators \cite{Do2} we can see that
\begin{equation}
\hbox{deg}\, (\Phi)=\hbox{\rm dim}\,\mathcal{H}(\Theta)=\hbox{\rm
deg\,(det\,$\Theta$\label{dettheta})}.
\end{equation}
We observe that if $\Theta \in H_{M_n}^2$ is an inner matrix
function then
\begin{equation}\label{2.18-8}
\hbox{\rm deg}\, (\Theta)< \infty \Longleftrightarrow \hbox{$\Theta$
is a finite Blaschke-Potapov product.}
\end{equation}
Thus if $\Phi \in H_{M_n}^2$, then $\Phi$ is rational if and only if
$\hbox{\rm deg}\, (\Phi) < \infty$.

On the other hand, the following lemma shows that the notion of
degree given in (\ref{2.18}) is also well defined in the sense that
it is independent of the left or the right coprime factorization. \

\begin{lemma}\label{lem2.9}
Let $\Phi\in H^2_{M_n}$ be such that $\Phi^*$ is of bounded type. \
If
$$
\Phi =\Theta_r A_r^*\ \ \hbox{\rm (right coprime)}
=A_{l}^*\Theta_{l}\ \ \hbox{\rm(left coprime)}\,,
$$
then $\hbox{\rm dim}\, \mathcal{H}(\Theta_r)=\hbox{\rm
dim}\,\mathcal{H}(\Theta_l)$. \
\end{lemma}

\begin{proof} Observe that
$$
\begin{aligned}
\hbox{dim}\,\mathcal{H}(\Theta_r)
&=\hbox{dim}\,\Bigl(\hbox{ker}\,H_{A_r\Theta_r^*}\Bigr)^\perp\\
    &=\hbox{rank}\,H_{A_r\Theta_r^*}=\hbox{rank}\,H_{\Theta_l^* A_l}\\
    &
       =\hbox{dim}\,\Bigl(\hbox{ker}\,H_{\widetilde A_l \widetilde{\Theta}_l^*}\Bigr)^\perp\\
    &=\hbox{dim}\,\mathcal{H}(\widetilde\Theta_l)=\hbox{dim}\,\mathcal{H}(\Theta_l)\,,
\end{aligned}
$$
where the last equality follows from (\ref{dettheta}). \
\end{proof}

\medskip

We should not expect that $\hbox{deg}\,(\Phi)=\hbox{deg}\,
\bigl(\hbox{det}\,\Phi\bigr)$ for a rational function $\Phi\in
H^2_{M_n}$. \  To see this, let
$$
\Phi:=\begin{pmatrix}z&-b_{\alpha}z\\0&1\end{pmatrix} \quad
(b_{\alpha}(z):=\frac{z-\alpha}{1-\overline{\alpha}z}).
$$
Then $\hbox{det}\, \Phi=z$, so that $\hbox{deg}\, ( \hbox{det}\,
\Phi)=1$. \  But $\Phi$ can be written as
$$
\Phi=\begin{pmatrix} b_{\alpha}z&0\\0&1\end{pmatrix}
\begin{pmatrix} b_{\alpha}&0\\-1&1\end{pmatrix}^* \equiv \Theta A^*
\quad (\hbox{right coprime})\,,
$$
so that $\hbox{deg}\, (\Phi)=\hbox{deg}\, ( \hbox{det}\,
\Theta)=\hbox{deg}\,(b_\alpha z)=2$. \

\bigskip

\bigskip

We recall that if $\Phi\in H^2_{M_n}$, define the {\it local rank}
of $\Phi$ by
$$
\hbox{Rank}\,\Phi:=\hbox{max}_{\zeta \in \mathbb
D}\,\hbox{rank}\,\Phi(\zeta),
$$
where $\hbox{rank}\,\Phi(\zeta):=\dim \Phi(\zeta)(\mathbb C^n)$.

\bigskip

The following lemma provides an information on the local rank.

\begin{lemma}\label{localrank}{\rm (}cf. \cite[p.44]{Ni}{\rm )} \
Suppose $\Phi\in H^2_{M_n}$ has the inner-outer factorization
$$
\Phi=\Phi_i\Phi_e.
$$
Then $\hbox{\rm cl}\, \Phi \mathcal P_{\mathbb C^n}=\Phi_i
H^2_{\mathbb C^m}$, where $m=\hbox{\rm Rank}\, \Phi=\hbox{\rm
Rank}\, \Phi_i.$
\end{lemma}

\medskip

We then have:

\begin{theorem}\label{thmcodimgh2.41}
Suppose $\Phi \in H^{\infty}_{M_n}$. Then we have:
\begin{itemize}
\item[(a)]  \ If $\det \Phi =0$, then $\hbox{\rm codim}\bigl(\hbox{\rm ran}\, T_{\Phi}\bigr)=\infty$;
\item[(b)]  \ If $\det \Phi \neq 0$, then $\hbox{\rm codim}\bigl(\hbox{\rm ran}\, T_{\Phi}\bigr) \leq \hbox{\rm
rank}\,H_{\Phi^*}$.
\end{itemize}
\end{theorem}

\begin{proof} \ Suppose that $\Phi$ has the following inner-outer factorization.
$$
\Phi=\Phi_i\Phi_e \qquad(\Phi_i\in H^{\infty}_{M_{n\times m}}, \
\Phi_e\in H^{\infty}_{M_{m\times n}} \ (m \leq n))
$$
Since $\Phi_e$ is an outer function, it follows from Lemma
\ref{localrank} that
\begin{equation}\label{dkfgkmjhg}
\hbox{cl ran}\, T_{\Phi}=\hbox{cl}\,\Phi H^2_{\mathbb
C^n}=\hbox{cl}\,\Phi \mathcal P_{\mathbb C^n}= \Phi_iH^2_{\mathbb
C^m} \quad(\hbox{Rank}\, \Phi=m).
\end{equation}
If $\det \Phi = 0$, then by Lemma \ref{localrank}, $\hbox{Rank}\,
\Phi_i=\hbox{Rank}\, \Phi=m<n$, so that $\Phi_i \in
H^{\infty}_{M_{n\times m}}$. \ Thus $\mathcal H(\Phi_i)\cong
H^2_{\mathbb C^{n-m}}$, so that by (\ref{dkfgkmjhg}), $\hbox{\rm
codim}\bigl(\hbox{\rm ran}\, T_{\Phi}\bigr)=\dim\,\mathcal
H(\Phi_i)=\infty$, which gives (a). \ Towards (b),  suppose that
$\hbox{det}\, \Phi \neq 0$. \ If $\Phi$ is not rational, then
$\hbox{\rm rank}\,H_{\Phi^*}=\infty$, which gives the result. \  If
instead $\Phi \in H^{\infty}_{M_n}$ is rational, then we may write
$$
\Phi=A^*\Theta \quad \ (\hbox{left coprime}),
$$
where $\Theta \in H^{\infty}_{M_n}$.  \ Then $\det A
\neq 0$, and hence $T_A$ is injective. \ Thus
$$
\aligned \ker T_{\Phi}^*&=\ker T_{\Theta^*}T_A=\{f\in H^2_{\mathbb
C^n}: T_{\Theta^*}T_Af=0\}\\
&\cong T_A\Bigl\{f:T_{\Theta^*}T_Af=0\Bigr\}\\
&=(\hbox{\rm ran}\, T_A) \bigcap (\ker T_{\Theta^*})\subseteq
\mathcal H(\Theta).
\endaligned
$$
Therefore
$$
\hbox{\rm codim}\bigl(\hbox{\rm ran}\,T_{\Phi}\bigr) =\dim \ker
T_{\Phi}^*\le\dim \mathcal H(\Theta) =\dim\,\bigl(\ker
H_{\Phi^*}^*\bigr)^{\perp}=\hbox{\rm rank}\,H_{\Phi^*},
$$
which gives the result.   \
\end{proof}

\bigskip

\begin{corollary}\label{thmcodimgh}
Suppose $\Phi \in H^{\infty}_{M_n}$ be such that $\Phi^*$ is of
bounded type. \ Thus we may write
$$
\Phi=
\begin{cases}
\, \Phi_i\Phi_e\  &\hbox{\rm (inner-outer factorization)}\\
\, A^*\Theta  &\hbox{\rm (left coprime)}.
\end{cases}
$$
If  $\det \Phi \neq 0$, then $\deg(\Phi_i)\le \deg(\Theta)$. \ In
particular, if $\Phi$ is a rational function, then $\Phi_i$ is a
finite Blaschke-Potapov product.
\end{corollary}

\begin{proof}   \ Since $\det \Phi \neq 0$,
it follows from Theorem \ref{thmcodimgh2.41} that
$$
\hbox{deg}\, (\Phi_i)=\hbox{\rm codim}\bigl(\hbox{\rm ran}\,
T_{\Phi}\bigr) \leq \hbox{\rm rank}\,H_{\Phi^*}=\dim \mathcal H
(\widetilde{\Theta})=\hbox{deg}\, (\Theta).
$$
The last assertion is obvious. \
\end{proof}

\medskip

The following example illustrates Corollary \ref{thmcodimgh}. \

\begin{example}
(i) Let
$$
\Phi:=\begin{pmatrix} z&z\\ 0&0
\end{pmatrix}.
$$
Then $\Phi$ has the following inner-outer factorization:
$$
\Phi=\begin{pmatrix} z\\0\end{pmatrix} \begin{pmatrix}
1&1\end{pmatrix}\equiv \Phi_{i}\Phi_e.
$$
Thus we have
$$
\deg\,(\Phi_i)=\dim \begin{pmatrix} zH^2\\
0\end{pmatrix}^\perp =\infty,
$$
On the other hand,
$$
\Phi=\Biggl(\frac{1}{\sqrt{2}}\begin{pmatrix}2&0\\0&0\end{pmatrix}\Biggr)^*
\Biggl(\frac{1}{\sqrt{2}}\begin{pmatrix}
z&z\\1&-1\end{pmatrix}\Biggr)\equiv A^* \Theta \quad(\hbox{left
coprime}).
$$
Thus $\hbox{deg}\ (\Theta)=1< \hbox{deg}\, (\Phi_i)$. \ Note that
$\det \Phi=0$.

\medskip

(ii) Let
$$
\Phi:=\begin{pmatrix} z&0\\ 0&z+2\end{pmatrix}.
$$
Then we have
$$
\Phi=\begin{cases}
\begin{pmatrix} z&0\\ 0&1\end{pmatrix}\begin{pmatrix} 1&0\\ 0&z+2\end{pmatrix}\equiv \Phi_i\Phi_e\ \ \hbox{(inner-outer factorization)}\\
\begin{pmatrix} 1&0\\ 0&1+2z\end{pmatrix}^*\begin{pmatrix} z&0\\0&z\end{pmatrix}\equiv A^*\Theta\ \ \hbox{(left coprime)}.
\end{cases}
$$
Thus by (\ref{dettheta}), we have $\deg
(\Theta)=\deg(\det\Theta)=\deg (z^2)=2$, but $\deg(\Phi_i)=\deg
(\det \Phi_i)=\deg (z)=1$.
\end{example}

\bigskip

\begin{corollary} \label{lemou234567} \
Suppose $\Phi \in H^{\infty}_{M_n}$  is a rational function with
$\det \Phi \neq 0$. \ Then for a finite Blaschke product $\omega$,
we have
$$
\Phi \circ \omega =(\Phi_i \circ \omega) \cdot (\Phi_e \circ \omega)
\quad (\hbox{\rm inner-outer factorization}).
$$
\end{corollary}
\begin{proof} \ Suppose that $\Phi$ is a rational function and $\det \Phi \neq 0$.  \
Then by Theorem \ref{thmcodimgh}, $\Phi_i$ is a finite
Blaschke-Potapov product. Thus $\Phi_i \circ \omega$ is a finite
Blaschke-Potapov product, so that by Theorem \ref{thmouter567} we
have
$$
\Phi \circ \omega =(\Phi_i \circ \omega) \cdot (\Phi_e \circ \omega)
\quad (\hbox{inner-outer factorization}).
$$
This completes the proof. \ \end{proof}

To proceed, we need:


\begin{lemma}\label{asdlem2.1}
If $\Theta_1, \Theta_2\in H^\infty_{M_n}$ are rational inner
functions then the following are equivalent:
\begin{itemize}
\item[(a)] $\Theta_1$ and $\Theta_2$ are right coprime;
\item[(b)] $\ker\Theta_1(\alpha)\cap \ker \Theta_2(\alpha)=\{0\}$ for any $\alpha\in\mathbb D$.
\end{itemize}
\end{lemma}

\begin{proof}
Immediate from \cite[Corollary 3.2]{CHKL}.
\end{proof}

\begin{theorem}\label{factorization6879}
Let $\Phi \in H_{M_n}^{\infty}$ be a matrix rational function with
$\det \Phi\ne 0$. \ Thus we may write
$$
\Phi= A^*\Theta  \ \ (\hbox{\rm left coprime}),
$$
where $\Theta$ is a finite Blaschke-Potapov product. \ Then, for a
finite Blaschke product $\omega$, we have
$$
\Phi\circ \omega=\bigl(A \circ \omega\bigr)^*\bigl(\Theta \circ
\omega\bigr) \ \ \hbox{\rm(left coprime)}.
$$
\end{theorem}

\begin{proof} \
Note that $A$ is rational and $\det A\ne 0$. \ Write
$$
A=A_iA_e\quad \hbox{(inner-outer factorization)}.
$$
By Corollary \ref{lemou234567}, $A\circ \omega$ can be written as:
\begin{equation}\label{2.43-1}
A \circ \omega =(A_i \circ \omega) \cdot (A_e \circ \omega) \quad
(\hbox{inner-outer factorization}).
\end{equation}
Since $\Theta$ and $A_i$ are left coprime and hence
$\widetilde\Theta$ and $\widetilde{A_i}$ are right coprime it
follows from Lemma \ref{asdlem2.1} that $\widetilde{\Theta\circ
\omega}=\widetilde\Theta\circ\widetilde{\omega}$ and
$\widetilde{A_i\circ \omega}=\widetilde{A_i}\circ\widetilde{\omega}$
are also right coprime. \ Thus $\Theta\circ\omega$ and
$A_i\circ\omega$ are left coprime. \ It thus follows from
(\ref{2.43-1}) that $\Theta\circ \omega$ and $A\circ \omega$ are
also left coprime. \ This completes the proof.
\end{proof}

\begin{corollary}
Let $\Phi \in H_{M_n}^{\infty}$ be a matrix rational function of the
form
$$
\Phi=\theta A^* \ \ (\hbox{\rm coprime}),
$$
where $\theta$ is a finite Blaschke product. \ Then, for a finite
Blaschke product $\omega$, we have
$$
\Phi\circ \omega=\bigl(\theta \circ \omega\bigr)\cdot\bigl(A \circ
\omega\bigr)^*\ \ \hbox{\rm(coprime)}.
$$
\end{corollary}

\begin{proof} \
Since $I_\theta$ and $A$ are coprime, we have, by Lemma
\ref{lem1.7}, $\det A\ne 0$, and hence, $\det\Phi\ne 0$. Thus the
result follows at once from Theorem \ref{factorization6879}. \
\end{proof}

\begin{theorem}\label{lem5.17}
Let $\Phi \in H_{M_n}^{2}$ be such that $\Phi^*$ is of bounded type.
\ Thus we may write
$$
\Phi=\Theta A^* \ \ (\hbox{\rm right coprime}).
$$
Then, for a Blaschke factor $\theta$, we have
$$
\Phi\circ \theta=\bigl(\Theta \circ \theta\bigr)\bigl(A \circ
\theta\bigr)^*\ \ \hbox{\rm(right coprime)}.
$$
\end{theorem}

\begin{proof} \ Let $\theta\equiv b_{\alpha}$ be arbitrary Blaschke factor. \ Suppose that $\Theta$ and $A$ are right coprime, and
assume, to the contrary, that $\Theta \circ \theta$ and $A \circ
\theta$ are not right coprime. \ Write
$$
\Theta_1:= \Theta \circ \theta \quad \hbox{and} \quad A_1:= A \circ
\theta.
$$
Since $\Theta_1$ and $A_1$ are not right coprime, there exists a
nonconstant inner matrix function $\Delta \in H^{\infty}_{M_n}$ such
that
$$
\Theta_1=\Theta_2 \Delta \quad \hbox{and} \quad A_1=A_2 \Delta
\qquad(\Theta_2, A_2 \in H^{\infty}_{M_n}).
$$
Note that $\theta^{-1}=b_{-\alpha}$. \ Put $\Omega:=\Delta \circ
b_{-\alpha}$. \ Then $\Omega$ is a nonconstant inner matrix
function, and
$$
\Theta=\Theta_1\circ b_{-\alpha}=\left(\Theta_2 \circ b_{-\alpha}
\right) \Omega
$$
and
$$
A=A_1\circ b_{-\alpha}=\left(A_2 \circ b_{-\alpha} \right) \Omega.
$$
Since $\Theta_2\circ b_{-\alpha}$ and $A_2\circ b_{-\alpha}$ are in
$H^{\infty}_{M_n}$, $\Omega$ is a common right inner divisor of
$\Theta$ and $A$, so that $\Theta$ and $A$ are not right coprime, a
contradiction. \
\end{proof}


%
%
%
%

\newpage

\chapter{Tensored-scalar singularity}

\medskip

In this chapter we ask the following question: How does one define a
singularity for a matrix function of bounded type? \ Conventionally,
the singularity (or the pole) of matrix $L^\infty$-functions is
defined by a singularity (or a pole) of some entry of the matrix
functions (cf. \cite{BGR}, \cite{BR}). \ However, when the
singularity corresponds to a pole, we shall use the opposite
convention and say that the matrix $L^\infty$-function has a pole of
order $m$ at the point $\alpha\in\mathbb D$ if every nonzero entry
has a pole at $\alpha$ of order at least $m$ and some entry has a
pole of order exactly $m$. \ This notion of singularity (for the
pole case) is more suitable for our study of Toeplitz and Hankel
operators. \ 

In this chapter, we give a new notion of
``tensored-scalar singularity." \ This new definition takes
advantage of the Hankel operator, as if we were using it to
characterize functions of bounded type (cf.~(\ref{1.1})). \ This
notion contributes to give an answer to the question: Under what
conditions does it follow that if the product of two Hankel
operators with matrix-valued bounded type symbols is zero then
either of the operators is zero? \

\medskip

By our conventions to be defined here,
$\Phi:=\left(\begin{smallmatrix} \frac{1}{z}&0\\ 0&1 \end{smallmatrix}\right)$
does not have a singularity at all but
$\Psi:=\left(\begin{smallmatrix} \frac{1}{z}&0\\ 0&\frac{1}{z^2} \end{smallmatrix}\right)$
has a pole of order $1$ at zero. \

\medskip

To motivate our new idea,
let us carefully consider the scalar-valued case. \

If $\varphi\in L^\infty$ is of bounded type then we may write, in
view of (\ref{2.2}),
\begin{equation}\label{111}
\varphi_-=\omega \overline{a} \quad \hbox{(coprime)}.
\end{equation}
Since $\varphi=\frac{a}{\omega}+\varphi_+$, the singularities of
$\varphi$ come from $\omega$. \ Thus if $\theta$ is a nonconstant
inner divisor of $\omega$ then $\theta$ leads to singularities of
$\varphi$. \ Consequently, we can say that $\varphi$ has a
singularity (with respect to $\theta$) if and only if
 $\theta$ is an inner divisor of $\omega$ if and only if $\omega H^2\subseteq \theta H^2$;
 in other words,
\begin{equation}\label{scalar-singularity}
\ker H_\varphi =\omega H^2\subseteq \theta H^2.
\end{equation}
As a new notion of a singularity for a matrix function, we now adopt the
matrix-valued version of the scalar-valued case
(\ref{scalar-singularity}).

\bigskip

\begin{definition}\label{df4.5}
Let $\Phi\in L^\infty_{M_n}$ be of bounded type. \ We say that
$\Phi$ has a {\it tensored-scalar singularity {\rm (}with respect to $\theta${\rm )}} if
there exists a nonconstant inner function $\theta$ such that
$$
\ker H_\Phi\subseteq \theta H^2_{\mathbb C^n}\,.
$$
\end{definition}

\bigskip

If $\varphi$ is a rational function, then $\omega$ in
(\ref{scalar-singularity}) is a finite Blaschke product, so a tensored-scalar
singularity reduces to a matrix pole (cf. \cite[Definition
3.5]{CHKL}) -- hereafter (to avoid confusion with the convention),
we shall call it
a {\it tensored-scalar pole}. \
In particular, we say that $\Phi$ has a {\it tensored-scalar singularity
of order $p$} ($p \in \mathbb{Z}_+$)
if $p$ is the degree of $\theta_0$, where
$$
\theta_0:=\hbox{l.c.m.} \{\theta:\ \Phi\ \hbox{has a  tensored-scalar singularity with
respect to $\theta$}\}.
$$
Moreover, if $\theta_0=b_\alpha^p$, then we say that $\Phi$ has a tensored-scalar pole of
order $p$ at $\alpha$. \

On the other hand, we note that every non-analytic bounded type
function $\varphi\in L^\infty$ has, trivially, a tensored-scalar singularity;
for, if $\varphi\notin H^\infty$, then by Beurling's Theorem, $\ker
H_\varphi=\theta H^2$ for a nonconstant inner function $\theta\in
H^\infty$. \

\bigskip

We also observe:

\begin{lemma}\label{lem55.2-1}
Let $\Phi \in L^{\infty}_{M_n}$ be of bounded type, and write
$\Phi=A\Theta^*$ \hbox{\rm (right coprime)}. Then $\Phi$ has a
tensored-scalar singularity if and only if $\Theta$ has a nonconstant
diagonal-constant inner divisor.
\end{lemma}

\begin{proof}
Observe that $\ker H_\Phi=\Theta H^2_{\mathbb C^n}$, so that $\ker
H_\Phi \subseteq \theta H^2_{\mathbb C^n}$ ($\theta\in H^\infty$ is
inner) if and only if $\Theta H^2_{\mathbb C^n}\subseteq \theta
H^2_{\mathbb C^n}$ if and only if $I_\theta$ is an inner divisor of
$\Theta$.
\end{proof}

In view of Lemma \ref{lem55.2-1}, if $\Phi\in L^\infty_{M_n}$ has
the following coprime factorization:
$$
\Phi_-=\theta B^*\ \ \hbox{\rm (coprime)};
$$
then, clearly, $\Phi$ has a tensored-scalar singularity. \

\medskip

On the other hand, it is well-known that if $\varphi,\psi\in
L^\infty$, then
\begin{equation}\label{5.14}
H_{\varphi}H_{\psi}=0\ \Longrightarrow\ H_\varphi=0\ \ \hbox{or}\ \
H_\psi=0.
\end{equation}
But (\ref{5.14}) may fail for matrix-valued functions. \ For
example, if we take
\begin{equation}\label{3.5}
\Phi:=\begin{pmatrix}1&0\\0&\overline{z}
\end{pmatrix}\quad\hbox{and}\quad
\Psi:=\begin{pmatrix}\overline{z}&0\\0&1
\end{pmatrix},
\end{equation}
then $H_{\Phi}H_{\Psi}=0$, but $H_\Phi\ne 0$ and $H_\Psi\ne 0$.

\medskip

We will now try to find a general condition for (\ref{5.14}) to hold
in the matrix-valued case. \ To do so, we need:

\begin{proposition}\label{pro55.10}
Let $\Phi \in L^{\infty}_{M_n}$ be of bounded type. \ Then $\Phi$
has a tensored-scalar singularity with respect to $\theta$ if and only if
$\widetilde{\Phi}$ has a tensored-scalar singularity with respect to
$\widetilde{\theta}$.
\end{proposition}

\begin{proof} \ Write
$$
\Phi= A \Theta^*\ \ (\hbox{right coprime}).
$$
In view of Lemma \ref{lem55.2-1}, we may write $\Theta=I_\theta
\Theta_1$ ($\Theta_1$ inner). \
Since $A$ and $I_\theta$ are right coprime, by Theorem \ref{lem33.9}
$A$ and $I_\theta$ are left coprime, so that $\widetilde{A}$ and
$I_{\widetilde{\theta}}$ are right coprime. \ It thus follows from
(\ref{RCD}) that
$$
\ker H_{\widetilde{\Phi}}=\ker H_{\widetilde{\Theta}_1^*
\widetilde{A}\overline{I_{\widetilde{\theta}}}}\subseteq \ker
H_{\widetilde{A}\overline{I_{\widetilde{\theta}}}}=\widetilde{\theta}
H^2_{\mathbb C^n},
$$
which implies that $\widetilde{\Phi}$ has a tensored-scalar singularity with
respect to $\widetilde{\theta}$. \ For the converse, we use the fact
$\Phi=\widetilde{\widetilde{\Phi}}$ and
$\theta=\widetilde{\widetilde{\theta}}$. \
This completes the proof.
\end{proof}

\bigskip

We now have:

\begin{theorem}\label{lem33.5}
Let $\Phi,\Psi\in L^\infty_{M_n}$ be of bounded type.  \ If $\Phi$
or $\Psi$ has a tensored-scalar singularity then
$$
H_\Psi H_\Phi=0\ \Longrightarrow \ H_\Phi=0 \ \hbox{or} \ H_\Psi=0.
$$
\end{theorem}

\begin{proof} \  Write
$$
\Phi= \Delta^* B\ \ (\hbox{left coprime})\quad\hbox{and}\quad \Psi=A
\Theta^* \ \ (\hbox{right coprime}).
$$
Suppose that $\Phi$ has a tensored-scalar singularity. \ By Lemma
\ref{lem55.2-1}, there exists a nonconstant inner function $\delta$
such that
$$
\Delta= \delta \Delta_1 \quad(\Delta_1\ \hbox{inner}).
$$
Suppose $H_\Psi H_\Phi=0$. \ Since $\widetilde{\Delta}$ and
$\widetilde{B}$ are right coprime, it follows from (\ref{RCD}) that
\begin{equation}\label{33.19}
\mathcal{H} ({I_{\widetilde{\delta}}})
    \subseteq\mathcal{H} ({\widetilde{\Delta}})
=\Bigl(\hbox{ker}\,H_\Phi^*\Bigr)^\perp
        \subseteq \hbox{ker}\, H_\Psi
           =\Theta H^2_{\mathbb C^n}.
\end{equation}
Write $\Theta\equiv\begin{pmatrix}
\theta_{ij}\end{pmatrix}_{i,j=1}^n$. \ Since $\widetilde{\delta}$ is
not constant, $\mathcal{H}_{\widetilde{\delta}}$ has at least an
outer function $\xi$ that is invertible in $H^\infty$ (cf.
\cite[Lemma 3.4]{CHL1}). \ It thus follows from (\ref{33.19}) that
$\Theta^*(\xi,0,0,\cdots)^t\in H^2_{\mathbb C^n}$, which implies
$\overline\theta_{1j}\xi\in H^2$ for each $j=1,\cdots,n$. \
Similarly, we can show that
$$
\overline{\theta}_{ij} \,\xi \in H^2\quad\hbox{for each
$i,j=1,2,\cdots,n$},
$$
so that $\overline{\theta}_{ij}\in \frac{1}{\xi} H^2\subseteq H^2$
for each $i,j=1,2,\cdots,n$. \ Therefore each $\theta_{ij}$ is
constant and hence, $\Theta$ is a constant unitary. \ Therefore
 $\Psi \in H^2_{M_n}$, which gives $H_{\Psi}=0$.

If instead $\Psi$ has a tensored-scalar singularity, then by Proposition
\ref{pro55.10}, $\widetilde{\Psi}$ has a tensored-scalar singularity. \ Thus
if $H_\Psi H_\Phi=0$, then
$H_{\widetilde{\Phi}}H_{\widetilde{\Psi}}=0$, so that
$H_{\Phi}=H_{\widetilde{\Phi}}^*=0$ by what we proved just above.
\end{proof}

\medskip

We observe that if $\Phi\in L^\infty_{M_n}$ is normal then by
(\ref{1.3}),
$$
[T_\Phi^*, T_\Phi]=H_{\Phi^*}^* H_{\Phi^*} - H_\Phi^* H_\Phi.
$$
On the other hand, we recall (\cite[Theorem 3.3]{GHR},
\cite[Corollary 2]{Gu2}) that if $\Phi\in L^\infty_{M_n}$ then
\begin{equation}\label{21}
H_{\Phi^*}^* H_{\Phi^*} \ge H_\Phi^* H_\Phi\ \Longleftrightarrow
\exists \, K\in H^\infty_{M_n}\ \hbox{such that}\ ||K||_\infty\le 1\
\hbox{and}\ \Phi-K\Phi^*\in H^\infty_{M_n}.
\end{equation}

\bigskip

The following theorem is of independent interest. \

\medskip

\begin{theorem}\label{lem3.13}
Let $\Phi, \Psi\in L^\infty_{M_n}$ be of bounded type. \ If $\Phi$
or $\Psi$ has a tensored-scalar singularity then
$$
H_{\Phi}^*H_{\Phi}=H_{\Psi}^*H_{\Psi} \Longleftrightarrow \Phi-U\Psi
\in H^{\infty}_{M_n}
$$
for some unitary constant matrix $U \in M_n$.
\end{theorem}

\begin{proof}
Write
$$
\Phi=A\Theta^* \quad(\hbox{right coprime}).
$$
If $\Phi$ has a tensored-scalar singularity then by Lemma \ref{lem55.2-1},
$\Theta$ has an inner divisor of the form $I_\theta$, with $\theta$
a nonconstant inner function. \ Thus we may write $\Theta=\theta
\Theta_1$ for an inner matrix function $\Theta_1\in
H^{\infty}_{M_n}$. \ It follows that $A$ and $I_\theta$ are right
coprime, so that by Theorem \ref{lem33.9}, $A$ and $I_\theta$ are
left coprime, and hence $\widetilde{A}$ and $I_{\widetilde{\theta}}$
are right coprime. \ Thus
\begin{equation}\label{35-1}
\hbox{ker}H_{\Phi}^*=\hbox{ker}H_{\widetilde{\Phi}}
=\hbox{ker}H_{\overline{\widetilde{\theta}}\widetilde{\Theta}_1^*\widetilde{A}}
\subseteq\hbox{ker}H_{\overline{\widetilde{\theta}}\widetilde{A}}=\widetilde{\theta}
H^2_{\mathbb C^n}.
\end{equation}
If $ H_{\Phi}^*H_{\Phi}=H_{\Psi}^*H_{\Psi}$, then
$\hbox{ker}H_{\Psi}=\hbox{ker}H_{\Phi}\subseteq \theta H^2_{\mathbb
C^n}$, which implies that $\Psi$ also has a tensored-scalar singularity.  \
Therefore without loss of generality we may assume that $\Phi$ has a
tensored-scalar singularity.  \ Since
$H_{\Phi}^*H_{\Phi}=H_{\Psi}^*H_{\Psi}$, we have
$H_{\Phi_-^*}^*H_{\Phi_-^*}=H_{\Psi_-^*}^*H_{\Psi_-^*}$. \ Then by
(\ref{21}), there exist $U,U^\prime\in H^\infty_{M_n}$ with
$||U||_\infty\le 1$ and $||U^\prime||_\infty\le 1$ such that
\begin{equation}\label{36-1}
\Phi_-^*-U\Psi_-^*\in H_{M_n}^2 \quad \hbox{and} \quad
\Psi_-^*-U^{\prime}\Phi_-^*\in H_{M_n}^2.
\end{equation}
It follows from (\ref{36-1}) that $\Phi_-^*-UU^{\prime}\Phi_-^*\in
H_{M_n}^2$, so that $H_{\Phi_-^*}-H_{UU^{\prime}\Phi_-^*}=0$, and
hence $(I-T_{\widetilde {UU^{\prime}}}^*) H_{\Phi}=0$, which
together with (\ref{35-1}) implies
\begin{equation}\label{39-9}
\mathcal H({I_{\widetilde{\theta}}}) \subseteq \hbox{ran}\,H_{\Phi}
\subseteq \hbox{ker}\, (I-T_{\widetilde{ UU^{\prime}}}^*).
\end{equation}
Thus, we have $ F=T_{\widetilde{ UU^{\prime}}}^*F$  for each $F\in
\mathcal H ({I_{\widetilde{\theta}}})$. \ But since
$||\widetilde{UU^{\prime}}^*||_{\infty} = ||UU^{\prime}||_{\infty}
\leq 1$, we should have $\widetilde{ UU^{\prime}}^*F \in
H^2_{\mathbb C^n}$, so that $F=\widetilde{ UU^{\prime}}^*F$ for each
$F\in \mathcal H ({I_{\widetilde{\theta}}})$. \
Thus $UU^\prime$ is constant, and actually $UU^\prime=I_n$. \
Therefore $U$ is a unitary constant and by (\ref{36-1}),
$\Phi-U\Psi\in H^\infty_{M_n}$. \ The converse is evident.
\end{proof}

\medskip

For $A \in L^2_{M_n}$ and an inner matrix function $\Theta \in
H^2_{M_n}$, we define
$$
\left(T_A\right)_{\Theta} \label{TAT}:=P_{\mathcal H (\Theta)}T_A
\vert_{\mathcal H (\Theta)},
$$
that is, $\left(T_A\right)_\Theta f := P_{\mathcal{H}(\Theta)} Af$
for each $f\in\mathcal{H}(\Theta)$. \ Note that $T_A$ is densely
defined and possibly unbounded. \ However $(T_A)_\Theta$ may be
bounded under certain conditions. \ To see this, let
$$
\overline{(BMO)_{M_n}}\label{bmomn}:= \Bigl\{\Phi\equiv
(\varphi_{ij})\in L^2_{M_n}: \ \overline{\varphi_{ij}}\in
BMO\Bigr\}.
$$
We then have:

\begin{lemma}\label{BMOmn}
Let $A\in H^2_{M_n}$ and $\Theta\in H^\infty_{M_n}$ be an inner
matrix function. \ If $G\equiv A^*\Theta\in H^2_{M_n}\cap
\overline{(BMO)_{M_n}}$, then $(T_A)_\Theta$ is bounded.
\end{lemma}

\begin{proof}
Since $G^*\in (BMO)_{M_n}$, we can find a matrix function $F\in
H^2_{M_n}$ such that
$$
\Phi\equiv G^*+F\in L^\infty_{M_n}.
$$
(cf. \cite[Theorem 1.3]{Pe}). \ Thus $\Theta\Phi\in L^\infty_{M_n}$.
\ Observe that
$$
\begin{aligned}
P_{\mathcal H(\Theta)} T_{\Theta\Phi}\vert_{\mathcal H(\Theta)}
  &=P_{\mathcal{H}(\Theta)} T_{\Theta(\Theta^*A+F)}\vert_{\mathcal H(\Theta)}\\
    &=P_{\mathcal{H}(\Theta)} T_{A+\Theta F}\vert_{\mathcal H(\Theta)}\\
      &=P_{\mathcal{H}(\Theta)} T_{A}\vert_{\mathcal H(\Theta)}\\
         &=(T_A)_\Theta,
\end{aligned}
$$
which implies  $(T_A)_\Theta$ is bounded.
\end{proof}

\medskip

In view of Lemma \ref{BMOmn}, $(T_A)_\Theta$ is understood in the
sense that the compression $(T_A)_\Theta$ is bounded even though
$T_A$ is possibly unbounded if $A^*\Theta\in H^2_{M_n}\cap
\overline{(BMO)_{M_n}}$. \ (In particular, we are interested in the
case $A^*\Theta\equiv \Phi_-$ for a matrix function $\Phi\in
L^\infty_{M_n}$.) \

\medskip

Here we pause to ask, when is $\left(T_A\right)_\Theta$ injective. \
First of all,  we take a look at possible cases when $n=2$.

\medskip

\begin{itemize}
\item[(i)]
Let
$$
A:=\begin{pmatrix}0&1\\0&0\end{pmatrix} \quad \hbox{and} \quad
\Theta:=\begin{pmatrix}z&0\\0&1\end{pmatrix}.
$$
Put $f:=\begin{pmatrix} 1&0\end{pmatrix}^t$. \ Then $f \in \mathcal
H (\Theta)$, $Af=0$, and $\left(T_A\right)_{\Theta}f=0$. \ Thus
$\left(T_A\right)_{\Theta}$ is not injective. \

\medskip

\item[(ii)]
Let
$$
A:=\begin{pmatrix}0&1\\0&0\end{pmatrix} \quad \hbox{and} \quad
\Theta:=\begin{pmatrix}1&0\\0&z\end{pmatrix}.
$$
Put $f:=\begin{pmatrix} 0& 1\end{pmatrix}^t$. \ Then $f \in \mathcal
H (\Theta)$, $Af=\begin{pmatrix} 1&0\end{pmatrix}^t\ne 0$, and
$\left(T_A\right)_{\Theta}f=0$. \ Thus $\left(T_A\right)_{\Theta}$
is not injective. \

\medskip

\item[(iii)]
Let
$$
A:=\begin{pmatrix}z \delta&1\\0&1\end{pmatrix} \ \ \hbox{and} \ \
\Theta:=\begin{pmatrix}b_{\alpha}z&0\\0&z\end{pmatrix} \ \
(\hbox{where}\
\delta(z):=\frac{\sqrt{1-|\alpha|^2}}{1-\overline{\alpha}z}).
$$
Put $f:=\begin{pmatrix} b_\alpha& 0\end{pmatrix}^t$. \ Then $f \in
\mathcal H (\Theta)$, $Af=\begin{pmatrix} zb_\alpha\delta&
0\end{pmatrix}^t\ne 0$, and $\left(T_A\right)_{\Theta}f=0$. \ Thus
$\left(T_A\right)_{\Theta}$ is not injective. \ Note that $\Theta$
has a tensored-scalar singularity and
$$
\Theta A^*=\begin{pmatrix}
b_{\alpha}\overline{\delta}&0\\z&z\end{pmatrix}\in H^2_{M_2}.
$$
Note that $A$ and $\Theta$ are not right coprime.
\end{itemize}

\medskip

We observe:

\begin{lemma} \label{23.2222}
Let $A \in H^{2}_{M_n}$ and let $\Theta \in H^{\infty}_{M_n}$ be an
inner function such that
$$
\Phi\equiv A^* \Theta = \Delta B^*\in H^2_{M_n}\cap
\overline{(BMO)_{M_n}},
$$
where $\Delta$ and $B$ are right coprime. \ Then $ \ker \,
(T_A)_{\Theta}\label{TPhiTheta}=\mathcal H(\Theta) \cap \Delta
H^2_{\mathbb C^n}$. \ In particular, if $\Phi^*$  has a tensored-scalar
singularity with respect to $\theta$, then $ \ker \,
(T_A)_{\Theta}\label{TPhiTheta2}\subseteq I_{\theta} \mathcal
H(\Theta_1)$ with $\Theta_1:=\overline{\theta}\Theta$.
\end{lemma}

\begin{proof}  \
Let $f \in \mathcal H(\Theta)$. \ Then we have
$$
\begin{aligned}
(T_A)_\Theta f=0
   &\Longleftrightarrow\ B\Delta^* f=\Theta^* Af\in H^2_{\mathbb C^n}\\
     &\Longleftrightarrow\ f\in \ker H_{B\Delta^*}=\Delta H^2_{\mathbb  C^n},
\end{aligned}
$$
which implies $\ker \, (T_A)_{\Theta}\label{TPhiTheta3}=\mathcal
H(\Theta) \cap \Delta H^2_{\mathbb C^n}$. \ Suppose that $\Phi^*$
has a tensored-scalar singularity with respect to $\theta$. \
Then, by Lemma
\ref{lem55.2-1}, we may write
$$
\Theta=I_\theta  \Theta_1 \ \hbox{and} \ \Delta=I_{\theta}\Delta_1
\quad \hbox{(where $\Theta_1$ and $\Delta_1$ are inner)}.
$$
We thus have
$$
\aligned\ker \, (T_A)_{\Theta}\label{TPhiTheta4}&=\mathcal H(\Theta)
\cap \Delta H^2_{\mathbb C^n}\subseteq \mathcal H(\Theta)
\cap I_{\theta} H^2_{\mathbb C^n}\\
&=\left[\mathcal H(I_{\theta})\bigoplus I_{\theta} \mathcal
H(\Theta_1)\right]
\cap I_{\theta} H^2_{\mathbb C^n}\\
&=I_{\theta} \mathcal H(\Theta_1),
\endaligned
$$
which gives the result. \
\end{proof}

\medskip

\begin{corollary}\label{cor2.7} Let $A \in H^{2}_{M_n}$ and
$\Theta\equiv I_\theta \in H^{\infty}_{M_n}$ for an inner function
$\theta$ such that $A^*\Theta\in H^2_{M_n}\cap
\overline{(BMO)_{M_n}}$. \ If $A$ and $\Theta$ are coprime then
$(T_A)_{\Theta}\label{TPhiTheta5}$ is injective. \ In particular, if
$\theta$ is a finite Blaschke product, then $(T_A)_\Theta$ is
invertible.
\end{corollary}

\begin{proof} \ Since $A$ and $\Theta\equiv I_\theta$ are coprime
it follows that
$$
A^* \Theta = \Theta A^* \quad (\hbox{coprime}).
$$
It thus follows from Lemma \ref{23.2222} that $ \ker
(T_A)_{\Theta}\label{TPhiTheta6}=\mathcal H(\Theta) \cap \Theta
H^2_{\mathbb C^n}=\{0\}, $ which implies that $(T_A)_{\Theta}$ is
injective. \ If $\theta$ is a finite Blaschke product, then
$(T_A)_\Theta$ is a finite dimensional operator, and it follows that
$(T_A)_\Theta$ is invertible.
\end{proof}

\begin{theorem}\label{thm82.7}  \ Let $A \in H^{2}_{M_n}$ and let
$\Theta=I_{\theta} \in H^{\infty}_{M_n}$ for an inner function
$\theta$ such that $A^*\Theta\in H^2_{M_n}\cap
\overline{(BMO)_{M_n}}$. \ If $A$ and $\Theta$ are  coprime, then
$(T_A)_\Theta$ has a linear inverse (possibly unbounded).
\end{theorem}

\begin{proof} \ By Corollary \ref{cor2.7}, $(T_A)_\Theta$ is injective.  \
Thus it suffices to show that $(T_A)_\Theta^*$ is injective. \
Observe that for $U,V \in \mathcal H(\Theta)$,
$$
\bigl \langle (T_{A})_{\Theta}\, U, \ V \bigr\rangle
 =\bigl\langle A U, \ V \bigr\rangle
 = \int_{\mathbb T}\hbox{tr}\,\Bigl((A^*V)^*U\Bigr)\,d\mu
   =\bigl\langle U, A^*V\bigr\rangle,
$$
which implies
\begin{equation}\label{***}
(T_{A})^*_{\Theta} = (T_{A^*})_\Theta.
\end{equation}
On the contrary, assume that $(T_{A})^*_\Theta$ is not injective.  \
Then there exists a nonzero vector $f \in \mathcal H(\Theta)$ such
that
$$
0=(T_{A})^*_{\Theta} f =P_{\mathcal H(\Theta)}(A^*f),
$$
which gives $A^*f \in \Theta H^2_{\mathbb C^n}$. \ Since
$\Theta=I_{\theta}$ and $f \in \mathcal H(\Theta)$, it follows that
$$
\Theta^*(A^* f)=A^*(\Theta^* f) \in H^2_{\mathbb C^n} \cap
(H^2_{\mathbb C^n})^{\perp}=\{0\},
$$
which implies that $A^* f=0$. \ Since $A$ and $\Theta$ are coprime
it follows from Lemma \ref{lem1.7} that $A^*$ is invertible and
hence $f=0$, a contradiction. \ This proves $(T_{A})^*_\Theta$ is
injective.
\end{proof}

\medskip

\begin{remark} \
From Lemma \ref{23.2222}, we can see that if $A \in H^{2}_{M_n}$ and
$\Theta \in H^{\infty}_{M_n}$ is an inner matrix function such that
$$
A^* \Theta = \Delta B^*\in H^2_{M_n}\cap \overline{(BMO)_{M_n}},
$$
where $\Delta$ and $B$ are right coprime, then
\begin{equation}\label{oneone}
\Theta=\Delta\ \Longrightarrow\ (T_A)_\Theta\ \hbox{is injective.}
\end{equation}
But (\ref{oneone}) does not hold in general if $\Theta\ne \Delta$. \
For example, let
$$
A:=\begin{pmatrix}0&1\\0&0\end{pmatrix} \quad \hbox{and} \quad
\Theta:=\begin{pmatrix}z&0\\0&1\end{pmatrix}.
$$
Then
$$
A^* \Theta=\begin{pmatrix}
0&0\\z&0\end{pmatrix}=\begin{pmatrix}1&0\\0&z
\end{pmatrix}\begin{pmatrix} 0&1\\0&0 \end{pmatrix}^*\equiv \Delta B^*.
$$
Note that $\Delta$ and $B$ are right coprime. \ It thus follows from
Lemma \ref{23.2222} that
$$
\ker (T_A)_{\Theta}\label{TPhiTheta7}=\mathcal H(\Theta) \cap \Delta
H^2_{\mathbb C^n}= \mathbb C \oplus\{0\} \neq \{0\},
$$
which implies that  $(T_A)_{\Theta}$ is not injective. \ Here we
note that $\Theta$ and $A$ are not right coprime and $\hbox{det}\,A
= 0$. \ Thus we may expect that if
\begin{enumerate}
\item[(1)] \ $\Theta$ and $A$ are right coprime;
\item[(2)] \ $\hbox{det}\, A \neq 0$,
\end{enumerate}
then $(T_A)_{\Theta}$ is injective although $\Theta\ne \Delta$. \
However this is not the case. \ For example, let
$$
A:=\begin{pmatrix}0&1\\1&0\end{pmatrix} \quad \hbox{and} \quad
\Theta:=\begin{pmatrix}z&0\\0&1\end{pmatrix}.
$$
Then
$$
A^* \Theta=\begin{pmatrix}
0&1\\z&0\end{pmatrix}=\begin{pmatrix}1&0\\0&z
\end{pmatrix}\begin{pmatrix} 0&1\\1&0 \end{pmatrix}^*\equiv \Delta A^*.
$$
Note that $\Delta$ and $A$ are right coprime. \ It thus follows from
Lemma \ref{23.2222} that
$$
\ker (T_A)_{\Theta}\label{TPhiTheta8}=\mathcal H(\Theta) \cap \Delta
H^2_{\mathbb C^n}=\mathbb C\oplus \{0\} \neq \{0\},
$$
which implies that  $(T_A)_{\Theta}$ is not injective. \ Note that
$\Theta$ and $A$ are right coprime, and  $\hbox{det}\, A \neq 0$.

\end{remark}

%
%
%
%
%
%

\chapter{An interpolation problem and a functional calculus}

\medskip

In this chapter, we consider both an interpolation problem for
matrix functions of bounded type and a functional calculus for
compressions of the shift operator. \

We first review the classical Hermite-Fej\' er interpolation
problem, following \cite{FF}. \ Let $\theta$ be a finite Blaschke
product of degree $d$:
$$
\theta=e^{i\xi}\prod_{i=1}^N (\widetilde{b}_i)^{m_i} \quad
\left(\widetilde b_i(z):=\frac{z-\alpha_i}{1-\overline{\alpha}_i
z},\ \hbox{where}\ \alpha_i\in\mathbb D\right),
$$
where $d=\sum_{i=1}^N m_i$. \  For  our purposes, rewrite $\theta$
in the form
$$
\theta=e^{i\xi}\prod_{j=1}^d b_j,
$$
where
$$
b_j:=\widetilde{b}_k\quad\hbox{if}\ \sum_{l=0}^{k-1} m_l<j\le
\sum_{l=0}^k m_l
$$
and, for notational convenience, set $m_0:=0$. \  Let
\begin{equation}\label{2.10}
\varphi_j:=\frac{q_j}{1-\overline{\alpha}_j z} b_{j-1}b_{j-2}\cdots
b_1\quad (1\le j\le d),
\end{equation}
where $\varphi_1:=q_1 (1-\overline{\alpha}_1 z)^{-1}$ and
$q_j:=(1-|\alpha_j|^2)^{\frac{1}{2}}$ ($1\le j\le d$). \  It is well
known (cf. \cite{Ta}) that $\{\varphi_j\}_{j=1}^d$ is an orthonormal
basis for $\mathcal{H}(\theta)$. \

For our purposes we concentrate on the data given by sequences of
$n\times n$ complex matrices. \  Given the sequence $\{K_{ij}:\ 1\le
i\le N,\ 0\le j<m_i\}$ of $n\times n$ complex matrices and a set of
distinct complex numbers $\alpha_1,\hdots,\alpha_N$ in $\mathbb{D}$,
the classical Hermite-Fej\' er interpolation problem entails finding
necessary and sufficient conditions for the existence of a
contractive analytic matrix function $K$ in $H^{\infty}_{M_n}$
satisfying
\begin{equation} \label{2.11}
\frac{K^{(j)}(\alpha_i)}{j!} = K_{i,j} \qquad (1 \leq i \leq N, \ 0
\leq j < m_i).
\end{equation}
To construct a matrix polynomial $K(z)\equiv P(z)$ satisfying
(\ref{2.11}), let $p_i(z)$ be the polynomial of order $d-m_i$
defined by
$$
p_i(z):=\prod_{k=1,\\ k \neq i}^{N} \Bigl(\frac{z-
\alpha_k}{\alpha_i - \alpha_k}\Bigr)^{m_k}.
$$
Consider the matrix polynomial $P(z)$ of degree $d-1$ defined by
\begin{equation} \label{2.12}
P(z):= \sum_{i=1}^N \Biggl(K_{i,0}^{\prime} +K_{i,1}^{\prime}(z-
\alpha_i)+K_{i,2}^{\prime}(z- \alpha_i)^2+\cdots
+K_{i,m_i-1}^{\prime} (z-\alpha_i)^{m_i-1}\Biggr)\,p_i(z),
\end{equation}
where the $K_{i,j}^{\prime}$ are obtained by the following
equations:
$$
K_{i,j}^\prime=K_{i,j}-\sum_{k=0}^{j-1} \frac{K_{i,k}^\prime\,
p_i^{(j-k)}(\alpha_i)}{(j-k)!} \ \ (1\le i\le N;\ 0\le j<m_i)
$$
and $K_{i,0}^\prime=K_{i,0}$ ($1\le i\le N$). \  Then $P(z)$
satisfies (\ref{2.11}). \

Let $W\label{W}$ be the unitary operator from $\bigoplus_1^d \mathbb
C^n $ onto $\mathcal H (I_\theta)$ defined by
\begin{equation} \label{2.13}
W:=(I_{\varphi_1}, I_{\varphi_2}, \cdots , I_{\varphi_d}),
\end{equation}
where the $\varphi_j$ are the functions in (\ref{2.10}). \  It is
known \cite[Theorem X.1.5]{FF} that if $\theta$ is the finite
Blaschke product of order $d$, then $U_\theta$ is unitarily
equivalent to the lower triangular matrix $M\label{M}$ on
$\mathbb{C}^d$ defined by
\bigskip
\begin{equation} \label{2.14}
M:={\small \begin{pmatrix} \alpha_1&0&0&0&\cdots&\cdots\\
q_1 q_2&\alpha_2&0&0&\cdots&\cdots\\
-q_1 \overline{\alpha}_1 q_3&q_2 q_3&\alpha_3&0&\cdots&\cdots\\
q_1 \overline{\alpha}_2 \overline{\alpha}_3q_4&-q_2
\overline{\alpha}_3q_4&q_3 q_4&\alpha_4&\cdots&\cdots\\
-q_1 \overline{\alpha}_2 \overline{\alpha}_3
\overline{\alpha}_4q_5&q_2 \overline{\alpha}_3
\overline{\alpha}_4q_5&-q_3 \overline{\alpha}_4q_5&q_4
q_5&\alpha_5&\ddots\\
\vdots&\vdots&\vdots&\ddots&\ddots&\ddots
\end{pmatrix}.}
\end{equation}
Now let $P(z) \in H^{\infty}_{M_n}$ be a matrix polynomial of degree
$k$. \  Then the matrix $P(M)$ on $\mathbb C^{n \times d}$ is
defined by
\begin{equation} \label{2.15}
P(M):=\sum_{i=0}^{k} P_i \otimes M^i, \quad \text{where}\ P(z)=
\sum_{i=0}^{k} P_i z^i.
\end{equation}

If $M$ is given by (\ref{2.14}) and $P$ is the matrix polynomial
defined by (\ref{2.12}) then the matrix $P(M)$ is called the {\it
Hermite-Fej\' er matrix} determined by (\ref{2.15}). \  In
particular, if $(T_P)_{\Theta}:=P_{\mathcal H(\Theta)}
T_P\vert_{\mathcal H (\Theta)}$ is the compression of $T_P$ to
$\mathcal H(\Theta)$ (where $\Theta \equiv :=I_\theta$ for an inner
function $\theta$) (cf. p.\pageref{compress}), then it is known
\cite[Theorem X.5.6]{FF} that
\begin{equation}\label{2.17}
W^* (T_P)_{\Theta} W=P(M),
\end{equation}
which says that $P(M)$ is a matrix representation for
$(T_P)_{\Theta}\label{TPTheta}$. \

\bigskip

We now consider an interpolation problem for matrix functions of
bounded type. \ Our interpolation problem involves a certain
matrix-valued functional equation: $\Phi-K\Phi^*\in H^\infty_{M_n}$
(where $\Phi\in L^\infty_{M_n}$ and $K\in H^\infty_{M_n}$ is
unknown). \ We may ask, when does there exist such a matrix
$H^\infty$-function $K$? If such a function $K$ exists, how do we
find it? \ If $\Phi$ is a matrix-valued {\it rational} function,
this interpolation problem reduces to the classical Hermite-Fej\' er
Interpolation Problem. \

\medskip

More concretely, we consider the following question: For $\Phi\in
L^\infty_{M_n}$,
$$
\hbox{when does there exist a function $K\in H_{M_n}^\infty$
satisfying $\Phi-K\Phi^*\in H_{M_n}^\infty$\,?}
$$
For notational convenience we write, for $\Phi\in L^\infty_{M_n}$,
$$
\mathcal C(\Phi)\label{cphi}:=\Bigl\{K \in H^{\infty}_{M_n}:\
\Phi-K\Phi^*\in H^{\infty}_{M_n}\Bigr\}.
$$
Thus question can be rewritten as: For $\Phi\in L^\infty_{M_n}$,
\begin{equation}\label{interpo}
\hbox{when is $\mathcal{C}(\Phi)$ nonempty\,?}
\end{equation}
Question (\ref{interpo}) resembles an interpolation problem, as we
will see below. \ In this chapter we consider Question
(\ref{interpo}) for matrix functions of bounded type. \

\medskip

If $C(\Phi) \neq \emptyset$, i.e., there exists a function $K\in
H_{M_n}^\infty\ \hbox{such that}\ \Phi_-^*-K\Phi_+^*\in H_{M_n}^2$,
then $H_{\Phi_-^*}=T_{\widetilde K}^*H_{\Phi_+^*}$ for some $K\in
H_{M_n}^\infty$, which implies $\hbox{ker}\,H_{\Phi_+^*} \subseteq
\hbox{ker}\,H_{\Phi_-^*}$.\ Thus we have
\begin{equation}\label{necessary}
\mathcal C (\Phi)\ \neq\ \emptyset \Longrightarrow
\hbox{ker}\,H_{\Phi_+^*} \subseteq \hbox{ker}\,H_{\Phi_-^*}.
\end{equation}
Let $\Phi\equiv \Phi_-^*+\Phi_+\in L^\infty_{M_n}$ be such that
$\Phi$ and $\Phi^*$ are of bounded type. \ In view of (\ref{2.9}),
we may write
\begin{equation}\label{2.666}
\Phi_+= \Theta_1 A^*\quad \hbox{and}\quad  \Phi_- =\Theta_2 B^* \ \
\hbox{(right coprime)}.
\end{equation}
Thus, if $\mathcal C(\Phi)\ne \emptyset$, then by (\ref{necessary})
and (\ref{RCD}),  $\Theta_1 H^2_{\mathbb C^n}\subseteq \Theta_2
H^2_{\mathbb C^n}$. \ It then follows (cf. \cite[Corollary
IX.2.2]{FF}) that $\Theta_2$ is a left inner divisor of $\Theta_1$.
\ Therefore, whenever we consider the interpolation problem
(\ref{interpo}) for a function $\Phi\in L^\infty_{M_n}$ such that
$\Phi$ and $\Phi^*$ are of bounded type, we may assume that
$\Phi\equiv \Phi_-^* + \Phi_+\in L^{\infty}_{M_n}$ is of the form
\begin{equation}\label{3.222}
\Phi_+= \Theta_0 \Theta_1 A^*\quad \hbox{and}\quad \Phi_-=\Theta_0
B^*\quad\hbox{(right coprime)}.
\end{equation}

On the other hand, if $\Phi\equiv \Phi_-^*+\Phi_+\in L^\infty_{M_n}$
is such that $\Phi$ and $\Phi^*$ are of bounded type, then in view
of (\ref{2.6-1}), we may also write
\begin{equation}\label{2.666-6}
\Phi_+= \theta_1 A^*\quad \hbox{and}\quad \Phi_- =\theta_2 B^*,
\end{equation}
where $\theta_1, \theta_2\in H^\infty$ are inner functions. \ If
$\mathcal C (\Phi)\ne \emptyset$, then we can also show that
$\theta_2$ is an inner divisor of $\theta_1$ even though the
factorizations in (\ref{2.666-6}) are not right coprime (cf. \cite
[Proposition 3.2]{CHL2}). \ Thus if $\Phi\equiv \Phi_-^* + \Phi_+\in
L^{\infty}_{M_n}$ is such that $\Phi$ and $\Phi^*$ are of bounded
type then we may,  without loss of generality, assume that $\Phi$ is
of the form
\begin{equation}\label{3.222-1}
\Phi_+= \theta_0 \theta_1 A^*\quad \hbox{and}\quad \Phi_-=\theta_0
B^*,
\end{equation}
where $\theta_0, \theta_1\in H^\infty_{M_n}$ are inner functions and
$A,B\in H^2_{M_n}$. \ In view of Lemma \ref{lem55.2-1}, we may also
assume that the pairs $\{I_{\overline \theta_0\overline\theta_1},
A^*\}$  and  $\{I_{\overline\theta_0}, B^*\}$ have no common tensored-scalar
singularity.

\bigskip

First of all, we consider the case of matrix-valued rational
functions $\Phi\in L^\infty_{M_n}$. \ In this case we may write
$$
\Phi_+= \theta_1\theta_0 A^*\quad\hbox{and}\quad \Phi_-= \theta_1
B^*,
$$
where $\theta_0, \theta_1\in H^\infty $ are finite Blaschke
products. \ Observe that
\begin{equation}\label{3.17}
K \in \mathcal C (\Phi)\ \Longleftrightarrow\ \Phi-K\Phi^*\in
H^{\infty}_{M_n} \ \Longleftrightarrow\
 \theta_0B-KA \in \theta_1\theta_0H^{2}_{M_n}.
\end{equation}
Write
$$
(\theta_1\theta_0)(z)=\prod_{i=1}^N
\left(\frac{z-\alpha_i}{1-\overline{\alpha}_i z}\right)^{m_i} \qquad
(d:=\sum_{i=1}^N m_i),
$$
i.e., $\theta_1\theta_0$ is a finite Blaschke product of degree $d$.
\ Then the last assertion in (\ref{3.17}) holds if and only if the
following equations hold: for each $i=1,\hdots, N$,
\medskip
\begin{equation}\label{3.18}
\begin{pmatrix} B_{i,0}\\
B_{i,1}\\
B_{i,2}\\
\vdots\\
B_{i,m_i -1}
\end{pmatrix}
= \begin{pmatrix}
K_{i,0}&0&0&\cdots&0\\
K_{i,1}&K_{i,0}&0&\cdots&0\\
K_{i,2}&K_{i,1}&K_{i,0}&\cdots&0\\
\vdots&\ddots&\ddots&\ddots&\vdots \\
K_{i,m_i -1}&K_{i,m_i -2}&\hdots&K_{i,1}&K_{i,0}
\end{pmatrix}
\begin{pmatrix}
A_{i,0}\\
A_{i,1}\\
A_{i,2}\\
\vdots\\
A_{i,m_i -1}
\end{pmatrix},
\end{equation}
where
$$
K_{i,j}:= \frac{K^{(j)}(\alpha_i)}{j!},\quad A_{i,j}:=
\frac{A^{(j)}(\alpha_i)}{j!} \quad \text{and} \quad
B_{i,j}:=\frac{(\theta_0B)^{(j)}(\alpha_i)}{j!}.
$$
Thus $K$ is a function in $H^{\infty}_{M_n}$ for which
\begin{equation}\label{3.19}
\frac{K^{(j)}(\alpha_i)}{j!} = K_{i,j} \qquad (1 \leq i \leq N, \ 0
\leq j < m_i),
\end{equation}
where the $K_{i,j}$ are determined by the equation (\ref{3.18}). \
This is exactly the classical Hermite-Fej\' er interpolation
problem. \ Therefore, the solution (\ref{2.12}) for the classical
Hermite-Fej\' er interpolation problem provides a polynomial $K\in
\mathcal C (\Phi)$. \

\bigskip

We turn our attention to the case of matrix functions of bounded
type. \

To proceed, we need:

\medskip

\begin{proposition}\label{pro3.13-1}
Let $\Phi\equiv \Phi_-^*+\Phi_+\in L^\infty_{M_n}$ be such that
$\Phi$ and $\Phi^*$ are of bounded type. \ In view of {\rm
(\ref{3.222})}, we may write
$$
\Phi_+=\Theta_0\Theta_1 A^*\quad\hbox{and}\quad \Phi_-=\Theta_0 B^*
\ \ \hbox{\rm (right coprime)}.
$$
Suppose $\Theta_1 A^* = A_1^*\Theta$ {\rm (}where $A_1$ and $\Theta$ are
left coprime{\rm )}. \ Then the following hold:
\begin{itemize}
\item[(a)] If $ K \in \mathcal C (\Phi)$, then $K=K^{\prime}\Theta $ for some
$K^{\prime} \in H^{\infty}_{M_n}$;
\item[(b)] If $I_{\omega}$ is an inner divisor of $\Theta_1$, then
$I_{\omega}$ is an inner divisor of $\Theta$ and
\begin{equation}\label{21.44}
\mathcal{C}(\Phi^{1, \omega})=\bigl\{\overline\omega K : K \in
\mathcal C( \Phi)\bigr\},
\end{equation}
\end{itemize}
where $\Phi^{1,\omega}:=\Phi_-^*+P_{H_0^2}(\overline\omega\Phi_+)$
{\rm (cf. p.\pageref{qw})}.
\end{proposition}

\begin{proof}
Observe that
$$
K \in \mathcal C (\Phi) \Longleftrightarrow
   B\Theta_0^* -  K A \Theta_1^* \Theta_0^*\in H_{M_n}^2
 \Longleftrightarrow B \Theta_1 -KA  \in H_{M_n}^2  \Theta_0\Theta_1\,,
$$
which implies that $KA\in  H_{M_n}^2 \Theta_1$, and hence $ KA
\Theta_1^*\in  H_{M_n}^2$. \ Let $\Theta_1 A^* = A_1^*\Theta$, where
$A_1$ and $\Theta$ are left coprime. \ Then
\begin{equation}\label{22.41}
0=H_{ KA \Theta_1^*}=H_{ K \Theta^*
A_1}=T_{\widetilde{K}^*}H_{\Theta^* A_1}.
\end{equation}
Since $A_1$ and $\Theta$ are left coprime, $\widetilde{A}_1$ and
$\widetilde{\Theta}$ are right coprime, so that
$\hbox{ran}H_{\Theta^* A_1}
=\mathcal H(\widetilde{\Theta})$. \ Thus by (\ref{22.41}),
$\widetilde{K}^*\mathcal H(\widetilde{\Theta}) \subseteq
(H^2_{\mathbb C^n})^{\perp} $, so that
$\widetilde{K}=\widetilde{\Theta}\widetilde{K^{\prime}}$ for some
$K^{\prime} \in H^{\infty}_{M_n}$, which implies
$K=K^{\prime}\Theta$. This proves (a).

For (b), suppose $I_{\omega}$ is an inner divisor of $\Theta_1$. \
Thus we may write $\Theta_1=\Omega_1 I_{\omega}$ for some inner
function $\Omega_1$. \ Now we will show that $I_\omega$ is an inner
divisor of $\Theta$. \ Since $\Theta_1 A^*=A_1^* \Theta$, we have
$H_{A {I_{\omega}}^*}=H_{\Theta^* A_1\Omega_1}$, and hence $\ker
H_{\Theta^* A_1\Omega_1}=\omega H^2_{\mathbb C^n}$. \  Thus we may
write
$$
\Theta^* A_1\Omega_1 =D (I_{\omega})^*= (I_{\omega})^* D
\quad(\hbox{coprime}),
$$
which implies that $I_{\omega}$ is an inner divisor of $\Theta$.
Write $\Theta=\Omega I_{\omega}$. \ Thus if $K \in \mathcal C
(\Phi)$, then by (a), $\Phi_-^* - K^{\prime}\Theta \Phi_+^* \in
H_{M_n}^2$,
which implies
$$
(\Phi^{1,\omega})_-^* -( K^{\prime}\Omega)(\Phi^{1,\omega})_+^* \in
H_{M_n}^2.
$$
Thus we have $\overline\omega K=K^{\prime}\Omega \in \mathcal C
(\Phi^{1,\omega})$, which implies that
$$
 \Bigl\{\overline\omega K: K \in
\mathcal C( \Phi)\Bigr\}  \subseteq \mathcal{C}(\Phi^{1, \omega}).
$$
The above argument is reversible, and this proves (\ref{21.44}).
\end{proof}

\begin{corollary}\label{cor33.13-1}
Let $\Phi \in L_{M_n}^{\infty}$ be of bounded type of the form
$$
\Phi_+=\theta_0\theta_1 A^*\quad\hbox{and}\quad \Phi_-=\theta_0 B^*
\  \ \hbox{\rm (coprime)},
$$
where $\theta_0$  and $\theta_1$ are finite Blaschke products and
$A,B\in H^2_{M_n}$. \  If $\omega$ is an inner divisor of
$\theta_1$, then
\begin{equation}\label{3.23}
\mathcal C(\Phi)=\Bigl\{\omega K : K \in \mathcal C( \Phi^{1,
\omega})\Bigr\}.
\end{equation}
\end{corollary}

\begin{proof}
Immediate from Proposition \ref{pro3.13-1}.
\end{proof}

\medskip

Now let $\Theta \in H^{\infty}_{M_n}$ be an inner matrix function
and suppose
$$
\Theta :=\prod_{i=1}^N \Delta_i \quad(\Delta_i\ \hbox{is an inner
matrix function}).
$$
Write
$$
\Omega_j:=\prod_{i=1}^j \Delta_i\quad (j=1,\cdots, N-1).
$$
Then by  Lemma \ref{lem22.55} (a), we have
$$
\aligned \mathcal K_{\Theta}
&=\mathcal K_{\Delta_N}\oplus  \mathcal K_{\Omega_{N-1}}\Delta_N\\
&=\mathcal K_{\Delta_N}\oplus\bigr(\mathcal K_{\Delta_{N-1}}
\oplus\mathcal K_{\Omega_{N-2}}\Delta_{N-1}\bigl)\Delta_N \\
&=\mathcal K_{\Delta_N}\oplus\mathcal K_{\Delta_{N-1}}\Delta_N\oplus
\bigr(\mathcal K_{\Delta_{N-2}}\oplus \mathcal K_{\Omega_{N-3}}\Delta_{N-2}\bigl)\Delta_{N-1}\Delta_N\\
& \qquad \vdots\\
&=\mathcal K_{\Delta_N}\oplus\mathcal
K_{\Delta_{N-1}}(\Delta_N)\oplus \mathcal
K_{\Delta_{N-2}}(\Delta_{N-1}\Delta_N)+\cdots +\mathcal
K_{\Delta_1}(\Delta_2 \cdots \Delta_N ),
\endaligned
$$
Thus if $B\in \mathcal K_{\Theta}$, then we may write
\begin{equation}\label{123.28}
B=B_1+\sum_{i=2}^{N} B_i \left( \prod_{j=N+2-i}^{N} \Delta_j\right)
\quad (B_i \in \mathcal K_{\Delta_{N+1-i}}).
\end{equation}
Let $\Phi\equiv \Phi_-^*+\Phi_+\in L^\infty_{M_n}$ be such that
$\Phi$ and $\Phi^*$ are of bounded type. \ To consider the
interpolation problem $\mathcal C (\Phi)\ne \emptyset$, we may
write, in view of (\ref{3.222-1}),
$$
\Phi_+= \theta_0 \theta_1 A^*\quad \hbox{and}\quad \Phi_-=\theta_0
B^*,
$$
where  $\theta_0, \theta_1$ are inner functions and $A,B\in
H^2_{M_n}$. \ We may assume that $\Phi_\pm(0)=0$. \ Thus, in view of
Lemma \ref{lem2.4}, $A\in \mathcal K_{{\theta_0\theta_1}}$ and $B\in
\mathcal K_{{\theta_0}}$. \ Let
$$
\theta_0:=\prod_{i=1}^N \delta_i \qquad(\delta_i \in H^{\infty}\
\hbox{is inner}).
$$
Then by (\ref{123.28}) and  Lemma \ref{lem22.55} (a), we may write
\begin{equation}\label{123.2888}
\begin{cases}
A=A_1+\sum_{i=2}^{N} \left( \prod_{j=N+2-i}^{N}
{\delta_j}\right) A_i + \theta_0 A_{N+1};\\
B=B_1+\sum_{i=2}^{N} \left( \prod_{j=N+2-i}^{N} {\delta_j}\right)
B_i,
\end{cases}
\end{equation}
where $A_i, B_i \in  \mathcal K_{{\delta_{N+1-i}}} (i=1,2,\cdots,
N)$ and $A_{N+1} \in \mathcal K_{{\theta_1}}$.

\bigskip

We then have:

\begin{lemma}\label{lem32.377} \
Let $\Phi\equiv \Phi_-^*+\Phi_+\in L^\infty_{M_n}$ be such that
$\Phi$ and $\Phi^*$ are of bounded type. \ Then in view of
$(\ref{3.222-1})$, we may write
$$
\Phi_+={\theta_0\theta_1} A^* \quad \hbox{and} \quad
\Phi_-={\theta_0} B^*.
$$
Put $\theta_0:=\prod_{i=1}^N \delta_i$ ($\delta_i \ \hbox{inner})$.
\ Then in view of {\rm (\ref{123.2888})}, we may write
$$
A=\sum_{i=1}^{N} \lambda_i A_i + \theta_0 A_{N+1} \quad \hbox{and}
\quad B=\sum_{i=1}^{N} \lambda_i B_i,
$$
where $\lambda_1=1$, $\lambda_i\equiv \prod_{j=N+2-i}^{N} \delta_j \
(i=2,3,\cdots N)$, $A_{N+1}\in \mathcal K_{{\theta_1}}$ and $\ A_i,
B_i \in \mathcal K_{{\delta_{N+1-i}}}$. \ Suppose that $A$ and
$I_{\theta_1}$ are coprime. \  Then for $K \in H^{\infty}_{M_n}$,
$$
K \in \mathcal C(\Phi) \Longleftrightarrow P_{\mathcal
K_{{\theta_0\theta_1}}}K ={\theta_1} \left(\sum_{i=1}^{N} \lambda_i
K_i \right),
$$
where $K_i \in \mathcal K_{{\delta_{N+1-i}}} \ (i=1,2,\cdots, N)$
satisfies the equation
\begin{equation}\label{30}
\sum_{i=1}^{N} \lambda_i B_i-\left(\sum_{i=1}^{N} \lambda_i K_i
\right) \left(\sum_{i=1}^{N} \lambda_i A_i \right) \in \theta_0
H^2_{M_n}.
\end{equation}
\end{lemma}

\begin{proof} \ Suppose that $A$ and
$I_{\theta_1}$ are coprime. \ Let $K \in \mathcal C({\Phi})$. \ If
we put $\Psi:=\Phi^{1,\theta_1}$, then by Proposition
\ref{pro3.13-1}(b) we can show that $K=\theta_1K^{\prime}$ for some
$K^{\prime} \in \mathcal C(\Psi)$. \ It thus suffices to show that
\begin{equation}\label{55.22}
K^{\prime} \in \mathcal C(\Psi) \Longleftrightarrow P_{\mathcal
K_{{\theta_0}}}K^{\prime} =\sum_{i=1}^{N} \lambda_i K_i,
\end{equation}
where $K_i \in \mathcal K_{{\delta_{N+1-i}}} \ (i=1,2,\cdots, N-1)$
satisfies the equation (\ref{30}). \ By Lemma \ref{lem2.4}, we have
$$
\Psi_+=P_{H_0^2}(\overline{\theta}_1 \Phi_+) =
\theta_0\bigl(P_{\mathcal K_{{\theta_0}}}A\bigr)^* =\theta_0
\left(\sum_{i=1}^{N} \lambda_i A_i\right)^*.
$$
Then
\begin{equation}\label{123.29}
\aligned \Psi_-^*-K^{\prime}\Psi_+^* \in H^2_{M_n}
&\Longleftrightarrow
\sum_{i=1}^{N} \lambda_i B_i - K^{\prime} \sum_{i=1}^{N} \lambda_i A_i \in \theta_0H^2_{M_n}\\
 &\Longleftrightarrow
 P_{\mathcal K_{{\theta_0}}} \left(\sum_{i=1}^{N} \lambda_i B_i \right)-P_{\mathcal K_{{\theta_0}}} \left(K^{\prime}
\sum_{i=1}^{N} \lambda_i A_i \right) =0\\
 &\Longleftrightarrow
\sum_{i=1}^{N} \lambda_i B_i - P_{\mathcal K_{{\theta_0}}}\left(
\left(P_{\mathcal K_{{\theta_0}}}K^{\prime}
\right)\left(\sum_{i=1}^{N} \lambda_i A_i\right)\right) =0.
\endaligned
\end{equation}
In  view of {\rm (\ref{123.2888})}, we may write
$$
P_{\mathcal K_{{\theta_0}}} K^{\prime}\equiv\sum_{i=1}^{N} \lambda_i
K_i,
$$
where $\lambda_1=1, \ \lambda_i:= \prod_{j=N+2-i}^{N} \delta_j, \ \
K_i\in \mathcal K_{{\delta_{N+1-i}}}$. \ It thus follows from
(\ref{123.29}) that $K^{\prime} \in \mathcal C(\Psi)$  if and only
if
$$
\sum_{i=1}^{N} \lambda_i B_i -\left(\sum_{i=1}^{N} \lambda_i K_i
\right) \left(\sum_{i=1}^{N} \lambda_i A_i \right) \in \theta_0
H^2_{M_n},
$$
which gives (\ref{55.22}). \ This completes the proof.
\end{proof}

\begin{theorem}\label{cor32.378} \
Let $\Phi\equiv \Phi_-^*+\Phi_+ \in L^{\infty}_{M_n}$ be of bounded
type of the form
$$
\Phi_+={\theta}^N A^* \quad \hbox{and} \quad \Phi_-= {\theta}^m B^*
\quad(N \geq m)\,,
$$
where $\theta$ is inner. \ Then in view of $(\ref{123.2888})$, we
may write
$$
A=A_0+\sum_{i=1}^{m-1} \theta^i A_i + \theta^m A_{m}\quad \hbox{and}
\quad B=B_0 +\sum_{i=1}^{m-1} \theta^i B_i,
$$
where $A_m\in \mathcal K_{{\theta^{N-m}}}$ and $\ A_i, B_i \in
\mathcal K_{{\theta}}$ $(i=0,1,\cdots, m-1)$. \ Suppose that $A$ and
$I_{\theta}$ are coprime.  \  Then for $K \in H^{\infty}_{M_n}$,
$$
K \in \mathcal C(\Phi) \Longleftrightarrow P_{\mathcal
K_{{\theta^N}}}K ={\theta}^{N-m}\left(K_0 +\sum_{i=1}^{m-1} \theta^i
K_i \right),
$$
where the $K_i\in \mathcal K_{{\theta}}$ satisfy the equation
\begin{equation}\label{333330}
\begin{pmatrix} B_0\\B_1 \\ \vdots\\B_{m-1}\end{pmatrix}
=\begin{pmatrix}K_0&0&\cdots &0\\K_1
&K_0&\cdots&0\\ \vdots&\vdots&\ddots&\vdots\\
K_{m-1}&K_{m-2}&\cdots & K_0
\end{pmatrix}
\begin{pmatrix} A_0\\A_1 \\ \vdots\\A_{m-1}\end{pmatrix}+\begin{pmatrix} G_0\\G_1 \\
\vdots\\G_{m-1}\end{pmatrix},
\end{equation}
where
\begin{equation}\label{G_k}
G_k=P_{H^2_{M_n}}\left(\overline{\theta}\sum_{j=0}^{k-1} K_j
A_{k-j}\right)-P_{\theta H^2_{M_n}}\left(\sum_{j=0}^k K_j
A_{k-j}\right) \quad (k=0,1,\cdots, m-1).
\end{equation}

\end{theorem}

\begin{proof} \ Suppose that $A$ and
$I_{\theta}$ are coprime. \ It suffices to show that (\ref{30}) hods
if and only if (\ref{333330}) holds. \ Suppose that there exists
$K_i \in \mathcal K_{{\theta}}$  ($i=0,1,\cdots, m-1$) such that
$$
B_0+\sum_{i=1}^{m-1}\theta^i B_i - \left(K_0+\sum_{i=1}^{m-1}
\theta^i K_i\right) \left(A_0+\sum_{i=1}^{m-1} \theta^i A_i
\right)\in {\theta}^m H^2_{M_n},
$$
or equivalently,
\begin{equation}\label{semi}
\sum_{k=0}^{m-1} \theta^k \left(B_k-\sum_{j=0}^k K_j A_{k-j} \right)
\in {\theta}^{m}H^2_{M_n}.
\end{equation}
Put $D_k:=\sum_{j=0}^k K_j A_{k-j}$ ($k=0,1,\cdots m-1$). \ Note
that for each $i,j=0,1,\cdots,m-1$, we have $K_i, A_j \in \mathcal
K_{{\theta}}$, and hence $K_iA_j\in \mathcal
K_{{{\theta^2}}}=\mathcal K_{{\theta}} \oplus \theta K_{{\theta}}$.
\ Thus we may write, for $k=0,1,\cdots m-1$,
$$
D_k=D_k^{\prime}+\theta D_k^{\prime \prime}\quad (D_k^{\prime},
D_k^{\prime \prime} \in \mathcal K_{{\theta}}).
$$
It thus follows from (\ref{semi}) that
$$
\begin{cases}
B_0&=K_0A_0-\theta D_0^{\prime\prime}=K_0A_0-P_{\theta H^2_{M_n}}(D_0)\\
B_k&=\sum_{j=0}^k K_j A_{k-j} +\Bigl(P_{
H^2_{M_n}}(\overline{\theta}D_{k-1})-P_{\theta
H^2_{M_n}}(D_k)\Bigr)\ \ (k=1,\cdots,m-1),
\end{cases}
$$
which gives
$$
\begin{pmatrix} B_0\\B_1 \\ \vdots\\B_{m-1}\end{pmatrix}
=\begin{pmatrix}K_0&0&\cdots &0\\K_1
&K_0&\cdots&0\\ \vdots&\vdots&\ddots&\vdots\\
K_{m-1}&K_{m-2}&\cdots & K_0
\end{pmatrix}
\begin{pmatrix} A_0\\A_1 \\ \vdots\\A_{m-1}\end{pmatrix}+\begin{pmatrix} G_0\\G_1 \\
\vdots\\G_{m-1}\end{pmatrix},
$$
where $G_0=-P_{\theta H^2_{M_n}}(D_0)$ and $G_k=P_{
H^2_{M_n}}(\overline{\theta}D_{k-1})-P_{\theta H^2_{M_n}}(D_k)$. \
This argument is reversible. \ This completes the proof. \
\end{proof}

\medskip

\begin{theorem}\label{thm32.378} \ Let $\Phi\equiv \Phi_-^*+\Phi_+ \in L^{\infty}_{M_n}$
be normal of the form
$$
\Phi_+= {\theta}^N A^* \quad \hbox{and} \quad \Phi_- = {\theta}^m
B^* \quad(N \geq m),
$$
where $\theta$ is inner and $I_{\theta}$ and $B$ are coprime. \ If
$\mathcal C(\Phi)\ne \emptyset$, then we may write
$$
A=A_0+\sum_{i=1}^{m-1} \theta^i A_i + \theta^m A_{m} \quad
\hbox{and} \quad B=B_0 +\sum_{i=1}^{m-1} \theta^i B_i,
$$
where $A_0$ and $B_0$ are invertible a.e. on $\mathbb T$, $A_m\in
\mathcal K_{{\theta^{N-m}}}$ and $\ A_i, B_i \in \mathcal
K_{{\theta}} \ (i=0,1,\cdots, m-1)$. \ If $K_i \in \mathcal
K_{\theta}$ and $\sum_{j=0}^i K_j A_{i-j}\in \mathcal K_{\theta}$
for all $i=0,1,\cdots m-1$, put
\begin{equation}\label{thmm3.31}
\begin{pmatrix} \widetilde{K}_0\\\widetilde{K}_1 \\
\vdots\\\widetilde{K}_{m-1}\end{pmatrix}=\begin{pmatrix}\widetilde{A}_0&0&\cdots
&0\\\widetilde{A}_1
&\widetilde{A}_0&\cdots&0\\ \vdots&\vdots&\ddots&\vdots\\
\widetilde{A}_{m-1}&\widetilde{A}_{m-2}&\cdots & \widetilde{A}_0
\end{pmatrix}^{-1}\begin{pmatrix}\widetilde{B}_0\\\widetilde{B}_1 \\ \vdots\\\widetilde{B}_{m-1}\end{pmatrix}.
\end{equation}
Then for $K \in H^{\infty}_{M_n}$,
\begin{equation}\label{C(Phi)}
K\in \mathcal C(\Phi) \Longleftrightarrow P_{\mathcal
K_{{\theta^{N}}}}K= {\theta}^{N-m}\left(K_0 +\sum_{i=1}^{m-1}
\theta^i K_i \right),
\end{equation}

\end{theorem}

\begin{proof} Suppose that $I_{\theta}$ and $B$ are
coprime and $\mathcal C(\Phi)\ne \emptyset$. \ Then $I_\theta$ and
$B_0$ are coprime, so that by Lemma \ref{lem1.7}, $B_0$ is
invertible a.e. on $\mathbb T$. \ Now we will show that
\begin{equation}\label{inv}
{A}_0 \ \hbox{is invertible a.e. on $\mathbb T$}.
\end{equation}
Since  $I_{\theta^m}$ and $B$ are coprime, it follows that $\ker
H_{\Phi_-^*}={\theta^m}H^2_{\mathbb C^n}$. \ Since $\mathcal
C(\Phi)\ne \emptyset$, it follows from Proposition
\ref{pro3.13-1}(b) that $\mathcal C(\Phi^{1, \theta^{N-m}})\ne
\emptyset$. Thus we have
\begin{equation}\label{93.31}
{\theta^m}H^2_{\mathbb C^n}=\ker H_{\Phi_-^*} \supseteq \ker
H_{(\Phi^{1, \theta^{N-m}})_+^*}.
\end{equation}
Observe that
$$
(\Phi^{1, \theta^{N-m}})_+ = {\theta}^m(P_{\mathcal
K_{{\theta^m}}}A)^* = {\theta^m}\left(A_0+\sum_{i=1}^{m-1} \theta^i
A_i \right)^*\,.
$$
It thus follows from (\ref{93.31}) that
$A_0$ and $I_{\theta}$ are coprime. \ Thus  by Lemma \ref{lem1.7},
$\hbox{det}\, A_0 \neq 0$. \ This proves (\ref{inv}). \ Suppose
(\ref{thmm3.31}) holds. \ Then a direct calculation shows that $G_k$
in (\ref{G_k}) should be zero for each $k=0,1,\cdots, m-1$, so that
(\ref{C(Phi)}) follows from Theorem \ref{cor32.378}.
\end{proof}

\medskip

We now turn our attention to a functional calculus for compressions
of the shift operator. \

It is well known that the functional calculus for polynomials of
compressions of the shift results in the  Hermite-Fej\' er matrix
via  the classical Hermite-Fej\' er Interpolation Problem. \ We now
extend the polynomial calculus to an $H^\infty$-functional calculus
(so called the Sz.-Nagy-Foia\c s functional calculus) via the
triangularization theorem, and then extend it further to an
$\overline{H^\infty}+H^\infty$-functional calculus for compressions
of the shift operator. \

First of all, we extend the representation (\ref{2.17}) to the case
of matrix $H^{\infty}$-functions.  \  We refer to \cite{AC} and
\cite{Ni} for details on this representation. \  For an explicit
criterion, we need to introduce the Triangularization Theorem
concretely. \ There are three cases to consider. \

\medskip

\noindent {\it Case 1} : Let $B$ be a Blaschke product and let
$\Lambda:=\{\lambda_n : n \geq 1\}$ be the sequence of zeros of $B$
counted with their multiplicities. \  Write
$$
\beta_1:=1, \quad
\beta_k:=\prod_{n=1}^{k-1}\frac{\lambda_n-z}{1-\overline{\lambda}_n
z}\cdot \frac{|\lambda_n|}{\lambda_n}\qquad (k\geq 2),
$$
and let
$$
\delta_j\label{deltaj}:=\frac{d_j}{1-\overline{\lambda}_j z}\beta_j
\qquad (j \geq 1),
$$
where $d_j\label{dj}:=(1-|\lambda_j|^2)^{\frac{1}{2}}$. \  Let
$\mu_B\label{mub}$ be the measure on $\mathbb N$ given by
$\mu_B(\{n\}):=\frac{1}{2}d_n^2, (n \in \mathbb N). \ $ Then the map
$V_B\label{VB}: L^2(\mu_B)\to \mathcal H(B)$ defined by
\begin{equation}\label{13.2}
V_B(c):=\frac{1}{\sqrt{2}} \sum_{n \geq 1} c(n)d_n \delta_n, \quad c
\equiv \{c(n)\}_{n \geq 1},
\end{equation}
is unitary and $U_B$ is mapped onto the operator
\begin{equation}\label{13.3}
V_B^*U_B V_B=(I-J_B)M_B,
\end{equation}
where $(M_B\label{MB} c)(n):=\lambda_n c(n)$ ($n\in \mathbb N$) is a
multiplication operator and
$$
(J_B\label{JB} c)(n):=\sum_{k=1}^{n-1}c(k)|\lambda_k|^{-2} \cdot
\frac{\beta_n(0)}{\beta_k(0)}d_k d_n \ \ \hbox{$(n \in \mathbb N$)}
$$
is a lower-triangular Hilbert-Schmidt operator. \

\medskip

\noindent {\it Case 2} : Let $s$ be a singular inner function with
continuous representing measure $\mu\equiv \mu_{s}\label{mus}$. \
Let $\mu_{\lambda}$ be the projection of $\mu$ onto the arc
$\{\zeta: \zeta \in \mathbb T, \ 0< \text{arg}\zeta \leq
\text{arg}\lambda\}$ and let
$$
s_{\lambda}(\zeta)\label{slambdazeta}:=\text{exp}\Bigl(-\int_{\mathbb
T}\frac{t+\zeta}{t-\zeta}d\mu_{\lambda}(t) \Bigr) \ \ \hbox{($\zeta
\in \mathbb D$)}.
$$
Then the map $V_s\label{Vs}: L^2(\mu)\to \mathcal H (s)$ defined by
\begin{equation}\label{13.4}
(V_s c)(\zeta)=\sqrt{2}\int_{\mathbb
T}c(\lambda)s_{\lambda}(\zeta)\frac{\lambda
d\mu(\lambda)}{\lambda-\zeta}\quad\hbox{($\zeta \in \mathbb D$)}
\end{equation}
is unitary, and $U_s$ is  mapped onto the operator
\begin{equation}\label{13.5}
V_s^*U_s V_s=(I-J_s)M_s,
\end{equation}
where $(M_s c)(\lambda):=\lambda c(\lambda)$ ($\lambda\in\mathbb T$)
is a multiplication operator and
$$
(J_s\label{Js} c)(\lambda)=2 \int _{\mathbb T}
e^{\mu(t)-\mu(\lambda)} c(t)d_{\mu_{\lambda}}(t)\ \ \hbox{($\lambda
\in \mathbb T$)}
$$
is a lower-triangular Hilbert-Schmidt operator. \

\medskip

\noindent {\it Case 3} : Let $\Delta$ be a singular inner function
with pure point representing measure $\mu\equiv
\mu_{\Delta}\label{mudelta}$. \  We enumerate the set $\{t\in
\mathbb T: \mu(\{t\})>0 \}$ as a sequence $\{t_k\}_{k \in \mathbb
N}$. \  Write $\mu_k:=\mu(\{t_k\}), \ k \geq 1$. \  Further, let
$\mu_{\Delta}$ be a measure on $\mathbb R_{+}=[0, \infty)$ such that
$d\mu_{\Delta}(\lambda)=\mu_{[\lambda]+1}d \lambda$ and define a
function $\Delta_{\lambda}\label{deltalambda}$ on the unit disk
$\mathbb D$ by the formula
$$
\Delta_0:=1, \quad
\Delta_{\lambda}(\zeta):=\text{exp}\Biggl\{-\sum_{k=1}^{[\lambda]}\mu_k
\frac{t_k+\zeta}{t_k-\zeta}-(\lambda-[\lambda])\mu_{[\lambda]+1}
\frac{t_{[\lambda]+1}+\zeta}{t_{[\lambda]+1}-\zeta}\Biggr \},
$$
where $[\lambda]$ is the integer part of $\lambda$  ($\lambda \in
\mathbb R_{+}$). \  Then the map $V_{\Delta}\label{Vdelta}:
L^2(\mu_{\Delta})\to \mathcal H (\Delta)$ defined by
\begin{equation}\label{13.6}
(V_{\Delta}c)(\zeta):=\sqrt{2} \int_{\mathbb
R_{+}}c(\lambda)\Delta_{\lambda}(\zeta)(1-\overline{t}_{[\lambda]+1}
\zeta)^{-1}d\mu_{\Delta}(\lambda) \ \hbox{($\zeta \in \mathbb D$)}
\end{equation}
is unitary and $U_\Delta$ is mapped onto the operator
\begin{equation}\label{13.7}
V_{\Delta}^*U_\Delta V_{\Delta}=(I-J_{\Delta})M_{\Delta},
\end{equation}
where $(M_{\Delta}c)(\lambda):= t_{[\lambda]+1} c(\lambda)$,
($\lambda\in \mathbb R_+$) is a multiplication operator and
$$
(J_{\Delta}\label{Jdelta} c)(\lambda) := 2 \int _{0}^{\lambda}c(t)
\frac{\Delta_{\lambda}(0)}{\Delta_{t}(0)} d\mu_{\Delta}(t) \ \
\hbox{($\lambda \in \mathbb R_+$)}
$$
is a lower-triangular Hilbert-Schmidt operator. \

\medskip

Collecting together the above three cases we get:

\medskip

\noindent {\bf Triangularization Theorem.} (\cite [p.123] {Ni}) Let
$\theta$ be an inner function with the canonical factorization
$\theta=B\cdot s \cdot \Delta$, where $B$ is a Blaschke product, and
$s$ and $\Delta$ are singular inner functions with representing
measures $\mu_{s}$ and $\mu_{\Delta}$ respectively, with $\mu_{s}$
continuous and $\mu_{\Delta}$ a pure point measure. \  Then the map
$V:\,L^2(\mu_B)\times L^2(\mu_s) \times L^2(\mu_{\Delta})\to
\mathcal H(\theta)$ defined by
\begin{equation}\label{12.26}
V:=\begin{pmatrix}V_B&0&0\\0&BV_s&0\\0&0&BsV_{\Delta}\end{pmatrix}
\end{equation}
is unitary, where $V_B, \mu_B, V_S, \mu_S, V_\Delta, \mu_\Delta$ are
defined in (\ref{13.2}) - (\ref{13.7}) and $U_\theta$ is mapped onto
the operator
\begin{equation}\label{12.27}
M:= V^*U_\theta V =
\begin{pmatrix}M_B&0&0\\0&M_{s}&0\\0&0&M_{\Delta}\end{pmatrix}+J,
\end{equation}
where $M_B, M_S, M_{\Delta}$ are defined in (\ref{13.3}),
(\ref{13.5}) and (\ref{13.7}) and
$$
J:=-\begin{pmatrix}J_BM_B&0&0\\0&J_sM_s&0\\0&0&J_{\Delta}M_\Delta\end{pmatrix}+A
$$
is a lower-triangular Hilbert-Schmidt operator, with $A^3=0$,
$\text{rank}A\leq 3$. \

\bigskip

Now we note that every compression of the shift operator is
completely non-unitary (cf. \cite[Proof of Theorem 3.3]{CHL2}).
Therefore if $M$ is given by {\rm (\ref{12.27})} then $M$ is an
absolutely continuous contraction. \ Thus if $\Phi \in
H^{\infty}_{M_n}$, then we can define $\Phi(M)$  as a
$H^\infty$-functional calculus (the so-called Sz.-Nagy-Foia\c s
functional calculus). \ Then we have:

\bigskip

\begin{theorem}\label{thm2.16}
 Let $\Phi \in H_{M_n}^{\infty}$ and let $\theta\in H^\infty$ be an inner
function. \  If we write
\begin{equation}\label{12.288}
M:= V^*U_\theta V \quad \hbox{and} \quad \mathcal V := V \otimes
I_n,
\end{equation}
where $V: L\equiv L^2(\mu_B) \times L^2(\mu_s) \times
L^2(\mu_{\Delta}) \to \mathcal H(\theta)$ is unitary as in
$(\ref{12.26})$, then
\begin{equation}\label{12.28}
{\mathcal V}^*(T_{\Phi})_{\Theta} \mathcal V = \Phi(M)\quad\hbox{\rm
($\Theta:=I_\theta$)}.
\end{equation}
\end{theorem}

\begin{remark}
$\Phi(M)$ is called a {\it matrix representation} for
$(T_{\Phi})_{\Theta}$.
\end{remark}

\begin{proof}[Proof of Theorem \ref{thm2.16}]
If $\Phi(z)\equiv \begin{pmatrix} \phi_{rs}(z)\end{pmatrix}_{1\le
r,s\le n} \in H^{\infty}_{M_n}$, we may write
$$
\Phi(z)=\sum_{i=0}^{\infty} A_i z^i \qquad (A_i \in M_n).
$$
We also write $\phi_{rs}(z):=\sum_{0}^\infty a_i^{(rs)} z^i$ and
then $A_i=\begin{pmatrix} a_i^{(rs)}\end{pmatrix}_{1\le r,s\le n}$.
\ We thus have
$$
(T_{\Phi})_{\Theta} =P_{\mathcal H (\Theta)}T_{\Phi}\vert_{\mathcal
H (\Theta)} =\sum_{i=0}^\infty \left(U_{\theta}^i\, \otimes\,
\begin{pmatrix} a_i^{(rs)}
                          \end{pmatrix}_{1\le r,s\le n}\right)
=\sum_{i=0}^{\infty} \left( U_{\theta}^i\, \otimes\, A_i\right).
$$
Let $\{\psi_j\}$ be an orthonormal basis for $\mathcal H (\theta)$
and put $e_j:=V^* \psi_j$. \  Then $\{e_j\}$ forms an orthonormal
basis for $L^2(\mu_B) \times L^2(\mu_s) \times L^2(\mu_{\Delta})$. \
Thus for each $f \in \mathbb C^n$, we have $\mathcal V (e_j \otimes
f)= \phi_j \otimes f$. \  It thus follows that
$$
\aligned \Bigl \langle (T_{\Phi})_{\Theta} (\psi_j \otimes f)\, , \
\psi_k \otimes g \Bigr \rangle &=\sum_{i=0}^{\infty} \Bigl \langle
(M^i \otimes A_i) (e_j \otimes f),
                           \ e_k \otimes g \Bigr \rangle \\
&=\Bigl \langle \Phi(M) (e_j \otimes f)\, , \ e_k \otimes g \Bigr
\rangle,
\endaligned
$$
which gives $ \mathcal V^*(T_{\Phi})_{\Theta}\mathcal V =\Phi(M). $
\end{proof}

\bigskip

We can now extend the representation (\ref{12.28}) to
$\overline{H^{\infty}_{M_n}} +H^{\infty}_{M_n}$ (where
$\overline{H^{\infty}_{M_n}}$ denotes the set of $n\times n$ matrix
functions whose entries belong to $\overline{H^\infty}:=\{g:
\overline g\in H^\infty\}$). \ Let $Q \in
\overline{H^{\infty}_{M_n}} +H^{\infty}_{M_n}$ be of the form
$Q=Q_-^*+Q_+$. \ If $\Theta:=I_\theta$ for an inner function
$\theta$, then we define
$$
(T_Q)_{\Theta}\label{TQT}:=P_{\mathcal H(\Theta)}T_{Q}|_{\mathcal
H(\Theta)}\,.
$$
Then
$$
\aligned
(T_Q)_{\Theta}&=(T_{Q_-^*})_{\Theta}+(T_{Q_+})_{\Theta}\\
&=(T_{Q_-})^*_{\Theta}+(T_{Q_+})_{\Theta}\, \quad(\hbox{by
(\ref{***})}).
\endaligned
$$
If $M:= V^*U_\theta V$, where $V: L\equiv L^2(\mu_B) \times
L^2(\mu_s) \times L^2(\mu_{\Delta}) \to \mathcal H(\theta)$ is
unitary as in $(\ref{12.26})$, we also define $Q(M)\label{QM}$ by
\begin{equation}\label{3.13}
Q(M):=\left(Q_-(M)\right)^*+Q_+(M),
\end{equation}
where $Q_{\pm}(M)$ is defined by the Sz.-Nagy-Foia\c s functional
calculus. \

\medskip

We then have:

\begin{lemma}\label{lem3.6}
Let $Q \in \overline{H^{\infty}_{M_n}} +H^{\infty}_{M_n}$ and
$\Theta:=I_\theta$ for an inner function $\theta$. \ Then
\begin{enumerate}
\item[(a)] $\mathcal V^*(T_Q)_{\Theta}\mathcal V=Q(M)$;
\item[(b)] $Q(M)^*=Q^*(M)$,
\end{enumerate}
where $\mathcal V$ and $M$ are given by {\rm (\ref{12.288})}. \
\end{lemma}

\begin{proof} \ It follows from Theorem \ref{thm2.16} that
$$
\mathcal V^*({T_Q})_{\Theta}\mathcal V=\mathcal
V^*\Bigl((T_{Q_-})_{\Theta}^*+(T_{Q_+})_{\Theta}\Bigr)\mathcal V
=(Q_-(M))^*+Q_+(M)=Q(M),
$$
which gives (a). \ By definition of $Q(M)$, we have
$$
\begin{aligned}
Q(M)^*
&=\Bigl(Q_-(M)^*+Q_+(M)\Bigr)^*=Q_-(M)+Q_+(M)^*\\
&=Q_-(M)+Q_+^*(M)=Q^*(M),
\end{aligned}
$$
which gives (b). \
\end{proof}

\medskip

We are tempted to guess that $Q(M)^*Q(M)=(Q^*Q)(M)$. \ But this is
not the case. \ To see this, let $\theta=z^3$ and let
$Q(z):=\left(\begin{smallmatrix} z^2&0\\0&z^2
\end{smallmatrix}\right)$. \  Then we have
$$
M=\begin{pmatrix}0&0&0\\1&0&0\\0&1&0\end{pmatrix}
$$
and
$$
Q(M)\equiv \begin{pmatrix}1&0\\
0&1\end{pmatrix}
\otimes M^2 =\begin{pmatrix}0&0&0&0&0&0\\0&0&0&0&0&0\\0&0&0&0&0&0\\0&0&0&0&0&0\\
1&0&0&0&0&0\\0&1&0&0&0&0\end{pmatrix}.
$$
Thus
$$
Q(M)^* Q(M)=\begin{pmatrix}1&0&0&0&0&0\\0&1&0&0&0&0\\0&0&0&0&0&0\\0&0&0&0&0&0\\
0&0&0&0&0&0\\0&0&0&0&0&0\end{pmatrix}\ \ \hbox{and}\ \
(Q^*Q)(M)=\begin{pmatrix}1&0&0&0&0&0\\0&1&0&0&0&0\\0&0&1&0&0&0\\0&0&0&1&0&0\\
0&0&0&0&1&0\\0&0&0&0&0&1\end{pmatrix}\,,
$$
which gives $Q^*(M)Q(M) \neq (Q^*Q)(M)$. \

\bigskip

In some sense, the following lemma shows that the operator induced
by $Q$ in (\ref{3.13}) is closed under constant matrix
multiplication.

\begin{lemma}\label{lem3.7}
Let $Q \in \overline{H^{\infty}_{M_n}} +H^{\infty}_{M_n}$ and and
$\Theta:=I_\theta$ for an inner function $\theta$. \ Then for any
constant matrix $\Lambda \in M_n$,
\medskip
\begin{enumerate}
\item[(a)] $(\Lambda Q)(M)=(\Lambda \otimes I) Q(M)$;
\item[(b)] $(Q \Lambda )(M)= Q(M)(\Lambda \otimes I)$,
\end{enumerate}
\medskip
where $M$ is given by {\rm (\ref{12.288})}. \
\end{lemma}

\begin{proof} Let us write $Q(z):=\sum_{j=-\infty}^{\infty} Q_j z^j$. \  Since $\Lambda$ is a
constant matrix, $(\Lambda Q)(z)=\sum_{j=-\infty}^{\infty} \Lambda
Q_j z^j$. \ Thus for each $U, W \in \mathcal H(\Theta)$,
$$
\begin{aligned}
\Bigl\langle \bigl((\Lambda \otimes I)Q(M)\bigr)\, \mathcal V^*U, \
\mathcal V^*W \Bigr\rangle &
=\sum_{j=-\infty}^{\infty}\Bigl\langle\bigl( (\Lambda \otimes I)(Q_j \otimes M^j)\bigr)\,\mathcal V^* U, \mathcal V^*W \Bigr\rangle \\
& = \sum_{j=-\infty}^{\infty}\Bigl\langle (\Lambda Q_j \otimes M^j)\, \mathcal V^*U, \ \mathcal V^*W \Bigr\rangle \\
& =\Bigl\langle \bigl( (\Lambda Q)(M)\bigr)\, \mathcal V^*U, \
\mathcal V^*W \Bigr\rangle,
\end{aligned}
$$
where $\mathcal V$ is given by {\rm (\ref{12.288})}. \  This proves
(a). \ Observe that
$$
\aligned (Q \Lambda)(M)^*&=(Q \Lambda)^*(M) \qquad (\hbox{by  Lemma \ref{lem3.6}(b)})\\
&=(\Lambda^* Q^*)(M)\\
&=(\Lambda^* \otimes I)(Q^*(M)) \qquad (\hbox{by (a)})\\
&=(\Lambda \otimes I)^*(Q(M))^* \qquad (\hbox{by  Lemma \ref{lem3.6}(b)})\\
&=\Bigl(Q(M) (\Lambda \otimes I)\Bigr)^*,
\endaligned
$$
which gives (b). \
\end{proof}

\bigskip

%
%
%
%
%
%

\chapter{Abrahamse's Theorem for matrix-valued symbols}

\medskip

In 1970, P.R. Halmos posed the following problem, listed as Problem
5 in his lecture ``Ten problems in Hilbert space" \cite{Hal1},
\cite{Hal2}:
$$
\hbox{Is every subnormal Toeplitz operator either normal or
analytic\,?}
$$
A Toeplitz operator $T_\varphi$ is called {\it analytic} if
$\varphi\in H^\infty$. \ Any analytic Toeplitz operator is easily
seen to be subnormal: indeed, $T_\varphi h=P(\varphi h)=\varphi h
=M_\varphi h$ for $h\in H^2$, where $M_\varphi$ is the normal
operator of multiplication by $\varphi$ on $L^2$. \ The question is
natural because normal and analytic Toeplitz operators are fairly
well understood, and they are both subnormal. \ In 1984, Halmos'
Problem 5 was answered in the negative by C. Cowen and J. Long
\cite{CoL}. \ However,  Cowen and Long's construction does not
provide an intrinsic connection between subnormality and the theory
of Toeplitz operators. \ Until now researchers have been unable to
characterize subnormal Toeplitz operators in terms of their symbols.
\ In this sense, we reformulate Halmos' Problem 5:
$$
\hbox{Which Toeplitz operators are subnormal\,?}
$$
The most interesting partial answer to Halmos' Problem 5 was given
by M.B. Abrahamse \cite{Ab}, who gave a general sufficient condition
for the answer to Halmos' Problem 5 to be affirmative. \

\medskip

\noindent{\bf Abrahamse's Theorem} (\cite[Theorem]{Ab}). \quad Let
$\varphi\in L^\infty$ be such that $\varphi$ or $\overline \varphi$
is of bounded type. \ If
\begin{itemize}
\item[(i)] $T_\varphi$ is hyponormal;
\item[(ii)] $\hbox{\rm ker}\,[T_\varphi^*, T_\varphi]$ is invariant under $T_\varphi$,
\end{itemize}
then $T_\varphi$ is normal or analytic.

\medskip

Consequently, if $\varphi\in L^\infty$ is such that $\varphi$ or
$\overline \varphi$ is of bounded type, then every subnormal
Toeplitz operator must be either normal or analytic, since
$\ker\,[S^*,S]$ is invariant under $S$ for every subnormal operator
$S$. \ It is actually sufficient to assume that $S$ is
$2$-hyponormal. \ We say that a block Toeplitz operator $T_\Phi$ is
{\it analytic} if $\Phi\in H^\infty_{M_n}$. \ Evidently, any
analytic block Toeplitz operator with a normal symbol is subnormal
because the multiplication operator $M_\Phi$ is a normal extension
of $T_\Phi$. \ As a first inquiry in the above reformulation of
Halmos' Problem 5 the following question can be raised:
$$
\hbox{\it Is Abrahamse's Theorem valid for Toeplitz operators with
matrix-valued symbols}\,?
$$
In general, a straightforward matrix-valued version of Abrahamse's
Theorem is doomed to fail: for instance, if
$$
\Phi:=\begin{pmatrix} z+\overline z&0\\ 0&z\end{pmatrix},
$$
then clearly, both $\Phi$ and $\Phi^*$ are of bounded type and
$$
T_\Phi=\begin{pmatrix}U + U^*&0\\ 0&U
\end{pmatrix}\quad\hbox{(where $U$ is the shift on $H^2$) }
$$
is subnormal, but neither normal nor analytic. \

In this chapter we extend the above result to the case of
bounded type symbols: we shall say that $T_\Phi$ has a {\it bounded type symbol} if both $\Phi$ and
$\Phi^*$ are of bounded type. \

Recently, it was shown in
\cite{CHL1} that if $\Phi\in L^\infty_{M_n}$ is such that $\Phi$ and $\Phi^*$ are
of bounded type of the form
\begin{equation}\label{77}
\Phi_-=\theta B^*\quad\hbox{(coprime)}
\end{equation}
and if $T_\Phi$ is hyponormal and
$\hbox{\rm ker}\,[T_\varphi^*, T_\varphi]$ is invariant under $T_\varphi$,
then $T_\Phi$ is normal or analytic. \
However, the condition (\ref{77}) forces the inner part of the right coprime factorization
(\ref{2.9}) of $\Phi_-$ to be diagonal-constant. \
Also, it was shown in \cite{CHKL} that
if $\Phi$ is a matrix-valued {\it rational} function then the condition (\ref{77})
can be weakened to the condition that the inner part of the right coprime factorization
(\ref{2.9}) of $\Phi_-$ has a nonconstant diagonal-constant inner divisor. \
We note that in view of Lemma \ref{lem55.2-1}, those conditions of \cite{CHL1} and \cite{CHKL} are
special cases of the condition of ``having a tensored-scalar
singularity." \ Indeed, in this chapter, we will show that
for a bounded type symbols $\Phi\in L^\infty_{M_n}$,
if $\Phi$ has a tensored-scalar singularity then we get a full-fledged
matrix-valued version of Abrahamse's Theorem. \

\medskip

To proceed, we need the following result. \

\begin{lemma}\label{lem6.1}
Let $\Phi\in L^\infty_{M_n}$ be such that $\Phi^*$ is of bounded
type. \ If $T_\Phi$ is hyponormal, then there exists an inner
function $\theta$ such that $\hbox{ran}\, [T_\Phi^*,
T_\Phi]\subseteq \mathcal H (I_\theta)$.
\end{lemma}

\begin{proof}
If $\Phi^*$ is of bounded type then, in view of (\ref{2.6}),
$$
\Phi^*=A\Theta^*\ \ (\Theta\equiv I_\theta,\ A\in H^2_{M_n}).
$$
If $T_\Phi$ is hyponormal then by Lemma \ref{gu}, $\Phi$ is normal. \
Thus by
(\ref{1.3}), $[T_\Phi^*, T_\Phi]=H_{\Phi^*}^* H_{\Phi^*}-
H_\Phi^* H_\Phi$. \ Thus if $T_\Phi$ is hyponormal then
$\hbox{ker}\, H_{\Phi^*}\subseteq \hbox{ker}\, H_\Phi$, and hence
$\hbox{ker}\, H_{\Phi^*}\subseteq \hbox{ker}\, [T_\Phi^*, T_\Phi]$,
which gives $\Theta H^2_{\mathbb C^n}\subseteq \hbox{ker}\,
[T_\Phi^*, T_\Phi]$, or $\hbox{ran}\, [T_\Phi^*,
T_\Phi]\subseteq \mathcal H (\Theta)$.
\end{proof}

\medskip

The following lemma shows what tensored-scalar singularities do in the
passage from hyponormality to subnormality.

\medskip

\begin{lemma}\label{lem6.3} \
Let $\Phi\in L^\infty_{M_n}$ be such that $\Phi^*$ is of bounded
type. \ Suppose $T_\Phi$ is hyponormal; thus, in view of Lemma
\ref{lem6.1}, there exists an inner function $\Omega\equiv I_\omega\in
H^\infty_{M_n}$ {\rm (}$\omega$ a nonconstant scalar inner function{\rm )}
such that $\hbox{\rm ran}\,[T_\Phi^*,
T_\Phi]\subseteq \mathcal H(\Omega)$. \ If $\Phi^* \Omega$ has a
tensored-scalar singularity then $T_\Phi$ is normal.
\end{lemma}

\begin{proof} \
By assumption, $\Omega H^2_{\mathbb C^n} \subseteq
\hbox{ker}[T_{\Phi}^*, T_{\Phi}]$. \
If $T_\Phi$ is hyponormal then by Lemma \ref{gu}, $\Phi$ is normal. \
Thus by
(\ref{1.3}), we have
$$
\aligned 0 &=T_{\Omega}^*[T_{\Phi}^*,
  T_{\Phi}]T_{\Omega}
       =T_{\Omega^*\Phi^*}T_{\Phi\Omega}-T_{\Omega^*\Phi}T_{\Phi^*\Omega}\\
&=\Bigl(T_{\Omega^*\Phi^*\Phi\Omega}-H_{\Phi\Omega}^*H_{\Phi\Omega}\Bigr)-
  \Bigl(T_{\Omega^*\Phi\Phi^*\Omega}-H_{\Phi^*\Omega}^*H_{\Phi^*\Omega}\Bigr)\\
&=H_{\Phi^*\Omega}^*H_{\Phi^*\Omega}-H_{\Phi\Omega}^*H_{\Phi\Omega}.
\endaligned
$$
Thus if $\Phi^*\Omega$ has a tensored-scalar singularity, then it follows
from Theorem \ref{lem3.13} that
\begin{equation}\label{37}
\Phi^*\Omega-U\Phi\Omega \in H^{\infty}_{M_n}\ \ \hbox{for some
unitary constant $U\in M_n$}.
\end{equation}
Since $T_{\Phi}$ is hyponormal, by Lemma \ref{gu}, there exists a
matrix function $K\in H^\infty_{M_n}$ such that $\Phi-K\Phi^*\in
H^\infty_{M_n}$. \ Thus it follows from (\ref{37}) that
$\Phi^*\Omega-UK\Phi^*\Omega\in H^{\infty}_{M_n}$, so that
$(I-T_{\widetilde{UK}}^*)H_{\Phi^*\Omega}=0$, which gives
\begin{equation}\label{37-1}
\hbox{ran}\,H_{\Phi^* \Omega}\subseteq \hbox{ker}\,
(I-T_{\widetilde{UK}}^*).
\end{equation}
On the other hand, since $\Phi^*\Omega$ has a tensored-scalar singularity, it
follows  that $\hbox{ker}\, H_{\Phi^*\Omega} \subseteq \zeta
H^2_{\mathbb C^n}$ ($\zeta$ nonconstant inner). \ Thus by Lemma
\ref{lem55.2-1}, we can write
$$
\Phi^*\Omega= B(\zeta \Delta)^* \quad(\hbox{right coprime}).
$$
We thus have
$$
\widetilde{\Phi^*\Omega}=\overline{\widetilde{\zeta}}
\widetilde{\Delta}^*\widetilde{B}.
$$
Since $I_\zeta$ and $B$ are right coprime, it follows from Theorem
\ref{lem33.9} that  $I_\zeta$ and $B$ are left coprime, so that
$I_{\widetilde{\zeta}}$ and $\widetilde{B}$ are right coprime. \ We
then have $\hbox{ker}\, H_{\Phi^*\Omega}^* =\hbox{ker}\,
H_{\widetilde{\Phi^*\Omega}} \subseteq \widetilde{\zeta}
H^2_{\mathbb C^n}$, which together with (\ref{37-1}) implies
$$
\mathcal H (I_{\widetilde{\zeta}}) \subseteq
\hbox{cl ran}\,H_{\Phi^* \Omega} \subseteq \hbox{ker}\,
(I-T_{\widetilde{UK}}^*).
$$
The same argument as the one used right after (\ref{39-9}) in
Theorem \ref{lem3.13} shows that $UK=I$, and hence $K=U^*$ is a
constant unitary. \ Therefore $[T_{\Phi}^*, T_{\Phi}]=
H_{\Phi^*}^*(I-T_{\widetilde{K}}T_{\widetilde{K}}^*)H_{\Phi^*}=0$. \
This completes the proof.
\end{proof}

\medskip

The main theorem of this chapter now follows.

\begin{theorem}\label{thm2.2} {\rm (Abrahamse's Theorem for matrix-valued symbols)} \
Let $\Phi\in L^\infty_{M_n}$ be such that $\Phi$ and $\Phi^*$ are of
bounded type. \ Assume $\Phi$ has a tensored-scalar singularity. \ If
\begin{itemize}
\item[(i)] $T_\Phi$ is hyponormal;
\item[(ii)] $\hbox{\rm ker}\,[T_\Phi^*, T_\Phi]$ is invariant under $T_\Phi$,
\end{itemize}
then $T_\Phi$ is normal. \ Hence, in particular, if $T_\Phi$ is
subnormal then $T_\Phi$ is normal.
\end{theorem}

\medskip

\begin{proof}
Suppose $T_\Phi$ is hyponormal and $\text{\rm ker}\,[T_{\Phi}^*,
T_{\Phi}]$ is invariant under $T_{\Phi}$. \ Let $\Phi$ and $\Phi^*$
be of bounded type. \ Thus in view of (\ref{2.9}), we may write
$$
\Phi_- = \Theta B^*\,\quad\hbox{\rm (right coprime)}.
$$
Suppose $\Phi$ has a tensored-scalar singularity, so that by Lemma
\ref{lem55.2-1}, $\Theta$ has a nonconstant diagonal-constant inner
function $I_\theta $. \ Thus $\Theta = \theta \Theta_1$ for some
inner matrix function $\Theta_1$. \ Since $T_\Phi$ is hyponormal we
may write, in view of (\ref{3.222}),
$$
\Phi_+=\Theta\Theta_2 A^*\quad\hbox{(right coprime)},
$$
where $\Theta_2$ is an inner matrix function. \ Since by Lemma \ref{gu}, $\Phi$ is
normal it follows from (\ref{1.3}) that
\begin{equation}\label{3.2.1}
[T_\Phi^*,T_\Phi]
=H_{A\Theta_2^*\Theta^*}^*H_{A\Theta_2^*\Theta^*} -
H_{B\Theta^*}^*H_{B\Theta^*},
\end{equation}
which implies $\Theta\Theta_2 H^2_{\mathbb C^n}\subseteq
\hbox{ker}\,[T_\Phi^*,T_\Phi]=
\bigl(\hbox{ran}\,[T_\Phi^*,T_\Phi]\bigr)^\perp$, so that
\begin{equation}\label{3.3-0}
\hbox{ran}\,[T_\Phi^*,T_\Phi]\subseteq \mathcal{H}
(\Theta\Theta_2)\,.
\end{equation}
We first suppose $\hbox{ran}\,[T_\Phi^*,T_\Phi] \subseteq \mathcal H
(\Theta_1\Theta_2)$. \ Since $\ker H_{\Phi^*\Theta_1\Theta_2}= \ker
H_{\overline\theta A}=\theta H^2_{\mathbb C^n}$, so that
$\Phi^*\Theta_1\Theta_2$ has a tensored-scalar singularity it follows from
Lemma \ref{lem6.3} with $\Omega\equiv \Theta_1\Theta_2$ that
$T_\Phi$ is normal. \ We thus suppose that
$\hbox{ran}\,[T_\Phi^*,T_\Phi]$ is not contained in $\mathcal H
(\Theta_1\Theta_2)$. \ In this case, we assume to the contrary that
$T_\Phi$ is not normal. \ In view of Lemma \ref{lem6.1}, there
exists a diagonal-constant inner function $\Omega\equiv I_\omega$
such that
\begin{equation}\label{3.4}
\hbox{ran}\,[T_\Phi^*,T_\Phi]\subseteq \mathcal{H}
(\Omega),\quad\hbox{or equivalently,}\quad \Omega^*
\left(\hbox{ran}\,[T_\Phi^*,T_\Phi]\right)\subseteq (H^2_{\mathbb
C^n})^\perp.
\end{equation}
By Lemma \ref{lem6.3}, $\Phi^*\Omega$ has no tensored-scalar singularity. \
But since $\Phi_+^*\Omega = (\overline\theta\omega)
A\Theta_2^*\Theta_1^*$, it follows from Lemma \ref{lem55.2-1} that
$\theta$ is an inner divisor of $\omega$. \ Let
$h\in\hbox{ran}\,[T_\Phi^*,T_\Phi]$ be an arbitrary vector. \ Since
by assumption, $\hbox{\rm ker}\,[T_\Phi^*, T_\Phi]$ is invariant for
$T_\Phi$, we have $T_\Phi^*
\left(\hbox{ran}\,[T_\Phi^*,T_\Phi]\right)\subseteq
\hbox{ran}\,[T_\Phi^*,T_\Phi]$, and hence, $T_{\Phi}^*
h\in\hbox{ran}\,[T_\Phi^*,T_\Phi]$. \ From (\ref{3.4}), $\Omega^*
\Phi^* h\in (H^2_{\mathbb C^n})^\perp$. \ Thus we have
$$
\Omega^*\Phi^*h
          =\Omega^*(\Phi_+^*+ \theta \Theta_1 B^*)h
          = \Phi_+^*(\Omega^*h) + \theta \Theta_1\Omega^* B^* h \in (H^2_{\mathbb C^n})^\perp.
$$
Since by (\ref{3.4}), $\Omega^* h\in (H^2_{\mathbb C^n})^\perp$, it
follows that
\begin{equation}\label{3.10}
\theta B^*\Omega^* h\in (H^2_{\mathbb C^n})^\perp.
\end{equation}
Since $\Phi_- = \Theta B^*\in zH^2_{M_n}$, we can choose $B\in
\mathcal{K}_\Theta$. \ Thus by Lemma \ref{lem22.55}
 we may write
\begin{equation}\label{15}
B=\theta B_1 + B_2,
\end{equation}
where $B_1 \in \mathcal{K}_{\Theta_1}$ and $B_2 \in
\mathcal{K}_\theta$. \ Then it follows from (\ref{3.10}) and
(\ref{15}) that
$$
B_1^* \Omega^* h+ (\overline{\omega}\theta) B_2^* h \in
(H^2_{\mathbb C^n})^\perp,\ \ \hbox{so that}\ \
(\overline{\omega}\theta) B_2^* h \in (H^2_{\mathbb C^n})^\perp.
$$
We thus have
\begin{equation}\label{17}
\bigl(\overline{z}\overline{\widetilde{\theta}}\widetilde{B}_2\bigr)\bigl(\widetilde{\omega}
\overline{\widetilde h}\bigr) \in H^2_{\mathbb C^n}, \quad\hbox{or
equivalently,}\quad\widetilde{\omega} \overline{\widetilde h} \in
\hbox{ker}\,H_{\overline{z}\overline{\widetilde{\theta}}\widetilde{B}_2}.
\end{equation}
We now claim that
\begin{equation}\label{299}
\widetilde{\omega} \overline{\widetilde h} \in
 z\widetilde{\theta} H^2_{\mathbb C^n}.
\end{equation}
To see this we first suppose that $\theta$ and $z$ are not coprime.
\ Then $z$ is an inner divisor of $\theta$.  \ Since $B$ and
$\Theta$ are right coprime, it follows from Theorem \ref{lem33.9}
that $B$ and $I_\theta$ are coprime, and hence  $B$ and
$I_{z\theta}$ are coprime.\ Thus by (\ref{15}), $B_2$ and
$I_{z\theta}$ are coprime. \ Therefore $\widetilde{B}_2$ and
$I_{z\widetilde{\theta}}$ are coprime, which implies that, by
(\ref{RCD}) together with (\ref{17}), $ \widetilde{\omega}
\overline{\widetilde h} \in \hbox{ker}\,H_{\overline{z}
\overline{\widetilde{\theta}}\widetilde{B}_2}=z\widetilde{\theta}H^2_{\mathbb
C^n}$, giving (\ref{299}).

Suppose now that $\theta$ and $z$ are coprime. \ Then by (\ref{RCD})
and (\ref{17}) we have
\begin{equation}\label{18}
\widetilde{\omega} \overline{\widetilde h} \in
\hbox{ker}\,H_{\overline{z}\overline{\widetilde{\theta}}\widetilde{B}_2}\subseteq
\hbox{ker}\,H_{\overline{\widetilde{\theta}}\widetilde{B}_2}=
\widetilde{\theta}H^2_{\mathbb C^n}.
\end{equation}
But since $\Omega^*h\in (H^2_{\mathbb C^n})^\perp$, and hence
$\widetilde{\omega} \overline{\widetilde h} \in z H^2_{\mathbb
C^n}$, we have
$$
\widetilde{\omega} \overline{\widetilde h} \in
\widetilde{\theta}H^2_{\mathbb C^n} \bigcap z H^2_{\mathbb C^n} =
z\widetilde{\theta} H^2_{\mathbb C^n},
$$
giving (\ref{299}). \ Now by (\ref{299}), we have
$$
(\overline{z}\overline{\widetilde{\theta}}\widetilde{\omega})
\overline{\widetilde h} \in H^2_{\mathbb C^n}, \ \ \hbox{so that} \
\ (\theta\overline{\omega}) h \in (H^2_{\mathbb C^n})^{\perp}.
$$
Therefore we can see that
$$
(\overline\theta\Omega)^* h \in (H^2_{\mathbb C^n})^\perp \
\hbox{for each} \ h \in\hbox{ran}\,[T_\Phi^*,T_\Phi],
$$
which implies that $ \hbox{ran}\,[T_\Phi^*,T_\Phi]\subseteq
\mathcal{H} ({\overline{\theta}\Omega})= \mathcal{H} ({I_{\omega
\overline{\theta}}})$. \ Thus we can repeat the above argument with
$\overline\theta\Omega$ in place of $\Omega$ in (\ref{3.4}). \ Then
induction on $p$ shows that
$$
\hbox{ran}\,[T_\Phi^*,T_\Phi]\subseteq  \mathcal{H} (I_{\omega
\overline{\theta}^p}) \ \hbox{for all} \ p\in \mathbb Z_+.
$$
In particular, $\omega \overline{\theta}^p \in H^2$, and hence
$\theta^p$ is an inner divisor of $\omega$ for all $p \in \mathbb
Z_+$. \ Thus it follows from Lemma \ref{lem1.4} that
$$
\omega H^2 \subseteq \bigcap_{p=1}^{\infty}{\theta}^p H^2=\{0\},
$$
a contradiction. \ Therefore $T_\Phi$ should be normal. \ This
proves the first assertion. \

The second assertion follows at once from the fact that if $T_\Phi$
is subnormal then $\hbox{\rm ker}\,[T_\Phi^*, T_\Phi]$ is invariant
for $T_\Phi$. This completes the proof of Theorem \ref{thm2.2}.
\end{proof}

\medskip

\begin{remark}
(a) We note that the assumption ``$\Phi$ has a tensored-scalar singularity"
is essential in Theorem \ref{thm2.2}. \ As we saw before, if
$\Phi:=\left(\begin{smallmatrix} \overline z+z&0\\
0&z\end{smallmatrix}\right)$, then $T_\Phi$ is neither normal nor
analytic. \ But since $\ker H_\Phi=\ker
H_{\left(\begin{smallmatrix} \overline z&0\\
0&0\end{smallmatrix}\right)} =\left(\begin{smallmatrix} z&0\\
0&1\end{smallmatrix}\right) H^2_{\mathbb C^n}$, it follows that
$\Theta\equiv \left(\begin{smallmatrix} z&0\\
0&1\end{smallmatrix}\right)$ does not have any nonconstant
diagonal-constant inner divisor, so that $\Phi$ does not have a
tensored-scalar singularity. \

(b) If $n=1$, then $\Theta\equiv \theta\in H^\infty$ is vacuously
diagonal-constant, so that Theorem \ref{thm2.2} reduces to the
original Abrahamse's Theorem. \hfill$\square$
\end{remark}

\bigskip

We will next give an example that illustrates Theorem \ref{thm2.2}.\
To do so we need some auxiliary observations. \ If $\Phi=\Phi_-^* +
\Phi_+\in L^{\infty}_{M_n}$ is such that $\Phi$ and $\Phi^*$ are of
bounded type and for which $T_\Phi$ is hyponormal, we will, in view
of (\ref{2.8}) and (\ref{3.222-1}), assume that
\begin{equation}\label{5-0}
\Phi_+=  A^* \Omega_1\Omega_2\quad \hbox{and} \quad \Phi_-=
B_\ell^*\Omega_2\ \hbox{(left coprime)},
\end{equation}
where $\Omega_1 \Omega_2=\Theta=I_\theta$. \

\medskip

\begin{lemma}\label{lem3.11} \
Let $\Phi=\Phi_-^*+ \Phi_+\in L^\infty_{M_n}$ be such that $\Phi$
and $\Phi^*$ are of bounded type. \ Then in view of (\ref{5-0}), we
may write
$$
\Phi_+=A^*\Theta_0\Theta_2\quad\hbox{and}\quad \Phi_- = B^*
\Theta_2\ \hbox{(left coprime)},
$$
where $\Theta_0\Theta_2=I_\theta $, with $\theta$ an inner function.
\ If  $T_\Phi$ is hyponormal and $\text{\rm ker}\,[T_{\Phi}^*,
T_{\Phi}]$ is invariant under $T_{\Phi}$\,, then
\begin{equation}\label{3.9}
\Theta_0 H^2_{\mathbb C^n}\ \subseteq\ \hbox{\rm ker}\,[T_{\Phi}^*,
T_{\Phi}].
\end{equation}
Moreover if $K\in \mathcal{C}(\Phi)$, then
\begin{equation}\label{3.15}
\hbox{\rm cl ran}\, H_{A\Theta_2^*}\subseteq \hbox{\rm
ker}\,(I-T_{\widetilde{K}}T_{\widetilde{K}}^*).
\end{equation}
\end{lemma}

\begin{proof}
The inclusion (\ref{3.9}) follows from a slight extension of
\cite[Theorem 3.5]{CHL2} , in which $\Theta_2$ is a
diagonal-constant inner function of the form
$\Theta_2=I_{\theta_2}$. \ However, a careful analysis of
\cite[Proof of Theorem 3.5, Step 1]{CHL2} shows that the proof does
not employ the diagonal-{\it constancy} of $\Theta_2$, but uses only
the diagonal- constancy of $\Theta_0\Theta_2$. \ Towards
(\ref{3.15}), we observe that if $K\in \mathcal{C}(\Phi)$, then by
(\ref{1.4}),
\begin{equation}\label{3.13-2}
[T_\Phi^*,
T_\Phi]=H_{\Phi_+^*}^*H_{\Phi_+^*}-H_{K\Phi_+^*}^*H_{K\Phi_+^*}
=H_{\Phi_+^*}^* (I- T_{\widetilde K}T_{\widetilde K}^*)
H_{\Phi_+^*}\,,
\end{equation}
so that $\hbox{ker}\,[T_\Phi^*, T_\Phi]=\hbox{ker}\,(I-
T_{\widetilde K}T_{\widetilde K}^*) H_{\Phi_+^*}$. \ Thus by
(\ref{3.9}),
$$
\{0\}=(I-T_{\widetilde{K}}T_{\widetilde{K}}^*)H_{A\Theta_2^*\Theta_0^*}(\Theta_0
H_{\mathbb C^n}^2) =
(I-T_{\widetilde{K}}T_{\widetilde{K}}^*)H_{A\Theta_2^*}(H_{\mathbb
C^n}^2),
$$
giving (\ref{3.15}).
\end{proof}

\medskip

In \cite{GHR}, the normality of block Toeplitz operator $T_\Phi$ was
also characterized in terms of the symbol $\Phi$ under a
``determinant" condition on the symbol $\Phi$.

\medskip

\noindent
\begin{lemma}\label{lem1.3} {\rm (Normality of Block Toeplitz Operators)
(\cite{GHR})} Let $\Phi\equiv \Phi_+ +\Phi_-^*$ be normal and
$\Phi_+^0:\label{phi+0}=\Phi_+-\Phi_+(0)$. If $\det\,\Phi_+\ne 0$,
then
\begin{equation}\label{1.3-1}
T_\Phi\ \hbox{is normal}\Longleftrightarrow \Phi_+^0=\Phi_-U\ \
\hbox{for some constant unitary matrix}\ U.
\end{equation}
\end{lemma}

\medskip

We then have:

\medskip

\begin{example} \label{ex5.9}
Let $\varphi, \psi \in L^{\infty}$ be of bounded type and consider
$$
\Phi:=\begin{pmatrix} {\overline b_\alpha} & \varphi\\ \psi&
{\overline b_\alpha}\end{pmatrix}\quad\hbox{($\alpha\in\mathbb
D$)}\,,
$$
where $b_\alpha$ is a Blaschke factor of the form
$b_\alpha(z):=\frac{z-\alpha}{1-\overline \alpha z}$. \ In view of
(\ref{2.2}) we may write
$$
\varphi_-:=\theta_0\overline{a}\quad\hbox{and}\quad \psi_-:=\theta_1
\overline{b}\quad\hbox{(coprime)}.
$$
If $\theta_0(\alpha)=0$ and $\theta_1(\alpha)\ne 0$, then $T_\Phi$
is never subnormal.
\end{example}

\begin{proof} \
Write
$$
\Phi\equiv \left(\begin{matrix} \overline b_\alpha&\varphi\\
\psi&\overline b_\alpha\end{matrix}\right) \equiv \Phi_-^*+\Phi_+
=\begin{pmatrix} b_\alpha& \psi_-\\
\varphi_-& b_\alpha\end{pmatrix}^{\ast}+\begin{pmatrix} 0&\varphi_+\\
\psi_+& 0\end{pmatrix}
$$
and assume that $T_\Phi$ is subnormal. \ Since by Lemma \ref{gu}
$\Phi$ is normal, a straightforward calculation shows that
\begin{equation}\label{9999}
|\varphi|=|\psi|.
\end{equation}
Also there exists a matrix function
$K\equiv \left(\begin{smallmatrix} k_1&k_2\\
k_3&k_4\end{smallmatrix}\right) \in \mathcal{E}(\Phi)$, i.e.,
$||K||_\infty\le 1$ such that $\Phi-K\Phi^*\in H^\infty_{M_2}$. \
Thus since $\Phi_-^*-K\Phi_+^*\in H^2_{M_2}$, and hence
$$
\begin{pmatrix} \overline b_\alpha&\overline{\theta}_0 a\\
\overline{\theta}_1 b &\overline b_\alpha
\end{pmatrix}-\begin{pmatrix}
k_2\overline{\varphi_+}&k_1\overline{\psi_+}\\k_4
\overline{\varphi_+}&k_3\overline{\psi_+}\end{pmatrix} \in
H^2_{M_2},
$$
we have
\begin{equation}\label{6.7-77}
\begin{cases}
\overline b_\alpha-k_2 \overline{\varphi_+}\in H^2,\quad
\overline{\theta}_1 b-k_4\overline{\varphi_+}\in H^2\\
\overline b_\alpha-k_3 \overline{\psi_+}\in H^2,\quad
\overline{\theta}_0 a-k_1\overline{\psi_+}\in H^2,
\end{cases}
\end{equation}
which via Cowen's Theorem gives that the following Toeplitz
operators are all hyponormal:
\begin{equation}\label{6.7-55}
T_{\overline b_\alpha+\varphi_+},\ \ T_{\overline{\theta}_1 b +
\varphi_+},\ \ T_{\overline b_\alpha+\psi_+},\ \
T_{\overline{\theta}_0 a + \psi_+}.
\end{equation}
Note that $\varphi_+\psi_+$ is not identically zero, so that
$\hbox{det}\,\Phi_+$ is not. Put
$$
\theta_0=b_{\alpha}^m\theta_0^{\prime}\quad (m\ge 1;\
\theta_0^{\prime}(\alpha)\ne 0).
$$
Note $a(\alpha)\ne 0$ because $\theta_0(\alpha)=0$ and $\theta_0$
and $a$ are coprime. \ Then a straightforward calculation  together
with (\ref{3.15}) shows that
\begin{equation}\label{6.7-p}
m=1.
\end{equation}
Thus we may write, by a direct calculation,
$$
\Phi_-=\begin{pmatrix} \theta_0^\prime & a \\ b_\alpha b &
\theta_1\end{pmatrix}^*
           \begin{pmatrix} \theta_0 &0\\ 0& b_\alpha \theta_1\end{pmatrix} \equiv
                B^*\Omega_2\ \ \hbox{(left coprime).}
$$
But since $\Omega_2\equiv \begin{pmatrix} \theta_0 &0\\
0& b_\alpha \theta_1\end{pmatrix}$ has a diagonal inner divisor
$I_{b_\alpha}$, it follows from Theorem \ref{thm2.2} that $T_\Phi$
is normal. Since $\hbox{det}\,\Phi_+\ne 0$, it follows form Lemma
\ref{lem1.3} that $\Phi_+^0=\Phi_-U$ for some
constant unitary matrix $U\equiv \left(\begin{smallmatrix} c_1&c_2\\
c_3&c_4\end{smallmatrix}\right)$. \ We observe
$$
\begin{aligned}
\Phi_+^0=\Phi_-U
      &\Longleftrightarrow \begin{pmatrix} 0& b_\alpha \theta_1\theta_3\overline d\\
                            \theta_0\theta_2\overline c&0\end{pmatrix}
                              = \begin{pmatrix} b_\alpha &\theta_1\overline b\\ \theta_0\overline a& b_\alpha\end{pmatrix}
                                  \begin{pmatrix} c_1&c_2\\ c_3&c_4\end{pmatrix}\\
      &\Longrightarrow
         \begin{cases} 0=c_1 b_\alpha+c_3\theta_1\overline b\\
                       0=c_4 b_\alpha+c_2\theta_0\overline a
         \end{cases}\\
      &\Longrightarrow
        \begin{cases} c_1=0,\ \theta_1\overline b=0\ \hbox{(i.e., $\theta_1=1$)}\\
                      c_4\ne 0,\ \theta_0=b_\alpha
        \end{cases}\\
      &\Longrightarrow U=\begin{pmatrix} 0&c_2\\ c_3&c_4\end{pmatrix}\quad(c_4\ne 0),
\end{aligned}
$$
which contradicts the fact that $U$ is unitary. \ Therefore $T_\Phi$
is never subnormal.
\end{proof}

\bigskip

%
%
%
%
%
%

\chapter{A subnormal Toeplitz completion}

\bigskip

In this chapter
we consider a subnormal ``Toeplitz" completion problem. \

Given a partially specified operator matrix with some known entries,
the problem of finding suitable operators to complete the given
partial operator matrix so that the resulting matrix satisfies
certain given properties is called a {\it completion problem}. \ A
{\it subnormal completion} of a partial operator matrix is a
particular specification of the unspecified entries resulting in a
subnormal operator. \ A {\it partial block Toeplitz matrix} is
simply an $n\times n$ matrix some of whose entries are specified
Toeplitz operators and whose remaining entries are unspecified. \ A
{\it subnormal Toeplitz completion} of a partial block Toeplitz
matrix is a subnormal completion whose unspecified entries must be
Toeplitz operators. \

We now consider:

\begin{problem}\label{probA} \
Let $b_\lambda$ be a Blaschke factor of the form
$b_\lambda(z):=\frac{z-\lambda}{1-\overline \lambda z}$ {\rm
(}$\lambda\in \mathbb D${\rm )}. \ Complete the unspecified Toeplitz
entries of the partial block Toeplitz matrix
$$
A:=\begin{pmatrix} T_{\overline b_\alpha} & ?\\ ?& T_{\overline
b_\beta}\end{pmatrix}\quad\hbox{{\rm (}$\alpha,\beta\in\mathbb
D${\rm )}}
$$
to make $A$ subnormal.
\end{problem}

\medskip

Recently, in \cite{CHL2}, we have considered Problem \ref{probA} for the cases
$\alpha=\beta=0$. \ The solution given in \cite[Theorem 5.1]{CHL2}
relies upon very intricate and long computations using the symbol
involved. \ However our solution in this chapter provides a
shorter and more insightful proof by employing the results of the
previous chapter. \

We now give an answer to Problem \ref{probA}:

\medskip

\begin{theorem} \label{thm4.2}
Let $\varphi, \psi \in L^{\infty}$ and consider
$$
A:=\begin{pmatrix} T_{\overline b_\alpha} & T_\varphi\\ T_\psi&
T_{\overline b_\beta}\end{pmatrix}\quad\hbox{{\rm
(}$\alpha,\beta\in\mathbb D${\rm )}}\,.
$$
The following statements are equivalent.
\begin{itemize}
\item[(a)] $A$ is normal;
\item[(b)] $A$ is subnormal;
\item[(c)] $A$ is $2$-hyponormal,
\end{itemize}
except in the following special case:
\begin{equation}\label{6.6-0}
\hbox{$\varphi_-=b_\alpha \theta_0^{\prime} \overline a$ and
$\psi_-=b_\alpha \theta_1^{\prime} \overline b$ {\rm (}coprime{\rm
)}\  with $(ab)(\alpha)=(\theta_0^{\prime}
\theta_1^{\prime})(\alpha)\ne 0$\,.}
\end{equation}
However, unless only one of $\theta_0^\prime$ and $\theta_1^\prime$
is constant, the exceptional case (\ref{6.6-0}) implies that if $A$
is subnormal then either $A$ is normal or $A-\beta$ is quasinormal
for some $\beta\in\mathbb C$.
\end{theorem}

\begin{proof} \
Clearly (a) $\Rightarrow$ (b) and (b) $\Rightarrow$ (c). \

(c) $\Rightarrow$ (a): \ Write
$$
\Phi\equiv \left(\begin{matrix} \overline b_\alpha&\varphi\\
\psi&\overline b_\beta\end{matrix}\right) \equiv \Phi_-^*+\Phi_+
=\begin{pmatrix} b_\alpha& \psi_-\\
\varphi_-& b_\beta\end{pmatrix}^{\ast}+\begin{pmatrix} 0&\varphi_+\\
\psi_+& 0\end{pmatrix}
$$
and assume that $T_\Phi$ is $2$-hyponormal. \ Since $\hbox{ker}\,
[T^*, T]$ is invariant under $T$ for every $2$-hyponormal operator
$T \in \mathcal{B}(\mathcal{H})$ (cf. \cite{CuL2}), we note that
Theorem \ref{thm2.2} holds for $2$-hyponormal operators $T_{\Phi}$
under the same assumption on the symbol. \ If $T_\Phi$ is hyponormal
then by Lemma \ref{gu} there exists a matrix function
$K\equiv \left(\begin{smallmatrix} k_1&k_2\\
k_3&k_4\end{smallmatrix}\right) \in H^\infty_{M_2}$ such that
$||K||_\infty\le 1$ and $\Phi-K\Phi^*\in H^\infty_{M_2}$, i.e.,
\begin{equation}\label{6.6-1}
\begin{pmatrix} \overline b_\alpha&\overline{\varphi_-}\\
\overline{\psi_-}& \overline b_\beta\end{pmatrix}\, -\, \begin{pmatrix} k_1&k_2\\
k_3&k_4\end{pmatrix}\, \left(\begin{matrix} 0&\overline {\psi_+}\\
\overline {\varphi_+}&0\end{matrix}\right) \in H^2_{M_2},
\end{equation}
which implies that
\begin{equation}\label{6.6-1-1}
H_{\overline b_\alpha}=H_{k_2\overline {\varphi_+}}=H_{\overline
{\varphi_+}}T_{k_2} \quad\hbox{and}\quad H_{\overline
b_\beta}=H_{k_3\overline{\psi_+}}=H_{\overline {\psi_+}}T_{k_3}.
\end{equation}
If $\overline{\varphi_+}$ is not of bounded type then
$\hbox{ker}\,H_{\overline{\varphi_+}}=\{0\}$, so that $k_2=0$, a
contradiction; and if $\overline{\psi_+}$ is not of  bounded type
then $\hbox{ker}\,H_{\overline{\psi_+}}=\{0\}$, so that $k_3=0$, a
contradiction. \ Thus since $\overline{\varphi_+}$ and
$\overline{\psi_+}$ are of bounded type, it follows that $\Phi^*$ of
bounded type. \ Since $T_{\Phi}$ is hyponormal, it follows from
(\ref{gu3}) that $\Phi$ is also of bounded type. \ Thus we can write
$$
\varphi_-\equiv \theta_0\overline{a}\quad\hbox{and}\quad \psi_-
\equiv \theta_1 \overline{b}\quad\hbox{(coprime)}.
$$
Since $\Phi$ is normal, i.e., $\Phi\Phi^*=\Phi^*\Phi$, a
straightforward calculation shows that $\alpha=\beta$. \ Thus by
(\ref{6.6-1}), we have
\begin{equation}\label{6.7-77-1}
\begin{cases}
\overline b_\alpha-k_2 \overline{\varphi_+}\in H^2,\quad
\overline{\theta}_0 a-k_1\overline{\psi_+}\in H^2\\
\overline b_\alpha-k_3 \overline{\psi_+}\in H^2,\quad

\overline{\theta}_1 b-k_4\overline{\varphi_+}\in H^2\, ,
\end{cases}
\end{equation}
which implies that the following Toeplitz operators are all
hyponormal (via Cowen's Theorem): \
\begin{equation}\label{6.7-55-1}
T_{\overline b_\alpha+\varphi_+},\ \ T_{\overline{\theta}_1 b +
\varphi_+},\ \ T_{\overline b_\alpha+\psi_+},\ \
T_{\overline{\theta}_0 a + \psi_+}.
\end{equation}
Put
$$
\theta_0=b_{\alpha}^m\theta_0^{\prime}\quad\hbox{and}\quad
\theta_1=b_{\alpha}^n \theta_1^{\prime}\qquad (m,n\ge 0;\
\theta_0^{\prime}(\alpha)\ne 0,\ \theta_1^{\prime}(\alpha)\ne 0).
$$
By Example \ref{ex5.9}, if $m\ne 0$ and $n=0$ then we get a
contradiction. \ Also a similar argument to Example \ref{ex5.9}
shows that
$$
\hbox{either $m=n=0$ or $m=n=1$}.
$$
Thus we have to consider the case $m=n=0$ and the case $m=n=1$. \

\medskip

\noindent  {\bf Case A} ($m=n=0$):  A straightforward calculation
shows that $\ker H_\Phi=\Theta H^2_{\mathbb C^2}$, where
$$
\Theta\equiv \begin{pmatrix} b_\alpha\theta_1&0\\
0&b_\alpha\theta_0\end{pmatrix}.
$$
Since $\Theta$ has a diagonal-constant inner divisor $I_{b_\alpha}$,
it follows from Lemma \ref{lem55.2-1} and Theorem \ref{thm2.2} that
$T_\Phi$ is normal. \

\bigskip

%
%

\noindent {\bf Case B-1} ($m=n=1$; $(ab)(\alpha)\ne
(\theta_0^\prime\theta_1^\prime)(\alpha)$): A straightforward
calculation shows that $\ker H_\Phi=\Theta H^2_{\mathbb C^2}$, where
$$
\Theta\equiv \begin{pmatrix} b_\alpha\theta_1^\prime&0\\
0&b_\alpha\theta_0^\prime\end{pmatrix}.
$$
Since $\Theta$ has a diagonal-constant inner divisor $I_{b_\alpha}$,
it follows from Lemma \ref{lem55.2-1} and Theorem \ref{thm2.2} that
$T_\Phi$ is normal. \

\bigskip

%
%

\noindent {\bf Case B-2} ($m=n=1$; $(ab)(\alpha)=(\theta_0^{\prime}
\theta_1^{\prime})(\alpha)$): A straightforward calculation shows
that

$$
\hbox{ker} H_{\widetilde\Phi} =\widetilde \Theta_2 H^2_{\mathbb C^2}
,
$$
where
$$
\Theta_2:= \nu  \begin{pmatrix}
            \theta_0 & -\overline \gamma\theta_1 \\
               {\gamma}\theta_0^\prime & \theta_1^\prime \end{pmatrix}
                \quad  ({\gamma}=-\frac{b(\alpha)}{\theta_0^{\prime}(\alpha)}
                    =-\frac{\theta_1^{\prime}(\alpha)}{a(\alpha)}; \ \nu:= \frac{1}{\sqrt{|\gamma|^2+1}} )
$$
Thus we can write
$$
\widetilde\Phi_- =
\begin{pmatrix}
\widetilde b_\alpha & \widetilde\theta_0 \overline{\widetilde a}\\
\widetilde\theta_1 \overline{\widetilde b} & \widetilde b_\alpha
\end{pmatrix}
=\widetilde \Theta_2 \widetilde B^*\quad \hbox{(right coprime)},
$$
so that
\begin{equation}\label{6.8-90}
\Phi_-=
\begin{pmatrix}
b_\alpha &\theta_1 \overline b\\ \theta_0\overline a &b_\alpha
\end{pmatrix} =B^*\Theta_2\quad\hbox{(left coprime)}\,.
\end{equation}
By  a scalar-valued version of (\ref{3.222-1}) and (\ref{6.7-55-1}),
we can see that $\varphi_+=\theta_1 \theta_3\overline{d}$ and
$\psi_+=\theta_0\theta_2\overline{c}$ for some inner functions
$\theta_2,\theta_3$, where $d\in\mathcal{H} ({z\theta_1 \theta_3})$
and $c\in \mathcal H ({z\theta_0\theta_2})$. \ Thus,  in particular,
$c(\alpha)\ne 0$ and $d(\alpha)\ne 0$. \ Let
$$
\theta_2=b_\alpha^q \theta_2^\prime\ \ \hbox{and}\ \
\theta_3=b_\alpha^p \theta_3^\prime\quad(\hbox{where}\
\theta_2^\prime(\alpha)\ne 0\ne \theta_3^\prime(\alpha)).
$$
If we write $\Phi_+=\Theta_0\Theta_2 A^*$
($\Theta_0\Theta_2=I_\theta$ for an inner function $\theta$), then a
straightforward calculation shows that
$$
A= \begin{pmatrix} 0&\theta_1^\prime\theta_3^\prime c\\
            \theta_0^\prime\theta_2^\prime d &0 \end{pmatrix}.
$$
We thus have
\begin{equation}\label{6.8-899}
H_{A\Theta_2^*} \begin{pmatrix} 1\\0\end{pmatrix}
      =\nu \begin{pmatrix}-{\gamma} H_{\overline b_{\alpha}}(\theta_3^{\prime}c)\\
           H_{\overline b_{\alpha}}(\theta_2^{\prime}d)\end{pmatrix}.
\end{equation}
It was known that $\hbox{ran}\, H_{\overline{b}_{\alpha}}=\mathcal H
({b_{\overline{\alpha}}})=\bigvee \{\delta_1\}$ (where $\delta_1:=
\frac{\sqrt{1-|\alpha|^2}}{1-{\alpha}z}$). \ It thus follows from
(\ref{6.8-899}) that
\begin{equation}\label{6.7-81}
\begin{pmatrix} \delta_1\\ \beta \delta_1  \end{pmatrix} \in\hbox{cl ran}\, H_{A\Theta_2^*}\subseteq
\hbox{ker}\,(I-T_{\widetilde{K}}T_{\widetilde{K}}^*),
\end{equation}
where $\beta\ne 0$. \ We observe that if $k\in H^2$, then
\begin{equation}\label{6.7-80}
T_{k(\overline{z})}\delta_1=k(\alpha)\delta_1:
\end{equation}
indeed, if $k\in H^2$ and $n\geq 0$, then
$$
\langle k(\overline{z})\delta_1, \ z^n \rangle=\langle \delta_1, \
\overline{k(\overline{z})}z^n \rangle = \overline{\langle
\widetilde{k}z^n, \ \delta_1
\rangle}=\sqrt{1-|\alpha|^2}\overline{\widetilde{k(\overline{\alpha})}
      \overline{\alpha}^n}=\sqrt{1-|\alpha|^2}k(\alpha)\alpha^n\,,
$$
so that
$$
T_{k(\overline{z})}\delta_1=P(k(\overline{z})\delta_1)
=\sqrt{1-|\alpha|^2}k(\alpha)\sum_{n=0}^{\infty}\alpha^nz^n=k(\alpha)\frac{\sqrt{1-|\alpha|^2}}{1-\alpha
z}=k(\alpha)\delta_1\,,
$$
which proves (\ref{6.7-80}). \ Thus a straightforward calculation
together with (\ref{6.7-81}) shows that
\begin{equation}\label{6.7-10-10}
\alpha_1k_1+\alpha_2k_3=1 \quad \hbox{and} \quad
\alpha_1k_2+\alpha_2k_4=\overline{\beta}\,,
\end{equation}
where $\alpha_1=\overline{k_1(\alpha)+\beta k_2(\alpha)}$ and
$\alpha_2=\overline{k_3(\alpha)+\beta k_4(\alpha)}$. \ We also have,
from (\ref{6.7-81}),
\begin{equation}\label{6.7-8-1}
     \left|\left|\begin{pmatrix} \delta_1\\ \beta\delta_1\end{pmatrix}\right|\right|_2
       = \left|\left|T_{\widetilde{K}}^*\begin{pmatrix} \delta_1\\ \beta\delta_1\end{pmatrix}\right|\right|_2
           =\left|\left| \begin{pmatrix} (k_1(\alpha)+\beta k_2(\alpha))\delta_1\\ (k_3(\alpha)
             +\beta k_4(\alpha))\delta_1\end{pmatrix}\right|\right|_2
           =\left|\left| \begin{pmatrix}  \overline{\alpha_1}\delta_1\\ \overline{\alpha_2}\delta_1 \end{pmatrix}\right|\right|_2\,,
\end{equation}
which implies
\begin{equation}\label{6.7-10-11}
1+|\beta|^2= |\alpha_1|^2+|\alpha_2|^2.
\end{equation}
From (\ref{6.7-77-1}), we can see that
\begin{equation}\label{6.7-11-11}
k_1=\theta_2k_1^{\prime}, \
k_2=\theta_1^{\prime}\theta_3k_2^{\prime}, \ k_3=\theta_0^{\prime}
\theta_2k_3^{\prime}, \ k_4=\theta_3k_4^{\prime},
\end{equation}
where $k_i^{\prime} \in H^{\infty}$ for $i=1,\cdots,4$. \ Thus by
(\ref{6.7-10-10}) and (\ref{6.7-11-11}), we can see that $\theta_2$
and $\theta_3$ are both constant. \ Without loss of generality, we
may assume that $\theta_2=\theta_3=1$. \ We next claim that
\begin{equation}\label{6.7-11-4}
\theta_0=\theta_1=b_\alpha,\ \ \hbox{i.e.,}\ \
\hbox{$\theta_0^\prime$ and $\theta_1^\prime$ are both constant.}
\end{equation}
By our assumption, if $\theta_0^\prime$ or $\theta_1^\prime$ is
constant then
         both $\theta_0^\prime$ and $\theta_1^\prime$ are constant. \
First of all, suppose that both $\theta_0^\prime$ and
$\theta_1^\prime$ have nonconstant Blaschke factors. \ Thus there
exist $v,w\in\mathbb D$ such that $\theta_0'(v)=0=\theta_1'(w)$. \
But since $k_3=\theta_0'k_3'$ and $k_2=\theta_1'k_2'$, it follows
from (\ref{6.7-10-10}) that
\begin{equation}\label{6.7-11-5}
k_1(v)=\frac{1}{\alpha_1}\quad\hbox{and}\quad
k_4(w)=\frac{\overline\beta}{\alpha_2}
\end{equation}
(where we note that $\alpha_1\ne 0$ and $\alpha_2\ne 0$). \ Observe
that $|k_1(v)|=1=|k_4(w)|$: indeed, if $|k_1(v)|<1$, then
$|\alpha_1|>1$, so that by (\ref{6.7-10-11}), $|\alpha_2|<|\beta|$,
which implies $|k_4(w)|>1$, which contradicts the fact
$||K||_\infty\le 1$. If instead $|k_4(w)|<1$, we similarly get a
contradiction. \ Since $||k_1||_\infty\le 1$ and $||k_4||_\infty\le
1$, it follows from the Maximum Modulus Theorem that $k_1$ and $k_4$
are both constant, i.e., \
\begin{equation}\label{6.7-11-6}
k_1=\frac{1}{\alpha_1}\quad\hbox{and}\quad
k_4=\frac{\overline\beta}{\alpha_2}.
\end{equation}
Then from (\ref{6.7-10-10}), we should have $k_2=k_3=0$, which leads
to a contradiction, using (\ref{6.7-77-1}). \

Next, we assume that $\theta_0^\prime$ has a nonconstant Blaschke
factor and $\theta_1'$ has a nonconstant singular inner factor. \
Since $\theta_0'$ has a nonconstant Blaschke factor,
$$
\exists \ w\in\mathbb D\ \hbox{such that}\ \theta_0'(w)=0,\ \hbox{so
that by (\ref{6.7-11-11})},\ k_3(w)=0.
$$
Thus by (\ref{6.7-10-10}), $k_1(w)=\frac{1}{\alpha_1}$. \ But since
$|k_1(w)|<1$ (otherwise $k_1$ would be constant, so that $k_3\equiv
0$, a contradiction from (\ref{6.7-77-1})), it follows that
$1<|\alpha_1|$. \ Thus by (\ref{6.7-10-11}),
\begin{equation}\label{6.8-0}
|\alpha_2|<|\beta|.
\end{equation}
On the other hand, since $\theta_1'$ has a nonconstant singular
inner factor, we can see that there exists $\delta\in [0,2\pi)$ such
that $\theta_1'$ has nontangential limit $0$ at $e^{i\delta}$ (cf.
\cite[Theorem II.6.2]{Ga}). \ Thus  by (\ref{6.7-11-11}), $k_2$ has
nontangential limit $0$ at $e^{i\delta}$ and in turn, by
 (\ref{6.7-10-10}),  $k_4$ has nontangential limit $\frac{\overline\beta}{\alpha_2}$ at $e^{i\delta}$. \
But since $||k_4||_\infty\le 1$, it follows that
$\left|\frac{\overline\beta}{\alpha_2}\right|\le 1$, i.e.,
$|\beta|\le |\alpha_2|$, which contradicts (\ref{6.8-0}). \

Next, if both $\theta_0^\prime$ and $\theta_1^\prime$ have
nonconstant singular factors then a similar argument gives a
contradiction. \ This proves (\ref{6.7-11-4}). \


Now, in view of (\ref{6.8-90}) and (\ref{6.7-11-4}), we may write
$$
\Phi_-=\begin{pmatrix} b_\alpha & b_\alpha \overline b\\ b_\alpha
\overline a& b_\alpha\end{pmatrix}=B^*\Theta_2
      \ \ \hbox{(left coprime),}
\ \ \hbox{where} \
\Theta_2:=\nu\begin{pmatrix} b_\alpha &-\overline{\gamma}b_\alpha\\
\gamma&1\end{pmatrix}.
$$
We can also write
$$
\Phi_+ = \begin{pmatrix} 0& b_\alpha\overline d\\ b_\alpha \overline
c&0\end{pmatrix}
    = A^*\Theta_0\Theta_2 =A^* \left( \nu \begin{pmatrix}1&\overline{\gamma}b_\alpha\\ -{\gamma}&b_\alpha\end{pmatrix}\right)
              \left(\nu \begin{pmatrix} b_\alpha&-\overline{\gamma}b_\alpha\\ \gamma &1\end{pmatrix}\right)\,,
$$
where $\Theta_0:=\nu \begin{pmatrix} 1&\overline{\gamma}b_\alpha\\
-{\gamma}&b_\alpha\end{pmatrix}$ and $\Theta_0\Theta_2= b_\alpha
I_2$. \ Then by (\ref{3.9}),
\begin{equation}\label{6.7-17}
\Theta_0 H^2_{\mathbb C^2} \subseteq \hbox{ker}\, [T_\Phi^*,
T_\Phi]\,,\quad\hbox{so that}\quad \hbox{ran}\,[T_\Phi^*,
T_\Phi]\subseteq \mathcal{H} ({\Theta_0})\,.
\end{equation}
Since $\hbox{dim}\, \mathcal{H} ({\Theta_0})=1$, it follows that
$$
\hbox{ran}\,[T_\Phi^*, T_\Phi]=\{0\} \ \ \hbox{or}\ \
\hbox{ran}\,[T_\Phi^*, T_\Phi]=\mathcal{H} ({\Theta_0}).
$$
If $\hbox{ran}\,[T_\Phi^*, T_\Phi]=\{0\}$, then evidently $T_\Phi$
is normal. \ Suppose $\hbox{ran}\,[T_\Phi^*, T_\Phi]=\mathcal{H}
({\Theta_0})$. \ We recall a well-known result of B. Morrel
(\cite{Mor}; \cite[p.162]{Con}). If $T\in\mathcal{B(H)}$ satisfies
the following properties: (i) $T$ is hyponormal; (ii) $[T^*,T]$ is
rank-one; and (iii) $\hbox{ker}\,[T^*,T]$ is invariant for $T$, then
$T-\beta$ is quasinormal for some $\beta\in\mathbb C$, i.e.,
$T-\beta$ commutes with $(T-\beta)^*(T-\beta)$. \ Since $T_\Phi$
satisfies the above three properties, we can conclude that
$T_{\Phi-\beta}$ is quasinormal for some $\beta\in\mathbb C$. \
\end{proof}

On the other hand, due to a technical problem, we omitted a detailed
proof for the case B-2 from the proof of \cite[Theorem 5.1]{CHL2}. \
The proof of the case B-2 (with $\alpha=0$) in the proof of Theorem
\ref{thm4.2} above provides the portion of the proof that did not
appear in \cite{CHL2}. \ In particular, Theorem \ref{thm4.2}
incorporates  an extension of a corrected version of \cite[Theorem
5.1]{CHL2}, in which the exceptional case (\ref{6.6-0}) was
inadvertently omitted. \

\begin{corollary} \label{cor6.2}
If
$$
T_\Phi:=\begin{pmatrix} T_{\overline b_\alpha} & T_\varphi\\ T_\psi&
T_{\overline b_\beta}\end{pmatrix}\quad\hbox{{\rm
(}$\alpha,\beta\in\mathbb D$; $\varphi,\psi\in L^\infty${\rm )}}\,,
$$
then $T_\Phi$ is subnormal if and only if $\alpha=\beta$ and one of
the following holds:
\begin{itemize}
\item[(a)] $\varphi=e^{i\theta} b_\alpha + \zeta$\quad and\quad $\psi=e^{i\omega}\varphi$\quad
\hbox{\rm ($\zeta\in\mathbb C$; $\theta,\omega\in [0,2\pi)$);}
\item[(b)] $\varphi=\mu\, \overline{b_\alpha} + e^{i\theta}\sqrt{1+|\mu|^2}\,b_\alpha + \zeta$\quad and\quad
$\psi=e^{i\,(\pi-2\,{\rm arg}\,\mu)}\varphi$ \ {\rm
(}$\mu,\zeta\in\mathbb C$, $|\mu|\ne 0,1$, $\theta\in [0,2\pi)${\rm
)},
\end{itemize}
except the following special case:
\begin{equation}\label{6.6-0-1}
\hbox{$\varphi_-=b_\alpha \theta_0^{\prime} \overline a$ and
$\psi_-=b_\alpha \theta_1^{\prime} \overline b$ {\rm (}coprime{\rm
)}\  with $(ab)(\alpha)=(\theta_0^{\prime}
\theta_1^{\prime})(\alpha)\ne 0$\,.}
\end{equation}
\end{corollary}



\begin{proof}
We use the notations of the proof of Theorem \ref{thm4.2}. \ From
the viewpoint of the proof of Theorem \ref{thm4.2}, we should
consider Case A and Case B-1. \

\medskip

\noindent  {\bf Case A} ($m=n=0$): Since $T_\Phi$ is normal and
$\hbox{det}\,\Phi_+\ne 0$, it follows form Lemma \ref{lem1.3} that
$\Phi_+^0=\Phi_-U$ for some
constant unitary matrix $U\equiv \left(\begin{smallmatrix} c_1&c_2\\
c_3&c_4\end{smallmatrix}\right)$. We observe
\begin{equation}\label{6.7.81}
\begin{aligned}
\Phi_+^0=\Phi_-U
      &\Longleftrightarrow \begin{pmatrix} 0& \theta_1\theta_3\overline d\\
                            \theta_0\theta_2\overline c&0\end{pmatrix}
                              = \begin{pmatrix} b_\alpha &\theta_1\overline b\\ \theta_0\overline a& b_\alpha\end{pmatrix}
                                  \begin{pmatrix} c_1&c_2\\ c_3&c_4\end{pmatrix}\\
      &\Longleftrightarrow
         \begin{cases} 0=c_1 b_\alpha+c_3\theta_1\overline b\\
                       0=c_4 b_\alpha+c_2\theta_0\overline a\\
                       \theta_1\theta_3\overline d=c_2 b_\alpha + c_4\theta_1\overline b\\
                       \theta_0\theta_2\overline c=c_3 b_\alpha + c_1\theta_0\overline a
         \end{cases}\\
      &\Longrightarrow
        \begin{cases} c_1=0,\ \theta_1\overline b=0\\
                      c_4=0,\ \theta_0 \overline a=0\\
                      \theta_1\theta_3 \overline d =c_2 b_\alpha\\
                      \theta_0\theta_2 \overline c= c_3 b_\alpha\,.
         \end{cases}
\end{aligned}
\end{equation}
Since $U$ is unitary we have $c_2=e^{i\omega_1}$ and
$c_3=e^{i\omega_2}$ ($\omega_1,\omega_2\in [0,2\pi)$). Thus we have
$$
\varphi=e^{i\omega_1} b_\alpha + \beta_1\quad\hbox{and}\quad
\psi=e^{i\omega_2} b_\alpha +\beta_2.
$$
But since $|\varphi|=|\psi|$, it follows that
$$
\varphi=e^{i\theta} b_\alpha + \zeta\quad\hbox{and}\quad
\psi=e^{i\omega}\varphi
          \quad\hbox{($\theta,\omega\in [0,2\pi$, $\zeta\in\mathbb C$))}.
$$

\medskip

\noindent {\bf Case B-1} ($m=n=1$; $(ab)(\alpha)\ne
(\theta_0^\prime\theta_1^\prime)(\alpha)$): \ Since $T_\Phi$ is
normal, it follows from the same argument as in (\ref{6.7.81}), we
can see that
$$
\theta_2=\theta_3=1\quad\hbox{and}\quad
\theta_0^\prime=\theta_1^\prime =1,
$$
in other words,
$$
\varphi_+=m_1 b_\alpha +\beta_1,\ \ \psi_+=m_2 b_\alpha +\beta_2, \
\ \varphi_-=\nu b_\alpha,\ \ \psi_-=\mu b_\alpha
$$
($m_1,m_2,\beta_1,\beta_2,\mu,\nu\in\mathbb C$). \ Thus we can write
$$
\Phi_+=\begin{pmatrix}0&\varphi_+\\ \psi_+&0
\end{pmatrix}\quad\hbox{and}\quad
\Phi_-^*=\begin{pmatrix}\overline{b}_\alpha & \mu\overline b_\alpha\\
\nu \overline b_\alpha & \overline b_\alpha
\end{pmatrix}\quad (\mu\ne 0\ne \nu).
$$
Since $T_\Phi$ is normal we have
$$
\begin{pmatrix}
H_{\overline{\varphi_+}}^*H_{\overline{\varphi_+}}&0\\0&H_{\overline{\psi_+}}^*H_{\overline{\psi_+}}
\end{pmatrix}=\begin{pmatrix}
(1+|\nu|^2)H_{\overline b_\alpha}&(\mu + \overline{\nu})H_{\overline b_\alpha}\\
(\overline{\mu}+\nu)H_{\overline b_\alpha} & (1+|\mu|^2)H_{\overline
b_\alpha}\end{pmatrix},
$$
which implies that
\begin{equation}\label{6.17}
\begin{cases}
\nu=-\overline{\mu}\\
H_{\overline{\varphi_+}}^*H_{\overline{\varphi_+}}=(1+|\nu|^2)H_{\overline b_\alpha}\\
H_{\overline{\psi_+}}^*H_{\overline{\psi_+}}=(1+|\mu|^2)H_{\overline
b_\alpha}.
\end{cases}
\end{equation}
By the case assumption, $1\ne |ab|=|\mu\nu|=|\mu|^2$, i.e.,
$|\mu|\ne 1$. We thus have
$$
\varphi_+=e^{i\theta_1}\sqrt{1+|\mu|^2}\,b_\alpha  +
\beta_1\quad\hbox{and}\quad
\psi_+=e^{i\theta_2}\sqrt{1+|\mu|^2}\,b_\alpha + \beta_2,
$$
($\beta_1,\beta_2\in\mathbb C;\ \theta_1,\theta_2\in [0,2\pi)$). \
Since $|\varphi|=|\psi|$, a straightforward calculation shows that
\begin{equation}\label{6.18}
\varphi=\mu\, \overline b_\alpha + e^{i
\theta}\sqrt{1+|\mu|^2}\,b_\alpha + \zeta \quad\hbox{and}\quad
\psi=e^{i\,(\pi-2\,{\rm arg}\,\mu)}\varphi,
\end{equation}
where $\mu\ne 0,\ |\mu|\ne 1,\ \zeta\in\mathbb C$, and $\theta\in
[0,2\pi)$. \
\end{proof}

%
%
%
%
%
%

\chapter{Hyponormal Toeplitz pairs}

\bigskip

In this chapter, we consider (jointly) hyponormal
Toeplitz pairs with matrix-valued bounded type symbols. \
In \cite{CuL1}, the authors studied hyponormality of
pairs of Toeplitz operators (called Toeplitz pairs) when both
symbols are trigonometric polynomials. \
The core of the main result
of \cite{CuL1} is that the  hyponormality of $\mathbf{T}\equiv
(T_\varphi, T_\psi)$ ($\varphi, \psi$ trigonometric polynomials)
forces that the co-analytic parts of $\varphi$ and $\psi$
necessarily coincide up to a constant multiple, i.e.,
\begin{equation}\label{5.1*}
\varphi-\beta\psi\in H^2\ \ \hbox{for some $\beta\in\mathbb{C}$}.
\end{equation}
In \cite{HL4}, (\ref{5.1*}) was extended for Toeplitz pairs whose symbols are
rational functions with some constraint. \
As a result, the following question arises at once: Does (\ref{5.1*})
still hold for Toeplitz pairs whose symbols are {\it matrix-valued}
trigonometric polynomials or rational functions? \
This chapter is concerned with this question. \ More generally, we give
a characterization of hyponormal
 Toeplitz pairs with matrix-valued bounded type symbols  by using the theory established
in the previous chapters. \ Consequently, we will show that
(\ref{5.1*}) is
still true for matrix-valued trigonometric polynomials under some
invertibility and commutativity assumptions on the Fourier
coefficients of the symbols (those assumptions always hold vacuously
for scalar-valued cases). \
Moreover, we give a characterization of the
(joint) hyponormality of Toeplitz pairs with bounded type symbols,
consider the
self-commutators of the Toeplitz pairs with matrix-valued rational
symbols, and derive rank formulae for them. \

\medskip

We first observe that if $\mathbf {T}=(T_\varphi, T_\psi)$ then the
self-commutator of $\mathbf{T}$ can be expressed as:
\begin{equation}\label{5.2}
[\mathbf{T}^*, \mathbf{T}]=
\begin{pmatrix}
\hbox{$[T_{\varphi}^*, T_{\varphi}]$}&
\hbox{$[T_{\psi}^*, T_{\varphi}]$}\\
\hbox{$[T_{\varphi}^*, T_{\psi}]$}& \hbox{$[T_{\psi}^*, T_{\psi}]$}
\end{pmatrix}
=\begin{pmatrix} H_{\overline{\varphi_+}}^*
H_{\overline{\varphi_+}}-H_{\overline{\varphi_-}}^*
H_{\overline{\varphi_-}} &H_{\overline{\varphi_+}}^*
H_{\overline{\psi_+}}
       -H_{\overline{\psi_-}}^* H_{\overline{\varphi_-}}\\
H_{\overline{\psi_+}}^*
H_{\overline{\varphi_+}}-H_{\overline{\varphi_-}}^*
H_{\overline{\psi_-}} &H_{\overline{\psi_+}}^*
H_{\overline{\psi_+}}-H_{\overline{\psi_-}}^* H_{\overline{\psi_-}}
\end{pmatrix}.
\end{equation}
The hyponormality of Toeplitz pairs is also related to the kernels
of Hankel operators involved with the analytic and co-analytic parts
of the symbol. \  Indeed it was shown in \cite [Lemma 6.2] {Gu2}
that if neither $\varphi$ nor $\psi$ is analytic and if
$(T_{\varphi}, T_{\psi})$ is hyponormal, then
\begin{equation}\label{5.3-3}
\text{\rm ker}\, H_{\overline{\varphi_+}}\subseteq \text{\rm ker}\,
H_{\overline{\psi_-}}\quad\text{and}\quad \text{\rm ker}\,
H_{\overline{\psi_+}}\subseteq \text{\rm ker}\,
H_{\overline{\varphi_-}}\,.
\end{equation}

Tuples (or pairs) of Toeplitz operators will be called {\it Toeplitz
tuples} (or {\it Toeplitz} pairs). \

On a first perusal, one might be tempted to guess that (\ref{5.1*})
still holds for Toeplitz pairs whose symbols are matrix-valued
trigonometric polynomials. \  However this is not the case. \  To
see this we take
$$
\Phi:=\begin{pmatrix} z^{-1}+2z&0\\
0&0\end{pmatrix}\quad\hbox{and}\quad \Psi:=\begin{pmatrix} 0&0\\
0&z^{-1}+2z\end{pmatrix}.
$$
Then a straightforward calculation shows that if $\mathbf
T=(T_{\Phi},T_{\Psi})$ then $[\mathbf T^*, \mathbf T]\ge 0$, i.e.,
$\mathbf T$ is hyponormal, but evidently, $\Phi_- \ne \Lambda
\Psi_-$ for any constant matrix $\Lambda\in M_2$. \  However, we
note that
$$
\Phi_-=\begin{pmatrix} z&0\\ 0&0\end{pmatrix} =\begin{pmatrix} z&0\\
0&z\end{pmatrix}\,\begin{pmatrix} 1&0\\ 0&0\end{pmatrix}^*
$$
and by Theorem \ref{lem33.9},
$$
\hbox{$\Theta\equiv \begin{pmatrix} z&0\\ 0&z\end{pmatrix}$ and
$A\equiv \begin{pmatrix} 1&0\\ 0&0\end{pmatrix}$ are not right
coprime.}
$$
As we may expect, if the condition ``$\Theta$ and $A$ are right
coprime" is assumed then we might get a matrix-valued version of
(\ref{5.1*}). \ As we will see in the sequel, a corollary to the
main result of this chapter follows the spirit of (\ref{5.1*})
(Corollary \ref{cor5.23}). \ The main theorem of this chapter
(Theorem \ref{thm5.22-1}) gives a complete characterization of the
hyponormality of Toeplitz pairs with bounded type symbols. \
Roughly speaking, this characterization says that the hyponormality
of a Toeplitz pair can be determined by the hyponormality of a {\it
single} Toeplitz operator. \

\medskip

To proceed, we consider some basic facts.

\medskip

Recall (cf. p.\pageref{gu}) that for each $\Phi\in L^\infty_{M_n}$,
if we put
$$
\mathcal{E}(\Phi)\label{Ebigphi2}:=\Bigl\{K\in H^\infty_{M_n}:\
||K||_\infty \le 1\ \ \hbox{and}\ \ \Phi-K \Phi^*\in
H^\infty_{M_n}\Bigr\}. \
$$
then $T_\Phi$ is hyponormal if and only if $\Phi$ is normal and
$\mathcal{E}(\Phi)$ is nonempty.

\medskip

If $\Phi \in L^\infty_{M_n}$, then by (\ref{1.3}),
$$
[T_\Phi^*, T_\Phi]= H_{\Phi^*}^* H_{\Phi^*} - H_{\Phi}^*H_\Phi
        +T_{\Phi^*\Phi-\Phi\Phi^*}\,.
$$
Since the normality of $\Phi$ is a necessary condition for the
hyponormality of $T_\Phi$, the positivity of $H_{\Phi^*}^*
H_{\Phi^*} - H_{\Phi}^*H_\Phi$ is an essential condition for the
hyponormality of $T_\Phi$. \  Thus, we isolate this property as a
new notion, weaker than hyponormality. \  The reader will notice at
once that this notion is meaningful for non-scalar symbols. \

\bigskip

\begin{definition}\label{def8.111}
For $\Phi, \Psi \in L^\infty_{M_n}$, let
$$
[T_\Phi, T_\Psi]_p\label{pseudo}:= H_{\Psi^*}^* H_{\Phi} -
H_{\Phi^*}^*H_\Psi.
$$
Then $[T_\Phi^*, T_\Phi]_p$ is called the {\it
pseudo-selfcommutator} of $T_\Phi$. \ Also $T_{\Phi}$ is said to be
{\it pseudo-hyponormal} if $[T_\Phi^*, T_\Phi]_p$ is positive
semidefinite. \
\end{definition}

\medskip

As in the case of hyponormality of scalar Toeplitz operators, we can
see that the pseudo-hyponormality of $T_\Phi$ is independent of the
constant matrix term $\Phi(0)$. \  Thus whenever we consider the
pseudo-hyponormality of $T_\Phi$ we may assume, without loss of
generality, that $\Phi(0)=0$. \ Observe that if $\Phi\in
L^\infty_{M_n}$ then
$$
[T_\Phi^*, T_\Phi]= [T_\Phi^*, T_\Phi]_p +
T_{\Phi^*\Phi-\Phi\Phi^*}.
$$
Thus $T_\Phi$ is hyponormal if and only if $T_\Phi$ is
pseudo-hyponormal and $\Phi$ is normal; also (via Theorem 3.3 of
\cite{GHR}) $T_\Phi$ is pseudo-hyponormal if and only if
$\mathcal{E}(\Phi)\ne \emptyset$. \

\medskip

Recall that for $\Phi\equiv  \Phi_-^* + \Phi_+ \in
L^{\infty}_{M_n}$, we write
$$
\mathcal C(\Phi)\label{Cphi}:=\Bigl\{K \in H^{\infty}_{M_n}:\
\Phi-K\Phi^*\in H^{\infty}_{M_n}\Bigr\}.
$$
Thus if $\Phi\in L^\infty_{M_n}$, then
$$
K\in \mathcal{E}(\Phi)\ \Longleftrightarrow\ K\in\mathcal{C}(\Phi)\
\hbox{and}\ ||K||_\infty\le 1.
$$
We note (cf. \cite{GHR}) that if $T_\Phi$ is pseudo-hyponormal, then
$||\Phi_-||_2\le ||\Phi_+||_2\label{2norm}$: indeed, if $K\in
\mathcal{E}(\Phi)$, then $\Phi_-^*-K\Phi_+^*\in H^2_{M_n}$, so that
$$
||\Phi_-||_2=||\Phi_-^*||_2\le ||K\Phi_+^*||_2\le ||K||_\infty
||\Phi_+^*||_2 \le ||\Phi_+^*||_2=||\Phi_+||_2.
$$

\medskip

In view of (\ref{3.222}) in Chapter 6, whenever we study the
pseudo-hyponormality of Toeplitz operators with symbol $\Phi$ such
that $\Phi$ and $\Phi^*$ are of bounded type, we may assume that the
symbol $\Phi\equiv \Phi_-^* + \Phi_+\in L^{\infty}_{M_n}$ is of the
form
\begin{equation}\label{bts-1}
\Phi_+= \Theta_0 \Theta_1 A^*\quad \hbox{and}\quad \Phi_-=\Theta_0
B^*\quad \hbox{(right coprime)}.
\end{equation}

\medskip

We next consider hyponormality of Toeplitz operators with bounded
type symbols. \ To do so, we use an interpolation problem developed
in Chapter 6.

\begin{proposition}{\rm(Pull-back symbols)} \label{pro3.13}
Let $\Phi\in L^\infty_{M_n}$ be such that $\Phi$ and $\Phi^*$ are of
bounded type. \ In view of {\rm (\ref{bts-1})}, we may write
$$
\Phi_+=\Theta_0\Theta_1 A^*\quad\hbox{and}\quad \Phi_-=\Theta_0 B^*
\ \ \hbox{\rm (right coprime)}.
$$
Suppose $\Theta_1 A^* = A_1^*\Theta$ (where $A_1$ and $\Theta$ are
left coprime). \ Then the following hold:
\begin{itemize}
\item[(a)] If $I_{\omega}$ is an inner divisor of $\Theta_1$,
put $\Phi^{1, \omega}:=\Phi_-^*+P_{H_0^2}(\overline\omega\Phi_+) \
\hbox{\rm (cf. p.\pageref{qw})}$. Then $T_{\Phi}$ is
pseudo-hyponormal if and only if $T_{\Phi^{1, \omega}}$ is
pseudo-hyponormal;
\item[(b)] Put $\Upsilon:=\Phi_-^*+P_{H^2}(\Theta_0 A_1^*).$ \ Then
$T_{\Phi}$ is pseudo-hyponormal if and only if $T_{\Upsilon}$ is
pseudo-hyponormal.
\end{itemize}
\end{proposition}

\begin{proof} \
(a) By Proposition \ref{pro3.13-1}(b), we have
$$\mathcal{C}(\Phi^{1, \omega})=\bigl\{\overline\omega K : K \in
\mathcal C( \Phi)\bigr\}.
$$
Thus the result follows at once from the observation that
$||K||_\infty=||\overline\omega K||_\infty$. \

(b) By Proposition \ref{pro3.13-1}(a), we have
$$
\aligned T_{\Phi}\ \hbox{is pseudo-hyponormal} &\Longleftrightarrow
\Phi_-^* - K^{\prime}\Theta \Phi_+^* \in H_{M_n}^2 \ \ (K^{\prime}
\in H^2 \ \hbox{and} \ ||K^{\prime}||_{\infty} \leq 1)\\
&\Longleftrightarrow \Phi_-^* - K^{\prime}A_1 \Theta_0^* \in
H_{M_n}^2\\
&\Longleftrightarrow \Phi_-^* -
K^{\prime}\bigl(P_{H^2_{M_n}}(\Theta_0
A_1^*)\bigr)^* \in H_{M_n}^2\\
&\Longleftrightarrow  T_{\Upsilon}\ \hbox{is pseudo-hyponormal}.
\endaligned
$$
\end{proof}

\medskip

For an operator $S\in \mathcal{B(H)}$,
$S^\sharp\label{Ssharp}\in\mathcal{B(H)}$ is called the
Moore-Penrose inverse of $S$ if
$$
SS^\sharp S=S,\ \ S^\sharp SS^\sharp=S^\sharp,\ \ (S^\sharp S)^*=
S^\sharp S,\ \ \hbox{and}\ \ (SS^\sharp)^*=SS^\sharp.
$$
It is known (\cite[Theorem 8.7.2]{Har}) that if an operator $S$ on a
Hilbert space has closed range then $S$ has a Moore-Penrose inverse.
\  Moreover, the Moore-Penrose inverse is unique whenever it exists.
\ On the other hand, it is well-known that if
$$
S:=\begin{pmatrix} A&B\\ B^*&C\end{pmatrix}\quad\hbox{on}\
\mathcal{H}_1\oplus \mathcal{H}_2,
$$
(where the $\mathcal H_j$ are Hilbert spaces, $A\in
\mathcal{B}(\mathcal H_1)$, $C\in\mathcal{B}(\mathcal H_2)$, and
$B\in\mathcal{B}(\mathcal H_2, \mathcal H_1)$), then
\begin{equation}\label{2.19}
S\ge 0\ \Longleftrightarrow\ A\ge 0,\ C\ge 0,\ \hbox{and}\
B=A^{\frac{1}{2}}DC^{\frac{1}{2}}\ \hbox{for some contraction}\ D;
\end{equation}
moreover, in (\cite [Lemma 1.2]{CuL1} and \cite [Lemma 2.1] {Gu2})
it was shown that if $A\ge 0$, $C\ge 0$ and $\hbox{ran}\,A$ is
closed then
\begin{equation}\label{2.20}
S\ge 0\ \Longleftrightarrow\ B^*A^\sharp B\le C\ \hbox{and}\
\hbox{ran}\, B\subseteq \hbox{ran}\, A,
\end{equation}
or equivalently (\cite [Lemma 1.4] {CMX}),
\begin{equation}\label{2.21}
|\langle Bg, f\rangle)|^2\le \langle Af,f\rangle \langle Cg,
g\rangle\quad\hbox{for all}\ f\in\mathcal{H}_1,\ g\in\mathcal{H}_2
\end{equation}
and furthermore, if both $A$ and $C$ are of finite rank then
\begin{equation}\label{2.22}
\hbox{rank}\, S=\hbox{rank}\, A+\hbox{rank}\,(C-B^*A^\sharp B).
\end{equation}
In fact, if $A\ge 0$ and $\hbox{ran}\,A$ is closed then we can write
$$
A=\begin{pmatrix} A_0&0\\ 0&0\end{pmatrix}:\
\begin{pmatrix}
\hbox{ran}\,A\\
\hbox{ker}\,A \end{pmatrix} \to
\begin{pmatrix}
\hbox{ran}\,A\\
\hbox{ker}\,A\end{pmatrix},
$$
so that the Moore-Penrose inverse of $A$ is given by
\begin{equation}\label{2.23}
A^\sharp=\begin{pmatrix} (A_0)^{-1}&0\\ 0&0\end{pmatrix}.
\end{equation}

\bigskip

We introduce a notion which will help simplify our arguments.

\begin{definition}
Let $\Phi, \Psi\in L^\infty_{M_n}$. \ For a Toeplitz pair
$\mathbf{T} \equiv (T_{\Phi}, T_{\Psi})$, the {\it
pseudo-commutator} of $\mathbf{T}$ is defined by (cf. Definition
\ref{def8.111})
$$
\begin{aligned}
\hbox{$[\mathbf{T}^*, \mathbf{T}]_p\label{T^*TP}$} &:=
\begin{pmatrix}
\hbox{$[T_\Phi^*, T_\Phi]_p$} & \hbox{$[T_\Psi^*, T_\Phi]_p$}\\
\hbox{$[T_\Phi^*, T_\Psi]_p$} & \hbox{$[T_\Psi^*, T_\Psi]_p$}
\end{pmatrix}\\
&=
\begin{pmatrix} H_{\Phi_+^*}^* H_{\Phi_+^*}-H_{\Phi_-^*}^* H_{\Phi_-^*}
     & H_{\Phi_+^*}^* H_{\Psi_+^*}-H_{\Psi_-^*}^* H_{\Phi_-^*}\\
           H_{\Psi_+^*}^* H_{\Phi_+^*}-H_{\Phi_-^*}^* H_{\Psi_-^*}
     & H_{\Psi_+^*}^* H_{\Psi_+^*}-H_{\Psi_-^*}^* H_{\Psi_-^*}
\end{pmatrix}\,.\\
\end{aligned}
$$
Then $\mathbf{T}\equiv (T_{\Phi}, T_{\Psi})$ is said to be {\it
pseudo-(jointly) hyponormal} if $[\mathbf{T}^*, \mathbf{T}]_p\ge 0$.
\ Evidently, if $\mathbf{T}\equiv (T_{\Phi}, T_{\Psi})$ is
pseudo-hyponormal then $T_\Phi$ and $T_\Psi$ are pseudo-hyponormal.
\end{definition}

\bigskip

We begin with:

\begin{lemma}\label{lem5.1}
Let $\Phi, \Psi\in L^\infty_{M_n}$. \  If
$\mathbf{T}\equiv(T_{\Phi}, T_{\Psi})$ is hyponormal then  $\Phi$
and $\Psi$ commute.
\end{lemma}

\begin{proof} Suppose that $\mathbf{T}$ is hyponormal. \  Then $T_{\Phi}$ and
$T_{\Psi}$ are hyponormal, and hence $\Phi$ and $\Psi$ are normal. \
Thus by (\ref{1.3}) we have
\begin{equation*}
\begin{split}
 [\mathbf{T}^*, \mathbf{T}]
 &= \hbox{$\begin{pmatrix} \hbox{$[T_{\Phi}^*, T_{\Phi}]$}& \hbox{$[T_{\Psi}^*, T_{\Phi}]$}\\
       \hbox{$[T_{\Phi}^*, T_{\Psi}]$}& \hbox{$[T_{\Psi}^*, T_{\Psi}]$} \end{pmatrix}$}\\
 &=\hbox{$\begin{pmatrix} H_{\Phi_+^*}^* H_{\Phi_+^*}\hskip-.1cm-\hskip-.1cm H_{\Phi_-^*}^* H_{\Phi_-^*}
     &H_{\Phi_+^*}^* H_{\Psi_+^*}\hskip-.1cm - \hskip-.1cm H_{\Psi_-^*}^* H_{\Phi_-^*}
         \hskip-.1cm +\hskip-.1cm T_{\Psi^* \Phi-\Phi \Psi^*}\\
     H_{\Psi_+^*}^* H_{\Phi_+^*}-H_{\Phi_-^*}^* H_{\Psi_-^*}+T_{\Phi^*
     \Psi-\Psi \Phi^*} & H_{\Psi_+^*}^* H_{\Psi_+^*}-H_{\Psi_-^*}^* H_{\Psi_-^*} \end{pmatrix}$}.
\end{split}
\end{equation*}
But since $\mathbf{T}$ is hyponormal it follows from (\ref{2.21})
that for any $m \geq 0, \  x, y \in H^2_{\mathbb C^{n}},$
\begin{equation}\label{5.6}
\aligned &\Bigl \vert \Bigl\langle (H_{\Phi_+^*}^*
H_{\Psi_+^*}-H_{\Psi_-^*}^* H_{\Phi_-^*}+T_{\Psi^* \Phi-\Phi
\Psi^*})\,I_{z^m} y,\ I_{z^m} x \Bigr \rangle  \Bigl\vert^2\\
&\leq \Bigl\langle (H_{\Phi_+^*}^* H_{\Phi_+^*}-H_{\Phi_-^*}^*
H_{\Phi_-^*})(I_{z^m} x), \ I_{z^m} x \Bigr\rangle\\
&\qquad\qquad \cdot \Bigl\langle (H_{\Psi_+^*}^*
H_{\Psi_+^*}-H_{\Psi_-^*}^* H_{\Psi_-^*})(I_{z^m} y), \  I_{z^m} y
\Bigr\rangle.
\endaligned
\end{equation}
Note that
\begin{equation}\label{5.7}
\Bigl\langle T_{\Psi^* \Phi-\Phi \Psi^*}(I_{z^m} y),\ I_{z^m} x)
\Bigr \rangle =\Bigl\langle T_{I_{z^m}}^* T_{\Psi^* \Phi-\Phi
\Psi^*} T_{I_{z^m}} y,\ x \Bigr \rangle =\Bigl\langle T_{\Psi^*
\Phi-\Phi \Psi^*}y,\ x  \Bigr \rangle\,.
\end{equation}
But since
$$
\lim_{m\to \infty} H_C (I_{z^m\omega})=0\quad\hbox{for any $C\in
L^\infty_{M_n}$ and $\omega\in H^2_{\mathbb C^n}$},
$$
if we take the limits on $m$ in (\ref{5.6}) and (\ref{5.7}) then we
have
$$
\Bigl\langle T_{\Psi^* \Phi-\Phi \Psi^*}y,\ x \Bigr \rangle
=\text{lim}_{m \rightarrow \infty}\Bigl\langle (T_{\Psi^* \Phi-\Phi
\Psi^*})\, I_{z^m} y,\ I_{z^m} x \Bigr \rangle =0,
$$
which implies that $\Psi^* \Phi = \Phi \Psi^*$. \  Since $\Psi$ is
normal it follows from the Fuglede-Putnam Theorem that
$\Phi\Psi=\Psi\Phi$.
\end{proof}

\medskip

We thus have:

\begin{corollary}\label{cor5.2} Let $\Phi, \Psi\in L^\infty_{M_n}$
and let $\mathbf{T}:=(T_{\Phi}, T_{\Psi})$. \  Then the following
are equivalent:
\begin{itemize}
\item[(i)] $\mathbf{T}$ is hyponormal;
\item[(ii)] $\mathbf{T}$ is pseudo-hyponormal,
$\Phi$ and $\Psi$ are normal, and $\Psi \Phi=\Phi \Psi$.
\end{itemize}
\end{corollary}

\begin{proof} Immediate from Lemma \ref{lem5.1}.
\end{proof}

\begin{lemma}\label{lem5.3}
Let $\Phi\in H^\infty_{M_n}$ and $\Psi\in L^\infty_{M_n}$. \ If
$$
\Phi\equiv \Phi_+=A^*\Theta\quad \hbox{\rm (left coprime)}
$$
then
\begin{equation}\label{5.8}
\mathbf{T}=(T_{\Phi}, T_{\Psi})\ \text{is pseudo-hyponormal}
\Longleftrightarrow T_{\Psi^{1,\Theta}}\ \text{is
pseudo-hyponormal}\,,
\end{equation}
where $\Psi^{1,\Theta}:=\Psi_-^*+P_{H^2_0}(\Psi_+\Theta^*)$ {\rm
(cf. p.\pageref{HTheta})}.
\end{lemma}

\begin{proof} \
Since $\Phi\in H^\infty_{M_n}$, $\mathbf{T}=(T_{\Phi}, T_{\Psi})$ is
pseudo-hyponormal if and only if
$$
[\mathbf{T}^*, \mathbf{T}]_p=\begin{pmatrix} H_{\Phi_+^*}^*
H_{\Phi_+^*}
&H_{\Phi_+^*}^* H_{\Psi_+^*}\\
H_{\Psi_+^*}^* H_{\Phi_+^*}& [T_{\Psi}^*, \ T_{\Psi}]_p
\end{pmatrix}\ge 0,
$$
or equivalently, by (\ref{2.21}),
\begin{equation}\label{5.9}
\Bigl \vert \Bigl \langle H_{\Psi_+^*}^*H_{\Phi_+^*}x, \ y \Bigr
\rangle \Bigr \vert^2 \leq  \Bigl \langle H_{\Phi_+^*}^*H_{\Phi_+^*}
x, \ x \Bigr \rangle \Bigl \langle [T_{\Psi}^*, \ T_{\Psi}]_p y ,\ y
\Bigr \rangle\qquad \hbox{(all $x,y \in H^2_{\mathbb{C}^n}$)}.
\end{equation}
Since by assumption, $\widetilde{A}$ and $\widetilde{\Theta}$ are
right coprime, it follows that
$$
\text{cl\,ran}\, H_{\Phi_+^*}=\mathcal H
(\widetilde{\Theta})=\text{cl ran}\, H_{\Theta^*}.
$$
Therefore, the inequality (\ref{5.9}) becomes
$$
\Bigl \vert \Bigl \langle H_{\Theta^*}x, \ H_{\Psi_+^*}y \Bigr
\rangle \Bigr \vert \leq \Bigl \langle H_{\Theta^*}x, \
H_{\Theta^*}x\Bigr \rangle \Bigl \langle [T_{\Psi}^*, \ T_{\Psi}]_p
y ,\ y \Bigr \rangle \qquad \hbox{(all $x,y \in
H^2_{\mathbb{C}^n}$)}.
$$
We thus have that $[\mathbf{T}^*, \mathbf{T}]_p \geq 0$ if and only
if $[\mathbf{T}_{\Theta}^*, \mathbf{T}_{\Theta}]_p \geq 0$, where
$\mathbf{T}_{\Theta}:=(T_{\Theta}, T_{\Psi}). \ $ Observe that
\begin{equation}\label{5.10}
[\mathbf{T}_{\Theta}^*, \mathbf{T}_{\Theta}]_p =\begin{pmatrix}
H_{\Theta^*}^* H_{\Theta^*}
&H_{\Theta^*}^* H_{\Psi_+^*}\\
H_{\Psi_+^*}^* H_{\Theta^*}& [T_{\Psi}^*, \ T_{\Psi}]_p
\end{pmatrix}.
\end{equation}
By (\ref{1.5}) we can see that $H_{\Theta^*} H_{\Theta^*}^*$ is a
projection on $\mathcal H (\widetilde\Theta)$. \ Therefore, it
follows from (\ref{2.20}) that
$$
[\mathbf{T}_{\Theta}^*, \mathbf{T}_{\Theta}]_p \geq 0
\Longleftrightarrow [T_{\Psi}^*, \ T_{\Psi}]_p -H_{\Psi_+^*}^*
H_{\Theta^*}H_{\Theta^*}^* H_{\Psi_+^*} \geq 0
$$
Observe that by (\ref{1.5}),
\begin{equation*}
[T_{\Psi}^*, \ T_{\Psi}]_p
 - H_{\Psi_+^*}^* H_{\Theta^*}H_{\Theta^*}^* H_{\Psi_+^*}
=[T_{\Psi^{1, \Theta}}^*, \ T_{\Psi^{1, \Theta}}]_p\,.
\end{equation*}
Thus the inequality (\ref{5.9}) holds if and only if
$T_{\Psi^{1,\Theta}}$ is pseudo-hyponormal. \  This proves the
lemma. \
\end{proof}

\begin{remark}
In \cite [Problem 5.4] {CuL1}, the following problem was formulated:
{\sl For $n>1$, find a block Toeplitz operator $R$, not all of whose
diagonals are constant, for which $(U^n, R)$ is hyponormal, where
$U$ is the unilateral shift on $H^2$.} \ In fact, if $(U^n,R)$ is a
hyponormal pair then by (\ref{2.19}),
$$
[R^*, U^n]=[U^{n*}, U^n]^{\frac{1}{2}}D\,[R^*,R]^{\frac{1}{2}}\ \
\hbox{for some contraction}\ D.
$$
But since $P_n:=[U^{n*},U^n]$ is the orthogonal projection of rank
$n$ and $P_nU^n=0$, it follows that
$$
U^{n*}R-RU^{n*}=[R^*,R]^{\frac{1}{2}} D^*P_n\ \Longrightarrow\
U^{n*}RU^n-R=[R^*,R]^{\frac{1}{2}}D^*P_nU^n=0,
$$
which implies that $R$ is a block Toeplitz operator $T_\Phi$ with
symbol $\Phi\in L^\infty_{M_n}$. \  Thus Lemma \ref{lem5.3} together
with Corollary \ref{cor5.2} gives an answer to this problem: since
$U^n\cong T_{I_z}$, it follows that
$$
(U^n,R)\ \hbox{is hyponormal}\ \Longleftrightarrow\ R\cong T_\Phi\
\hbox{($\Phi$ is normal)}\ \hbox{and}\ T_{\Phi^{1, I_z}}\ \hbox{is
pseudo-hyponormal,}
$$
where $\cong\label{cong}$ denotes unitary equivalence. \  For
example, if $R=T_\Phi$ with
$$
\Phi=\begin{pmatrix} z&0\\0&2z\end{pmatrix},
$$
then $(U^2,R)\cong (T_{zI_2}, T_\Phi)$ is a hyponormal pair, but
$T_\Phi$ is not unitarily equivalent to any Toeplitz operator
(indeed, the essential spectrum of $T_\Phi$ is $\mathbb T\cup
2\,\mathbb T$, which is not connected). \hfill$\square$
\end{remark}

\begin{lemma} \label{lem5.4} Let $\Phi, \Psi \in L^\infty_{M_n}$. \
If $\mathbf T=(T_{\Phi}, T_{\Psi})$ is pseudo-hyponormal and
$\Theta\in H^\infty_{M_n}$ is an inner matrix function then $\mathbf
T_{\Theta}:=(T_{\Phi_{\Theta}}, T_{\Psi_{\Theta}})$ is
pseudo-hyponormal. \
\end{lemma}

\begin{proof} Suppose that $\mathbf T=(T_{\Phi}, T_{\Psi})$ is
pseudo-hyponormal. \  We then have
\begin{equation*}
\begin{split}
[(T_{\Phi_{\Theta}})^*, T_{\Psi_{\Theta}}]_p
&=H_{[P_{H_0^2}(\Theta^*\Psi_+)]^*}^*
H_{[P_{H_0^2}(\Theta^*\Phi_+)]^*}
   -H_{[P_{(H^2)^\perp}(\Phi_-^*\Theta)]}^* H_{[P_{(H^2)^\perp}(\Psi_-^*\Theta)]}\\
&=(H_{\Psi_+^*}T_{\Theta})^*H_{\Phi_+^*}T_{\Theta}-
(H_{\Phi_-^*}T_{\Theta})^*H_{\Psi_-^*}T_{\Theta}\\
&=T_{\Theta}^*(H_{\Psi_+^*}^*H_{\Phi_+^*}-
H_{\Phi_-^*}^*H_{\Psi_-^*})T_{\Theta}\\
&=T_{\Theta}^*\,[T_\Phi^*, T_\Psi]_p \,T_{\Theta},
\end{split}
\end{equation*}
Therefore, we have
$$
\begin{aligned}
\hbox{$[(\mathbf T_{\Theta})^*, \mathbf T_{\Theta}]_p$}
  &=  \begin{pmatrix} \hbox{$[(T_{\Phi_{\Theta}})^*, T_{\Phi_{\Theta}}]_p$}&
     \hbox{$[(T_{\Psi_{\Theta}})^*, T_{\Phi_{\Theta}}]_p$}\\
     \hbox{$[(T_{\Phi_{\Theta}})^*, T_{\Psi_{\Theta}}]_p$}&
     \hbox{$[(T_{\Psi_{\Theta}})^*, T_{\Psi_{\Theta}}]_p$}\end{pmatrix}\\
  &= \begin{pmatrix} T_{\Theta} & 0\\0&T_{\Theta}\end{pmatrix}^*
     [\mathbf {T^*}, \mathbf T ]_p
     \begin{pmatrix} T_{\Theta} & 0\\0&T_{\Theta}\end{pmatrix},
\end{aligned}
$$
which gives the result. \
\end{proof}

\begin{lemma}\label{lem5.5} Let $\mathbf T\equiv (T_{\Phi}, T_{\Psi})$ and $\mathbf
S\equiv (T_{\Phi-\Lambda \Psi}, T_{\Psi})$, where $\Lambda\in M_n$
is a constant normal matrix commuting with $\Psi_-$ and $\Phi_-$. \
Then
\begin{equation}\label{5.11}
\mathbf T\ \text{is pseudo-hyponormal} \Longleftrightarrow \mathbf
S\ \text{is pseudo-hyponormal}.
\end{equation}
\end{lemma}

\begin{proof} Put
$$
\mathcal {T} := \begin{pmatrix}I & -T_\Lambda
\\0&I\end{pmatrix}[\mathbf T^*, \mathbf T]_p
\begin{pmatrix}I&0\\-T_\Lambda^*& I\end{pmatrix}\quad\hbox{and}
\quad \mathcal S:=[\mathbf S^*, \ \mathbf S]_p.
$$
Then
\begin{equation*}
\mathcal S =\begin{pmatrix} H_{(\Phi-\Lambda
\Psi)^*}^*H_{(\Phi-\Lambda \Psi)^*}-H_{(\Phi-\Lambda
\Psi)}^*H_{(\Phi-\Lambda \Psi)}&H_{(\Phi-\Lambda
\Psi)^*}^*H_{\Psi^*}-H_{\Psi}^*H_{(\Phi-\Lambda \Psi)}\\
H_{\Psi^*}^*H_{(\Phi-\Lambda \Psi)^*}-H_{(\Phi-\Lambda
\Psi)}^*H_{\Psi}&H_{\Psi^*}^*H_{\Psi^*}-H_{\Psi}^*H_{\Psi}
\end{pmatrix}.
\end{equation*}
Note that if $\Phi\in L^\infty_{M_n}$ and $\Lambda$ is a constant
matrix such that $\Phi_- \Lambda=\Lambda \Phi_-$, then by the
Fuglede-Putnam Theorem,
\begin{equation}\label{3.12-1}
H_{\Phi}\, T_\Lambda = H_{\Phi\Lambda}=H_{\Lambda\Phi}=T_\Lambda\,
H_{\Phi}.
\end{equation}
Then by a straightforward calculation together with (\ref{3.12-1})
and the assumption that $\Lambda\Lambda^*=\Lambda^*\Lambda$ and
$\Lambda$ commutes with $\Psi_-$ and $\Phi_-$, we can show that
$\mathcal S= \mathcal T$. \  This proves the lemma. \
\end{proof}

\begin{lemma}\label{lem5.6}
Let $\mathbf T:=(T_\Phi, T_\Psi)$ be a pseudo-hyponormal Toeplitz
pair with bounded type symbols $\Phi,\Psi\in L^\infty_{M_n}$. \ If
$\Phi$ and $\Psi$ are not analytic and if $\Psi_+=\Theta A^*$
\hbox{\rm (right coprime)}, then
\begin{equation}\label{5.12}
H_{\Psi_-^*}^* H_{\Phi_-^*}T_{\Theta}=0. \
\end{equation}
\end{lemma}

\begin{proof} \
This follows from the positivity test for $[\mathbf T^*, \mathbf
T]_p$ via (\ref{2.21}).
\end{proof}

We now have

\begin{corollary}\label{cor55.6}
Let $\mathbf T:=(T_\Phi, T_\Psi)$ be a pseudo-hyponormal Toeplitz
pair with bounded type symbols $\Phi, \Psi \in L^\infty_{M_n}$,
which  are not analytic. \ If $\Psi$ has a tensored-scalar singularity then
\begin{equation}\label{5.16}
\hbox{\rm ker}\, H_{\Psi_+^*} \subseteq \hbox{\rm ker}\,
H_{\Phi_-^*}\,.
\end{equation}
\end{corollary}

\begin{proof} \
In view of (\ref{bts-1}), we may write
$$
\Phi_+ = \Theta_0 \Theta_1 A^*, \ \ \Phi_- =\Theta_0 B^*,  \ \
\Psi_+ = \Theta_2\Theta_3 C^* , \ \ \Psi_- = \Theta_2 D^* \quad
\hbox{\rm (right coprime)}.
$$
Since $\Psi$ has a tensored-scalar singularity, it follows from (\ref{5.12}),
(\ref{1.4}), and Theorem \ref{lem33.5} that
$H_{\Phi_-^*\Theta_2\Theta_3}=0$, and hence $\ker
H_{\Psi_+^*}=\Theta_2\Theta_3 H^2_{\mathbb C^n} \subseteq \ker
H_{\Phi_-^*}$.
\end{proof}

\medskip

\begin{lemma} \label{lem5.12} \
Let $\mathbf T:=(T_\Phi, T_\Psi)$ be a pseudo-hyponormal Toeplitz
pair with bounded type symbols $\Phi, \Psi \in L^\infty_{M_n}$ of
the form
\begin{equation}\label{5.26-1}
\Phi_+ = \Theta_0\Theta_1 A^*, \ \ \Phi_- =\Theta_0 B^* ,\ \ \Psi_+
= \Theta_2\Theta_3 C^*, \ \ \Psi_- =\Theta_2 D^* \quad(\hbox{right
coprime}).
\end{equation}
If $\Phi^*$ has a tensored-scalar singularity and
$\Theta_0^*\Theta_2=I_{\theta}$ for some inner function $\theta$,
then $I_{\theta}$ is an inner divisor of $\Theta_3$. \
\end{lemma}

\begin{proof} By assumption, $\Theta_2=\theta\Theta_0$. \
Thus we may write
$$
\Phi_+ = \Theta_0 \Theta_1 A^*, \ \ \Phi_- =\Theta_0 B^*, \ \ \Psi_+
= \theta\Theta_0 \Theta_3 C^*, \ \ \Psi_- = \theta \Theta_0 D^*
\quad(\hbox{right coprime}).
$$
By Corollary \ref{cor55.6}, we have $\Theta_0\Theta_1 H^2_{\mathbb
C^n} \subseteq \theta\Theta_0 H^2_{\mathbb C^n}$. \  It thus follows
from \cite [Corollary IX.2.2] {FF} that $\Theta_1= \theta \Delta$
for some square inner matrix function $\Delta$. \ Thus we can write
$$
\Phi_+ = \theta\Theta_0 \Delta A^*, \ \ \Phi_- =\Theta_0 B^*, \ \
\Psi_+ = \theta\Theta_0 \Theta_3 C^*, \ \ \Psi_- = \theta\Theta_0
D^*\quad(\hbox{right coprime}).
$$
Put
$$
\Phi^{(1)}:=\Phi_{\Theta_0}, \ \ \Psi^{(1)}:=\Phi_{\Theta_0}.
$$
Then by Lemma \ref{lem5.4}, $(T_{\Phi^{(1)}},  T_{\Psi^{(1)}})$ is
pseudo-hyponormal. \  By Proposition \ref{pro2.6} we get the
following right coprime factorizations:
$$
\Phi^{(1)}_+ = {\theta} \Delta A_1^*, \ \ \Phi^{(1)}_- =0, \ \
\Psi^{(1)}_+ = {\theta}\Theta_3 C_1^*, \ \ \Psi^{(1)}_- = {\theta}
D_1^*\,,
$$
where $A_1:=P_{\mathcal K_{{\theta} \Delta_1}}A$, $C_1:=P_{\mathcal
K_{{\theta}\Theta_3}}C$, $D_1:=P_{\mathcal K_{{\theta}}}D$. \  It
thus follows from Lemma \ref{lem5.3} that $T_{(\Psi^{(1)})^{1,
\Theta_1}}$ is pseudo-hyponormal. \ Observe that
$$
\left({(\Psi^{(1)})^{1, \Theta_1}}\right)_+
=P_{H_0^2}(\Psi^{(1)}_+\Delta^* I_{\theta}^*) =P_{H_0^2}(\Theta_3
(\Delta C_1)^*) =\Theta_3\bigl(P_{\mathcal K_{\Theta_3}}(\Delta
C_1)\bigr)^*,
$$
where the last equality follows from  Lemma \ref{lem2.4}. \ Thus we
can write
$$
\left({(\Psi^{(1)})^{1, \Theta_1}}\right)_+ =\Theta_3^{\prime} B^*
\qquad(\hbox{right coprime}),
$$
where $\Theta_3^{\prime}$ is a left inner divisor of $\Theta_3$. \
Since $T_{(\Psi^{(1)})^{1, \Theta_1}}$ is pseudo-hyponormal, it
follows from (\ref{bts-1}) that $I_{\theta}$ is a (left) inner
divisor of $\Theta_3^{\prime}$ and hence is a inner divisor of
$\Theta_3$. \
\end{proof}

\medskip

\begin{lemma} \label{cor5.16} \ Let $\mathbf{T}\equiv(T_{\Phi}, T_{\Psi})$ be a pseudo-hyponormal
Toeplitz pair with bounded type symbols $\Phi, \Psi \in
L^\infty_{M_n}$ of the form
$$
\Phi_+ = \theta \theta_1 A^*, \ \ \Phi_- =\theta_0 B^*, \ \ \Psi_+ =
\theta \theta_3 C^*, \ \ \Psi_- = \theta_2 D^*\quad \hbox{\rm
(coprime)},
$$
where $\theta:=\hbox{\rm l.c.m.}(\theta_0, \theta_2)$. If we let
$\delta:=\hbox{\rm g.c.d.}(\theta_1, \theta_3)$, then
$$
\mathbf{T}   \ \hbox{ pseudo-hyponormal} \Longleftrightarrow
 \mathbf{T}_{\Delta}:=(T_{\Phi^{1, \delta}}, T_{\Psi^{1,
\delta}}) \ \hbox{ pseudo-hyponormal}.
$$
\end{lemma}

\begin{proof}
This follows from a slight variation of \cite[Proof of Theorem
1]{HL4} for the matrix-valued case by using Theorem \ref{thm82.7}.
\end{proof}

\medskip

\begin{lemma}\label{lem5.9} Suppose $\theta$ is a finite Blaschke product of degree $n$. \
Let $\Phi,\Psi\in L^\infty_{M_n}$ be such that
$\mathbf{T}:=(T_{\Phi}, T_{\Psi})$ and
$\mathbf{T}':=(T_{\Phi\circ\theta}, T_{\Psi\circ\theta}). \ $ Then
\begin{equation}\label{5.25}
[\mathbf T'^*, \mathbf T'] \cong  \bigoplus _n [\mathbf T^*, \mathbf
T],
\end{equation}
and
\begin{equation}\label{5.25222}[\mathbf T'^*, \mathbf T']_p \cong
\bigoplus _n [\mathbf T^*, \mathbf T]_p\,,
\end{equation}
where $n=\hbox{\rm deg}\,(\theta)$ and \ $\cong$ means unitary
equivalence. \  In particular, $\mathbf{T}$ is {\rm (}pseudo{\rm )}
hyponormal if and only if $\mathbf{T}'$ is {\rm (}pseudo{\rm )}
hyponormal.
\end{lemma}

{\it Remark.} \ The unitary operator given in Lemma \ref{lem5.9}
depends only on the inner function $\theta$. \

\begin{proof}[Proof of Lemma \ref{lem5.9}] \ Let $\Phi=\left(\varphi_{ij}\right)\in
L^{\infty}_{M_n}$ be arbitrary. \ Then by a well-known fact due to
C. Cowen \cite [Theorem 1] {Co1}, there exists a unitary operator
$V$ such that
$$
T_{\varphi_{ij}\circ\theta}=V^*\left(\bigoplus_nT_{\varphi_{ij}}\right)V\,
 \quad(\hbox{for}\ i, j=1,2,\cdots, n),
$$
where $n=\hbox{\rm deg}\,(\theta)$. \ Put $\mathcal V := V \otimes
I_n$. \ Then
$$
[T_{\Psi\circ\theta}^*,\ T_{\Phi\circ\theta}] =T_{\Psi\circ\theta}^*
T_{\Phi\circ\theta}-T_{\Phi\circ\theta}T_{\Psi\circ\theta}^*
=\mathcal V^*\left(\bigoplus_n [T_{\Psi}^*, T_{\Phi}]\right)\mathcal
V.
$$
Therefore we have
$$
\aligned~[\mathbf T'^*, \mathbf T'] &=\begin{pmatrix}
\hbox{$[T_{\Phi\circ\theta}^*, T_{\Phi\circ\theta}]$}&
\hbox{$[T_{\Psi\circ\theta}^*, T_{\Phi\circ\theta}]$}\\
\hbox{$[T_{\Phi\circ\theta}^*, T_{\Psi\circ\theta}]$}&
\hbox{$[T_{\Psi\circ\theta}^*, T_{\Psi\circ\theta}]$}
\end{pmatrix}\\
&=\begin{pmatrix}\mathcal V&0\\0&\mathcal V
\end{pmatrix}^*\begin{pmatrix}
\hbox{$\bigoplus_n[T_{\Phi}^*, T_{\Phi}]$}&
\hbox{$\bigoplus_n[T_{\Psi}^*, T_{\Phi}]$}\\
\hbox{$\bigoplus_n[T_{\Phi}^*, T_{\Psi}]$}&
\hbox{$\bigoplus_n[T_{\Psi}^*, T_{\Psi}]$}
\end{pmatrix}\begin{pmatrix}\mathcal V&0\\0&\mathcal V
\end{pmatrix}\\
&\cong \bigoplus_n [\mathbf T^*, \mathbf T]\,,
\endaligned
$$
giving (\ref{5.25}). \  We also observe
$$
\aligned ~ [T_{\Psi\circ\theta}^*,\ T_{\Phi\circ\theta}]_p
&=[T_{\Psi\circ\theta}^*,\
T_{\Phi\circ\theta}]-T_{(\Psi^*\Phi-\Phi\Psi^*)(\theta)}\\
&=\mathcal V^*\left(\bigoplus_n \Bigl[[T_{\Psi}^*,
T_{\Phi}]-T_{(\Psi^*\Phi-\Phi\Psi^*)}\Bigr]\right)\mathcal V\\
&\cong \bigoplus_n [T_{\Psi}^*, T_{\Phi}]_p\,,
\endaligned
$$
giving (\ref{5.25222}). \ The remaining assertions are evident from
(\ref{5.25}) and (\ref{5.25222}). \
\end{proof}

\medskip

\begin{lemma}\label{lem5.18} Let
$\mathbf T=(T_{\Phi}, T_{\Psi})$ be a pseudo-hyponormal Toeplitz
pair with bounded type symbols $\Phi, \Psi \in L^\infty$ of the form
\begin{equation}
\Phi_+ = \theta^{p+1} A^*, \ \ \Phi_- =\theta B^* ,\ \ \Psi_+ =
\theta^{q+1} C^*, \ \ \Psi_- =\theta D^* \quad(\hbox{\rm right
coprime}),
\end{equation}
where the $\theta_i$ is inner. \ If $pq=0$, then $p=q=0$.
\end{lemma}

\begin{proof} \
Without loss of generality we may assume that $q=0$. \ Suppose that
$p\neq 0$. \ Since $\mathbf T$ is pseudo-hyponormal, it follows from
Lemma \ref{lem5.5} that $(T_{\Phi-\beta\Psi}, T_{\Psi})$ is
pseudo-hyponormal for all $\beta \in \mathbb R$. \  In particular
$T_{\Phi-\beta\Psi}$ is pseudo-hyponormal for all $\beta\in \mathbb
R$. \ Observe that
$$
\Phi-\beta \Psi=\overline\theta (B-\beta D) + \theta^{p+1} (A -
\overline{\beta} C \theta^{p})^*.
$$
Thus by Proposition \ref{pro3.13}, we can see that
$T_{\Upsilon_\beta}$ is pseudo-hyponormal, where
$$
\Upsilon_\beta= \overline\theta (B-\beta D) + \theta (P_{{\mathcal
K}_{I_{\theta}}}(A- \overline{\beta} \theta^{p}C))^* =
\overline\theta (B-\beta D) + \theta \bigl(P_{\mathcal
K_{I_\theta}}A\bigr)^*,
$$
which gives a contradiction because
$||(\Upsilon_\beta)_-||_2>||(\Upsilon_\beta)_+||_2$ if $|\beta|$ is
sufficiently large.
\end{proof}

\medskip

\begin{lemma} \label{thm5.19}
Let $\mathbf T:=(T_\Phi, T_\Psi)$ be a pseudo-hyponormal Toeplitz
pair with bounded type symbols $\Phi,\Psi\in L^\infty_{M_n}$. \
Suppose the inner parts of right coprime factorizations of $\Phi_+$
and $\Psi_+$ commute. \ If $\Phi$ and $\Psi$ have a common tensored-scalar
pole, then this pole has the same order.
\end{lemma}

\begin{proof}
In view of (\ref{bts-1}), we may write
$$
\Phi_+ = \Theta_0\Theta_1 A^*, \ \ \Phi_- =\Theta_0 B^* ,\ \ \Psi_+
= \Theta_2\Theta_3 C^*, \ \ \Psi_- =\Theta_2 D^* \quad(\hbox{right
coprime}).
$$
Assume that $\Theta_0\Theta_1$ and $\Theta_2\Theta_3$ commute. \
Also suppose $\Phi$ and $\Psi$ have a common tensored-scalar pole of order
$p$ and $r$, respectively, at $\alpha$ . \ Then we can write
$$
\Theta_0=b_{\alpha}^p \Delta_0, \ \ \Theta_1=b_{\alpha}^q \Delta_1,
\ \ \Theta_2=b_{\alpha}^r \Delta_2, \ \ \Theta_3=b_{\alpha}^s
\Delta_3 \ \  \left(b_\alpha(z):=\frac{z-\alpha}{1-\overline\alpha
z}\right),
$$
where $p,r \geq 1, q,s \geq 0$, $\Delta_i$ and $I_{b_{\alpha}}$ are
coprime for $i=0,1,2,3$.  \  By Theorem \ref{lem5.17} and Lemma
\ref{lem5.9}, we may assume that $\alpha=0$. \ Assume to the
contrary that $r\neq p$. \  Then
without loss of generality we may assume that $p < r$. \ Let
$$
\Delta^{\prime}:=\prod_{i=0}^3 \Delta_i \quad \hbox{and} \quad
\Delta:= z^{p-1} \Delta^{\prime},
$$
and write $\Phi^{(1)}:=\Phi_{\Delta}$ and
$\Psi^{(1)}:=\Psi_{\Delta}$. \ It then follows from Lemma
\ref{lem5.4} that $(T_{\Phi^{(1)}}, T_{\Psi^{(1)}})$ is
pseudo-hyponormal. \ Since by assumption, $\Theta_0\Theta_1$ and
$\Theta_2\Theta_3$ commute, it follows that
$$
\Theta_2\Theta_3={z}^{r+s}\Delta_2\Delta_3 \quad \hbox{and} \quad
\Delta={z}^{p-1}\Delta_2\Delta_3\Delta_0\Delta_1,
$$
so that $\hbox{left-g.c.d.} (\Delta,
\Theta_2\Theta_3)={z}^{p-1}\Delta_2\Delta_3$. \ Thus we have
$$
\bigl(\hbox{left-g.c.d.} (\Delta,
\Theta_2\Theta_3)\bigr)^*\Theta_2\Theta_3=I_{z}^{r+s-p+1}\equiv
I_{\theta_1}.
$$
By Proposition \ref{pro2.6} (c), we can write
$$
\Psi^{(1)}_+={z}^{r+s-p+1}C_1^* \qquad (\hbox{coprime}),
$$
where $C_1=P_{\mathcal K_{{\theta_1}}}(C\Delta_0\Delta_1)$. \
Similarly we also get the following coprime factorizations:
$$
\Phi^{(1)}_+ = {z}^{q+1}A_1^*, \ \ \Phi^{(1)}_-={z}B_1^*, \ \
\Psi^{(1)}_-={z}^{r-p+1} D_1^*
$$
for some $A_1,B_1,D_1\in H^2_{M_n}$, in particular $B_1 \in M_n$ is
invertible a.e. on $\mathbb T$. \ Applying Lemma \ref{lem5.12} with
$\theta=z^{r-p}$ and $\theta_3=z^s$, we have $0<r-p\leq s$. \ Then
by using Proposition \ref{pro3.13} and Lemmas \ref{lem5.3},
\ref{lem5.4}, and \ref{lem5.5}, we can show that $s \geq q$ and
$T_{\Phi^{(\gamma)}}$ is pseudo-hyponormal, where
$$
\Phi^{(\gamma)}_+ ={z}^{r-p+1}A_\gamma^*\quad\hbox{and}\quad
\Phi^{(\gamma)}_- = {z}^{r-p+1}\bigl(D_1+\gamma
{z}^{r-p}B_1\bigr)^*,
$$
with
$$
A_{\gamma}:=P_{\mathcal K_{{z}^{r-p+1}}}(C_1+\overline{\gamma}
{z}^{r+s-p-q}A_1).
$$
Since $T_{\Phi^{(\gamma)}}$ is pseudo-hyponormal it follows that
$s=q$. \ Thus we have
$$
A_\gamma = P_{\mathcal K_{{z}^{r-p+1}}}C_1 + P_{\mathcal
K_{{z}^{r-p+1}}} (\overline{\gamma} {z}^{r-p}A_1).
$$
Put
$$
P_{\mathcal K_{{z}^{r-p+1}}}C_1=\sum_{i=0}^{r-p}{z}^i C_1^{(i)}
\quad\hbox{and}\quad D_1=\sum_{i=0}^{r-p} {z}^i D_1^{(i)},
$$
where $C_1^{(i)}, D_1^{(i)}\in M_n$ for each $i=0,\cdots,r-p$. \
Then we have
$$
\aligned &\Phi^{(\gamma)}_+ ={z}^{r-p+1}\Biggl( \sum_{i=0}^{r-p-1}
{z}^i C_1^{(i)} + z^{r-p} \bigl(C_1^{(r-p)}
            +\overline{\gamma}P_{\mathcal K_{z}}A_1\bigr) \Biggr)^*\equiv z^{r-p+1}C^{\prime};\\
&\Phi^{(\gamma)}_-= z^{r-p+1}\Biggl ( \sum_{i=0}^{r-p-1} z^i
D_1^{(i)} + z^{r-p} \bigl(D_1^{(r-p)}
            +\gamma B_1\bigr) \Biggr)^*\equiv z^{r-p+1}D^{\prime}.
\endaligned
$$
Since $T_{\Phi^{(\gamma)}}$ is pseudo-hyponormal, there exists a
matrix function $K_{\gamma} \in \mathcal E(\Phi^{(\gamma)})$. \
Write
$$
K_{\gamma}(z)=\sum_{i=0}^{\infty}  z^i K_\gamma^{(i)}.
$$
Since $I_z$ and $C_1^{(0)}$, $D_1^{(0)}$ are coprime,  $I_z$ and
$C^{\prime}$, $D^{\prime}$ are coprime. \ Thus, it follows from
Theorem \ref{thm32.378} (with $\theta=z$) that
$$
\begin{pmatrix} D_1^{(0)}\\\vdots\\D_1^{(r-p-1)}\\D_1^{(r-p)}+\gamma B_1\end{pmatrix}
   =\begin{pmatrix}K_{\gamma}^{(0)}&0&\cdots &0\\K_{\gamma}^{(1)}
      &K_{\gamma}^{(0)}&\cdots&0\\ \vdots&\vdots&\ddots&\vdots\\
        K_{\gamma}^{(r-p)}& K_{\gamma}^{(r-p-1)}&\cdots & K_{\gamma}^{(0)}
\end{pmatrix}
\begin{pmatrix}C_1^{(0)}\\ \vdots \\C_1^{(r-p-1)}\\ C_1^{(r-p)}+
\overline{\gamma}A_1(0)\end{pmatrix}.
$$
Thus for each $\gamma\in \mathbb C$,
\begin{equation}\label{5.33-1}
\gamma B_1- \overline{\gamma}K_{\gamma}^{(0)}A_1(0)
=\sum_{j=0}^{r-p-1}K^{(r-p-j)}_{\gamma} C_1^{(j)}
      +K_{\gamma}^{(0)} C_1^{(r-p)}-D_1^{(r-p)}.
\end{equation}
If we put $\gamma:=\zeta+i\,\zeta$ ($\zeta\in \mathbb R$), then
(\ref{5.33-1}) can be written as
\begin{equation}\label{5.34-1}
\begin{aligned}
\zeta\,(B_1-K_{\gamma}^{(0)}A_1(0)) &+i\,\zeta\,(B_1+K_{\gamma}^{(0)}A_1(0))\\
&=\sum_{j=0}^{r-p-1}K^{(r-p-j)}_{\gamma} C_1^{(j)}
      +K_{\gamma}^{(0)} C_1^{(r-p)}-D_1^{(r-p)}.
\end{aligned}
\end{equation}
But since $||K_\gamma||_\infty\le 1$ (and consequently $\sup_\gamma
||\hbox{right-hand side of (\ref{5.34-1})}||<\infty$), letting
$\zeta\to\infty$ on both sides of (\ref{5.34-1}) gives
$$
B_1=K_{\gamma}^{(0)}A_1(0)=-K_{\gamma}^{(0)}A_1(0)\,,
$$
which implies that $B_1=0$. \  This contradicts the fact that
$\det\,B_1\ne 0$. \ This completes the proof. \ \end{proof}

\begin{theorem} \label{thm5.19-2} \
Let $\mathbf T:=(T_\Phi, T_\Psi)$ be a pseudo-hyponormal Toeplitz
pair with bounded type symbols $\Phi,\Psi\in L^\infty_{M_n}$. \
Suppose all inner parts of right coprime factorizations of
$\Phi_\pm$ and $\Psi_\pm$ are diagonal-constant. \ If $\Phi$ and
$\Psi$ have a common tensored-scalar pole then the inner parts of $\Phi_-$
and $\Psi_-$ coincide.
\end{theorem}

\begin{proof} \
By assumption, write
$$
\Phi_+ = \theta_0\theta_1 A^*, \ \ \Phi_- =\theta_0 B^* ,\ \ \Psi_+
= \theta_2\theta_3 C^*, \ \ \Psi_- =\theta_2 D^* \quad(\hbox{right
coprime}),
$$
where the $\theta_i$ are inner and $A,B,C,D\in H^2_{M_n}$. \ Suppose
$\Phi$ and $\Psi$ have a common tensored-scalar pole. \ We want to show that
$\theta_0=\theta_2$. \ By Theorem \ref{lem5.17} and Lemmas
\ref{lem5.9} and \ref{thm5.19}, we may write,  without loss of
generality,
$$
\theta_0={z}^p \delta_0, \ \ \theta_1={z}^q \delta_1, \ \
\theta_2={z}^p \delta_2, \ \ \theta_3={z}^s \delta_3 ,
$$
where $p \geq 1, q,s \geq 0$, $\delta_i(0)\ne 0$ for $i=0,1,2,3$.
Assume to the contrary that $\theta_0\ne \theta_2$. \ Then
$\delta_0\ne\delta_2$. \ Now if we combine Lemmas
\ref{lem544444.2-1}, \ref{lem5.4}, \ref{lem5.12} \ref{lem5.18} and
Proposition \ref{pro2.6}, then we can see that $\mathbf
T:=(T_{\Phi^\prime}, T_{\Psi^\prime})$ is pseudo-hyponormal, under
the following coprime factorizations:
$$
\Phi_+^\prime=z\omega A_1^*,\ \ \Phi_-^\prime=z B_1^*,\ \
\Psi_+^\prime= z\omega^p C_1^*,\ \ \Psi_-^\prime=z\omega D_1^*\ \
(p\ge 2),
$$
where $\omega=b_\beta$ for some nonzero $\beta$. \ Let
$\gamma\in\mathbb C$. \ By Lemma \ref{lem5.5},
$(T_{\Phi^\prime-\gamma \Psi^\prime}, T_{\Psi^\prime})$ is
pseudo-hyponormal. \ Observe that
$$
(\Phi^\prime-\gamma
\Psi^\prime)_+=z\omega^p\Bigl(\omega^{p-1}A_1-\overline\gamma
C_1\Bigr)^*.
$$
On the other hand, we can choose $\gamma_0\in\mathbb C$ such that
$\bigl(\omega^{p-1}A_1-\overline{\gamma_0}C_1\bigr)(0)$ is not
invertible (in fact, $\overline{\gamma_0}$ is an eigenvalue of
$\omega^{p-1}(0)C_1(0)^{-1}A_1(0)$). \ Then by Theorem
\ref{lem33.9}, $\omega^{p-1}A_1-\overline{\gamma_0}C_1$ and $I_z$
are not coprime. \ Thus there exists a nonconstant inner divisor
$\Omega_1$ of $I_z$ (say, $I_z=\Omega_1\Omega_2$) such that
$$
(\Phi^\prime-\gamma_0 \Psi^\prime)_+ =
\omega^p\Omega_2\Bigl(\Omega_1^*
\bigl(\omega^{p-1}A_1-\overline{\gamma_0} C_1\bigr)\Bigr)^*.
$$
We thus have
$$
\omega^p\Omega_2 H^2_{\mathbb C^n}\subseteq \ker
H_{(\Phi^\prime-\gamma_0 \Psi^\prime)_+^*}.
$$
Hence, by Corollary \ref{cor55.6},
$$
\omega^p\Omega_2 H^2_{\mathbb C^n}\subseteq \ker
H_{(\Phi^\prime-\gamma_0 \Psi^\prime)_+^*} \subseteq \ker
H_{(\Psi_-^\prime)^*}=z\omega H^2_{\mathbb C^n},
$$
which implies that $I_z$ is an inner divisor of $\Omega_2$, and
hence $\Omega_1$ is a constant unitary, a contradiction. \ This
completes the proof. \
\end{proof}

\medskip

\begin{corollary} \label{cor5.19} \
Let $\mathbf T:=(T_\Phi, T_\Psi)$ be a hyponormal Toeplitz pair with
bounded type symbols
 $\Phi,\Psi\in L^\infty_{M_n}$. \
Suppose all inner parts of right coprime factorizations of
$\Phi_\pm$ and $\Psi_\pm$ are diagonal-constant. \ If $\hbox{\rm
det}\,\Phi_-^*$ and $\hbox{\rm det}\,\Psi_-^*$ have a common pole
then the inner parts of $\Phi_-$ and $\Psi_-$ coincide.
\end{corollary}

\begin{proof} \ Write
$$
\Phi_-=\theta_0 B^*\quad\hbox{and}\quad \Psi_-=\theta_2
D^*\quad\hbox{(right coprime)}.
$$
Suppose $\hbox{\rm det}\,\Phi_-^*$ and $\hbox{\rm det}\,\Psi_-^*$
have a common pole. \ Since
$$
\hbox{det}\,\Phi_-^* =\frac{\hbox{det}\,B}{\theta_0^n}
\quad\hbox{and}\quad
\hbox{det}\,\Psi_-^*=\frac{\hbox{det}\,D}{\theta_2^n},
$$
there exists $\alpha \in \mathbb D$ such that
$\theta_0(\alpha)^n=\theta_2(\alpha)^n=0$, so that
$\theta_0(\alpha)=\theta_2(\alpha)=0$. \ Thus $\Phi_-^*$ and
$\Psi_-^*$ have a common tensored-scalar pole and hence it follows from
Theorem \ref{thm5.19-2}. \
\end{proof}

\begin{corollary}\label{cor5.20}
Let $\mathbf T:=(T_\Phi, T_\Psi)$ be a hyponormal Toeplitz pair with
matrix-valued trigonometric polynomial symbols whose outer
coefficients are invertible. \  Then
$$
\hbox{\rm deg}\,(\Phi_-)=\hbox{\rm deg}\,(\Psi_-)\,.
$$
\end{corollary}

\begin{proof} Let $\Phi(z)=\sum_{j=-m}^N A_j z^j$ and $\Psi(z)=
\sum_{j=-l}^M B_j z^j$. \  Then $\Phi$ and $\Psi$ have a common tensored-scalar pole
at $z=0$. \  Thus by Theorem \ref{thm5.19}, we have $m=l$, i.e.,
$\hbox{\rm deg}\,(\Phi_-)=\hbox{\rm deg}\,(\Psi_-)$. \
\end{proof}

The following theorem gives a characterization of hyponormality for
Toeplitz pairs with bounded type symbols. \

\begin{theorem}\label{thm5.22-1}
{\rm (Hyponormality of Toeplitz Pairs with Bounded Type Symbols)} \
Let $\mathbf{T}\equiv (T_{\Phi}, T_{\Psi})$ be a Toeplitz pair with
bounded type symbols $\Phi, \Psi \in L^\infty_{M_n}$ of the form
\begin{equation}\label{5.35}
\Phi_+ = \theta_0 \theta_1 A^*, \ \ \Phi_- =\theta_0 B^*, \ \ \Psi_+
= \theta_2 \theta_3 C^*, \ \ \Psi_- = \theta_2 D^*\quad\hbox{\rm
(coprime)} ,
\end{equation}
where the $\theta_i$ are inner functions. \ Assume that
\begin{itemize}
\item[(a)] $\theta_0$ or $\theta_2$ is a finite Blaschke product;
\item[(b)] $\theta_0$ and $\theta_2$ have a common Blaschke factor $b_\alpha$ such that
$B(\alpha)$ and $D(\alpha)$ are diagonal-constant.
\end{itemize}
If the pair $\mathbf T$ is hyponormal then
\begin{equation}\label{5.35-1}
\Phi-\Lambda \Psi\in H^2_{M_n}\quad\hbox{for some}\ \Lambda\in M_n.
\end{equation}
Moreover, $\mathbf T$ is hyponormal if and only if
\begin{enumerate}
\item[(i)] $\Phi$ and $\Psi$ are normal and $\Phi\Psi=\Psi\Phi$;
\item[(ii)] $\Phi_-=\Lambda^*\Psi_-$;
\item[(iii)] $T_{\Psi^{1, \Omega}}$ is pseudo-hyponormal with
$\Omega:=\theta_0\theta_1\theta_3\overline{\theta}\Delta^*$\,,
\end{enumerate}
where $\theta:=\hbox{\rm g.c.d.} (\theta_1, \theta_3)$ and
$\Delta:=\hbox{\rm left-g.c.d.}\, \bigl(I_{\theta_0\theta}, \
\overline\theta (\theta_3A- \theta_1 C\Lambda^*)\bigr)$.
\end{theorem}

\begin{proof}
Suppose $\mathbf T$ is hyponormal. \ Then by Theorem
\ref{thm5.19-2}, $\theta_0=\theta_2$. \ In view of Lemma
\ref{cor5.16} we may assume that $\theta_1$ and $\theta_3$ are
coprime. \ By our assumption, we note that $B(\alpha)$ and
$D(\alpha)$ are invertible. \ Put $\Lambda:=B(\alpha)D(\alpha)^{-1}$
(note that $\Lambda$ is diagonal-constant by assumption). \ Then
$(B-\Lambda D)(\alpha)=0$, so that
\begin{equation}\label{PhiPsi}
B-\Lambda D= b_\alpha B_{1} \quad \hbox{and} \quad \theta_0=b_\alpha
\delta_1\quad (B_{1}\in H^2_{M_n}; \ \delta_1 \ \hbox{ inner}).
\end{equation}
We now claim that
\begin{equation}\label{5.36}
\Phi_-=\Lambda^*\Psi_-.
\end{equation}
Assume to the contrary that $\Phi_-\ne \Lambda^*\Psi_-$. \ Then it
follows from (\ref{PhiPsi}) that
$$
\aligned \Phi -\Lambda\Psi &= \theta_0 \theta_1 \theta_3
(\overline{\theta_3} A^*-
\overline{\theta_1}\Lambda C^*) + \overline{\theta_0} (B -\Lambda D)\\
&= \theta_0 \theta_1 \theta_3 (\overline{\theta_3} A^*-
\overline{\theta_1}\Lambda C^*) + \overline{\delta}_1B_{1}.
\endaligned
$$
Suppose that $\Phi-\Lambda\Psi\notin H^2_{M_n}$. \ Then $\delta_1$
is nonconstant. \ Since, by Lemma \ref{lem5.5}, $(T_{ \Phi -\Lambda
\Psi}, T_{\Psi})$ is pseudo-hyponormal, applying Theorem
\ref{thm5.19-2} with $(T_{ \Phi -\Lambda \Psi}, T_{\Psi})$ in place
of $(T_{\Phi}, T_{\Psi})$ gives $\theta_0=\delta_1$, which is a
contradiction.  \ This proves (\ref{5.36}) and hence (\ref{5.35-1}).
\

Towards the second assertion, let $\theta:=\hbox{g.c.d.}(\theta_1,
\theta_3)$. \ Then we can write
$$
\theta_1=\theta \omega_1 \ \hbox{and} \ \theta_3=\theta \omega_3
\quad \hbox{(for some coprime inner functions $\omega_1$ and
$\omega_3$)}.
$$
We thus have
$$
\Phi_{\Lambda}:=\Phi-\Lambda \Psi=\theta_0\theta \omega_1 \omega_3
(\overline{\omega_3} A^*- \overline{\omega_1} \Lambda C^*) \in
H^2_{M_n}.
$$
We claim that
\begin{equation}\label{coprimeomega}
\omega_3 A - \omega_1 C\Lambda^* \ \hbox{and} \  \omega_1 \omega_3
\ \hbox{are coprime}.
\end{equation}
Assume, to the contrary, that $\omega_3 A - \omega_1 C\Lambda^*$ and
$\omega_1$ are not coprime. \ Then there exists a nonconstant inner
matrix function $\Upsilon \in H^{\infty}_{M_n}$ such that
$$
\omega_3 A - \omega_1 C\Lambda^*=\Upsilon A^{\prime} \quad
\hbox{and} \quad I_{\omega_1}=\Upsilon \Omega^{\prime}.
$$
Thus we have
\begin{equation}\label{upsilon}
\omega_3 A = \Upsilon \Omega^{\prime} C\Lambda^*+\Upsilon
A^{\prime}=\Upsilon(\Omega^{\prime} C\Lambda^*+A^{\prime}).
\end{equation}
Since $A$ and $I_{\theta_1}$ are coprime, it follows that $A$ and
$\Upsilon$ are coprime and hence $\widetilde{A}$ and
$\widetilde{\Upsilon}$ are coprime. \ It thus follows from
(\ref{upsilon}) that
$$
\aligned \Upsilon^*\omega_3 A \in H^2_{M_n} &\Longrightarrow
H_{\widetilde{A}\widetilde{\omega}_3\widetilde{\Upsilon}^*}=0\\
&\Longrightarrow \widetilde{I_{\omega_3}}H^2_{\mathbb C^n} \subseteq
\ker H_{\widetilde{A}\widetilde{\Upsilon}^*}=
\widetilde{\Upsilon}H^2_{\mathbb C^n},
\endaligned
$$
which implies that $\widetilde{\Upsilon}$ is a  inner divisor of
$\widetilde{I_{\omega_3}}$, and hence ${\Upsilon}$ is an inner
divisor of ${I_{\omega_3}}$, so that  ${\Upsilon}$ is a nonconstant
common inner divisor of $I_{\omega_1}$ and $I_{\omega_3}$, a
contradiction. \ This proves (\ref{coprimeomega}). \ Let
$$
\Delta:=\hbox{\rm left-g.c.d.}\, \bigl(I_{\theta_0\theta}, \
\omega_3 A - \omega_1 C\Lambda^*\bigr).
$$
Then we may write
$$
\Phi_\Lambda= (\Delta^*\theta_0\theta\omega_1\omega_3)
\Bigl[\Delta^*({\omega_3} A - {\omega_1}C\Lambda^*) \Bigr]^*
\quad\hbox{(left coprime)}.
$$
It follows from Corollary \ref{lem2.3} (with $A:=\Delta^*(\omega_3 A
- \omega_1 C\Lambda^*)$, $\theta:=\omega_1\omega_3$  and
$B:=\theta_0\theta\Delta^*$) that $\Delta^*(\omega_3 A - \omega_1
C\Lambda^*)$ and $\theta_0\theta\omega_1\omega_3\Delta^*$ are left
coprime, and by Lemma \ref{lem5.3}, $T_{\Psi^{1,\Omega}}$ is
pseudo-hyponormal, where $\Omega:=\theta_0 \theta \omega_1 \omega_3
\Delta^* =\theta_0\theta_1\theta_3\overline\theta\Delta^*$. \  The
converse is obtained by reversing the above argument. \  This
completes the proof. \
\end{proof}

\medskip

\begin{corollary}\label{thm5.22}
{\rm (Hyponormality of Rational Toeplitz Pairs)} Let
$\mathbf{T}\equiv (T_{\Phi}, T_{\Psi})$ be a Toeplitz pair with
rational symbols $\Phi, \Psi \in L^\infty_{M_n}$ of the form
\begin{equation}\label{5.35-35}
\Phi_+ = \theta_0 \theta_1 A^*, \ \ \Phi_- =\theta_0 B^*, \ \ \Psi_+
= \theta_2 \theta_3 C^*, \ \ \Psi_- = \theta_2 D^*\quad\hbox{\rm
(coprime)}.
\end{equation}
Assume that $\theta_0$ and $\theta_2$ are not coprime. \ Assume also
that $B(\gamma_0)$ and $D(\gamma_0)$ are diagonal-constant for some
$\gamma_0\in\mathcal{Z}(\theta_0)$. \ If the pair $\mathbf T$ is
hyponormal then
\begin{equation}\label{5.35-2}
\Phi-\Lambda \Psi\in H^2_{M_n}\quad\hbox{for some}\ \Lambda\in M_n.
\end{equation}
Moreover, the pair $\mathbf T$ is hyponormal if and only if
\medskip

\begin{enumerate}
\item[(i)] $\Phi$ and $\Psi$ are normal and $\Phi\Psi=\Psi\Phi$;
\item[(ii)] $\Phi_-=\Lambda^*\Psi_-$ (with $\Lambda:=B(\gamma_0)D(\gamma_0)^{-1}$);
\item[(iii)] $T_{\Psi^{1, \Omega}}$ is pseudo-hyponormal with
$\Omega:=\theta_0\theta_1\theta_3\overline{\theta}\Delta^*$\,,
\end{enumerate}
\medskip
\noindent where $\theta:=\hbox{\rm g.c.d.}\, (\theta_1, \theta_3)$
and $\Delta:=\hbox{\rm left-g.c.d.}\, \bigl(I_{\theta_0\theta}, \
\overline\theta (\theta_3A- \theta_1 C\Lambda^*)\bigr)$.
\end{corollary}

\begin{proof} \
This follows from Theorem \ref{thm5.22-1}. \
\end{proof}

\medskip

For the following result, we recall the notion of outer coefficient,
defined on page \pageref{outer}.

\begin{corollary}\label{cor5.23} {\rm (Hyponormality of
Polynomial Toeplitz Pairs)} Let $\Phi, \Psi\in L^\infty_{M_n}$ be
matrix-valued trigonometric polynomials of the form
\begin{equation}\label{5.36-1}
\Phi(z):=\sum_{j=-m}^N A_j z^j\quad\hbox{and}\quad
\Psi(z):=\sum_{j=-\ell}^M B_j z^j
\end{equation}
satisfying
\begin{itemize}
\item[(i)] the outer coefficients $A_{-m},A_N, B_{-\ell}$ and $B_M$ are invertible;
\item[(ii)] the ``co-analytic" outer coefficients $A_{-m}$ and $B_{-\ell}$ are diagonal-constant.
\end{itemize}
If $\mathbf T\equiv (T_\Phi, T_\Psi)$ is hyponormal then
\begin{equation}\label{5.37}
\Phi-\Lambda\Psi\in H^2_{M_n}\quad\hbox{for some constant matrix}\
\Lambda\in M_n.
\end{equation}
\end{corollary}

\noindent{\it Remark.} \ If $\Phi$ and $\Psi$ are scalar-valued
trigonometric polynomials then all assumptions of the above result
are trivially satisfied, and hence the above result reduces to
(\ref{5.1*}). \

\begin{proof}[Proof of Corollary \ref{cor5.23}]
This follows from Corollary \ref{thm5.22} together with the
following observations: If $\Phi$ and $\Psi$ are matrix-valued
trigonometric polynomials of the form (\ref{5.36-1}), then under the
notation of Corollary \ref{thm5.22} we have
\medskip

\begin{itemize}
\item[(i)] $\theta_i=z^{m_i}$ for some $m_i\in\mathbb N$,
               and hence $\mathcal{Z}(\theta_i)=\{0\}$;
\item[(ii)] $A(0)=A_N$, $B(0)=A_{-m}$, $C(0)=B_M$ and $D(0)=B_{-\ell}$;
\item[(iii)] if $A_N$, $A_{-m}$, $B_M$ and $B_{-\ell}$ are invertible then
the representation (\ref{5.36-1}) are coprime factorizations;
\item[(iv)] $B(0)$ and $D(0)$ are diagonal-constant.
\end{itemize}
\end{proof}

\bigskip

We next consider the self-commutators of the Toeplitz pairs with
matrix-valued rational symbols and derive rank formulae for their
self-commutators. \

\medskip

Let $\Phi, \Psi \in L_{M_n}^\infty$ be matrix-valued rational
functions of the form
\begin{equation}\label{5.39}
\Phi_+ = \theta_0 \theta_1 A^*, \ \ \Phi_- =\theta_0 B^*, \ \ \Psi_+
= \theta_2 \theta_3 C^*, \ \ \Psi_- = \theta_2 D^*\quad \hbox{\rm
(coprime)}.
\end{equation}
By Corollary \ref{cor5.19}, if the pair $(T_\Phi, T_\Psi)$ is
hyponormal and if $\theta_0$ and $\theta_2$ are not coprime then
$\theta_0=\theta_2$. \  The following question arises at once. \

\bigskip

\begin{question}\label{qA}
Let $\mathbf {T}\equiv(T_{\Phi}, T_{\Psi})$ be hyponormal, where
$\Phi$ and $\Psi$ are given in {\rm (}\ref{5.39}{\rm )}. \  Does it
follow that $\theta_1=\theta_3$?
\end{question}

The answer to Question \ref{qA} is negative even for scalar-valued
symbols, as shown by the following example. \

\begin{example} \label{ex5.27}
Let $\varphi=\overline{\varphi_-}+\varphi_+\in L^\infty$ and
$\psi=\overline{\psi_-}+\psi_+\in L^{\infty}$ be of the form
$$
\varphi_+=4z, \ \ \varphi_-=z, \ \ \psi_+=2zb, \ \ \psi_-=z\quad
\left(\hbox{\rm where}\
b(z):=\frac{z+\frac{1}{2}}{1+\frac{1}{2}z}\right). \
$$
Under the notation of Corollary \ref{thm5.22}, $\varphi_-=\psi_- $
and we have
$$
a\theta_3-\overline{\lambda}c\theta_1 =4b-2.
$$
Since $(4b-2)(0)=0$, we have $\delta=z$ and
$$
P_{H_0^2}(\theta \delta \overline{\theta}_1 c)=P_{H_0^2}(z\cdot
2)=2z.
$$
Thus we have $\upsilon=\overline z +2z$, and $T_{\upsilon}$ is
hyponormal.  \  Therefore by Corollary \ref{thm5.22}, $\mathbf
{T}:=(T_{\phi}, T_{\psi})$ is hyponormal. \  Note that $\theta_1=1
\neq b=\theta_3$. \hfill$\square$
\end{example}

\medskip

However, if the symbols are matrix-valued rational function then the
answer to Question \ref{qA} is indeed affirmative under some
assumptions on the symbols. \

\begin{theorem} \label{thm5.28} \
 Let
$\mathbf{T}\equiv (T_{\Phi}, T_{\Psi})$ be a hyponormal Toeplitz
pair with rational symbols $\Phi, \Psi \in L^\infty_{M_n}$ of the
form
\begin{equation}\label{5.35-35-35}
\Phi_+ = \theta_0 \theta_1 A^*, \ \ \Phi_- =\theta_0 B^*, \ \ \Psi_+
= \theta_2 \theta_3 C^*, \ \ \Psi_- = \theta_2 D^*\quad\hbox{\rm
(coprime)}.
\end{equation}
Assume that
\medskip
\begin{itemize}
\item[(i)] $\mathcal Z(\theta_0)=\mathcal Z(\theta_1)$;
\item[(ii)] $\mathcal Z(\theta_0) \cap \mathcal Z(\theta_2)\neq \emptyset$;
\item[(iii)] $B(\gamma_0)$ and $D(\gamma_0)$ are diagonal-constant for some
$\gamma_0\in\mathcal{Z}(\theta_0)$.
\end{itemize}
\medskip
Then $\theta_0=\theta_2$ and $\theta_1=\theta_3$.
\end{theorem}

\begin{remark}
Note that $\mathcal Z(\theta_0) \ne \mathcal Z(\theta_1)$ and
$\mathcal Z(\theta_2)\ne \mathcal Z(\theta_3)$ in Example
\ref{ex5.27}.
\end{remark}

\begin{proof}[Proof of Theorem \ref{thm5.28}]
By Corollary \ref{cor5.19}, we have $\theta_0=\theta_2$. \ In view
of Corollary \ref{cor5.16}, we may write
$$
\Phi_+ = \theta_0 \theta_1 A^*, \ \ \Phi_- =\theta_0 B^*, \ \ \Psi_+
= \theta_0 \theta_3 C^*, \ \ \Psi_- = \theta_0 D^*\ \ \hbox{\rm
(coprime)} ,
$$
where the $\theta_1$ and $\theta_3$ are coprime. \  By Corollary
\ref{thm5.22}, it follows that $T_{\Psi^{1, \Omega}}$ is
pseudo-hyponormal with $\Omega:=\theta_0\theta_1\theta_3\Delta^*$,
where $\Delta:=\hbox{\rm left-g.c.d.}\, \bigl\{ I_{\theta_0}, \
(\theta_3 A- \theta_1 C\Lambda^*)\bigr\}$ (with
$\Lambda:=B(\gamma_0)D(\gamma_0)^{-1}$).  \ Thus we can write
\begin{equation}\label{999-345}
I_{\theta_0} =\Delta \Theta_0^{\prime} \quad \hbox{and} \quad
\theta_3A-\theta_1 C\Lambda^*=\Delta A_1\quad (\Theta_0^{\prime},
A_1\in H^\infty_{M_n}).
\end{equation}
Suppose $\mathcal Z(\theta_0)=\mathcal Z(\theta_1)$. \ Then we can
write
$$
\theta_0=\prod_{k=1}^{N}b_{\alpha_k}^{p_k} \quad \hbox{and} \quad
\theta_1=\prod_{k=1}^{N}b_{\alpha_k}^{q_k} \quad\left(\alpha_k \in
\mathbb D, \ p_k,  q_k \in \mathbb Z_+\right).
$$
Clearly,
$$
\delta\equiv \hbox{g.c.d.}(\theta_0,
\theta_1)=\prod_{k=1}^{N}b_{\alpha_k}^{r_k}
\quad\bigl(r_k:=\hbox{min}\, (p_k, q_k)\bigr).
$$
Thus, we may write
\begin{equation}\label{999-3451}
\theta_0=\delta \delta_0 \quad \hbox{and} \quad \theta_1=\delta
\delta_1 \quad (\delta_i\ \hbox{is a finite Blaschke product for} \
i=0,1).
\end{equation}
Since $\theta_1$ and $\theta_3$ are coprime, it follows from
(\ref{Blaco}) that $\mathcal Z(\theta_1)\cap \mathcal
Z(\theta_3)=\emptyset$, and hence $\mathcal Z(\theta_0)\cap \mathcal
Z(\theta_3)=\emptyset$, so that $\theta_0$ and $\theta_3$ are
coprime. \ It follows from (\ref{999-345}) and (\ref{999-3451}) that
$$
I_{\theta_0}= I_{\delta \delta_0} =\Delta \Theta_0^{\prime} \quad
\hbox{and} \quad \theta_3A-\delta \delta_1 C\Lambda^*=\Delta A_1.
$$
By Corollary \ref{lem2.3} and Theorem \ref{lem33.9}, we can easily
see that
\begin{equation}\label{0109}
\hbox{$I_\delta$ and $\Delta$ are coprime}.
\end{equation}
By Theorem \ref{lem33.9}, $I_{\delta^{\prime}} \equiv D(\Delta)$ and
$I_\delta$ are coprime. \ Thus, $\delta^{\prime}$ and $\delta$ are
coprime, so that $\mathcal Z(\delta^{\prime})\cap \mathcal Z(\delta)
=\emptyset$, and hence $\mathcal Z(\delta^{\prime})\cap \mathcal
Z(\theta_0) =\emptyset$. \ Since $\delta^{\prime}$ is a inner
divisor of $\theta_0$, $\mathcal Z(\delta^{\prime}) \subseteq
\mathcal Z(\theta_0)$ and hence $\mathcal
Z(\delta^{\prime})=\emptyset$. \ Thus, $\delta^{\prime}$ is a
constant, so that $\Delta$ is a constant unitary. \ We now want to
show that $\theta_3=1$. \ Assume to the contrary that $\theta_3 \neq
1$. \ Observe that
\begin{equation}\label{5.40}
H_{({\Psi^{1, \Omega}})_+^*} =H_{[P_{H_0^2 (\Psi_+ \Omega^*)}]^*}
=H_{\Omega \Psi_+^* } =H_{\theta_1 C}=0
\quad(\Omega:=\theta_0\theta_1\theta_3)\,.
\end{equation}
It thus follows that
$$
\theta_0 H_{\mathbb C_n}^2 = \hbox{ker}H_{{(\Psi^{1,\Omega})}_-^*}
\supseteq \hbox{ker}H_{{(\Psi^{1,\Omega})}_+^*}= H_{\mathbb C^n}^2,
$$
a contradiction. \  Therefore we must have $\theta_3=1$. \ This
completes the proof.  \
\end{proof}

\begin{corollary} \label{cor5.28-1}
Let $\Phi,\Psi\in L^\infty_{M_n}$ be matrix-valued trigonometric
polynomials such that the outer coefficients of $\Phi$ and $\Psi$
are invertible and the co-analytic outer coefficients of $\Phi$ and
$\Psi$ are diagonal-constant. \  If $\mathbf T:=(T_\Phi, T_\Psi)$ is
hyponormal then
$$
\hbox{\rm deg}\,(\Phi_+)=\hbox{\rm deg}\,(\Psi_+).
$$
\end{corollary}

\begin{proof}
Immediate from Theorem \ref{thm5.28}.
\end{proof}

\bigskip

If the matrix-valued rational symbols $\Phi$ and $\Psi$ have the
same co-analytic and analytic degrees we get a general necessary
condition for the hyponormality of the pair $\mathbf T:=(T_{\Phi},
T_{\Psi})$. \  This plays an important role in getting a rank
formula for the self-commutator $[\mathbf T^*,\mathbf T]$. \

\begin{theorem} \label{thm5.29}   \ Let
$\mathbf{T}\equiv (T_{\Phi}, T_{\Psi})$ be a hyponormal Toeplitz
pair with rational symbols $\Phi, \Psi \in L^\infty_{M_n}$ of the
form
\begin{equation}\label{5.35A}
\Phi_+ = \theta_0 \theta_1 A^*, \ \ \Phi_- =\theta_0 B^*, \ \ \Psi_+
= \theta_2 \theta_1 C^*, \ \ \Psi_- = \theta_2 D^*\ \ \hbox{\rm
(coprime)}.
\end{equation}
Suppose

\begin{itemize}
\item[(i)] $\mathcal Z(\theta_0) \cap \mathcal Z(\theta_2)\neq \emptyset$;
\item[(ii)] \ $B(\gamma_0)$ and $D(\gamma_0)$ are diagonal-constant.
\end{itemize}
Then $\Phi-\Lambda \Psi \in \mathcal K_{z\theta_1}$ for some
$\Lambda\in M_n$.
\end{theorem}

\begin{proof}
By Corollary \ref{cor5.19}, we have $\theta_0=\theta_2$. \ It
follows from Corollary \ref{thm5.22} that $T_{\Psi^{1,\Omega}}$ is
pseudo-hyponormal with
$$
\Omega=\theta_0\theta_1\Delta^*\quad \hbox{(with
$\Delta:=\hbox{left-g.c.d.}\,(I_{\theta_0},\, A_0- C_0
\Lambda^*)$)}\,,
$$
where $A_0 := P_{\mathcal K_{\theta_0}}A$ and $C_0 := P_{\mathcal
K_{\theta_0 }}C$. \ Then we can easily see that
\begin{equation}\label{5.41}
\Delta=I_{\theta_0}.
\end{equation}
Also since $\Delta$ is a left inner divisor of $A_0- C_0\Lambda^*$,
it follows that $ A_0- C_0\Lambda^* \in \theta_0 H_{M_n}^2. $ \ By
Lemma \ref{lem2.4}, $C_0\Lambda^* \in \mathcal K_{\theta_0}$ and
hence $A_0- C_0\Lambda^* \in \mathcal K_{\theta_0}$. \  Therefore,
$$
A_0- C_0\Lambda^* \in \theta_0 H_{M_n}^2 \bigcap \mathcal
K_{\theta_0}=\{0\},
$$
which implies $A_0=C_0\Lambda^*$. \  Put $A_1:=A-A_0$, and
$C_1:=C-C_0$. \ Then we may write
$$
\Phi_+-\Lambda \Psi_+
=\theta_0\theta_1(A_0^*+A_1^*)- \theta_0\theta_1 \Lambda(C_0^*+C_1^*)\\
=\theta_1(A_2- C_2 \Lambda^*)^*,
$$
(where $A_2:=\overline{\theta_0}A_1, C_2:=\overline{\theta_1}C_1\in
H^2_{M_n}$) which implies $\Phi_+-\Lambda \Psi_+ \in \mathcal
K_{z\theta_1}$ by Lemma \ref{lem2.4}. \
It thus follows from Corollary \ref{thm5.22} that
$$
\Phi-\Lambda \Psi =\Phi_-^*-\Lambda \Psi_-^* + \Phi_+ -\Lambda
\Psi_+ = \Phi_+ -\Lambda \Psi_+ \in \mathcal K_{z\theta_1}\,,
$$
which proves the theorem. \
\end{proof}


We now obtain the rank formula for self-commutators. \

\begin{theorem} \label{thm5.34}
Let $\mathbf T\equiv(T_\Phi, T_\Psi)$ be a hyponormal Toeplitz pair
with rational symbols $\Phi, \Psi \in L^\infty_{M_n}$ of the form
$$
\Phi_+ = \theta_0 \theta_1 A^*, \ \ \Phi_- =\theta_0 B^*, \ \ \Psi_+
= \theta_2 \theta_1 C^*, \ \ \Psi_- = \theta_2 D^*\quad\hbox{\rm
(coprime)}.
$$
If $\theta_0$ and $\theta_2$ are not coprime and $B(\gamma_0)$ and
$D(\gamma_0)$ are diagonal-constant for some
$\gamma_0\in\mathcal{Z}(\theta_0)$, then
\begin{equation}\label{5.43}
\begin{aligned}
~ [\mathbf {T}^*, \mathbf T] &=\mathcal
W\left(\bigl(I-K(M)^*K(M)\bigr)\otimes
\begin{pmatrix} 1&1\\ 1&1\end{pmatrix}\right)\mathcal W^*\bigoplus 0\\
&=\left([T_{\Phi}^*, T_{\Phi}]\otimes \begin{pmatrix} 1&1\\
1&1\end{pmatrix}\right)\bigoplus 0
\end{aligned}
\end{equation}
with
$$
\mathcal W\label{mathcalW}: = \begin{pmatrix}
(T_A)_{\theta_0\theta_1 }^* W &0\\0& (T_C)_{\theta_0\theta_1 }^*
W\end{pmatrix}\,,
$$
where $K\in\mathcal C(\Phi)$ is a polynomial {\rm (}cf. {\rm
(}\ref{2.12}{\rm )}{\rm )}, and $W$ and $M$ are given by {\rm
(\ref{2.13})} and {\rm (\ref{2.14})} with $\theta=\theta_0\theta_1$.
\  In particular,\label{rankT^*T}
\begin{equation}\label{5.44}
\hbox{\rm rank}\,[\mathbf {T}^*, \mathbf T]=\hbox{\rm
rank}\,[T_{\Phi}^*, T_{\Phi}] =\hbox{\rm
rank}\,\Bigl(I-K^*(M)K(M)\Bigr).
\end{equation}
\end{theorem}

\begin{remark}
We note that the formula (\ref{5.43}) does not say that the
hyponormality of $T_\Phi$ implies the hyponormality of the pair
$\mathbf T$. \ Indeed, (\ref{5.43}) is true for the selfcommutators
only when the pair $\mathbf T$ is already hyponormal (via Corollary
\ref{thm5.22}).
\end{remark}

\begin{proof}[Proof of Theorem \ref{thm5.34}] \ By
Corollary \ref{cor5.19}, we have $\theta_0=\theta_2$. \  Let
$\Lambda\equiv B(\gamma_0) D(\gamma_0)^{-1}$. \ Since $\mathbf
T\equiv(T_{\Phi}, T_{\Psi})$ is hyponormal, it follows from
Corollary \ref{thm5.22} that $\Phi_-=\Lambda^* \Psi_-$. \ Since
$T_{\Phi}$ and $T_{\Psi}$ are hyponormal we can find functions $K
\in \mathcal C(\Phi)$ and $K_1 \in \mathcal C(\Psi)$. \  We thus
have $H_{\Phi_-^*}=H_{K\Phi_+^*}$ and
$H_{\Psi_-^*}=H_{K_1\Psi_+^*}$. \ A straightforward calculation with
Lemma \ref{lem3.6} shows that
$$
\begin{aligned}
\left(H_{\Phi_+^*}^*H_{\Psi_+^*}- H_{\Psi_-^*}^*
H_{\Phi_-^*}\right)|_{\mathcal H(I_\theta)}
  & =(T_A)_{\theta }^* \,W\Bigl(I-(\Lambda^{-1}
K)^*(M)(\Lambda K_1)(M)\Bigr)W^*(T_C)_{\theta };\\
\left(H_{\Phi_+^*}^*H_{\Phi_+^*}-H_{\Phi_-^*}^*H_{\Phi_-^*}\right)|_{\mathcal
H(I_\theta)}
 &=(T_A)_{\theta }^* \, W\Bigl(I - K^*(M)K(M)\Bigr)W^*(T_A)_{\theta };\\
\left(H_{\Psi_+^*}^*H_{\Psi_+^*}-H_{\Psi_-^*}^*H_{\Psi_-^*}\right)|_{\mathcal
H(I_\theta)} &=(T_C)_{\theta }^* \,
W\Bigl(I-K_1^*(M)K_1(M)\Bigr)W^*(T_C)_{\theta },
\end{aligned}
$$
where $W$ and $M$ are given by {\rm (\ref{2.13})} and {\rm
(\ref{2.14})} with $\theta=\theta_0\theta_1$ and $I$ is the identity
on $\mathbb C^{n\times d}$. \ But since $\hbox{ran}\,[\mathbf T^*,
\mathbf T] \subseteq \mathcal{H}(I_\theta)\oplus
\mathcal{H}(I_\theta)$, we have
\begin{equation}\label{5.45}
\begin{aligned}
~[\mathbf {T}^*, \mathbf T] &=[\mathbf {T}^*, \mathbf
T]\,\vert_{\mathcal{H}(I_\theta)
         \oplus \mathcal{H}(I_\theta)}\bigoplus 0\\
&= \small{
\begin{pmatrix}
\left(H_{\Phi_+^*}^*H_{\Phi_+^*}-H_{\Phi_-^*}^*H_{\Phi_-^*}\right)|_{\mathcal
H(I_\theta)}&
\left(H_{\Phi_+^*}^*H_{\Psi_+^*}-H_{\Psi_-^*}^*H_{\Phi_-^*}\right)|_{\mathcal
H(I_\theta)}\\
\left(H_{\Psi_+^*}^*H_{\Phi_+^*}-H_{\Phi_-^*}^*H_{\Psi_-^*}\right)|_{\mathcal
H(I_\theta)}&
\left(H_{\Psi_+^*}^*H_{\Psi_+^*}-H_{\Psi_-^*}^*H_{\Psi_-^*}\right)|_{\mathcal
H(I_\theta)}
\end{pmatrix}
}
\bigoplus 0\\
&=\mathcal W \small{
\begin{pmatrix} I-K(M)^*K(M)&I-(\Lambda^{-1}
K)^*(M)(\Lambda K_1)(M)\\I-(\Lambda K_1)^*(M)(\Lambda ^{-1}
K)(M)&I-K_1^*(M)K_1(M)\end{pmatrix} } \mathcal W^* \bigoplus 0,
\end{aligned}
\end{equation}
where
$$
\mathcal W: = \begin{pmatrix} (T_A)_{\theta }^* \, W &0\\0&
(T_C)_{\theta }^* \, W\end{pmatrix}.
$$
On the other hand, since $\mathbf T$ is hyponormal it follows from
Theorem \ref{thm5.29} that $ \Phi-\Lambda \Psi \in \mathcal
K_{z\theta_1}$. \ Thus we can see that
$$
\mathcal C(\Phi^{1, \theta_1 })=\mathcal C(\Lambda \Psi^{1, \theta_1
})=\mathcal C((\Lambda \Psi)^{1, \theta_1 }).
$$
It thus follows from Proposition \ref{pro3.13-1} that
\begin{equation}\label{5.49}
\mathcal C(\Phi)=\mathcal C(\Lambda \Psi),
\end{equation}
which implies that
\begin{equation}\label{5.50}
\mathcal C(\Psi)=\{\Lambda^* \Lambda^{-1}K : K \in \mathcal
C(\Phi)\}.
\end{equation}
Thus we can choose $K_1\in\mathcal{C}(\Psi)$ such that
$$
K_1=\Lambda^*\Lambda^{-1}K.
$$
In what follows, we recall that $K$ is a polynomial and $K(M)$ is
defined by (\ref{2.15}). \ Therefore, it follows from Lemma
\ref{lem3.7} that
\begin{equation}\label{5.51}
\begin{aligned}
(\Lambda^{-1} K)^*(M)(\Lambda K_1)(M)
&=(\Lambda^{-1} K)^*(M)(\Lambda^* K)(M)\\
&=K^*(M)(\Lambda^{*-1} \otimes I)(\Lambda^* \otimes I)K(M)\\
&=K^*(M)K(M)
\end{aligned}
\end{equation}
and similarly,
\begin{equation}\label{5.52}
(\Lambda K_1)^*(M)(\Lambda^{-1} K)(M)=K^*(M) K(M).
\end{equation}
Also, by (\ref{5.50}) and Lemma \ref{lem3.7}, we have
\begin{equation}\label{5.53}
\aligned
K_1^*(M)K_1(M)&=(\Lambda^*\Lambda^{-1}K)^*(M)(\Lambda^*\Lambda^{-1}K)(M)\\
&=K^*(M)((\Lambda^{*-1}\Lambda) \otimes I)(\Lambda^* \Lambda^{-1}\otimes I)K(M)\\
&=K^*(M)K(M).
\endaligned
\end{equation}
Thus, by (\ref{5.45}), (\ref{5.51}), (\ref{5.52}), and (\ref{5.53}),
we have
$$
\aligned~ [\mathbf {T}^*, \mathbf T] &=\mathcal W\begin{pmatrix}
I-K(M)^*K(M)&I-K(M)^*K(M)\\I-K(M)^*K(M)&I-K(M)^*K(M)\end{pmatrix}
\mathcal W^* \bigoplus 0\\
&=\mathcal W\left(\bigl(I-K(M)^*K(M)\bigr)\otimes
\begin{pmatrix} 1&1\\ 1&1\end{pmatrix}
\right)\mathcal W\bigoplus 0\\
&=\left([T_{\Phi}^*, T_{\Phi}]\otimes \begin{pmatrix} 1&1\\
1&1\end{pmatrix}\right)\bigoplus 0\,
\endaligned
$$
which proves (\ref{5.43}). \

On the other hand, since $A(\alpha)$ is invertible for each zero
$\alpha$ of $\theta_0\theta_1$, it follows from Corollary
\ref{cor2.7} that $(T_A)_{\theta_0\theta_1 }$ and
$(T_C)_{\theta_0\theta_1 }$ are invertible. \  Therefore $\mathcal
W$ is invertible, so that the rank formula (\ref{5.44}) is obtained
from (\ref{5.43}). \
\end{proof}

On the other hand, Theorem \ref{thm5.34} can be extended to Toeplitz
$m$-tuples. \ To see this we observe:

\begin{lemma} \label{lem5.36} Let $\Phi_i \in L^\infty_{M_n}  \ (i=1,2,\hdots,m)$
and let $\sigma$ be a permutation on $\{1,2,\hdots, m\}$. \  Then $
\mathbf T:=(T_{\Phi_1}, \hdots, T_{\Phi_m})$ is hyponormal if and
only if $\mathbf T_{\sigma}\label{Tsigma}:=(T_{\Phi_{\sigma(1)}},
\hdots, T_{\Phi_{\sigma(m)}})$ is hyponormal. \  Furthermore, we
have
$$
\hbox{\rm rank}\,[\mathbf T^*, \mathbf T]=\hbox{\rm rank}\, [\mathbf
T_{\sigma}^*, \mathbf T_{\sigma}].
$$
\end{lemma}

\begin{proof} Obvious.
\end{proof}

\begin{lemma}\label{lem5.37} Let $\Phi_i \in L^\infty_{M_n} \  (i=1,2,\hdots,m)$. \  Then the
m-tuple $ \mathbf T:=(T_{\Phi_1}, \hdots, T_{\Phi_m})$ is hyponormal
if and only if every sub-tuple of $\mathbf T$ is hyponormal. \
\end{lemma}

{\it Remark.} \ A tuple $\mathbf T$ is considered to be its own
sub-tuple.

\begin{proof}[Proof of Lemma \ref{lem5.37}] \ In view of Lemma \ref{lem5.36}, it suffices to show that if $\mathbf T$
is hyponormal then for every $m_0 \leq m$, the sub-tuple $\mathbf
T_{m_0}:=(T_{\Phi_1}, \hdots, T_{\Phi_{m_0}})$ is hyponormal. \  But
this is obvious since
$$
[\mathbf T^*,  \mathbf T]=\begin{pmatrix}
[\mathbf T_{\Phi_{m_0}}^*, \mathbf T_{\Phi_{m_0}}]& \ast\\
\ast&\ast\end{pmatrix}.
$$
\end{proof}

\begin{corollary}\label{cor5.45} For each $i=1,2,\hdots,m$, suppose
$\Phi_i=(\Phi_{i})_-^*+(\Phi_{i})_+ \in L^\infty_{M_n}$ is a
matrix-valued rational function of the form
$$
(\Phi_{i})_+ = \theta_i \delta A_i^*\quad\hbox{and}\quad
(\Phi_{i})_- =\theta_i B_i^* \quad\hbox{\rm (coprime)}.
$$
Assume that there exists $j_0$ \hbox{\rm ($1\le j_0\le m$)} such
that $\theta_{j_0}$ and $\theta_{i}$ are not coprime for each
$i=1,2,\hdots,m$. \  If $B_i(\gamma_0)$ is diagonal-constant for
some $\gamma_0\in\mathcal{Z}(\theta_i)$ and for each
$i=1,2,\hdots,m$, then the following statements are equivalent:
\smallskip
\begin{enumerate}
\item[(a)] The m-tuple $\mathbf T:=(T_{\Phi_1}, T_{\Phi_2}, \cdots T_{\Phi_m})$ is hyponormal;
\item[(b)] Every subpair of $\mathbf T$ is hyponormal\,.
\end{enumerate}
\smallskip
\noindent Moreover, if $\mathbf T$ is hyponormal, then
$$
\aligned  ~ [\mathbf {T}^*, \mathbf T]_p &= \left([T_{\Phi_1}^*,
T_{\Phi_1}]_p\otimes
\left(\begin{smallmatrix}1&1&\cdots&1\\1&1&\cdots&1\\
\vdots&\vdots & & \vdots\\1&1&\cdots&1\\
\end{smallmatrix}\right)\right) \bigoplus \ 0\\
&=\mathcal W^*\left(\bigl(I-K(M)^*K(M)\bigr) \otimes \left(\begin{smallmatrix}1&1&\cdots&1\\1&1&\cdots&1\\
\vdots&\vdots & & \vdots\\1&1&\cdots&1\\
\end{smallmatrix}\right)\right)\mathcal
W\bigoplus\ 0,
\endaligned
$$
where $K \in \mathcal C(\Phi_1)$ is a polynomial \hbox{\rm (cf.
(\ref{2.12}))}, $W$ and $M$ are given by {\rm (\ref{2.13})} and {\rm
(\ref{2.14})} with $\theta=\theta_1\delta$, and $\mathcal W: =
\hbox{\rm diag}\bigl((T_{A_i})_{\theta_i\delta }^* W\bigr)$. \  In
particular,
$$
\hbox{\rm rank}\,[\mathbf {T}^*, \mathbf T]_p =\hbox{\rm rank}\,
[T_{\Phi_1}^*, T_{\Phi_1}]_p =\hbox{\rm rank}\,(I-K^*(M)K(M)).
$$
\end{corollary}

\begin{proof} (a) $\Rightarrow$ (b): Immediate from Lemma \ref{lem5.37}. \

(b) $\Rightarrow$ (a): We follow the idea in \cite [Corollary 2.11]
{CuL1}. \  Since every subpair $\mathbf T_{ij}=(T_{\Phi_i},
T_{\Phi_j})$ of $\mathbf T$ is hyponormal for each $i,j$, it follows
from Theorem \ref{thm5.34} that
$$
[\mathbf {T}_{ij}^*, \mathbf T_{ij}]_p=\Bigl([T_{\Phi_i}^*,
T_{\Phi_i}]\otimes
\begin{pmatrix} 1&1\\ 1&1\end{pmatrix}\Bigr)\bigoplus 0 \quad(1\leq i,j\leq m).
$$
Therefore
$$
\aligned~ [\mathbf T^*,  \mathbf T]_p&=\begin{pmatrix}
\hbox{$[T_{\Phi_1}^*, T_{\Phi_1}]_p$}& \hbox{$[T_{\Phi_2} ^*,
T_{\Phi_1}]_p$}
& \hdots & \hbox{$[T_{\Phi_m}^*,T_{\Phi_1}]_p$}\\
\hbox{$[T_{\Phi_1}^*, T_{\Phi_2}]_p$}& \hbox{$[T_{\Phi_2} ^*,
T_{\Phi_2}]_p$}
& \hdots &\hbox{$[T_{\Phi_m}^*,T_{\Phi_2}]_p$}\\
\vdots & \vdots & \ddots & \vdots\\
\hbox{$[T_{\Phi_1}^*,T_{\Phi_m}]_p$} &
\hbox{$[T_{\Phi_2}^*,T_{\Phi_m}]_p$} & \hdots &
\hbox{$[T_{\Phi_m}^*,T_{\Phi_m}]_p$}
\end{pmatrix}\\
&=\left([T_{\Phi_1}^*, T_{\Phi_1}]_p\otimes
\left(\begin{smallmatrix}1&1&\cdots&1\\1&1&\cdots&1\\
\vdots&\vdots & & \vdots\\1&1&\cdots&1\\
\end{smallmatrix}\right)\right) \bigoplus \ 0\\
&=\mathcal W^*\left(\bigl(I-K(M)^*K(M)\bigr) \otimes
\left(\begin{smallmatrix}1&1&\cdots&1\\1&1&\cdots&1\\
\vdots&\vdots & & \vdots\\1&1&\cdots&1\\
\end{smallmatrix}\right)\right)\mathcal
W\bigoplus\ 0,
\endaligned
$$
where $\mathcal W: = \hbox{diag}\bigl((T_{A_i})_{\theta_i\delta }^*
W\bigr)$. \ The rank formula is obvious. \  This completes the
proof. \
\end{proof}

%
%
%
%
%
%

\newpage

\chapter{Concluding remarks}

\bigskip

In this paper we have tried to answer a number of questions
involving matrix functions of bounded type; there are, however,
still many questions that we were unable to answer. \ These
questions have to do with whether properties involving matrix
rational functions can be transmitted to the case of matrix
functions of bounded type. \ Concretely, this means that if we know
property $X$ in the case when $\theta$ is a finite Blaschke product
in the decomposition (\ref{2.6}) of a matrix function $\Phi$ whose
adjoint is of bounded type, then property $X$ is still true for the
cases of any inner function $\theta$. \ Consequently, the main
problem lies on the cases where $\theta$ is a singular inner
function. \ Below we pose some questions involving matrix functions
of bounded type.


\bigskip

\noindent {\bf 1. Mutual singularity of two finite positive Borel
measures.} \ In Chapter 3, we have considered coprime singular inner
functions, to understand well functions of bounded type; this is
helpful when considering coprime factorizations for functions of
bounded type. \ Here is a relevant question: \ Let $\mu$ and
$\lambda$ be finite positive Borel measures on $\mathbb T$. \ For $x
\in \mathbb T$ and $r>0$, write
$$
B(x,r):= \Bigl\{ xe^{i\theta}: 0 < |\theta| <r\Bigr\}.
$$
For $x \in \hbox{supp}(\lambda)$, let
$$
Q_r(\mu/\lambda)(x)\label{qrmu}:=\frac{\mu(B(x,r))}{\lambda(B(x,r))}.
$$
Now we may define {\it a derivative} of $\mu$ with respect to
$\lambda$ at $x$ to be
$$
D(\mu)_\lambda (x)\label{dmulam}:=\lim_{r \to 0}Q_r(\lambda/\mu)(x)
$$
at those points $x\in \hbox{supp}(\lambda)$ at which this limit
exists. \ In this case,  we may ask:

\medskip

\begin{problem}\label{proAA} \
Let $\mu$ and $\lambda$ be finite positive Borel measures on
$\mathbb T$. \  Does it follow that $\mu \perp \lambda$ if and only
if $D(\mu)_\lambda(x)=0$ a.e.\hbox{\rm [$\lambda$]}\,?
\end{problem}
If $\lambda$ is the Lebesgue measure on $\mathbb T$, then the answer
to Problem \ref{proAA} is affirmative (cf. \cite[Theorem 7.14]{Ru}).

We also have:

\begin{problem}\label{AAA}
Let $\theta_1,\theta_2\in H^\infty$ be singular inner functions with
singular measures $\mu_1$ and $\mu_2$, respectively. \ Are the
following statements equivalent?
\begin{itemize}
\item[(a)] $\theta_1$ and $\theta_2$  are not coprime;
\item[(b)] There exists $x \in S\equiv \hbox{\rm supp}\,(\mu_1)\cap \hbox{\rm supp}\,(\mu_2)$
and $0<m <1$ such that
$$
 m \leq \left\{\frac{\mu_1(B(x,r))}{\mu_2(B(x,r))}: r \neq 0\right\} ,
$$
where $B(x,r)=\{ xe^{i\theta}: 0 < |\theta| <r\}$;
\item[(c)] $D(\mu_1)_{\mu_2}\neq 0 \ [\mu_2]$.
\end{itemize}
\end{problem}



\bigskip

\noindent {\bf 2. Coprime factorizations for compositions.} \ In
Theorem \ref{lem5.17}, we have shown that if $\Phi \in
H_{M_n}^{\infty}$ is such that $\Phi^*$ is of bounded type of the
form $\Phi=\Theta A^*$ (right coprime), then for a Blaschke factor
$\theta$, we have
\begin{equation}\label{cropB}
\Phi\circ \theta=\bigl(\Theta \circ \theta\bigr)\bigl(A \circ
\theta\bigr)^*\quad\hbox{\rm(right coprime)}.
\end{equation}
However we were unable to decide whether (\ref{cropB}) holds for any
inner function $\theta$. \

\begin{problem}\label{pro5.1}
Let $\Theta\in H^\infty_{M_n}$ be inner and $A\in H^\infty_{M_n}$. \
If $\Theta$ and $A$ are right coprime, does it follow that
$\Theta\circ \theta$ and $A\circ \theta$ are right coprime for any
inner function $\theta\in H^\infty$?
\end{problem}

\bigskip

\noindent {\bf 3. Invertibility of the compressions.} \ Let $A\in
H^2_{M_n}$ and $\Theta\equiv I_\theta\in H^\infty_{M_n}$ be a
rational function, i.e., $\theta$ is a finite Blaschke product. \ In
Corollary \ref{cor2.7}, we have shown that if $A$ and $\Theta$ are
coprime then $(T_A)_\Theta$, the compression of $T_A$ to $\mathcal
H(\Theta)$, is invertible. \ However we were unable to decide when
$(T_A)_\Theta$ is invertible.

\begin{problem}\label{pro5.2}
Give a necessary and sufficient condition for $(T_A)_\Theta$ to be
invertible if $\Theta\equiv I_\theta$ for an {\it arbitrary} inner
function $\theta$.
\end{problem}


\bigskip

\noindent {\bf 4. An interpolation problem for matrix functions of
bounded type.} \ In Chapter 6, we have considered an interpolation
problem: For a matrix function $\Phi\in L^\infty_{M_n}$, when does
there exist a bounded analytic matrix function $K\in H^\infty_{M_n}$
satisfying $\Phi-K\Phi^*\in H^\infty_{M_n}$? \ In particular, we are
interested in the cases when $\Phi$ and $\Phi^*$ are of bounded
type. \ As we have discussed in Chapter 6, $\ker
H_{\Phi_+^*}\subseteq \ker H_{\Phi_-^*}$ is a necessary condition
for the existence of a solution. \ If $\Phi$ is a matrix-valued
rational function then using the Kronecker Lemma we can show that
this condition is also sufficient for the existence of a solution
(cf. \cite[Proposition 3.9]{CHL2}). \ Moreover, in this case, the
solution $K$ is given by a polynomial via the classical
Hermite-Fej\' er Interpolation Problem. \ However we were unable to
determine whether the condition
 $\ker H_{\Phi_+^*}\subseteq \ker H_{\Phi_-^*}$
 is sufficient for the existence of a solution when
 $\Phi\in L^\infty_{M_n}$ is such that $\Phi$ and $\Phi^*$ are of bounded type. \

\begin{problem}\label{pro5.3}
Let $\Phi\in L^\infty_{M_n}$ be such that $\Phi$ and $\Phi^*$ are of
bounded type. \ If  $\ker H_{\Phi_+^*}\subseteq \ker H_{\Phi_-^*}$,
does there exist a solution $K\in H^\infty_{M_n}$ satisfying
$\Phi-K\Phi^*\in H^\infty_{M_n}$?
\end{problem}


\bigskip

\noindent {\bf 5. Square-hyponormal Toeplitz operators.} \ In
Theorem \ref{thm2.2}, we have shown that if $\Phi\in L^\infty_{M_n}$
is such that $\Phi$ and $\Phi^*$ are of bounded type, and if $\Phi$
has a tensored-scalar singularity, then the subnormality and the normality of
$T_\Phi$ coincide. \ Also in \cite[Theorem 4.5]{CHL2}, it was shown
that if $\Phi\in L^\infty_{M_n}$ is a matrix-valued rational
function whose inner part of the coprime factorization of its
co-analytic part is diagonal-constant, and  if $T_\Phi$ and
$T_\Phi^2$ are hyponormal, then $T_\Phi$ is either normal or
analytic. \ However we were unable to decide whether this result
still holds for matrix-valued bounded type symbols. \

\begin{problem}\label{pro5.4}
Let $\Phi\in L^\infty_{M_n}$ be such that $\Phi$ and $\Phi^*$ are of
bounded type of the form
$$
\Phi_-=\theta B^* \quad \hbox{\rm (coprime)},
$$
where $\theta$ is an inner function in $H^\infty$. \ If $T_\Phi$ and
$T_\Phi^2$ are hyponormal, does it follow that $T_\Phi$ is either
normal or analytic?
\end{problem}


\vskip 1cm

%
%
%
%
%
%

\bibliographystyle{amsalpha}

%
%
%
%
%
%

\chapter*{List of Symbols}




\halign{#\qquad\vrule\qquad&#\cr

\vtop{\hsize=3.5in\parindent=0pt\halign{#\hfil\hskip
0.375in&\hfil#\cr

$\mathcal{B(H,K)}$               \hskip 4cm   &\pageref{B(H,K)}\cr
$\mathcal{B(H)}$                              &\pageref{B(H)}\cr
$[A,B]$                                       &\pageref{[A,B]}\cr
$[T^*,T]$                                     &\pageref{[T^*,T]}\cr
$\hbox{cl}\,M$                                &\pageref{clM}\cr
$M^\perp$                                     &\pageref{Mperp}\cr
$\hbox{ker}\,T$                               &\pageref{ker}\cr
$\hbox{ran}\,T$                               &\pageref{ran}\cr
$\mathbb T$                                   &\pageref{T}\cr
$H^2(\mathbb T)$                              &\pageref{H^2T}\cr
$L^2(\mathbb T)$                              &\pageref{L^2T}\cr
$L^\infty$                                    &\pageref{L^infty}\cr
$H^\infty$                                    &\pageref{H^infty}\cr
$T_\varphi$                                   &\pageref{Tvarphi}\cr
$H_\varphi$                                   &\pageref{Hvarphi}\cr
$P$                                           &\pageref{P}\cr
$P^\perp$                                     &\pageref{Pperp}\cr
$J$                                           &\pageref{J}\cr $BMO$
&\pageref{BMO}\cr $L^2_{\mathcal X}$
&\pageref{L2X}\cr $H^2_{\mathcal X}$
&\pageref{H2X}\cr $M_n$
&\pageref{M_n}\cr $M_{n\times r}$
&\pageref{M_nr}\cr $L^\infty_{\mathcal X}$
&\pageref{LiX}\cr $H^\infty_{\mathcal X}$
&\pageref{HiX}\cr $\Phi_+$
&\pageref{Phi_+}\cr $\Phi_-$
&\pageref{Phi_-}\cr $\langle A,B\rangle$
&\pageref{AB}\cr $T_\Phi$
&\pageref{T_Phi}\cr $H_\Phi$
&\pageref{H_Phi}\cr $P_n$
&\pageref{P_n}\cr $\widetilde\Phi$
&\pageref{widetildePhi}\cr $\Theta$
&\pageref{Theta}\cr $||A||_\infty$
&\pageref{A_infty}\cr $||A||_2$
&\pageref{A_2}\cr }}&

\vtop{\hsize=3.5in\parindent=0pt\halign{#\hfil\hskip
0.375in&\hfil#\cr

$\hbox{tr}\,(\cdot)$       \hskip 4cm         &\pageref{tr}\cr $S$
&\pageref{S}\cr $[\mathbf T^*,\mathbf T]$
&\pageref{bfT^*bfT}\cr $\mathcal{E}(\varphi)$
&\pageref{Ephi}\cr $\mathcal{E}(\Phi)$
&\pageref{Ebigphi}\cr $I_\theta$
&\pageref{Itheta}\cr $H_0^2$
&\pageref{h20}\cr $\mathcal{H}(\Theta)$
&\pageref{HTheta}\cr $\mathcal{H}_{\Theta}$
&\pageref{H_Theta}\cr $\mathcal{K}_{\Theta}$
&\pageref{K_Theta}\cr $\Phi_{\Delta_1,\Delta_2}$
&\pageref{PhiLow}\cr $\Phi^{\Delta_1,\Delta_2}$
&\pageref{Phiupper}\cr $\Phi_{\Delta}$
&\pageref{Phi_delta}\cr $\Phi^{\Delta}$
&\pageref{Phi^delta}\cr $U_\theta$
&\pageref{U_theta}\cr $\mathcal{Z}(\theta)$
&\pageref{ztheta}\cr $b_\lambda(z)$
&\pageref{blambdaz}\cr $\hbox{left-g.c.d.}$
&\pageref{lgcd}\cr $\hbox{left-l.c.m.}$
&\pageref{llcm}\cr $\hbox{right-g.c.d.}$
&\pageref{rgcd}\cr $\hbox{right-l.c.m.}$
&\pageref{rlcm}\cr $\Theta_d$
&\pageref{thetad}\cr $\Theta_m$
&\pageref{thetam}\cr $\hbox{g.c.d.}$
&\pageref{GCD}\cr $\hbox{l.c.m.}$
&\pageref{LCM}\cr $P_{\mathcal{H}_{\Theta}}$
&\pageref{PHtheta}\cr $P_{\mathcal{K}_{\Theta}}$
&\pageref{PKTheta}\cr $D$
&\pageref{Dz}\cr $\mathcal{P}_E$
&\pageref{pe}\cr $D(\Delta)$
&\pageref{ddelta}\cr $G_\mu$
&\pageref{gmu}\cr $\mathcal{R}(f;z_0)$
&\pageref{rfz0}\cr $\hbox{deg}\,(\Phi)$
&\pageref{degphi}\cr $\hbox{det}\,\Theta$
&\pageref{dettheta}\cr $\overline{(BMO)_{M_n}}$
&\pageref{bmomn}\cr


\hbox{\phantom{$M$, $M(n)$, $M(\gamma)$,
$M(n)(\gamma)$}}&\hbox{\phantom{03}}\cr}}\cr}


\newpage

\hfil\break

\vskip 2cm


\halign{#\qquad\vrule\qquad&#\cr

\vtop{\hsize=3.5in\parindent=0pt\halign{#\hfil\hskip
0.375in&\hfil#\cr

$(T_A)_\Theta$               \hskip 4cm       &\pageref{TAT}\cr $W$
&\pageref{W}\cr $M$
&\pageref{M}\cr $(T_P)_\Theta$
&\pageref{TPTheta}\cr $\mathcal{C}(\Phi)$
&\pageref{cphi}\cr $d_j$
&\pageref{dj}\cr $\delta_j$
&\pageref{deltaj}\cr $\mu_B$
&\pageref{mub}\cr $V_B$
&\pageref{VB}\cr $M_B$
&\pageref{MB}\cr $J_B$
&\pageref{JB}\cr $s_\lambda(\zeta)$
&\pageref{slambdazeta}\cr $\mu_s$
&\pageref{mus}\cr $\mu_\Delta$
&\pageref{mudelta}\cr $J_s$
&\pageref{Js}\cr $V_S$
&\pageref{Vs}\cr }}\\&

\vtop{\hsize=3.5in\parindent=0pt\halign{#\hfil\hskip
0.375in&\hfil#\cr

$\Delta_\lambda$        \hskip 4cm
&\pageref{deltalambda}\cr $V_\Delta$
&\pageref{Vdelta}\cr $J_\Delta$
&\pageref{Jdelta}\cr $\mathcal{V}$
&\pageref{12.288}\cr $(T_{Q_\pm})_\Theta$
&\pageref{TQT}\cr $Q(M)$
&\pageref{QM}\cr $\Phi^0_+$
&\pageref{phi+0}\cr $[T_\Phi, T_\Psi]_p$
&\pageref{pseudo}\cr
$S^\sharp$                                    &\pageref{Ssharp}\cr
$[\mathbf T^*,\mathbf{T}]_p$                  &\pageref{T^*TP}\cr
$\cong$                                       &\pageref{cong}\cr
$\mathcal{W}$                                 &\pageref{mathcalW}\cr
$\hbox{rank}\,[\mathbf T^*,\mathbf T]$        &\pageref{rankT^*T}\cr
$\mathbf T_\sigma$                            &\pageref{Tsigma}\cr
$Q_r(\mu/\lambda)$                            &\pageref{qrmu}\cr
$D(\mu)_\lambda$                              &\pageref{dmulam}\cr


\hbox{\phantom{$M$, $M(n)$, $M(\gamma)$,
$M(n)(\gamma)$}}&\hbox{\phantom{03}}\cr}}\cr}


\end{document}